\documentclass[reqno,12pt, letterpaper]{amsart}
\usepackage{amsmath,amssymb,amsthm,graphicx,mathrsfs,url}
\usepackage[usenames,dvipsnames]{color}
\usepackage[colorlinks=true,linkcolor=Red,citecolor=Green]{hyperref}
\usepackage{amsxtra}
\usepackage{graphicx}
\usepackage{subcaption}
\usepackage{tikz}

\definecolor{myorange}{RGB}{180,90,0}
\usepackage{hyperref}
\hypersetup{pdfborder=0 0 0,
	    colorlinks=true,
	    citecolor=blue,
	    linkcolor=blue,
	    urlcolor=blue,
	    pdfauthor={Michael Hitrik, Johannes Sjoestrand, Martin Vogel}
}

\usepackage{graphicx,color}
\usepackage[font=small,labelfont=bf]{caption}

\usepackage{array}
\newcolumntype{P}[1]{>{\centering\arraybackslash}m{#1}}

\def\wrtext#1{\relax\ifmmode{\leavevmode\hbox{#1}}\else{#1}\fi}

\def\abs#1{\left|#1\right|}
\def\begeq{\begin{equation}}
\def\endeq{\end{equation}}

\def\?[#1]{\textbf{[#1]}\marginpar{\Large{\textbf{??}}}}

\let\epsilon=\varepsilon 

\newcommand{\RR}{{\mathbb R}}
\newcommand{\NN}{{\mathbb N}}
\newcommand{\CC}{{\mathbb C}}

\newcommand{\ZZ}{{\mathbb Z}}
\newcommand{\e}{\mathrm{e}}
\newcommand{\mO}{\mathcal{O}}

\newcommand{\dist}{\mathrm{dist}}
\def\norm#1{||\,#1\,||}

\newtheorem{theo}{Theorem} 
\newtheorem{prop}{Proposition}[section]	
\newtheorem{defi}[prop]{Definition}

\newtheorem{lemm}[prop]{Lemma}

\newtheorem{rem}[prop]{Remark}
\numberwithin{equation}{section}

\let\Im=\Imag

\let\Re=\Real

\DeclareMathOperator{\supp}{supp}
\DeclareMathOperator{\vol}{vol}

\DeclareMathOperator{\tr}{tr}
\DeclareMathOperator{\neigh}{neigh}

\usepackage{scalerel}

\newcommand\reallywidehat[1]{\arraycolsep=0pt\relax%
\begin{array}{c}
\stretchto{
  \scaleto{
    \scalerel*[\widthof{\ensuremath{#1}}]{\kern-.5pt\bigwedge\kern-.5pt}
    {\rule[-\textheight/2]{1ex}{\textheight}} 
  }{\textheight} %
}{0.5ex}\\           
#1\\                 
\rule{-1ex}{0ex}
\end{array}
}

\def\blue#1{\textcolor{blue}{#1}}

\begin{document}

\title[Singular values]{Weyl laws for exponentially small singular values of the $\overline{\partial}$ operator}

\author{Michael Hitrik}
\address{Department of Mathematics, University of California,
Los Angeles, CA 90095, USA.}
\email{hitrik@math.ucla.edu}
\author{Johannes Sj\"ostrand}
\address{Universit\'e Bourgogne Europe, CNRS, IMB UMR 5584, 21000 Dijon, France.}
\email{jsjostra@ube.fr}

\author{Martin Vogel}
\address{Institut de Recherche Math\'ematique Avanc\'ee - UMR 7501, CNRS et
Universit\'e de Strasbourg, 7 rue Ren\'e-Descartes, 67084 Strasbourg Cedex, France.}
\email{vogel@math.unistra.fr}

\begin{abstract}
We study the number of exponentially small singular values of the semiclassical $\overline{\partial}$ operator on exponentially weighted $L^2$ spaces on a compact Riemann surface. Accurate upper and lower bounds on the number of such singular values are established in terms of auxiliary notions of upper and lower bound weights. Assuming that the Laplacian of the exponential weight changes sign along a curve, we construct optimal such weights by solving a free boundary problem, which yields Weyl asymptotics for the counting 
function of the singular values in an interval of the form $[0,\mathrm{e}^{-\tau/h}]$, for $\tau>0$ smaller than the oscillation of the 
weight. We also provide a precise description of the leading term in the Weyl asymptotics, in the regime of small $\tau > 0$.
\end{abstract}

\maketitle

\setcounter{tocdepth}{1}
\tableofcontents

\section{Introduction and statement of results}
\label{sec_intro}

\medskip
\noindent
The present paper is a contribution to the spectral theory of non-self-adjoint operators in the semiclassical limit. An essential feature making the study of non-self-adjoint operators particularly challenging is the potential instability of their spectra. Thus, while the operator norm of the resolvent of a self-adjoint operator $P$ acting on a complex Hilbert space, satisfies
\begin{equation*}
\|(P-z)^{-1}\| = \frac{1}{\dist(z, {\rm Spec}(P))}, \quad z\in \rho(P)\subseteq \CC,
\end{equation*}
in contrast, the resolvent norm can be much larger than $(\dist(z, {\rm Spec}(P)))^{-1}$, when the operator is non-self-adjoint. Here $\rho(P)$ denotes the resolvent set of $P$. This implies that the spectrum of $P$ can be highly sensitive to small perturbations, complicating considerably the numerical computation of eigenvalues and illustrating their fragility "deep in the complex domain", see~\cite{TrEm05}. The phenomenon of spectral instability, and the closely related notion of pseudospectrum, are also at the heart of the fascinating recent developments concerning the distribution of eigenvalues for non-self-adjoint operators subject to random perturbations and the associated probabilistic Weyl laws. See~\cite{Sj19},~\cite{BOV24}, \cite{NoVo21}, \cite{Ol23}, \cite{Vo20}.

\medskip
\noindent
Turning the attention to the realm of semiclassical analysis, let us recall a basic mecha\-nism for the spectral instability of large classes of non-normal operators in the semiclassical limit. Let
\begin{equation*}
P = P(x,hD_x;h) = \sum_{|\alpha|\leq m} a_\alpha(x;h)(hD_x)^\alpha, \quad 0<h \leq 1, \quad \alpha \in \NN^d,
\end{equation*}
with $(hD_x)^\alpha=(hD_{x_1})^{\alpha_1}\cdots(hD_{x_d})^{\alpha_d}$, be a semiclassical differential operator on a smooth $d$-dimensional compact manifold $M$, with $a_{\alpha}(x;h) = a_{\alpha}(x;0) + \mathcal O(h)$ in $C^{\infty}(M)$, so that the semiclassical principal symbol of $P$ is given by
\begin{equation*}
p(x,\xi) = \sum_{|\alpha|\leq m} a_\alpha(x;0)\xi^\alpha \in C^\infty(T^*M).
\end{equation*}
Using a complex WKB construction, it was shown in \cite{Horm60a,Horm60b, Da99a, Da99b, Zw01,DSZ04}
that if $\CC \ni z = p(x_0,\xi_0)$ for some $ (x_0,\xi_0)\in T^*M$ and
\begin{equation}
\label{eq:0}
\frac{1}{i}\{p,\overline{p}\}(x_0,\xi_0) >0,
\end{equation}
where $\{a,b\} = \partial_\xi a \cdot \partial_x b - \partial_x a \cdot \partial_\xi b$ denotes the Poisson bracket of two functions $a,b\in C^1(T^*M)$, then
\begin{equation}
\label{eqWF}
\exists u_h \in C^\infty_0(M),~ \|u_h\|_{L^2(M)}=1,~\mathrm{WF}_h(u_h)=\{(x_0,\xi_0)\},
\end{equation}
such that
\begeq
\label{eqQM}
\|(P-z)u_h\|_{L^2(M)} = \mathcal{O}(h^\infty).
\endeq
Here $\mathrm{WF}_h(u_h)$ stands for the semiclassical wave front set of $u_h$, see~\cite[Chapter 2]{Ma}. When $M$ is real analytic and $P$ has analytic coefficients, then (\ref{eqQM}) can be improved to
\begin{equation}
\label{eqExpdecay}
\|(P-z)u_h\| \leq C\, \mathrm{e}^{-1/Ch},
\end{equation}
for some fixed $C>0$, see~\cite{Sj82,DSZ04}, and we still have (\ref{eqWF}), with $\mathrm{WF}_h(u_h)$ replaced by the analytic semiclassical wave front set of $u_h$, see~\cite[Chapter 3]{Ma}. Further refinements on such quasimode constructions under iterated Poisson bracket conditions have been obtained in \cite{Pr04,Pr06,Pr08}.

\medskip
\noindent
Keeping the analyticity assumptions and letting $P$ stand also for the corresponding closed densely defined operator on $L^2(M)$, equipped with the maximal domain $\mathcal D(P) = \{u\in L^2(M); Pu \in L^2(M)\}$, we get in view of (\ref{eqExpdecay}),
\begin{equation*}
\| (P-z)^{-1}\|_{\mathcal L(L^2(M),L^2(M))} \geq C^{-1} e^{1/Ch}, 
\quad z\in \rho(P).
\end{equation*}
Taking $z=0$ for simplicity and assuming that $P$ is elliptic in the classical sense, we conclude also that $P$ has singular values, i.e. eigenva\-lues of the self-adjoint operator $(P^* P)^{1/2}$, that are $\mathcal O(1)\, e^{-1/Ch}$. Furthermore, under suitable assumptions, the norm of $P^{-1}$ is given by the reciprocal of the smallest singular value of $P$. Following the classical results on microlocal analytic hypoellipticity for analytic operators of principal type~\cite{SKK73},~\cite{Him86}, we further expect the region of the energy surface $p^{-1}(0) \subseteq T^*M$ where
\[
\frac{1}{i}\{p,\overline{p}\} < 0
\]
to be "classically forbidden" for the singular states of $P$ associated to exponentially small singular values. Assuming that
\[
\{p,\{p,\overline{p}\}\} \neq 0 \,\,\, \text{along}\,\,\, p = 0, \,\, \{p,\overline{p}\} = 0,
\]
we also expect the singular states to be concentrated away from the region where the bracket vanishes, thanks to the recent results \cite{HiZw25},~\cite{Sj23}. See also~\cite{Tr84}, \cite{Him86}. Speaking heuristically, the study of exponentially small singular values for the non-self-adjoint operator $P$ can be viewed therefore as a quantum tunneling problem, connecting points $\rho_0 \in p^{-1}(0)$ in the classically allowed region where
\[
\frac{1}{i}\{p,\overline{p}\}(\rho_0) > 0,
\]
via a path in the complexification of $T^*M$, to points $\rho_1 \in p^{-1}(0)$ in the classically forbidden region where
\[
\frac{1}{i}\{p,\overline{p}\}(\rho_1) < 0.
\]
We shall return to these heuristic considerations at the end of the introduction, when discussing an outline of the paper.

\medskip
\noindent
In this work we shall be interested in exponentially small singular values for a model non-self-adjoint operator. Specifically, our purpose here is to obtain precise asymptotics for the number of exponentially small singular values of the semiclassical $\overline{\partial}$ operator on an exponentially weighted space of the form $e^{\varphi/h}L^2(M)$, where $M$ is a compact Riemann surface. We shall next describe the assumptions, state the results, and outline some aspects of the proofs.

\subsection{The setting}
\label{sec:Setting}
We shall work on a compact Riemann surface $M$, i.e. a compact connected complex manifold of complex dimension one without boundary. We shall assume that $M$ is equipped with a conformal Riemannian metric $g$, i.e. a Riemannian metric compatible with the complex structure on $M$. Thus, if $(U,z)$ is a local holomorphic coordinate chart on $M$ with $z = x +iy$, we have 
\begeq
\label{eq_metric0} 
g = g_U(z)(dx\otimes dx + dy\otimes dy), \quad 0 < g_U \in C^{\infty}(U).
\endeq
Let $TM$ be the tangent bundle of $M$ in the sense of real smooth manifolds. Writing $2 dx = dz + d\overline{z}$, $2i dy = dz - d\overline{z}$, we obtain the following expression for the Riemannian metric $g$ as a symmetric positive definite bilinear form in the fibers of $TM$, 
\begeq
\label{eq_metric1} 
g = \frac{1}{2} g_U(z)\left(dz \otimes d\overline{z} + d\overline{z} \otimes dz\right).
\endeq
Associated to the metric $g$ is the area (K\"ahler) $(1,1)$-form $\omega$, which in the coordinate chart $U$ is given by 
\begin{equation}
\label{eq_metric}
\omega = g_U(z) \frac{d\overline{z}\wedge dz}{2i}.
\end{equation}
Introducing the canonical almost complex structure $J: TM \rightarrow TM$ on $M$, locally given by  
\begeq
\label{eq_J}
J(\partial_x) = \partial_y,\ \quad J(\partial_y) = -\partial_x,
\endeq
we may write 
\begeq
\label{eq_Kahler}
\omega(t,s) = g(Jt,s), \quad t,s\in T_z M.
\endeq

\medskip
\noindent 
For $\varphi \in C^{\infty}(M;\mathbb R)$ non-constant let
\begin{equation}\label{eq:dbar1}
h\overline{\partial}:L^2_\varphi (M, \omega) \to L^2_\varphi (M, T^*_{0,1}M), \quad h\in(0,1].
\end{equation}
The definition of these exponentially weighted $L^2$ spaces and the associated scalar products is recalled in Section \ref{Sec:MetIP} below. Moreover, we define the operator $P = P_{\varphi}$ by 
\begin{equation}\label{eq:DefOp}
P:=(h\overline{\partial})_\varphi := {e}^{-\varphi/h}\circ h\overline{\partial}\circ {e}^{\varphi/h}.
\end{equation}
We shall view $P$ as a closed densely defined operator: $L^2 (M, \omega) \to L^2(M, T^*_{0,1}M)$, equipped with the maximal domain 
\begin{equation*}
\mathcal{D}(P) = \{u\in L^2(M,\omega); Pu \in L^2(M, T^*_{0,1}M)\}. 
\end{equation*}
We have $\mathcal D(P) = H^1(M)$, the standard Sobolev space on $M$, in view of the classical ellipticity of $h\overline{\partial}$. Our main problem is to study the distribution of exponentially small singular values of $P$. Let therefore $P^*$ denote the adjoint of $P$ and note that in view of the classical ellipticity, the domain of $P^*$ is given by $\mathcal{D}(P^*)=H^1_{0,1}(M)$, i.e. the space of $(0,1)$-forms whose coefficients in every local coordinate chart are in $H^1$. We shall view $P^*P$ as a closed densely defined operator: $L^2(M,\omega)\to L^2(M,\omega)$ equipped with the domain
\begin{equation*}
\mathcal{D}(P^*P) = \{u\in\mathcal{D}(P); Pu\in \mathcal{D}(P^*) \} = H^2(M),
\end{equation*}
where the last equality again follows by the ellipticity. We also note that $P^*P$ is self-adjoint on this domain, since it is an elliptic second order formally self-adjoint operator. The resolvent of $P^*P$ is compact, in view of the compactness of the embedding $H^2(M) \hookrightarrow L^2(M,\omega)$~\cite[Theorem 2.34]{Au82}, which in turn implies that the spectrum of the self-adjoint non-negative operator $P^*P$ is purely discrete. The eigenvalues 
of $(P^*P)^{1/2}$,
\begin{equation}
\label{eq:sgvals}
0 = t_1 \leq t_2 \leq \dots \nearrow +\infty
\end{equation}
are the singular values of $P$. 

\medskip
\noindent
Associated to the conformal metric $g$ in (\ref{eq_metric0}) is the (negative) Laplacian $\Delta = \Delta_g$ acting on functions. We have in each local holomorphic coordinate chart $(U,z)$ on $M$, 
\begeq
\label{eq_Lapl}
\Delta_g = \frac{4}{g_U(z)} \partial_z \,\partial_{\overline{z}},
\endeq
where the definition of the standard operators $\partial_z$, $\partial_{\overline{z}}$ is recalled in (\ref{eq1.4.0.1}) below. Using also (\ref{eq_metric}), we obtain the following equality of (1,1)-forms, 
\begin{equation}
\label{eq_Laplace}
(\Delta_g u)\, \omega = 2i \partial \overline{\partial}\, u, \quad u \in C^2(M). 
\end{equation}
Here the right hand side is independent of the choice of a conformal metric on $M$.

\bigskip
\noindent
The starting point for this work is the following result. 
\begin{theo}
\label{decay_rate_large} 
Let $\tau > 0$ be such that
\begeq
\label{eq_large} 
\tau > \underset{M}{\rm max}\, \varphi - \underset{M}{\rm min}\, \varphi.
\endeq
For all $h>0$ small enough, the operator $P: L^2(M,\omega) \rightarrow L^2(M,T^*_{0,1}M)$, or equivalently, $h\overline{\partial}: L^2_{\varphi}(M,\omega) \rightarrow L^2_{\varphi}(M,T^*_{0,1}M)$, has precisely one singular value  in the interval $[0, e^{-\tau/h}]$, namely $0$. Here $\omega$ is the area form associated to the metric $g$. 
\end{theo}

\medskip
\noindent
In view of Theorem \ref{decay_rate_large}, established in Appendix \ref{sec:app_large_decay}, in the main part of the paper we shall be concerned with counting singular values of $P$ in an interval of the form $[0, e^{-\tau/h}]$, with the exponential decay rate $\tau$ satisfying the condition $0 < \tau < \underset{M}{\rm max}\, \varphi - \underset{M}{\rm min}\, \varphi$.

\subsection{Upper and lower bounds on the number of exponentially small singular values.}
Let $0 < \tau < \underset{M}{\rm max}\, \varphi - \underset{M}{\rm min}\, \varphi$ be fixed. Given a function $\psi \in C(M;\mathbb R)$ such that
\begin{equation}
\label{eq2.7_n}
\varphi - \tau \leq \psi \leq \varphi \quad \wrtext{on}\,\,\, M,
\end{equation}
the compact sets
\begin{equation}\label{eq2.7_b}
M_+(\psi)=\{x\in M;\psi(x) = \varphi(x)\} \quad \text{and} \quad M_-(\psi)=\{x\in M;\psi(x) = \varphi(x)-\tau\}
\end{equation}
will be referred to as \emph{the contact sets} of $\psi$.

\medskip
\noindent
Let $\Omega \subseteq M$ be open. We recall that a function $u\in C(\Omega;\mathbb R)$ is subharmonic in $\Omega$ if for each local holomorphic coordinate chart $z: \mathbb C \supseteq V \rightarrow U \subseteq \Omega$, the function $u = u(z)$ is subharmonic on the open set $V \subseteq \mathbb C$ in the usual sense of complex analysis: for each $K \subseteq V$ compact and each $h\in C(K;\mathbb R)$ harmonic in the interior of $K$ such that $u \leq h$ on $\partial K$, we have $u\leq h$ in $K$. We say that $u\in C(\Omega;\mathbb R)$ is superharmonic in $\Omega$ if $-u$ is subharmonic. 

\medskip
\noindent
When deriving upper and lower bounds on the number of singular values of the operator $P$ in the interval $[0, e^{-\tau/h}]$, it will be convenient to work with the following notions of auxiliary weight functions.

\medskip
\noindent
\begin{defi}\label{def_ubw_n}
We say that $\psi \in C(M;\mathbb R)$ is an upper bound weight if $\psi$ satisfies {\rm (\ref{eq2.7_n})} and $\psi$ is superharmonic on the open set $\{z\in M;\, \psi(z) < \varphi(z)\}$.
\end{defi}

\medskip
\noindent
\begin{defi}
\label{def_lbw_n}
We say that $\psi \in C(M;\mathbb R)$ is a lower bound weight if $\psi$ satisfies {\rm (\ref{eq2.7_n})} and $\psi$ is subharmonic on the open set
$\{z\in M;\, \varphi(z) - \tau < \psi(z)\}$.
\end{defi}

\medskip
\noindent
The following are the first two main results of this work.
\begin{theo}\label{thm:UpperBd}
Let $\psi\in C(M;\mathbb R)$ be an upper bound weight, in the sense of Definition {\rm \ref{def_ubw_n}}. Assume that the contact set $M_+(\psi)$ is contained in $M_+ = \{z\in M; \Delta \varphi(z) > 0\}$ and that ${\rm int}(M_+(\psi)) \neq \emptyset$. Then the number $N([0,e^{-\tau/h}]$ of singular values of $P$ in the interval $[0,e^{-\tau/h}]$ satisfies
\begin{equation*}
N([0,e^{-\tau/h}]) \leq \frac{1}{2\pi h}\int_{M_+(\psi)} \Delta \varphi(z)\, \omega(z,dz d\overline{z}) + \frac{o(1)}{h}, \quad h\to 0^+.
\end{equation*}
\end{theo}

\medskip
\noindent
{\it Remark}. The assumption $M_+(\psi) \subseteq M_+$ implies that the upper bound weight $\psi$ in Theorem \ref{thm:UpperBd} is strictly subharmonic in ${\rm int}(M_+(\psi))$.

\begin{theo}\label{thm:LowerBd}
Let $\psi\in C(M;\mathbb R)$ be a lower bound weight, in the sense of Definition {\rm \ref{def_lbw_n}}, and let $0\leq \omega(\delta)$
be the modulus of continuity of $\psi$. Assume that $M_+(\psi) \subseteq M_+ = \{z\in M; \Delta \varphi(z) > 0\}$ and that ${\rm int}(M_\pm(\psi)) \neq \emptyset$. Let $\widetilde{\tau} = \tau - \omega(h)$. Then the number $N([0,e^{-\widetilde{\tau}/h}])$ of singular values of $P$ in the interval $[0,e^{-\widetilde{\tau}/h}]$ satisfies
\begin{equation*}
N([0,e^{-\widetilde{\tau}/h}]) \geq \frac{1}{2\pi h}\int_{{\rm int}(M_+(\psi))} \Delta \varphi(z)\, \omega(z,dz d\overline{z}) - \frac{o(1)}{h}, \quad h\to 0^+.
\end{equation*}
When $\psi$ is Lipschitz continuous, the lower bound above holds also for the number $N([0,e^{-\tau/h}])$ of singular values of $P$ in the interval $[0,e^{-\tau/h}]$.
\end{theo}

\medskip
\noindent
The value of such general results as Theorem \ref{thm:UpperBd} and Theorem \ref{thm:LowerBd} depends very much on our ability to construct and analyze, for a given weight function $\varphi$, upper and lower bound weights satisfying also the additional assumptions above. In order to do so, we shall now introduce some additional hypotheses on $\varphi$.

\subsection{Optimal weights and Weyl asymptotics}
Under additional assumptions on the weight $\varphi$ in (\ref{eq:DefOp}), we shall show the existence of an \emph{optimal weight} $\psi \in C^{1,1}(M;\mathbb R)$, which is both an upper and a lower bound weight, in the sense of Definition \ref{def_ubw_n} and Definition \ref{def_lbw_n}. When doing so, we let $\tau >0$ be fixed such that
\begin{equation}
\label{eq:opw2_2a_0}
0 < \tau < \max_M \varphi - \min_M \varphi.
\end{equation}
Moreover, suppose that
\begin{equation}\label{eq:cs4_2a}
d \Delta \varphi \neq 0\,\,\, \text{along}\,\,\,(\Delta \varphi)^{-1}(0).
\end{equation}
Since an optimal weight $\psi$ is both an upper and a lower bound weight we see that it is given by the solution of the following double obstacle
problem,
\begin{equation}\label{eq:din1_2a}
\begin{cases}
\Delta \psi \geq 0 \quad \text{on } \{ \psi  > \varphi-\tau\}, \\
\Delta \psi  \leq 0 \quad \text{on } \{ \psi <\varphi\}, \\
\varphi-\tau \leq \psi \leq \varphi \text{ a.e on } M.
\end{cases}
\end{equation}
Recall the contact sets $M_{\pm}(\psi)$ defined in \eqref{eq2.7_b}. Their boundaries will be called \emph{free boundaries}.
\begin{theo}\label{thm:DPPex2}
Suppose that $\varphi\in C^{\infty}(M;\RR)$ is non-constant and that \eqref{eq:opw2_2a_0} and \eqref{eq:cs4_2a} hold. Then the double obstacle problem \eqref{eq:din1_2a} has a unique solution $\psi$ such that $\psi\in W^{2,p}(M)$ for all $2 < p < +\infty$. Moreover,
\begin{itemize}
\item $\psi\in C^{1,\alpha}(M)$, for every $0<\alpha<1$,
\item the contact sets $M_\pm(\psi)$ satisfy $\vol(M_\pm(\psi)) > 0$ and $M_{\pm}(\psi) \subseteq \{\pm\Delta \varphi >0\}$. We have
\begin{equation*}
\int_{M_+(\psi)}\Delta \varphi(z) \, \omega(z, dz\,d\overline{z}) =  -\int_{M_-(\psi)}\Delta \varphi(z) \, \omega(z, dz\,d\overline{z}), 
\end{equation*}
\item the free boundaries $\partial M_\pm(\psi)$ have $\vol(\partial M_\pm(\psi))= 0$.
\end{itemize}
\end{theo}

\medskip
\noindent
{\it Remark}. Theorem \ref{thm:DPPex2} is a special case of Theorems \ref{thm:DPPex} and \ref{thm:FreeBoundary}, and of Propositions \ref{prop:cs1} and \ref{prop:cs2} in Section \ref{sec:FB}, where we prove these results for general double obstacle problems on a $C^{\infty}$ compact $d$--dimensional Riemannian manifold without boundary. The $C^{1,\alpha}$ regularity of the unique solution $\psi$ to the double obstacle problem given in Theorem \ref{thm:DPPex2} is not optimal. One can prove that $\psi\in C^{1,1}(M)$ which can be obtained following standard methods as outlined in Remark \ref{rem:OptReg}. This is optimal since $\Delta \psi$ in general exhibits a jump discontinuity when passing from the contact sets $M_\pm(\psi)$ to the region $\{\varphi-\tau <\psi <\varphi\}$. However, since we do not the require optimal regularity of $\psi$ for this work we opted not to provide a detailed proof of this fact.
%

\medskip
\noindent
Combining Theorems \ref{thm:UpperBd} and \ref{thm:LowerBd} with the existence of an optimal weight provided by Theorem \ref{thm:DPPex2}
immediately yields the following result.
\begin{theo}
\label{thm:Weyl}
Let $\varphi\in C^{\infty}(M;\RR)$ be such that \eqref{eq:opw2_2a_0} and \eqref{eq:cs4_2a} hold, and let $\psi$ be the solution of 
{\rm (\ref{eq:din1_2a})} given in Theorem {\rm \ref{thm:DPPex2}}. Then the number $N([0,e^{-\tau/h}])$ of singular 
values of $P$ in the interval $[0,e^{-\tau/h}]$ satisfies
\begin{equation*}
N([0,e^{-\tau/h}]) = \frac{1}{2\pi h}\int_{M_+(\psi)} \Delta \varphi(z)\, \omega(z, dz\,d\overline{z}) + \frac{o(1)}{h}, \quad h\to 0^+.
\end{equation*}
\end{theo}

\medskip
\noindent
In the model case when $M = \mathbb C/(2\pi \mathbb Z + 2\pi i \mathbb Z)$ is the standard $2$-torus and $\varphi = \varphi(z) = \sin ({\rm Im}\, z)$, Theorem \ref{thm:Weyl} has been obtained by two of the authors in \cite{SV}, using methods based on the separation of variables.

\subsection{The case when $\tau $ is small: thin bands of separation}
In the case when the band $\{\varphi-\tau <\psi <\varphi\}$, where the optimal weight $\psi$ in (\ref{eq:din1_2a}) is harmonic, is very thin, we obtain a more precise version of Theorem \ref{thm:Weyl}. We work in the regime when $0<\tau\ll1$ is small and we keep the hypothesis (\ref{eq:cs4_2a}), assuming that $\varphi \in C^{\infty}(M;\mathbb{ R})$ is such that
\begin{equation}
\label{eq:Dweight}
d\Delta \varphi \ne 0 \,\,\, \text{along}\,\,\, \gamma = (\Delta \varphi)^{-1}(0).
\end{equation}

\begin{theo}\label{1csp4_n}
Let $\tau = \epsilon ^3\widehat{\tau } > 0$, where $\varepsilon>0$ is small and $\widehat{\tau } \asymp 1$, and let $M_+ = \{z\in M; \Delta \varphi(z) > 0\}$ be as in Theorems {\rm \ref{thm:UpperBd}}, {\rm \ref{thm:LowerBd}}. Assume that {\rm (\ref{eq:Dweight})} holds. Then the number $N([0,e^{-\tau/h}])$ of singular values of $P$ in the interval $[0,e^{-\tau/h}]$ satisfies, for every fixed $N\ge 4$,
\begin{equation}\label{3.csp4}
N([0,e^{-\tau /h}])= \frac{1}{2\pi h}\left(\int_{\Omega_+} \Delta \varphi(z)\, \omega(z, dz\,d\overline{z}) + \mathcal{ O}(\epsilon ^{(N+1)/2})+o(1) \right),\ h\to 0^+.
\end{equation}
Here, the "$\mathcal{ O}(\epsilon ^{(N+1)/2})$" is uniform in $\epsilon,h$ and the ``$o(1)$'' is uniform for $\epsilon $ varying in each interval of the form $[\beta /2,\beta]$, with $0<\beta \ll 1$. Furthermore, $\Omega_+ \subseteq M_+$ is 
such that
\begin{equation}
\label{4.csp4}
\int_{\Omega_+} \Delta \varphi(z)\, {\omega} = \int_{M_+} \Delta \varphi(z)\, {\omega} - \tau ^{2/3}\frac{1}{2}\left(\frac{3}{2} \right)^{2/3}\int_\gamma (\partial _n\Delta \varphi )^{1/3} dx +\mathcal{O}(\epsilon ^3),
\end{equation}
where $n$ is the unit normal to $\gamma$ such that $\partial_n \Delta \varphi|_{\gamma} > 0$, and $dx$ is the arc length element along $\gamma$.
\end{theo}

\medskip
\noindent
We remark that the set $\Omega_+$ in Theorem \ref{1csp4_n} is an asymptotic approximation of the contact set $M_+(\psi)$ from Theorem
\ref{thm:Weyl}, in the regime of $\tau > 0$ small.

\medskip
\noindent
{\it Remark}. It follows from Theorem \ref{1csp4_n} that, assuming that (\ref{eq:Dweight}) holds, the number $N([0,e^{-\tau/h}])$ of singular values of $P$ in the interval $[0,e^{-\tau/h}]$ satisfies 
\[ 
\lim_{\tau \rightarrow 0^+}\, \lim_{h \rightarrow 0^+} h\, N([0,e^{-\tau/h}]) = \frac{1}{2\pi} \int_{M_+} \Delta \varphi(z)\, \omega(z,dz\,d\overline{z}).
\]
Here $M_+ = \{z\in M; \Delta \varphi(z) > 0\}$. 
\begin{figure}[ht]
\centering
\includegraphics[width=0.9\textwidth]{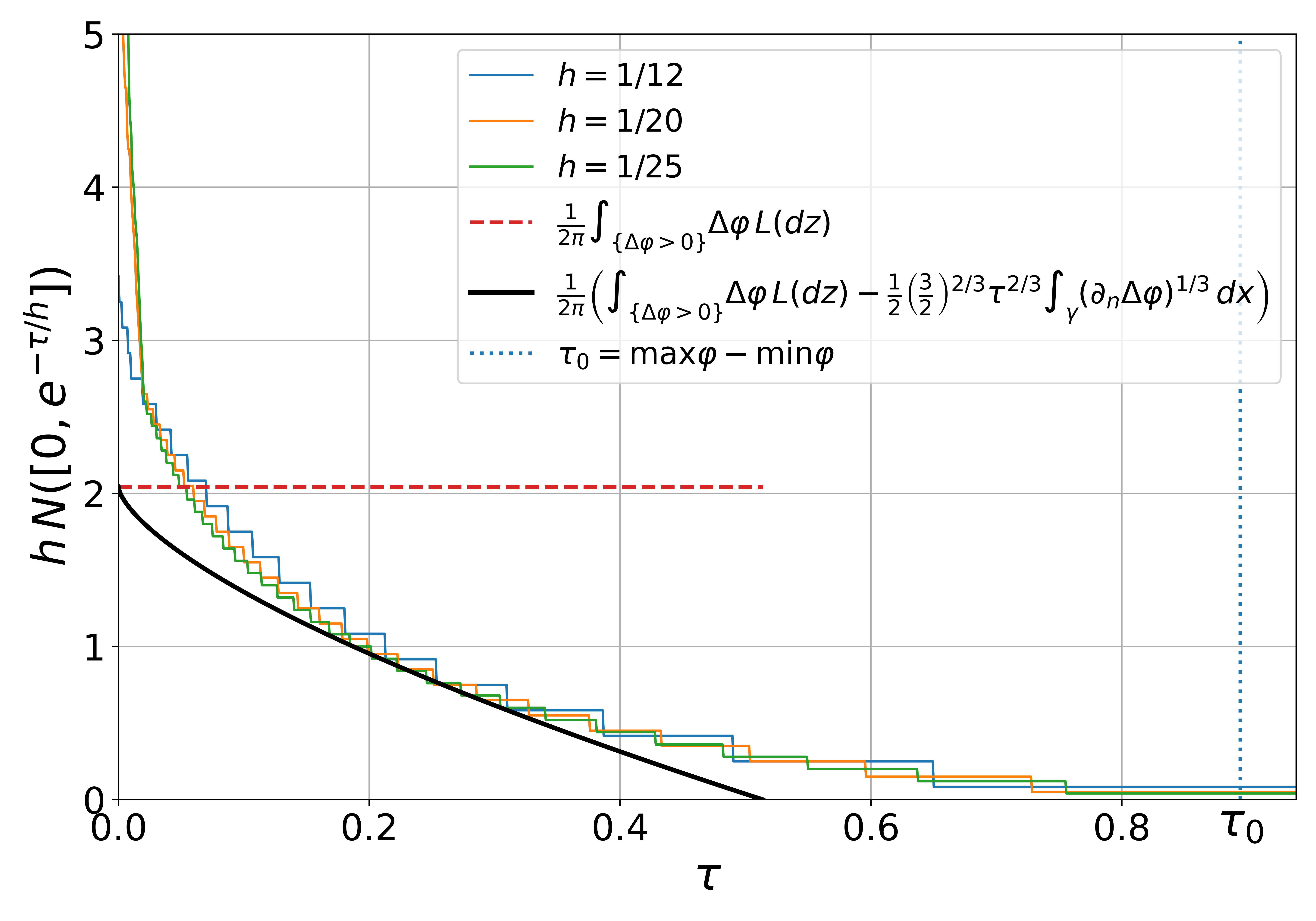}
\caption{The numerically obtained number of exponentially small singular values $N([0,\e^{-\tau/h}])$, normalized by a factor of $h$, for a given $\tau$ in the case when $\varphi(z)=0.2\left(\cos(\Im z)-\cos(\Re z)+\cos(\Re z+2\Im z)\right)$ for $z\in M = \mathbb C/(2\pi \mathbb Z + 2\pi i \mathbb Z)$. The green, orange and blue line correspond to different values of $h$. The thick black line shows the leading term approximation from Theorem \ref{1csp4_n}. For $\tau > \tau_0$ there is only one singular value in $[0,\e^{-\tau/h}]$ as proven in Theorem \ref{decay_rate_large}.}
\label{figA}
\end{figure}

\medskip
\noindent
We illustrate Theorem \ref{1csp4_n} with a numerical simulation in Figure \ref{figA}. Recall that Theorem \ref{1csp4_n} is valid only for small and $h$-independent $\tau$. In Figure \ref{figA}, we see that for $\tau$ large the approximation of $N([0,e^{-\tau /h}])$ by the two leading terms on the right hand side of \eqref{4.csp4} ceases to hold. A discrepancy appears when $\tau$ becomes too small as a function of $h$. In this regime Theorem \ref{1csp4_n} is no longer valid and we expect to see a transition towards the standard semiclassical Weyl asymptotics whose onset one can see in Figure \ref{figA} for values of $\tau$ close to $0$ (this corresponds to the stark growth of the singular value counting function seen in Figure \ref{figB} for values of $\lambda>0.6$). It would be a very interesting problem to provide an asymptotic formula for $N([0,e^{-\tau /h}])$ in the regime when $\tau = \tau(h)=o_{h\to 0}(1)$. 
\begin{figure}[ht]
\centering
\includegraphics[width=0.8\textwidth]{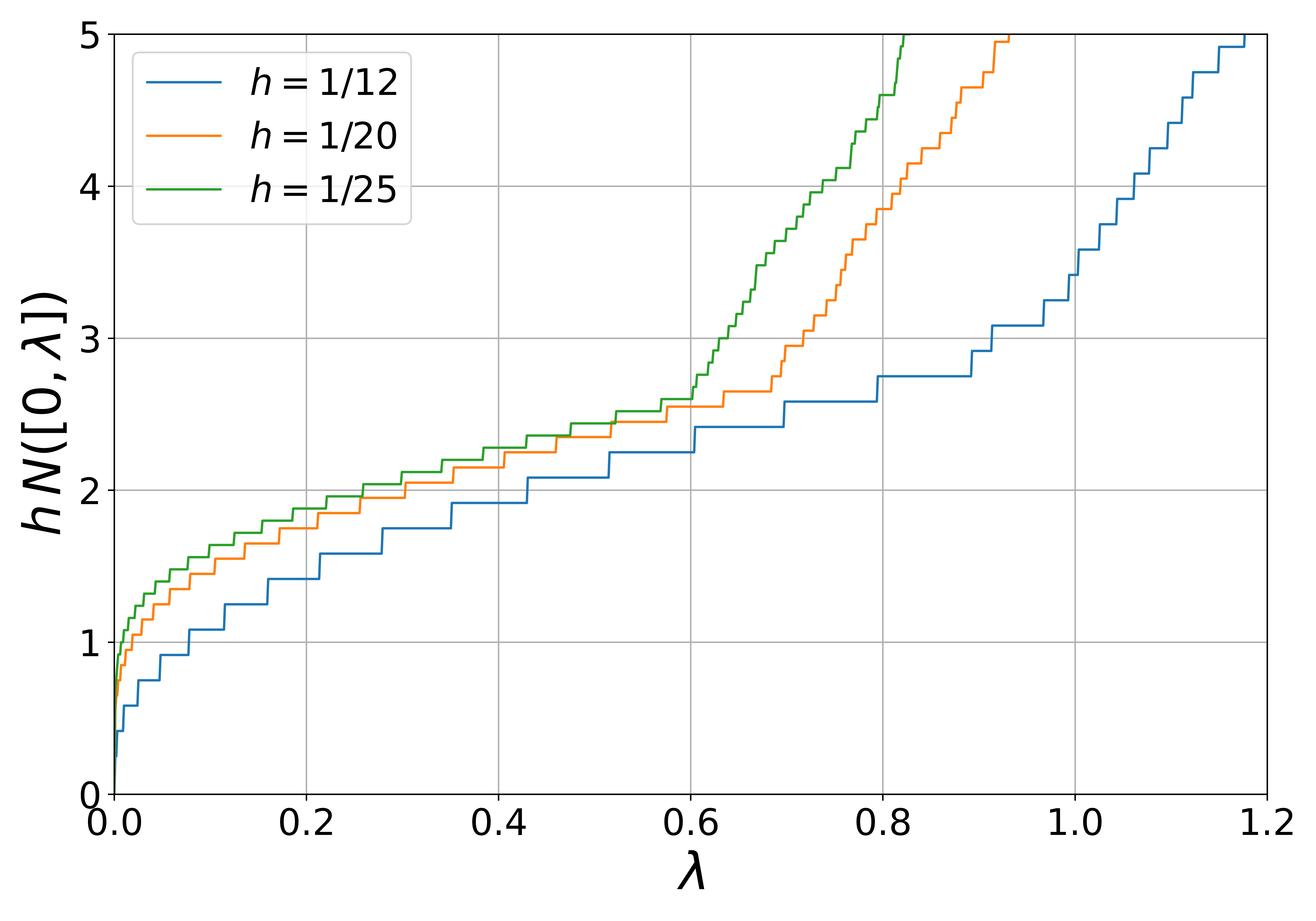}
\caption{The numerically obtained number of exponentially small singular values 
$N([0,\lambda])$ normalized by a factor of $h$ for a given $\lambda$ for various values of $h$. The function $\varphi$ and the manifold $M$ for this simulation are the same as in Figure \ref{figA}.}
\label{figB}
\end{figure}

\subsection{Connection to recent works on tunneling for operators with magnetic fields}
Assume for simplicity that $M = \mathbb C/(2\pi \mathbb Z + 2\pi i\mathbb Z)$ is the standard $2$-torus. The self-adjoint operator
\[
(2P^*)(2P) = \left(-2h\partial_z + 2\varphi'_z\right)\left(2h\partial_{\overline{z}} + 2\varphi'_{\overline{z}}\right) = (hD_x + \varphi'_y)^2 + (hD_y - \varphi'_x)^2 -h\Delta \varphi,
\]
studied in this work, is closely related to the two-dimensional magnetic Schr\"odinger operator as well as to the Pauli operator
\[
\begin{pmatrix} 0 & 2P
\\ 2P^*  & 0  \end{pmatrix}^2,
\]
for which the study of tunneling and exponentially small eigenvalues has been the subject of intense recent activity. The papers \cite{EKP16}, \cite{HeSu17a}, \cite{HeSu17b}, \cite{HKS19} are mainly concerned with the bottom of the spectrum for such operators on bounded domains in $\mathbb R^2$ with the Dirichlet boundary conditions. Of particular relevance here is the work~\cite{HKS19}, which considers the case when the magnetic field, which is the analogue of $\Delta \varphi$, is allowed to change sign. Accurate semiclassical estimates for the exponentially small eigenvalues of the Dirichlet Pauli operator are established in~\cite{BLTRS21}, assuming that the magnetic field is positive.

\medskip
\noindent
The exponentially small splitting of the first two eigenvalues for the two-dimensional magnetic Laplacian has been studied in \cite{BoHeRa22}, \cite{FoMoRa25}. A similar study is made in \cite{HeKaSu24} for the magnetic Schr\"odinger operator, showing a braid structure for the lowest eigenvalues in the presence of a triangular symmetry. See also \cite{FeShWe22}, \cite{FeShWe25}, \cite{HeKa24}.

\medskip
\noindent
The work~\cite{Be14} considers the average distribution of exponentially small eigenvalues for Dolbeault Laplacians, associated to high powers of a holomorphic line bundle over a compact complex manifold. We would also like to mention the recent work~\cite{Fi25}, which studies the asymptotic logarithmic distribution of exponentially small eigenvalues for Toeplitz operators on complex projective manifolds. 

\medskip
\noindent
Finally, earlier results concerning tunneling for semiclassical Schrödinger operators with electric potentials have been obtained by Simon~\cite{Si84} and by Helffer and the second named author \cite{HeSj84,HeSj85a,HeSj85b,HeSj85c}.

\subsection{Organization of the paper} We shall conclude the introduction by reviewing the structure of the paper and outlining some aspects of the proofs. In Section \ref{sec_HC} we derive H\"ormander's $L^2$ estimates for the $\overline{\partial}$ operator with superharmonic weights on an open Riemann surface, which are then used to establish the exponential decay of the singular states of $P$ in the "classically forbidden" region where $\psi_{ub} < \varphi$, with $\psi_{ub}$ being an upper bound weight. In the same spirit, we also review the existence theory for $\overline{\partial}$ on Riemann surfaces, in exponentially weighted $L^2$ spaces with subharmonic weights. Section \ref{sec:bounds} is devoted to the proof of Theorem \ref{thm:UpperBd} and Theorem \ref{thm:LowerBd}. Roughly, when establishing upper bounds on the number of exponentially small singular values, we rely on the localization of the corresponding singular states to a small neighborhood of the contact set $M_+(\psi_{ub})$, assumed to be confined to the region where $\Delta \varphi > 0$. This allows us to introduce and to exploit the classical Bergman projection in the region of the strict subharmonicity of $\varphi$, for which a precise asymptotic description in the semiclassical limit is available, thanks to \cite{Catlin}, \cite{Ze98}, \cite{BBSj}. Passing to trace class estimates for the associated truncated Bergman projections, with the truncations supported in small neighborhoods of the contact set, we obtain the result of Theorem \ref{thm:UpperBd} essentially by a trace computation. The starting point for the proof of Theorem \ref{thm:LowerBd} is a family of $\mathcal O(h^{\infty})$ quasimodes of $P$, given by the coherent states associated to Toeplitz quantizations of symbols localized in the interior of the contact set $M_+(\psi_{lb})$. Here $\psi_{lb}$ is a lower bound weight. Using $\overline{\partial}$ arguments, we then show that such quasimodes are reproduced by the orthogonal projection onto the spectral subspace associated to the singular values in $[0,e^{-\tau/h}]$, up to a small error, and the number of singular values can be bounded from below by the trace class norm of the Toeplitz operators in question. Analyzing the trace class norms by passing to the corresponding asymptotic Bergman projections~\cite{BBSj}, similarly to the proof of Theorem \ref{thm:UpperBd}, we obtain the result of Theorem \ref{thm:LowerBd}. 

\medskip
\noindent
Theorem \ref{1csp4_n} is established in Section \ref{tb}, by developing a detailed asymptotic analysis of the Dirichlet problem for the Laplacian in a thin band of variable width $\asymp \tau^{1/3}$ around each of the finitely many curves comprising the connected components of the set $\gamma = (\Delta \varphi)^{-1}(0)$. This is accomplished by means of the techniques for treating two-scale operators, using the semiclassical pseudodifferential calculus in the slow variable with operator valued symbols, see~\cite{GeMaSj91}, \cite[Chapter 13]{DiSj}. This analysis is then used to construct upper and lower bound weights in the regime of $\tau > 0$ small, so that Theorem \ref{thm:UpperBd} and Theorem \ref{thm:LowerBd} can be applied. In particular, the discussion in Section \ref{tb} does not depend on the results of Section \ref{sec:FB}.

\medskip
\noindent
Section \ref{sec:FB}, which is independent of the previous sections, is devoted to the proof of existence, uniqueness, and regularity results for solutions of the double obstacle problem (\ref{eq:din1}) below, posed on an arbitrary $d$-dimensional compact Riemannian manifold without boundary. Theorem \ref{thm:DPPex2}, concerning the problem (\ref{eq:din1_2a}), is then obtained as a special case of these general results. We study the regularity of the free boundaries $\partial M_{\pm}(\psi)$ by reducing to a one obstacle problem. In particular, we show that the free boundaries $\partial M_{\pm}(\psi)$ are of measure $0$ which is crucial here since it guarantees, together with the contact sets $M_{\pm}(\psi)$ having strictly positive measure, that the interior of the contact sets is non-empty, as required by Theorem \ref{thm:UpperBd} and Theorem \ref{thm:LowerBd}. Let us also recall that the study of free boundary problems emerging from obstacle problems has a long and rich history, see \cite{PSU12,Ro87}, \cite{Fi18,Ca98}. However, the typical setting considered for obstacle problems is that of bounded domains with suitable boundary conditions. In our case, we work on a compact manifold without boundary, which required us to extend the existing theory to this setting. In particular, one novelty in Section \ref{sec:FB} is a new form of the method of penalization adapted from \cite{LePa23}.
%

\medskip
\noindent
The paper is concluded by four appendices. Appendix \ref{sec:almholexten} discusses almost holomorphic extensions of smooth vector fields and smooth functions with values in a totally real submanifold of a Riemann surface. Appendix \ref{sec:Prop3.9} provides a self-contained proof of the rapid off-diagonal decay of the Bergman kernel on an open Riemann surface, needed in Section \ref{sec:bounds}. Appendix \ref{sec:app_large_decay} is devoted to the proof of Theorem \ref{decay_rate_large}. In Appendix \ref{sec:appHoSoSpace}, we collect some basic notions and results concerning Sobolev and H\"older spaces on compact manifolds, used in Section \ref{sec:FB}.

\medskip
\noindent
We shall finally say a few words about the perspectives offered by this work. A more precise asymptotic study of the case when $\tau$ is close to 
$\underset{M}{\rm max}\, \varphi - \underset{M}{\rm min}\, \varphi$, somewhat in the spirit of Section \ref{tb}, may also be possible and would 
be interesting and natural to undertake. Another direction, which seems promising, concerns the extension to the case of high powers of a complex line bundle over a compact Riemann surface, equipped with a Hermitian metric whose curvature changes sign. Finally, the problem of quantum tunneling for general analytic non-self-adjoint operators remains largely unexplored, and it would be most interesting to address it as well.

\bigskip
\noindent
{\bf Acknowledgments.} This project started when M.H. was visiting Universit\'e de Bourgogne in October 2023. He is most grateful to its Institut de Math\'ematiques for the generous hospitality and excellent working conditions. M.H. would also like to thank Robert Greene, Inwon Kim, and Maciej Zworski for stimulating discussions and helpful advice. He gratefully acknowledges partial support from the Simons Foundation grant MPS-TSM-00007843. J.S. acknowledges that the IMB receives support from the EIPHI Graduate School (contract ANR-17-EURE-0002).  M.V. is partially funded by the Agence Nationale de la Recherche, through the project ADYCT (ANR-20-CE40-0017). M.V. is very grateful to the Department of Mathematics at the Universit\`a di Bologna for hosting him as a visiting researcher during parts of this project and for providing excellent working conditions. We would like to thank Robert Berman and Siarhei Finski for bringing the works~\cite{Be14} and ~\cite{Fi25}, respectively, to our attention. We would like to thank Henry Zeng for kindly providing us with the code on which the numerical simulations in Figures \ref{figA} and \ref{figB} are based.

\section{Carleman-H\"ormander estimates on Riemann surfaces}
\setcounter{equation}{0}
\label{sec_HC}

\medskip
\noindent
The purpose of this section is to establish some essentially well known a pri\-ori weighted $L^2$ estimates and existence theorems for the $\overline{\partial}$--operator on a compact Riemann surface $M$, instrumental when proving Theorem \ref{thm:UpperBd} and Theorem \ref{thm:LowerBd}.  See~\cite[Chapter 2]{NaRa11},~\cite[Chapter 11]{Va11},~\cite{Ch15}. 

\subsection{Metrics and $L^2$ scalar products on $M$}
\label{Sec:MetIP}
We shall denote by $T^{1,0}M $ and $T^{0,1}M$ the holomorphic and anti-holomorphic tangent bundle of $M$, respectively, so that 
\[
TM \otimes \mathbb C = T^{1,0}M \oplus T^{0,1}M. 
\]
Here $TM \otimes \mathbb C$ is the fiberwise complexification of $TM$, the tangent bundle of $M$ in the sense of real smooth manifolds. The space of smooth sections of the complex vector bundle of $(p,q)$ forms on $M$ will be denoted by $\Omega^{p,q}(M)$. The conformal Riemannian metric $g$ in (\ref{eq_metric0}), (\ref{eq_metric1}) extends to a complex bilinear form $g^{\mathbb C}$ in the fibers of $TM \otimes \mathbb C$, and if $z = x+iy$ is a local holomorphic coordinate on $M$, we have in view of (\ref{eq_metric1}), 
\begeq
\label{eq_ext0}
g^{\mathbb C}\left(\frac{\partial}{\partial z}, \frac{\partial}{\partial z}\right) = g^{\mathbb C}\left(\frac{\partial}{\partial \overline{z}}, \frac{\partial}{\partial \overline{z}}\right) =0, 
\endeq
\begeq
\label{eq_ext0.1}
g^{\mathbb C}\left(\frac{\partial}{\partial z}, \frac{\partial}{\partial \overline{z}}\right) = g^{\mathbb C}\left(\frac{\partial}{\partial \overline{z}}, \frac{\partial}{\partial {z}}\right) = \frac{1}{2} g_U(z). 
\endeq
Here and below we write in a local holomorphic coordinate chart $(U,z)$, 
\begin{equation}
\label{eq1.4.0.1}
\frac{\partial}{\partial {z}} = \frac{1}{2}\left(\frac{\partial}{\partial x} + \frac{1}{i} \frac{\partial}{\partial y}\right), \quad 
\frac{\partial}{\partial \overline{z}} = \frac{1}{2}\left(\frac{\partial}{\partial x} + i \frac{\partial}{\partial y}\right), \quad z = x+iy. 
\end{equation}
It follows that the representation (\ref{eq_metric1}) still holds for $g^{\mathbb C}$, 
\begeq
\label{eq_ext0.2}
g^{\mathbb C} = \frac{1}{2} g_U(z)\left(dz\otimes d\overline{z} + d\overline{z}\otimes dz\right).
\endeq
Taking a complex linear extension of $J$ in (\ref{eq_J}) to the fibers of $TM \otimes \mathbb C$, we get similarly to (\ref{eq_Kahler}),
\[
\omega(t,s) = g^{\mathbb C}(Jt,s), \quad t,s \in T_z M \otimes \mathbb C. 
\]

\medskip
\noindent
The bilinear form $g^{\mathbb C}$ gives rise to a positive definite Hermitian form 
\begeq
\label{eq_ext03}
h(t,s) = g^{\mathbb C}(t,\overline{s}), \quad t,s \in T_z^{1,0}M,
\endeq
defined in the fibers of the holomorphic tangent bundle, varying smoothly with the base point $z \in M$, and there is an analogous positive definite Hermitian form in the fibers of $T^{0,1}M$. Here in (\ref{eq_ext03}), $\overline{s} \in T_z^{0,1}M$ is obtained by applying the operation of complex conjugation to $s\in T_z^{1,0}M$, well defined on the complexification $T_zM \otimes \mathbb C$ of $T_z M$. We then also have the corresponding dual Hilbert space norms and scalar products in the fibers of the bundles $T^*_{1,0} M$ and $T^*_{0,1} M$ of $(1,0)$-forms and $(0,1)$-forms on $M$, respectively. If $z$ is a local holomorphic coordinate on $M$, following (\ref{eq_metric}), we shall view the $(1,1)$-form 
\begin{equation}
\label{eq:ext1}
\omega = g(z) \frac{d\overline{z}\wedge dz}{2i} 
\end{equation}
as a positive smooth density of integration on the corresponding coordinate chart. Here $(2i)^{-1} d\overline{z}\wedge dz \simeq L(dz)$ can be viewed as the Lebesgue measure when identifying the coordinate chart with an open subset of $\CC$. We also notice that $\omega$ is a globally well defined positive density on $M$. Indeed, in the intersection of two local holomorphic coordinate charts $(V,q)$ and $(U,z)$, we have $q=\chi(z)$, where $\chi$ is holomorphic, and using that $g_U(z) = g_V(q) |\chi'(z)|^2$, in view of (\ref{eq_metric0}), we get 
\begin{equation}\label{eq:ext2}
\omega_{q} = g_V(q) \frac{d\overline{q}\wedge dq}{2i} = g_V(q) \frac{|\chi'(z)|^2 d\overline{z}\wedge dz}{2i} = g_U(z) \frac{d\overline{z}\wedge dz}{2i} = \omega_z. 
\end{equation}

\medskip
\noindent 
If $u,v \in C(M)$ are continuous functions on $M$, we define the $L^2$ scalar product 
\begin{equation}\label{eq:ext3}
(u,v) = \int_M u\overline{v}\,\omega, 
\end{equation}
and notice that when at least one of the functions $u,v$ has compact support in a coordinate chart $(U,z)$ 
then 
\begin{equation}\label{eq:ext3a}
(u,v) = \int_M u\overline{v}\,\omega = \iint_U u(z)\overline{v(z)}\, g_U(z)\frac{d\overline{z}\wedge dz}{2i}.
\end{equation}
We denote the corresponding $L^2$ space by $L^2(M,\omega)$ and notice that $C^\infty(M)$ is a dense subspace. 

\medskip
\noindent 
If $\nu,\mu$ are continuous $(0,1)$-forms on $M$ we define their $L^2$ scalar product by 
\begin{equation}\label{eq:ext4}
(\nu,\mu) = \frac{1}{2i} \int_M \nu \wedge \overline{\mu}.
\end{equation}
If $\nu,\mu$ have compact support in a coordinate chart $(U,z)$, we write $\nu = n(z)d\overline{z}$, $\mu = m(z)d\overline{z}$ and 
get 
\begin{equation}\label{eq:ext5}
(\nu,\mu) = \iint_U n(z)\overline{m(z)}\frac{d\overline{z}\wedge dz}{2i}.
\end{equation}
We have therefore a well defined scalar product and the corresponding $L^2$ space $L^2(M,T^*_{0,1}M)$ of $(0,1)$-forms on $M$. Similarly, if $\nu,\mu$ are continuous $(1,0)$-forms on $M$ we define their scalar product by 
\begin{equation}\label{eq:ext4.1}
(\nu,\mu) = -\frac{1}{2i} \int_M \nu \wedge \overline{\mu}.
\end{equation}
If $\nu,\mu$ have compact support in a coordinate chart $(U,z)$,  we write $\nu = n(z)dz$, $\mu = m(z)dz$ and get 
\begin{equation}\label{eq:ext5.1}
  (\nu,\mu) 
  = -\frac{1}{2i}\iint_U n(z)\overline{m(z)}\frac{dz\wedge d\overline{z}}{2i}
  =\iint_U n(z)\overline{m(z)}\frac{d\overline{z}\wedge dz}{2i}.
\end{equation}
We have therefore a scalar product and the corresponding $L^2$ space $L^2(M,T^*_{1,0}M)$ of $(1,0)$-forms on $M$. Notice that the definitions of the $L^2$ spaces of forms, $L^2(M,T^*_{0,1}M)$ and $L^2(M,T^*_{1,0}M)$, are independent of the choice of a conformal Riemannian metric on $M$. The spaces $\Omega^{0,1}(M)$ and $\Omega^{1,0}(M)$ are dense in $L^2(M,T^*_{0,1}M)$ and $L^2(M,T^*_{1,0}M)$, respectively.

\medskip
\noindent
An open connected proper subset $\Omega\subseteq M$ is naturally an open Riemann surface and the inclusion map $\iota:\Omega \hookrightarrow M$ is a holomorphic embedding. For an $(p,q)$-form $\alpha$ on $M$, the restriction $\iota^*\alpha = \alpha|_\Omega$ is a well-defined $(p,q)$-form $\Omega$. 
Thus, using $\omega|_\Omega$ instead of $\omega$ and proceeding as in \eqref{eq:ext3}, we obtain the $L^2$ space $L^2(\Omega,\omega|_\Omega)$. 
Similarly to \eqref{eq:ext4}, \eqref{eq:ext4.1}, we define the $L^2$ spaces $L^2(\Omega,T^*_{1,0}\Omega)$ and $L^2(\Omega,T^*_{0,1}\Omega)$. 
Restricting functions and forms to $\Omega$ by means of $\iota$ we get that 
\begin{equation}\label{eq:ext5.5} 
\begin{split}
  & \iota^*:   L^2(M,\omega) \rightarrow L^2(\Omega,\omega|_\Omega), \\
  & \iota^*:L^2(M,T^*_{p,q}M) \rightarrow L^2(\Omega,T^*_{p,q}\Omega), ~~ p+q=1.
\end{split}
\end{equation}
To simplify the notation we shall not indicate the restriction of $\omega$ to $\Omega$ and write $L^2(\Omega,\omega)$.

%
%
%

\subsection{Carleman estimates and weighted $L^2$ spaces on $M$}
Let $\varphi:M\to \RR$ be smooth and consider the exponentially weighted $L^2$ spaces 
\begin{equation}\label{eq:ext6}
\begin{split}
& L^2_\varphi (M, \omega) = {e}^{\varphi/h}L^2(M,\omega), \\ 
& L^2_\varphi (M, T^*_{0,1}M) = {e}^{\varphi/h}L^2(M,T^*_{0,1}M). \\ 
\end{split}
\end{equation}
We shall be concerned with the semiclassical $\overline{\partial}$ operator 
\begin{equation}\label{eq:ext7}
h\overline{\partial}:L^2_\varphi (M, \omega) \to L^2_\varphi (M, T^*_{0,1}M), \quad h\in(0,1],
\end{equation}
or, equivalently, with the conjugated operator 
\begin{equation}\label{eq1.3}
P:=(h\overline{\partial})_\varphi := {e}^{-\varphi/h}\circ h\overline{\partial}\circ {e}^{\varphi/h}
:L^2 (M, \omega) \to L^2(M, T^*_{0,1}M).
\end{equation}
We have in view of (\ref{eq1.3}), 
\begin{equation}\label{eq:ext9}
P = (h\overline{\partial})_\varphi = h\overline{\partial} + (\overline{\partial}\varphi)^\wedge, 
\end{equation}
and therefore, in a local holomorphic coordinate chart $(U,z)$, the operator $P$ is given by 
\begin{equation}\label{eq:ext10}
Pu =  (h\overline{\partial})_\varphi u =
  \left(h\frac{\partial u}{\partial \overline{z}}+ \frac{\partial\varphi}{\partial \overline{z}} u
  \right)\, d\overline{z}.
\end{equation}

\medskip
\noindent 
We shall view $P$ in (\ref{eq1.3}) as a closed densely defined operator on $L^2(M,\omega)$ equipped with the maximal domain,
\begeq
\label{eq1.4}
\mathcal{D}(P) = \{u\in L^2(M,\omega); Pu \in L^2(M, T^*_{0,1}M)\}, 
\endeq
and it follows from (\ref{eq:ext9}) that 
\begin{equation}
\label{eq1.4.1}
\mathcal{D}(P) = \{u\in L^2(M,\omega); h\overline{\partial}u\in L^2(M, T^*_{0,1}M)\} = H^1(M),
\end{equation}
the standard Sobolev space on $M$, in view of the classical ellipticity of $h\overline{\partial}$.

\bigskip
\noindent
Similarly to \eqref{eq1.3}, we set 
\begin{equation}\label{eq:ext10.1}
(h\partial)_{-\varphi} := {e}^{\varphi/h}\circ h\partial\circ {e}^{-\varphi/h} :L^2 (M, \omega) \to L^2(M, T^*_{1,0}M), 
\end{equation}
so that 
\begin{equation}\label{eq:ext10.2}
(h\partial)_{-\varphi} = h\partial  - (\partial\varphi)^\wedge. 
\end{equation}
In a local holomorphic coordinate chart $(U,z)$ we have therefore 
\begin{equation}
\label{eq:ext10.3}
(h\partial)_{-\varphi} u = \left(h\frac{\partial u}{\partial z}- \frac{\partial\varphi}{\partial z} u \right)\, dz.
\end{equation}

\medskip
\noindent
We shall now compute the formal adjoint $P^*$ of $P$ in (\ref{eq1.3}). When doing so, let $u\in C^\infty(M)$, $\mu \in \Omega^{0,1}(M)$, and let us write using (\ref{eq:ext4}),  
\begin{multline}
\label{eq:rs14}
(Pu,\mu) = \frac{1}{2i} \int_M Pu \wedge \overline{\mu } = \frac{1}{2i} \int_M h\overline{\partial}(e^{\varphi/h}u) \wedge e^{-\varphi/h} \overline{\mu } \\
 = \frac{1}{2i} \int_M h\overline{\partial}(u\overline{\mu }) - \frac{1}{2i} \int_M e^{\varphi/h}u\, h\overline{\partial}(e^{-\varphi/h} \overline{\mu}).
\end{multline} 
Here $u\overline{\mu}$ is a $(1,0)$-form, and therefore 
\begeq
\label{eqStokes}
\int_M h\overline{\partial}(u\overline{\mu}) = \int_M hd(u\overline{\mu}) = 0, \quad d = \partial + \overline{\partial},
\endeq
by Stokes' theorem. It follows from (\ref{eq:ext3}), (\ref{eq:rs14}), and (\ref{eqStokes}) that $P^* \mu$ is the unique $C^{\infty}$ function on $M$ such that we have, on the level of $(1,1)$-forms, 
\begeq
\label{eq1.4.0.2}
(P^* \mu)\,\omega = \frac{1}{2i} e^{\varphi/h}\circ h\partial\circ e^{-\varphi/h} (\mu) = \frac{1}{2i} (h\partial)_{-\varphi} (\mu).
\endeq
Letting  
\begin{equation}
\label{eq:ext8.0}
P^*: L^2(M, T^*_{0,1}M)\to L^2 (M, \omega),
\end{equation}
stand also for the $L^2$--adjoint of $P$ with the domain (\ref{eq1.4.1}), we conclude by ellipticity that the domain $\mathcal{D}(P^*)$ of $P^*$ satisfies $\mathcal{D}(P^*) = H^1_{0,1}(M)$, the space of $(0,1)$-forms with coefficients in $H^1$, and that (\ref{eq1.4.0.2}) still holds for $\mu \in \mathcal D(P^*)$. 

\medskip
\noindent
{\it Remark.} For future reference, we notice that (\ref{eq1.4.0.2}) gives, taking $\varphi = 0$, 
\begeq
\label{eq1.4.0.3}
\left((h\overline{\partial})^* \mu\right) \omega = -\frac{i}{2} h\partial \mu, \quad \mu \in \Omega^{0,1}(M). 
\endeq
We get therefore, letting $\mu = h \overline{\partial} u$, $u\in C^{\infty}(M)$, 
\begeq
\label{eq1.4.0.4} 
\left(h^2 \overline{\partial}^* \overline{\partial} u\right) \omega = -\frac{i}{2} h^2 \partial \overline{\partial} u = - \frac{1}{4}(h^2 \Delta_g u)\omega, 
\endeq
leading to the factorization 
\begeq
\label{eq1.4.0.5}
-\Delta_g = 4 \overline{\partial}^* \overline{\partial}.
\endeq
Here in (\ref{eq1.4.0.4}) we have also used (\ref{eq_Laplace}). 

\medskip
\noindent
Let $z = x+iy$ be a local holomorphic coordinate on $M$, and let $(x,y;\xi,\eta)$ be the corresponding local canonical coordinates on $T^*M$, the cotangent bundle of $M$ in the sense of real smooth manifolds. Setting
\begeq
\label{eq1.7}
\zeta = \frac{1}{2}(\xi - i\eta), \quad \overline{\zeta} = \frac{1}{2}(\xi + i\eta),
\endeq
we may express the standard symplectic 2-form on $T^*M$ as follows, 
\begin{equation}
\label{eq1.9}
\sigma_{\mathbb R} = d\xi \wedge dx + d\eta \wedge dy =  2 \Re d \zeta \wedge d z = d \zeta \wedge d z + d \bar \zeta \wedge d \bar z. 
\end{equation}
Introducing $\partial_{\zeta} = \partial_{\xi} + i\partial_{\eta}$, $\partial_{\bar{\zeta}} = \partial_{\xi} - i\partial_{\eta} \in T\mathbb C \otimes \mathbb C$, which is the dual basis to $d\zeta$, $d\overline{\zeta} \in T^* \mathbb C \otimes \mathbb C$, we get using also (\ref{eq1.4.0.1}), 
\[
\begin{pmatrix} \partial_z \\ \partial_{\bar{z}}
\end{pmatrix} = A \begin{pmatrix} \partial_x \\ \partial_{y}
\end{pmatrix}, \quad A = \begin{pmatrix} \frac{1}{2} & -\frac{i}{2} \\
\frac{1}{2} & \frac{i}{2}\end{pmatrix},
\]
\[
\begin{pmatrix} \partial_{\zeta} \\ \partial_{\bar {\zeta}}
\end{pmatrix} = B \begin{pmatrix} \partial_{\xi} \\ \partial_{\eta}
\end{pmatrix}, \quad B = \begin{pmatrix} 1 & i \\
1 & -i\end{pmatrix}.
\]
Using that $B^t A = 1$, we obtain 
\[
\partial_\zeta a\, \partial_z b + \partial_{\bar \zeta } a\, \partial_{\bar z } b = B \begin{pmatrix} \partial_{\xi} a \\ \partial_{\eta} a \end{pmatrix} \cdot A \begin{pmatrix} \partial_{x} b \\ \partial_{y} b \end{pmatrix} = \begin{pmatrix} \partial_{\xi} a \\ \partial_{\eta} a \end{pmatrix}\cdot B^t A \begin{pmatrix} \partial_{x} b \\ \partial_{y} b \end{pmatrix} =  \partial_{\xi} a\, \partial_{x} b + \partial_{\eta} a\, \partial_y b,
\]
\[
\partial_z a\, \partial_{\zeta} b + \partial_{\bar z} a\, \partial_{\bar \zeta } b = A \begin{pmatrix} \partial_x a \\ \partial_y a \end{pmatrix} \cdot B \begin{pmatrix} \partial_{\xi} b \\ \partial_{\eta} b \end{pmatrix} = B^tA \begin{pmatrix} \partial_{x} a \\ \partial_{y} a \end{pmatrix}\cdot \begin{pmatrix} \partial_{\xi} b \\ \partial_{\eta} b \end{pmatrix} =  \partial_{x} a\, \partial_{\xi} b + \partial_{y} a\, \partial_{\eta} b.
\]
Consequently, in the $(z,\zeta)$ coordinates, the Poisson bracket of two functions $a, b \in C^1(T^*M)$ is given by
\begin{equation}
\label{eq1.10}
\{ a , b \} = \partial_\zeta a\, \partial_z b  + \partial_{\bar \zeta } a\, \partial_{\bar z } b - \partial_z a\, \partial_\zeta b - \partial_{\bar z } a\, \partial_{\bar \zeta } b.
\end{equation} 

\medskip
\noindent
The semiclassical symbol $p(\gamma)$ of $P$ in (\ref{eq1.3}) at $\gamma = (x,y;\xi, \eta) \in T^*M$ is a linear map from $\mathbb C$ to the fiber of $T^*_{0,1} M$ at $z = x+ iy$ given by 
\begeq
\label{eq1.6}
p(\gamma)\lambda = \lambda \left(\frac{i}{2}(\xi + i\eta) + \partial_{\overline{z}}\varphi\right)\,d\overline{z} =  \lambda\left(i \overline{\zeta} + \partial_{\overline{z}}\varphi\right)\, d\overline{z}, \quad \lambda \in \mathbb C,
\endeq
see~\cite[Chapter 6]{Horm_ALPDO}. 

\medskip
\noindent
{\it Remark}. Assume that $M = \mathbb C/(2\pi \mathbb Z + 2\pi i \mathbb Z)$ is the standard $2$-torus equipped with the usual complex structure and the Euclidean metric. The $(0,1)$-cotangent bundle $T^*_{0,1} M$ of $M$ is trivial, and the operators $P$, $P^*$ in (\ref{eq1.3}), (\ref{eq:ext8.0}) can naturally be viewed as scalar operators, acting on $C^{\infty}(M)$. In the $(z,\zeta)$ coordinates, the semiclassical symbols of $P$, $P^*$ are, respectively, 
\begeq
\label{eq1.8}
p = i\overline{\zeta} + \partial_{\overline{z}}\varphi, \quad p^* = \overline{p} = -i \zeta +  \partial_z \varphi.
\endeq
We get, in view of (\ref{eq:ext10}), (\ref{eq1.4.0.2}), (\ref{eq1.10}), and (\ref{eq1.8}),
\begeq
\label{eq1.11}
[P,P^*] = 2h\partial_z \partial_{\overline{z}}\varphi = \frac{h}{2} \Delta \varphi, \quad \frac{1}{i}\{p,p^*\} =
2\partial_z \partial_{\overline{z}}\varphi = \frac{1}{2} \Delta \varphi.
\endeq
It follows, in particular, from (\ref{eq1.10}), (\ref{eq1.8}), (\ref{eq1.11}) that conditions (\ref{eq:cs4_2a}), (\ref{eq:Dweight}) are equivalent to the non-vanishing of the second Poisson bracket along the zero set of the first bracket,
\begin{equation}
\label{eq1.11.1}
\{p,\{p,\overline{p}\}\} \neq 0 \,\, \text{ when } \,\, \{p,\overline{p}\} =0.
\end{equation}

\medskip
\noindent
Let us return to the operator $P$ given in (\ref{eq1.3}), (\ref{eq1.4}). When equipped with the domain 
\begin{equation}\label{eq:domVN}
\mathcal{D}(P^*P) =\{ u \in  \mathcal{D}(P); Pu \in  \mathcal{D}(P^*)\} \subseteq L^2(M,\omega), 
\end{equation}
the operator $P^*P$ is non-negative self-adjoint by a theorem of von Neumann~\cite[Theorem 3.24]{Ka95}. By ellipticity, we see that $\mathcal{D}(P^*P) = H^2(M)$, which is compactly embedded into $L^2(M,\omega)$~\cite[Theorem 2.34]{Au82}. It follows that the spectrum of $P^*P$ is discrete. The eigenvalues are of the form $t_j^2$, where by definition, $t_j \geq 0$ are the singular values of $P$, see also~\cite{SV} and the discussion in Section \ref{sec:Setting}. The corresponding eigenfunctions of $P^*P$ will be called singular functions or singular states.

\bigskip
\noindent
Let $0\leq t \leq e^{-\tau/h}$ be an exponentially small singular value of $P$, $\tau > 0$, and let $v \in L^2(M,\omega)$ be a corresponding $L^2$--normalized singular state, $P^*P v = t^2 v$. Then with the $L^2$ scalar products and norms we have
\begeq
\label{eq1.12}
\|Pv\|^2 = (P^*Pv,v) = t^2 \|v\|^2 = t^2,\quad \|Pv\| = t.
\endeq
Recalling (\ref{eq1.3}), we see that $u = e^{\varphi/h}v$ satisfies
\begeq
\label{eq1.13}
\|u\|_{L^2_{\varphi}(M,\omega)} = \|v\|_{L^2(M,\omega)} = 1,\quad \|h\overline{\partial}u\|_{L^2_{\varphi}(M, T^*_{0,1} M)} = \|Pv\|_{L^2(M,T^*_{0,1}M)} = t \leq e^{-\tau/h}.
\endeq
Here the $L^2_{\varphi}$ and $L^2$ spaces of functions and $(0,1)$-forms are as in \eqref{eq:ext6}, $\|u\|_{L^2_{\varphi}(M,\omega)} = \|e^{-\varphi/h}u\|_{L^2(M,\omega)}$. We think of 
(\ref{eq1.13}) as
\begeq
\label{eq1.14}
u = \mathcal O(1)\, e^{\varphi/h},\quad h\overline{\partial}u
= \mathcal O(1)\, e^{(\varphi - \tau)/h}, \quad \wrtext{in the}\,\,
L^2 \,\,\wrtext{sense}.
\endeq

\medskip
\noindent
By the standard Carleman-H\"ormander weighted $L^2$ estimates~\cite[Chapter 4]{Horm_CASV},~\cite[Section 4.2]{Horm_Conv}, we get the following result.

\begin{prop}\label{propCH}
Let $\Omega \subseteq M$ be an open subset such that ${\rm ext}(\Omega) = {\rm int}(M\setminus \Omega) \neq \emptyset$ and let 
$\psi \in C(\overline{\Omega};\mathbb R)$ be superharmonic in $\Omega$. Let $\Theta \subseteq \overline{\Omega}$ be a neighborhood of $\partial \Omega$ in $\overline{\Omega}$ and let $\varepsilon > 0$. Then there exists $C = C(\Theta, \varepsilon)$, depending also on $\Omega$ but not on $\psi$, such that
\begin{equation}
\label{eq1.15}
\|e^{-\psi/h} u\|_{L^2(\Omega,\omega)} \leq C e^{\varepsilon/h} \left(\|e^{-\psi/h} 
h\overline{\partial}u\|_{L^2(\Omega,T^*_{0,1}\Omega)} + \|1_{\Theta} e^{-\psi/h} u\|_{L^2(\Omega,\omega)}\right),
\end{equation}
for all $u\in H^1(\Omega)$.
\end{prop}
\begin{proof}
We shall first show that for each $\varepsilon > 0$ there exists $C_{\varepsilon} > 0$, depending also on $\Omega$ but not on $\psi$, such that for all $u\in C^{\infty}_0(\Omega)$ we have
\begin{equation}
\label{eq1.15.1}
\norm{e^{-\psi/h} u}_{L^2(\Omega,\omega)} 
\leq C_{\varepsilon}\, e^{\varepsilon/h} \norm{e^{-\psi/h} h\overline{\partial} u}_{L^2(\Omega,T^*_{0,1}\Omega)}. 
\end{equation}
When doing so, we shall assume first that $\psi \in C^{\infty}(\Omega; \mathbb R)$ is strictly superharmonic in $\Omega$,
\begeq
\label{eq1.15.2}
-\Delta \psi \geq \alpha > 0.
\endeq
Estimate (\ref{eq1.15.1}) is then equivalent to an estimate for the conjugated operator $P_{\psi} = (h\overline{\partial})_{\psi} = e^{-\psi/h} \circ h\overline{\partial} \circ e^{\psi/h}$,
\begeq
\label{eq1.16}
\norm{u}_{L^2(\Omega,\omega)} \leq C_{\varepsilon}\, e^{\varepsilon/h} 
\norm{P_{\psi} u}_{L^2(\Omega,T^*_{0,1}\Omega)}, \quad u\in C^{\infty}_0(\Omega).
\endeq
When establishing (\ref{eq1.16}), we begin with a local argument: let $(U,z)$ be a local holomorphic coordinate chart in $\Omega$ and given 
$u\in C_0^{\infty}(U)$, $v\in C^{\infty}_0(\Omega)$, let us consider 
\begin{equation}
\label{eq:ext11}
\begin{split}
(P_{\psi} u, P_{\psi} v) = ((h\overline{\partial})_\psi u, (h\overline{\partial})_\psi v) 
&= \frac{1}{2i} \iint_U (h\overline{\partial})_\psi u \wedge \overline{(h\overline{\partial})_\psi v} \\
&= \iint_U \left(h\frac{\partial u}{\partial \overline{z}} + \frac{\partial\psi}{\partial \overline{z}}u
\right)\overline{\left(h\frac{\partial v}{\partial \overline{z}} + \frac{\partial\psi}{\partial \overline{z}}v
\right)}\frac{d\overline{z}\wedge dz}{2i} \\
&= \iint_U \left[\left(-h\frac{\partial}{\partial z}+ \frac{\partial\psi}{\partial z}
\right)\!\left(h\frac{\partial}{\partial \overline{z}}+ \frac{\partial\psi}{\partial \overline{z}}
\right)u\right] \overline{v}\frac{d\overline{z}\wedge dz}{2i}. 
\end{split}
\end{equation}
Similarly, using (\ref{eq:ext4.1}), \eqref{eq:ext10.3} we get   
\begin{equation}\label{eq:ext11.1}
\begin{split}
  ((h\partial)_{-\psi} u, (h\partial)_{-\psi} v)
  &= -\frac{1}{2i} \iint_U (h\partial)_{-\psi}  u \wedge \overline{(h\partial)_{-\psi}  v}
  \\
  &= \iint_U \left[\left(-h\frac{\partial}{\partial \overline{z}}- \frac{\partial\psi}{\partial \overline{z}}
  \right)\left(h\frac{\partial}{\partial z} - \frac{\partial\psi}{\partial z}
  \right)u\right]\overline{v}\frac{d\overline{z}\wedge dz}{2i} \\
  & = \iint_U \left[\left(h\frac{\partial}{\partial \overline{z}}+ \frac{\partial\psi}{\partial \overline{z}}
  \right)\left(-h\frac{\partial}{\partial z} + \frac{\partial\psi}{\partial z}
  \right)u\right]\overline{v}\frac{d\overline{z}\wedge dz}{2i}. 
\end{split}
\end{equation}
Subtracting \eqref{eq:ext11.1} from \eqref{eq:ext11} gives 
\begin{equation}\label{eq:ext11.2}
\begin{split}
((h\overline{\partial})_\psi u, (h\overline{\partial})_\psi v)
&- ((h\partial)_{-\psi} u, (h\partial)_{-\psi} v) \\
&= \iint_U \left(\left[-h\frac{\partial}{\partial z}+ \frac{\partial\psi}{\partial z}
  ,h\frac{\partial}{\partial \overline{z}}+ \frac{\partial\psi}{\partial \overline{z}}
  \right]u\right) \overline{v}\frac{d\overline{z}\wedge dz}{2i}
  \\
  &=
  \iint_U \left(-2h\partial_z\partial_{\overline{z}} \psi
  \right) u \,\overline{v} \,\frac{d\overline{z}\wedge dz}{2i}. 
\end{split}
\end{equation}
Here we notice that 
\begin{equation}\label{eq:ext11.3}
-2 \partial_z\partial_{\overline{z}} \psi\, \frac{d\overline{z}\wedge dz}{2i} = -i \partial\overline{\partial}\psi 
\end{equation} 
is invariant under changes of local holomorphic coordinates and is therefore a globally defined real $(1,1)$-form on $\Omega$.

\medskip
\noindent
Let $u \in C^{\infty}_0(\Omega)$. Decomposing $u$ by means of a partition of unity subordinate to an open cover of $\supp u$ by local holomorphic coordinate charts, we obtain from \eqref{eq:ext11.2} that for all $u,v\in C^{\infty}_0(\Omega)$, we have 
\begin{equation}\label{eq:ext11.4}
  ((h\overline{\partial})_\psi u, (h\overline{\partial})_\psi v)
  -
  ((h\partial)_{-\psi} u, (h\partial)_{-\psi} v)
  = -ih \int_\Omega  u \,\overline{v}\, \partial\overline{\partial}\psi. 
\end{equation}
In particular, when $u=v$ we get from \eqref{eq:ext11.4}, 
\begin{equation}
\label{eq1.18}
\|P_{\psi} u\|^2 \geq -ih \int_\Omega  |u|^2 \, \partial\overline{\partial}\psi = -\frac{h}{2}\int_\Omega (\Delta\psi)|u|^2 \,\omega  \geq \frac{h\alpha}{2}\|u\|^2.
\end{equation}
Here we have also used (\ref{eq_Laplace}) and \eqref{eq1.15.2}. The estimate (\ref{eq1.18}) can be reformulated as follows,
\begeq
\label{eq1.19}
\left(\frac{h\alpha}{2}\right)^{1/2} \norm{e^{-\psi/h} u}_{L^2(\Omega,\omega)}
\leq \norm{e^{-\psi/h} h\overline{\partial}u}_{L^2(\Omega,T^*_{0,1}\Omega)},
\endeq
for all $u\in C^{\infty}_0(\Omega)$. 

\medskip
\noindent
We shall next drop assumption (\ref{eq1.15.2}) and assume only that $\psi \in C^{\infty}(\Omega;\mathbb R)$ satisfies $-\Delta \psi \geq 0$ 
in $\Omega$. To this end, let $z_0 \in {\rm ext}(\Omega)$ and let us set $\widetilde{\Omega} = M \setminus \{z_0\}$, so that $\Omega \Subset \widetilde{\Omega}$. The open Riemann surface $\widetilde{\Omega}$ is a Stein manifold~\cite[Section 2.2]{Fo11}, and hence there exist 
$f_j \in {\rm Hol}(\widetilde{\Omega})$, $1 \leq j \leq 3$, such that the map $\widetilde{\Omega} \ni z \mapsto (f_1(z),f_2(z),f_3(z))\in \mathbb C^3$ 
is a proper holomorphic embedding, see~\cite[Theorem 5.3.9]{Horm_CASV}. The non-negative $C^{\infty}$ function
\begin{equation}
\label{eq1.22}
F(z) = \sum_{j=1}^3 \abs{f_j(z)}^2 \geq 0, \quad z\in \widetilde{\Omega},
\end{equation}
is therefore strictly subharmonic in $\widetilde{\Omega}$. (Alternatively, given $0 < b \in C^{\infty}(M)$, we may consider the elliptic equation $\Delta F = b - C_b \delta_{z_0}$ on $M$, with $C_b = \int b(z)\, \omega(z, dz d\overline{z}) > 0$. Here $\Delta: H^{s+2}(M) \rightarrow H^s(M)$ is Fredholm of index zero for each $s\in \mathbb R$, and we have $f\in \mathcal R(\Delta: H^{s+2}(M) \rightarrow H^s(M))$ precisely when $\langle{f,1\rangle} = 0$. In our case $b - C_b \delta_{z_0} \in H^{s}(M)$ for $\displaystyle s < -\frac{{\rm dim}_{\mathbb R}\, M}{2} = -1$, and therefore we get a solution $F \in H^{s+2}(M)$, unique up to an additive constant. By the elliptic regularity, we have $F \in C^{\infty}(\widetilde{\Omega})$ and $F$ is strictly subharmonic in $\widetilde{\Omega}$. Modifying $F$ by a constant, we may achieve that $F\geq 0$ 
in $\Omega$.)

\medskip
\noindent
When $\alpha >0$, the function $\psi_{\alpha} = \psi - \alpha F \in C^{\infty}(\Omega)$ satisfies in $\Omega$, 
\begeq
\label{eq1.22.0.1}
-\Delta \psi_{\alpha} = -\Delta \psi + \alpha\, \Delta F \geq \alpha\, \Delta F \geq \alpha c,
\endeq
for some $c>0$ independent of $\psi$, and applying (\ref{eq1.19}) with $\psi$ replaced by $\psi_{\alpha}$, we get, using that $\psi_{\alpha} \leq \psi$,
\begeq
\label{eq1.19.2}
\norm{e^{-\psi/h} u}_{L^2(\Omega,\omega)} \leq \left(\frac{2}{h\alpha c}\right)^{1/2}
\norm{e^{-\psi_{\alpha}/h} h\overline{\partial} u}_{L^2(\Omega,T^*_{0,1}\Omega)}, \quad u \in C^{\infty}_0(\Omega). 
\endeq 
For a given $\varepsilon>0$, we choose $\alpha >0$ small enough, depending also on $\|F\|_{L^{\infty}(\Omega)}$ but not on $\psi$, so that
\begeq
\label{eq1.19.3}
\psi(z) - \frac{\varepsilon}{4} \leq \psi_{\alpha}(z),\quad z\in \Omega, 
\endeq and hence we obtain (\ref{eq1.15.1}) for all $u\in C^{\infty}_0(\Omega)$, assuming that $\psi \in C^{\infty}(\Omega; \mathbb R)$ is superharmonic. 

\medskip
\noindent
We shall next remove the smoothness assumption on $\psi$ by an approximation argument. Applying~\cite[Proposition 5.1.2]{Br10}, we obtain that for every $\Omega_1\Subset \Omega$ open, there exists a sequence of strictly superharmonic smooth functions $\psi_j$, defined on an open neighborhood of $\overline{\Omega}_1$, which increases pointwise to $\psi$ on $\Omega_1$. When $u\in C^{\infty}_0(\Omega_1) \subseteq C^\infty_0(\Omega)$, we can therefore apply (\ref{eq1.15.1}) with $\psi$ replaced by $\psi_j$, and letting $j\rightarrow \infty$, we obtain (\ref{eq1.15.1}) for all $u\in C^{\infty}_0(\Omega_1)$, and hence for all $u\in C^{\infty}_0(\Omega)$. Recalling that $\psi \in C(\overline{\Omega}; \mathbb R)$, by a density argument we next extend (\ref{eq1.15.1}) to all $u\in H^1_0(\Omega)$. Here the latter space is defined as the closure of $C^{\infty}_0(\Omega)$ in $H^1(\Omega)$. 

\medskip
\noindent
Let now $u\in H^1(\Omega)$. If $\Theta$ is an open neighborhood of $\partial \Omega$ in $\overline{\Omega}$, let $\chi \in C^{\infty}_0(\Omega; [0,1])$ 
be equal to $1$ in $\Omega \backslash \Theta$, and let us apply (\ref{eq1.15.1}) 
to $\chi u \in H^1_0(\Omega)$. We get
\begin{multline} 
\label{eq1.20}
\norm{e^{-\psi/h} \chi u}_{L^2(\Omega,\omega)} \\ 
\leq C_{\varepsilon}\, e^{\varepsilon/h}\left(\norm{e^{-\psi/h} h\overline{\partial} u}_{L^2(\Omega,T^*_{0,1}\Omega)} 
+ h\, \norm{\nabla \chi}_{L^{\infty}(\Omega)}\,\norm{1_{\Theta}e^{-\psi/h} u}_{L^2(\Omega,\omega)} \right).
\end{multline} 
Hence we get for $0 < h \leq 1$,
\begin{multline}
\label{eq1.21}
\norm{e^{-\psi/h} u}_{L^2(\Omega,\omega)} \leq \norm{e^{-\psi/h} \chi u}_{L^2(\Omega,\omega)} + \norm{e^{-\psi/h} (1-\chi) u}_{L^2(\Omega,\omega)} \\
\leq C_{\varepsilon}\, e^{\varepsilon/h} \norm{e^{-\psi/h} h\overline{\partial} u}_{L^2(\Omega,T^*_{0,1}\Omega)} +
\left(C_{\varepsilon}\,\norm{\nabla \chi}_{L^{\infty}(\Omega)}\, e^{\varepsilon/h} +1\right)\norm{1_{\Theta}e^{-\psi/h} u}_{L^2(\Omega,\omega)} .
\end{multline}
We obtain (\ref{eq1.15}), completing the proof.
\end{proof}

\medskip
\noindent
As an application of Proposition \ref{propCH}, we obtain the following conclusion concerning singular states of $P$ in \eqref{eq1.3}, associated 
to exponentially small singular values.
\begin{prop}
\label{appl_sstate}
Let $\Omega \subseteq M$ be open such that ${\rm ext}(\Omega) \neq \emptyset$. Let $\psi \in C(\overline{\Omega}; \mathbb R)$ be superharmonic in $\Omega$ satisfying
\begin{equation}
\label{eq1.25}
\psi \geq \varphi \,\,\, \wrtext{on}\,\,\, \partial \Omega, \quad
\psi \geq \varphi - \tau \,\,\, \wrtext{on}\,\,\, \overline{\Omega},
\end{equation}
where $0 < \tau < {\rm max}\, \varphi - {\rm min}\, \varphi$. Let $0 \leq t \leq e^{-\tau/h}$ be
a singular value of $P$ in {\rm (\ref{eq1.3})}, or equivalently of $h\overline{\partial}:L^2_\varphi (M, \omega) \to L^2_\varphi (M, T^*_{0,1}M)$,
and let $u\in L^2_{\varphi}(M,\omega)$ be a corresponding normalized singular state of $h\overline{\partial}$, so that we have {\rm (\ref{eq1.13})}. Then for every $\varepsilon > 0$, we have
\begin{equation}
\label{eq1.26}
\norm{\e^{-\psi/h} u}_{L^2(\Omega,\omega)} \leq \mathcal O_{\varepsilon} (1)\,
\e^{\varepsilon/h}.
\end{equation}
This holds uniformly with respect to $(u,t)$.
\end{prop}
\begin{proof}
We have in view of (\ref{eq1.25}) and the second equation in (\ref{eq1.13}),
\begin{equation}
\label{eq1.27}
\norm{e^{-\psi/h}h\overline{\partial}u}_{L^2(\Omega,T^*_{0,1}\Omega))} 
\leq \norm{e^{-(\varphi - \tau)/h}h\overline{\partial}u}_{L^2(\Omega,T^*_{0,1}\Omega)}
\leq 1.
\end{equation}
Furthermore, the first equation in (\ref{eq1.13}) gives that
\begin{equation}
\label{eq1.28}
\norm{e^{-\varphi/h}u}_{L^2(\Omega,\omega)} \leq 1,
\end{equation}
and using (\ref{eq1.25}) again, we conclude that for every $\varepsilon > 0$ there exists a neighborhood $\Theta_{\varepsilon} \subseteq  \overline{\Omega}$  of $\partial \Omega$ in $\overline{\Omega}$ such that
\begin{equation}
\label{eq1.29}
\norm{1_{\Theta_\varepsilon}e^{-\psi/h}u}_{L^2(\Omega,\omega)} \leq 
\norm{1_{\Theta_\varepsilon}e^{-(\varphi - \varepsilon)/h}u}_{L^2(\Omega,\omega)}
\leq e^{\varepsilon/h}.
\end{equation}
Combining (\ref{eq1.27}), (\ref{eq1.29}), and Proposition \ref{propCH}, we obtain (\ref{eq1.26}). This completes the proof.
\end{proof}

\medskip
\noindent
{\it Remark}. Let $\psi$, $\tau$ be as in Proposition \ref{appl_sstate}. In view of (\ref{eq1.26}), the interesting case occurs when $\psi \leq \varphi$ in $\Omega$.

\bigskip
\noindent
We shall next review briefly H\"ormander's $L^2$ estimates and existence theorems for the $\overline{\partial}$ operator on the Riemann surface $M$, on the level of 0-forms. See also~\cite[Chapter 11]{Va11},~\cite{Ch15}. Let $\mu \in \Omega^{0,1}(M)$ be such that ${\rm supp}\, \mu$ is a compact subset of a local holomorphic coordinate chart $(U,z)$, and let us write $\mu = m(z)\, d\overline{z}$. It follows from (\ref{eq1.4.0.2}) that we have in $U$, 
\begeq
\label{eq1.29.1}
(P^* \mu)\,\omega = \left(-h \frac{\partial m}{\partial z} + \frac{\partial \varphi}{\partial z} m\right)\, \frac{d\overline{z}\wedge dz}{2i},
\endeq
and recalling also (\ref{eq:ext1}), it will be convenient to write in the local chart $(U,z)$, 
\begeq
\label{eq1.29.2}
\omega = e^{-2f}\, \frac{d\overline{z}\wedge dz}{2i}, \quad f = f_U \in C^{\infty}(U).
\endeq
Let also $\nu \in \Omega^{0,1}(M)$ and let us write $\nu = n(z)\, d\overline{z}$ locally in the chart $(U,z)$. We have therefore, in view of (\ref{eq:ext3a}), (\ref{eq1.29.1}), and (\ref{eq1.29.2}), 
\begin{multline}
\label{eq1.29.3} 
(P^* \mu, P^* \nu)_{L^2(M, \omega)} = \iint_U e^{2f} \left(-h \frac{\partial m}{\partial z} + \frac{\partial \varphi}{\partial z} m\right) 
\overline{\left(-h \frac{\partial n}{\partial z} + \frac{\partial \varphi}{\partial z} n\right) }\, \frac{d\overline{z}\wedge dz}{2i} \\
= \iint_U \left[\left(h \frac{\partial}{\partial \overline{z}} + \frac{\partial \varphi}{\partial \overline{z}}\right) \left(e^{2f} \left(-h \frac{\partial m}{\partial z} + \frac{\partial \varphi}{\partial z} m\right)\right)\right]\, \overline{n}\,\frac{d\overline{z}\wedge dz}{2i}. 
\end{multline}
Here 
\begin{multline}
\label{eq1.29.4} 
\left(h \frac{\partial}{\partial \overline{z}} + \frac{\partial \varphi}{\partial \overline{z}}\right) \circ e^{2f} \circ \left(-h \frac{\partial }{\partial z} + \frac{\partial \varphi}{\partial z}\right) \\ 
= e^{f} \circ \left(e^{-f} \circ \left(h \frac{\partial}{\partial \overline{z}} + \frac{\partial \varphi}{\partial \overline{z}}\right) \circ e^{f}\right) \left(e^{f}\circ \left(-h \frac{\partial }{\partial z} + \frac{\partial \varphi}{\partial z}\right)\circ e^{-f}\right) \circ e^{f} \\
= e^{f} \circ \left(h \frac{\partial}{\partial \overline{z}} + \frac{\partial \varphi}{\partial \overline{z}} + h\frac{\partial f}{\partial \overline{z}}\right) \left(-h \frac{\partial }{\partial z} + \frac{\partial \varphi}{\partial z} + h\frac{\partial f}{\partial z}\right)\circ e^{f},
\end{multline}
and we get using (\ref{eq1.29.3}), (\ref{eq1.29.4}),
\begin{multline}
\label{eq1.29.5}
(P^* \mu, P^* \nu)_{L^2(M, \omega)} \\
= \iint_U \left[\left(h \frac{\partial}{\partial \overline{z}} + \frac{\partial \varphi}{\partial \overline{z}} + h \frac{\partial f}{\partial \overline{z}}\right) \left(-h \frac{\partial }{\partial z} + \frac{\partial \varphi}{\partial z} + h \frac{\partial f}{\partial z}\right) (e^{f}m)\right]\, (e^{f} \overline{n})\, \frac{d\overline{z}\wedge dz}{2i}.
\end{multline}

\medskip
\noindent
We shall next make the following general observation. Let $\beta \in \Omega^{0,1}(M)$ and let us write 
$\beta = u(z)\, d\overline{z}$, locally in the chart $(U,z)$. Proceeding similarly to~\cite[Chapter 11]{Va11}, we claim that the expression 
\begeq
\label{eq1.29.5.1}
\overline{\nabla} \beta := \left(h\frac{\partial u}{\partial \overline{z}} + \frac{\partial \varphi}{\partial \overline{z}} u + 2h\frac{\partial f}{\partial \overline{z}} u\right)\, d\overline{z} \otimes d\overline{z}
\endeq
defines a global section of the complex line bundle $T^*_{0,1}M \otimes T^*_{0,1}M$. Indeed, in the intersection of two local holomorphic coordinate charts $(V,w)$ and $(U,z)$, we have $w = \chi(z)$, where $\chi$ is holomorphic and writing $\beta = u(z)\, d\overline{z} = v(w)\, d\overline{w}$, we get 
\begeq
\label{eq1.29.5.2}
u(z) = v(w) \overline{\chi'(z)}.
\endeq 
It follows from (\ref{eq1.29.5.2}) that 
\begeq
\label{eq1.29.5.3}
\frac{\partial u}{\partial \overline{z}}(z) = \frac{\partial v}{\partial \overline{w}}(w) (\overline{\chi'(z)})^2 + v(w) \overline{\chi''(z)},
\endeq
and we also have 
\begeq
\label{eq1.29.5.4}
\frac{\partial \varphi}{\partial \overline{z}}(z) = \frac{\partial \varphi}{\partial \overline{w}}(w) \overline{\chi'(z)}. 
\endeq
Finally, expressing the area form $\omega$ in (\ref{eq1.29.2}) in the $w$--coordinates, we get 
\begeq
\label{eq1.29.5.4.1}
2 f_U(z) = 2f_V(w) - \log \abs{\chi'(z)}^2, 
\endeq
implying that 
\begeq
\label{eq1.29.5.5}
2 \frac{\partial f_U}{\partial \overline{z}}(z) = 2 \frac{\partial f_V}{\partial \overline{w}}(w) \overline{\chi'(z)} - \frac{\chi'(z)}{\abs{\chi'(z)}^2} \overline{\chi''(z)}.
\endeq
We get therefore, combining (\ref{eq1.29.5.2}), (\ref{eq1.29.5.3}), (\ref{eq1.29.5.4}), and (\ref{eq1.29.5.5}), in the intersection of the local holomorphic coordinate charts $(V,w)$ and $(U,z)$, 
\[
\left(h\frac{\partial u}{\partial \overline{z}} + \frac{\partial \varphi}{\partial \overline{z}} u + 2h\frac{\partial f_U}{\partial \overline{z}} u\right)\, d\overline{z} \otimes d\overline{z} = \left(h\frac{\partial v}{\partial \overline{w}} + \frac{\partial \varphi}{\partial \overline{w}} v + 2h\frac{\partial f_V}{\partial \overline{w}} v\right)\, d\overline{w} \otimes d\overline{w}.
\]
This establishes the claim and a partition of unity argument shows that the local expression (\ref{eq1.29.5.1}) gives rise to a globally well defined linear map 
\begeq
\label{eq1.29.5.6}
\overline{\nabla}: \Omega^{0,1}(M) \rightarrow C^{\infty}(M; T^*_{0,1}M \otimes T^*_{0,1}M).
\endeq 
The fibers of the complex line bundle $T^*_{0,1}M \otimes T^*_{0,1}M$ carry a natural Hermitian scalar product induced by the Hermitian form in (\ref{eq_ext03}), which in a local holomorphic coordinate chart $(U,z)$ is given by 
\begeq
\label{eq1.29.5.6.1}
(u_1 d\overline{z} \otimes d\overline{z}, u_2 d\overline{z} \otimes d\overline{z})_z = e^{4f(z)} u_1(z) \overline{u_2(z)}.
\endeq
Here the locally defined function $f\in C^{\infty}(U)$ has been introduced in (\ref{eq1.29.2}). Using (\ref{eq1.29.5.1}), (\ref{eq1.29.5.6.1}) we see that the pointwise scalar product of the sections $\overline{\nabla} \mu$ and $\overline{\nabla} \nu$ is a well defined function on $M$ given by 
\begeq 
\label{eq1.29.5.7}
(\overline{\nabla} \mu, \overline{\nabla} \nu)_z  
= e^{4f(z)} \left(h\frac{\partial m}{\partial \overline{z}} + \frac{\partial \varphi}{\partial \overline{z}} m + 2h\frac{\partial f}{\partial \overline{z}} m\right) \overline{\left(h\frac{\partial n}{\partial \overline{z}} + \frac{\partial \varphi}{\partial \overline{z}} n + 2h\frac{\partial f}{\partial \overline{z}} n\right)}. 
\endeq
Here we recall that the support of $\mu \in \Omega^{0,1}(M)$ is a compact subset of a local holomorphic coordinate chart $(U,z)$. Writing 
\[ 
h\frac{\partial}{\partial \overline{z}} + \frac{\partial \varphi}{\partial \overline{z}} + 2h\frac{\partial f}{\partial \overline{z}} = \
e^{-f}\circ \left(h\frac{\partial}{\partial \overline{z}} + \frac{\partial \varphi}{\partial \overline{z}} + h\frac{\partial f}{\partial \overline{z}}\right)\circ e^{f},
\] 
we get using (\ref{eq1.29.5.7}), 
\begeq
\label{eq1.29.5.8}
(\overline{\nabla} \mu, \overline{\nabla} \nu)_z  
= e^{2f(z)} \left(h\frac{\partial }{\partial \overline{z}} + \frac{\partial \varphi}{\partial \overline{z}} + h\frac{\partial f}{\partial \overline{z}} \right)(e^f m) \overline{\left(h\frac{\partial}{\partial \overline{z}} + \frac{\partial \varphi}{\partial \overline{z}} + h\frac{\partial f}{\partial \overline{z}}\right)(e^f n)}, 
\endeq
and it follows from (\ref{eq1.29.2}), (\ref{eq1.29.5.8}) that 
\begin{multline}
\label{eq1.29.5.9}
(\overline{\nabla} \mu, \overline{\nabla} \nu)_{L^2(M,\omega)} := \int_M (\overline{\nabla} \mu, \overline{\nabla} \nu)_z \,\omega \\ 
= \iint_U \left(h\frac{\partial }{\partial \overline{z}} + \frac{\partial \varphi}{\partial \overline{z}} + h\frac{\partial f}{\partial \overline{z}} \right)(e^f m) \overline{\left(h\frac{\partial}{\partial \overline{z}} + \frac{\partial \varphi}{\partial \overline{z}} + h\frac{\partial f}{\partial \overline{z}}\right)(e^f n)}\, \frac{d\overline{z}\wedge dz}{2i} \\
= \iint_U \left[\left(-h\frac{\partial}{\partial {z}} + \frac{\partial \varphi}{\partial {z}} + h\frac{\partial f}{\partial {z}}\right)
\left(h\frac{\partial }{\partial \overline{z}} + \frac{\partial \varphi}{\partial \overline{z}} + h\frac{\partial f}{\partial \overline{z}} \right)(e^f m)\right] (e^f \overline{n})\, \frac{d\overline{z}\wedge dz}{2i}. 
\end{multline} 
Subtracting (\ref{eq1.29.5.9}) from (\ref{eq1.29.5}) gives that 
\begin{multline}
\label{eq1.29.6}
(P^* \mu, P^* \nu)_{L^2(M, \omega)} - (\overline{\nabla} \mu, \overline{\nabla} \nu)_{L^2(M,\omega)} \\
= \iint_U \left[\left(h \frac{\partial}{\partial \overline{z}} + \frac{\partial \varphi}{\partial \overline{z}} + h \frac{\partial f}{\partial \overline{z}}\right) \left(-h \frac{\partial }{\partial z} + \frac{\partial \varphi}{\partial z} + h \frac{\partial f}{\partial z}\right) (e^{f}m)\right]\, (e^{f} \overline{n})\, \frac{d\overline{z}\wedge dz}{2i} \\
-  \iint_U \left[\left(-h \frac{\partial }{\partial z} + \frac{\partial \varphi}{\partial z} + h \frac{\partial f}{\partial z}\right) 
\left(h \frac{\partial}{\partial \overline{z}} + \frac{\partial \varphi}{\partial \overline{z}} + h \frac{\partial f}{\partial \overline{z}}\right) (e^{f}m)\right]\, (e^{f} \overline{n})\, \frac{d\overline{z}\wedge dz}{2i} \\
= \iint_U \left(\left[h \frac{\partial}{\partial \overline{z}} + \frac{\partial \varphi}{\partial \overline{z}} + h \frac{\partial f}{\partial \overline{z}}, -h \frac{\partial }{\partial z} + \frac{\partial \varphi}{\partial z} + h \frac{\partial f}{\partial z}\right] e^{f} m\right)\, (e^{f} \overline{n})\, \frac{d\overline{z}\wedge dz}{2i} \\
= 2h \iint_U \left(\partial_z \partial_{\overline{z}}\, \varphi + h \partial_z \partial_{\overline{z}}\, f\right) e^{2f} m\overline{n}\,  \frac{d\overline{z}\wedge dz}{2i}. 
\end{multline} 
Here, similarly to (\ref{eq:ext11.3}), 
\[
2 \partial_z \partial_{\overline{z}}\, \varphi\, \frac{d\overline{z}\wedge dz}{2i} = i \partial \overline{\partial} \varphi
\]
is a globally defined $(1,1)$ form, and we remark that the $(1,1)$ form $i \partial \overline{\partial} f$ is globally defined as well, since in the intersection of the local holomorphic coordinate charts $V$ and $U$ on $M$, the difference $f_U - f_V$ is harmonic, see (\ref{eq1.29.5.4.1}). It follows therefore from (\ref{eq1.29.6}) that
\begin{multline}
\label{eq1.29.7}
(P^* \mu, P^* \nu)_{L^2(M, \omega)} - (\overline{\nabla} \mu, \overline{\nabla} \nu)_{L^2(M,\omega)} \\
= h \iint_U e^{2f} m \overline{n} \left(i \partial \overline{\partial} \varphi + ih \partial \overline{\partial} f\right) 
= \frac{h}{2} \iint_U \Delta(\varphi + hf) e^{2f} m \overline{n}\, \omega \\ 
= \frac{h}{2} \iint_U \Delta(\varphi + hf) m \overline{n}\, \frac{d\overline{z}\wedge dz}{2i} = \frac{h}{2} \int_M \Delta(\varphi + hf) \frac{\mu \wedge \overline{\nu}}{2i}.
\end{multline} 
Here we have also used (\ref{eq_Laplace}). By a partition of unity argument, we can now eliminate the support condition on $\mu \in \Omega^{0,1}(M)$, and conclude that 
\begeq
\label{eq1.29.7.1}
(P^* \mu, P^* \nu)_{L^2(M, \omega)} - (\overline{\nabla} \mu, \overline{\nabla} \nu)_{L^2(M,\omega)} = \frac{h}{2} \int_M \Delta(\varphi + hf) \frac{\mu \wedge \overline{\nu}}{2i},
\endeq
for all $\mu, \nu \in \Omega^{0,1}(M)$. See also~\cite[Theorem 11.2.1]{Va11}.

\medskip
\noindent
{\it Remark.} Assume that $M = \mathbb C /(2\pi \mathbb Z + 2\pi i \mathbb Z)$ is the standard $2$-torus equipped with the usual complex structure and the Euclidean metric. As observed previously, the operators $P$ and $P^*$ in (\ref{eq1.3}) and (\ref{eq:ext8.0}) can then be viewed as scalar operators acting on $C^{\infty}(M)$, and (\ref{eq1.29.7.1}) becomes the following commutator identity, valid for all $u,v \in C^{\infty}(M)$, 
\[
(P^*u, P^*v)_{L^2(M)} - (P u, Pv)_{L^2(M)} = ([P,P^*]u,v)_{L^2(M)} = \frac{h}{2} (\Delta \varphi\, u,v)_{L^2(M)}.
\]

\medskip
\noindent 
Let $W \subseteq M$ be open such that $\varphi$ is strictly subharmonic in a neighborhood of $\overline{W}$, so that $\Delta \varphi \geq \alpha$ near $\overline{W}$, for some $\alpha > 0$. In what follows, we shall write $\mu \in \mathscr D_{0,1}(W)$ if $\mu \in \Omega^{0,1}(M)$ is such that ${\rm supp}\, \mu$ is a compact subset of $W$. It follows from (\ref{eq1.29.7.1}) that for all $h>0$ small enough and all $\mu \in \mathscr D_{0,1}(W)$, we have 
\begeq
\label{eq1.29.13}
h^{1/2} \|\mu\|_{L^2(W,T^*_{0,1}W)} \leq \mathcal O(1) \|P^* \mu\|_{L^2(M,\omega)}.
\endeq 
Combining (\ref{eq1.29.13}) with classical arguments involving the Hahn-Banach theorem and the Riesz representation theorem, we obtain the following solvability result for $\overline{\partial}$, see also~\cite[Theorem 11.3.1]{Va11}. 

\begin{prop}
\label{exist_dbar}
Let $M$ be a compact Riemann surface equipped with a conformal Riemannian metric $g$ and let $\varphi \in C^{\infty}(M;\mathbb R)$ be non-constant. Let $W\subseteq M$ be an open set such that $\varphi$ is strictly subharmonic in a neighborhood of $\overline{W}$. There exists $h_0 > 0$ such that for all $h \in (0,h_0]$ and all $\beta \in L^2(W, T^*_{0,1} W)$, there exists a solution $u\in L^2(W,\omega)$ of the equation 
\[
e^{-\varphi/h} \circ h\overline{\partial} \circ e^{\varphi/h} u = \beta \quad \wrtext{in}\quad W,
\] 
which satisfies 
\[
\|u\|_{L^2(W,\omega)} \leq \frac{\mathcal O(1)}{h^{1/2}} \|\beta \|_{L^2(W, T^*_{0,1}W)}.
\]
Here $\omega$ is the area (K\"ahler) form associated to $g$. 
\end{prop}

\bigskip
\noindent
We shall finish this section by making some standard computations and remarks. We let $\sigma$ stand for the canonical complex symplectic $(2,0)$-form on $T^*_{1,0}M$. In the natural local holomorphic coordinates $(z,\zeta)$ on $T^*_{1,0}M$ we can write $\sigma$ as 
\begin{equation}
\label{eq1.30.0}
\sigma = d\zeta\wedge dz = d(\zeta\, dz). 
\end{equation}
We have a fiberwise bijection between $T^*_{1,0}M$ and $T^*M$, the cotangent bundle of $M$ in the sense of real smooth manifolds, 
where associated to a $(1,0)$-form $\zeta\, dz$ is the 1-form ${\rm Re}\, ({\zeta}\, dz)$. Writing $z = x + iy$, $\zeta = \xi - i \eta$, we get
\[
{\rm Re}\, ({\zeta}\, dz) = \xi\, dx + \eta\, dy,
\]
so in local coordinates we have a bijection between $(z,\zeta) \in T^*_{1,0}M$ and $(x,y;\xi,\eta) \in T^*M$. Notice that this identification is different from the one used in (\ref{eq1.7}), (\ref{eq1.9}). 

\medskip
\noindent
Associated to the weight $\varphi \in C^{\infty}(M;\mathbb R)$ in (\ref{eq1.3}) and the exponentially weighted space $L^2_{\varphi}(M,\omega) = 
e^{\varphi/h} L^2(M,\omega)$ is the I-Lagrangian manifold, i.e. a Lagrangian manifold for ${\rm Im}\, \sigma$, 
\begin{equation}
\label{eq1.30}
\Lambda_{\varphi} 
= \left\{\left(z,\frac{2}{i}(\partial \varphi)_z \right); \, z\in M, ~(\partial \varphi)_z \in T^*_{z,1,0}M\right\} \subseteq T^*_{1,0} M, 
\end{equation}
and the restriction of $\sigma$ to $\Lambda_{\varphi}$ is given by
\begin{equation}
\label{eq1.31}
\sigma|_{\Lambda_{\varphi}} = 
d\left(\frac{2}{i}\partial \varphi\right) = \frac{2}{i} \overline{\partial} \partial \varphi = 2i \partial \overline{\partial} \varphi = (\Delta \varphi)\,\omega. 
\end{equation}
Here we have also used (\ref{eq_Laplace}). It follows that the open subsets
\begin{equation}
\label{eq2.3}
\left\{\left(z,\frac{2}{i}(\partial \varphi)_z \right); \, z\in M_\pm, 
  ~(\partial \varphi)_z \in T^*_{z,1,0}M_\pm\right\}\subseteq \Lambda_{\varphi},
\end{equation}
where the open sets $M_{\pm} \subseteq M$ are defined by
\begin{equation}\label{eq2.4}
M_{\pm} = \left\{z\in M;\, \pm \Delta \varphi(z) > 0\right\},
\end{equation}
carry the natural symplectic volume forms given by
\begin{equation}
\label{eq2.5}
\pm \sigma|_{\Lambda_{\varphi}} \simeq \pm (\Delta \varphi)\,\omega|_{M_\pm}.
\end{equation}
By Stokes' theorem we have
$$
\int_M \Delta \varphi\, \omega = 0,
$$
and therefore, identifying the sets in (\ref{eq2.3}) with $M_{\pm}$ via the natural projection map $\pi_z: \Lambda_{\varphi} \rightarrow M$, we get that
\begin{equation}
\label{eq2.6}
{\rm Vol}(M_+) = \int_{M_+} \Delta \varphi\, \omega = - \int_{M_-} \Delta \varphi\, \omega= {\rm Vol}(M_-).
\end{equation}

\section{Bounds on the number of exponentially small singular values}
\label{sec:bounds}

\medskip
\noindent
We shall start this section by making some general remarks concerning the notions of upper bound and lower bound weights, as introduced in Definition \ref{def_ubw_n} and Definition \ref{def_lbw_n}. First, let $\psi_{\alpha}$, $\alpha \in J$, be a finite family of upper bound weights. Then the infimum of $\psi_{\alpha}$ is an upper bound weight and we have
\begin{equation}
\label{eq3.00}
M_+\left(\inf_{\alpha \in J} \psi_{\alpha}\right) = \bigcap_{\alpha \in J} M_+(\psi_{\alpha}).
\end{equation}
Analogously, let $\psi_{\alpha}$, $\alpha \in J$, be a finite family of lower bound weights. Then the supremum of $\psi_{\alpha}$ is a lower bound weight and we have
\begin{equation}
\label{eq3.02}
M_+\left(\sup_{\alpha \in J} \psi_{\alpha}\right) = \bigcup_{\alpha \in J} M_+(\psi_{\alpha}).
\end{equation}
Here the notion of the contact set $M_+(\psi_{\alpha})$ has been introduced in (\ref{eq2.7_b}).

\subsection{Upper bounds}
Let $0 < \tau < {\rm max}\, \varphi - {\rm min}\, \varphi$ be fixed and let $\psi \in C(M;\mathbb R)$ be an upper bound weight in the sense
of Definition \ref{def_ubw_n}. The purpose of this subsection is to prove Theorem \ref{thm:UpperBd}. More precisely, we wish to establish
an upper bound on the number of singular values of the operator
\[
h\overline{\partial}: L^2_{\varphi}(M,\omega) \rightarrow L^2_{\varphi}(M, T^*_{0,1}M)
\] 
on the interval $[0,e^{-\tau/h}]$, in terms of the weight $\psi$. When doing so, we shall assume that the contact set $M_+(\psi)$ defined in (\ref{eq2.7_b}) satisfies
\begin{equation}
\label{eq3.1.0.0.1}
M_+(\psi) \subseteq M_+ = \{x\in M; \Delta \varphi(x) > 0\}.
\end{equation}
Let us also introduce the open set
\begin{equation}
\label{eq3.1}
\Omega = M_+(\psi)^{\complement} = \{x\in M;\, \psi(x) < \varphi(x)\},
\end{equation}
where $\psi$ is superharmonic. Following Proposition \ref{appl_sstate}, we shall assume that ${\rm ext}(\Omega) = {\rm int}(M_+(\psi)) \neq \emptyset$.

\medskip
\noindent
Let $0\leq t\leq \e^{-\tau/h}$ be a singular value of $h\overline{\partial}$ and let $u\in L^2_{\varphi}(M,\omega)$ be a corresponding normalized singular state as in \eqref{eq1.12}, \eqref{eq1.13}. We have (\ref{eq1.25}), and Proposition \ref{appl_sstate} shows therefore that
\[
\norm{e^{-\psi/h} u}_{L^2(\Omega,\omega)} = \mathcal{O}_\varepsilon(1)\, \e^{\varepsilon/h},
\]
for every $\varepsilon>0$. It follows that for each compact set $K \subseteq \Omega$ there exists $\delta > 0$ such that
\begin{equation}
\label{eq3.1.0.1}
\norm{e^{-\varphi/h} u}_{L^2(K,\omega)} = \mathcal{O}(1)\, \e^{-\delta/h}.
\end{equation}
In particular, (\ref{eq3.1.0.1}) holds when $K = M \setminus V$, where $V \subseteq M$ is an arbitrary open neighborhood of $M_+(\psi)$.

\bigskip
\noindent
Next, let $W \subseteq M$ be open such that $\varphi$ is strictly subharmonic in a neighborhood of $\overline{W}$. We have therefore 
\begeq
\label{eq3.1.0.2}
\Delta \varphi \geq \alpha \quad \wrtext{on} \quad \overline{W},
\endeq
for some $\alpha >0$. Let us set 
\begin{equation}
\label{eq3.1.0.3} 
H_{\varphi}(W) = {\rm Hol}(W) \cap L^2_{\varphi}(W,\omega), 
\end{equation}
and observe that for each $K \subseteq W$ compact there is a constant $C = C(K,\varphi,\omega,h) > 0$ such that 
\[
\|f\|_{L^{\infty}(K)} \leq C \|f\|_{L^2_{\varphi}(W,\omega)},\quad f \in H_{\varphi}(W).
\]
It follows that the subspace $H_{\varphi}(W) \subseteq L^2_{\varphi}(W,\omega)$ is closed and we may introduce the corresponding orthogonal (Bergman) projection 
\begeq
\label{eq3.1.0.4}
\Pi: L^2_{\varphi}(W,\omega) \rightarrow H_{\varphi}(W).
\endeq
An application of the Schwartz kernel theorem shows that the operator $\Pi$ is of the form 
\begeq
\label{eq3.1.0.5} 
\Pi u(x) = \int_W K(x,y) u(y) e^{-2\varphi(y)/h}\, \omega(y,dy\,d\overline{y}),
\endeq
where $K \in {\rm Hol}(W \times W^{\dagger})$ is such that $W\ni y \mapsto \overline{K(x,y)} \in H_{\varphi}(W)$ for each $x\in W$, see~\cite[Section 5]{RSV}, ~\cite[Appendix A]{CoHiSj}. Here for a complex manifold $Y$, the notation $Y^{\dagger}$ stands for the manifold with the complex conjugate structure. Following~\cite{BBSj},~\cite{RSV},~\cite{HiSt}, we shall now review some basic facts concerning the asymptotic behavior of the Bergman kernel $K(x,y)$ on $W$, in the semiclassical limit $h \rightarrow 0^+$. These results will play a crucial role in the proof of Theorem \ref{thm:UpperBd}. 

\medskip
\noindent
Let $p \in W$ and let $(U,x)$ be a local holomorphic coordinate chart around $p$ such that $U \Subset W$ and $x(p) = 0$. When working locally near $p$, we shall identify the coordinate chart with the corresponding open neighborhood of $0$ in $\mathbb C_x$. When $\Omega \subseteq \mathbb C$ is open, we let $\rho(\Omega) = \{x\in \mathbb C;\, \overline{x} \in \Omega\}$, and in what follows, we shall say that a function $f  = f(x,\theta) \in  C^{\infty}(\Omega \times \rho(\Omega))$ is almost holomorphic along the anti-diagonal $\theta = \overline{x}$ if we have for all $N$,
\[
\partial_{\overline{x}} f(x,\theta) = \mathcal O_N(\abs{\theta - \overline{x}}^N), \quad \partial_{\overline{\theta}} f(x,\theta) = \mathcal O_N(\abs{\theta - \overline{x}}^N),
\]
uniformly on compact subsets of $\Omega \times \rho(\Omega)$. We then have the following well known result going back to~\cite{Catlin} and~\cite{Ze98}, see also~\cite{BBSj}, ~\cite{HiSt}. 

\begin{prop}
\label{Bergman_kernel}
There exist open neighborhoods $\widetilde{U} \Subset V \Subset U$ of $0$ such that 
\begeq
\label{eq3.1.0.6}
e^{-\varphi(x)/h}\left(K(x,y) - \frac{1}{h} e^{\frac{2}{h}\Psi(x,\overline{y})} a(x,\overline{y};h)\right)e^{-\varphi(y)/h}  = \mathcal O(h^\infty),
\endeq
uniformly for $x, y\in \widetilde{U}$. Here $\Psi \in  C^{\infty}(V \times \rho(V))$ is almost holomorphic along the anti-diagonal $\theta = \overline{x}$, with $\Psi(x,\overline{x}) = \varphi(x)$, $x\in V$, and 
\begeq
\label{eq3.1.0.7} 
a(x,\theta;h) \sim \sum_{j=0}^{\infty} h^j a_j(x,\theta), \quad h \rightarrow 0^+, 
\end{equation}
in $C^{\infty}(V\times \rho(V))$. Each function $a_j \in C^{\infty}(V\times \rho(V))$ is almost holomorphic along the anti-diagonal $\theta = \overline{x}$, and we have 
\begeq
\label{eq3.1.0.8} 
a_0(x,\overline{x}) = \frac{1}{2\pi} \Delta_g \varphi(x), \quad x\in V.
\endeq 
\end{prop} 
\begin{proof}
In the case of Bergman projections associated to high powers of a positive complex line bundle over a compact complex manifold, this result has been established in~\cite{BBSj}. The present setting of exponentially weighted spaces of holomorphic functions on an open Riemann surface is quite analogous, and it will therefore be sufficient to recall the main steps of the argument of~\cite{BBSj}, indicating the minor modifications required. Our starting point is~\cite[Proposition 2.4, Lemma 2.6]{BBSj}, which gives that for all $u\in H_{\varphi}(U)$, we have 
\begin{multline}
\label{eq3.1.0.9}
u(x) = \frac{1}{2\pi h} \int\!\!\!\int_{\Gamma_{V_2}} e^{\frac{2}{h}(\Psi(x,\theta) - \Psi(y,\theta))} \chi(y) m(x,y,\theta) u(y)\, d\theta \wedge dy \\ 
+ \mathcal O(h^{\infty}) e^{\varphi(x)/h} \|u\|_{H_{\varphi}(U)}, \quad x\in V_1.
\end{multline} 
Here $V_1 \Subset V_2 \Subset U$ are small open (convex) neighborhoods of $0$, $\chi\in C^{\infty}_0(V_2; [0,1])$ is such that $\chi =1$ in a neighborhood of $\overline{V_1}$, the contour $\Gamma_{V_2}$ is given by $\theta = \overline{y}$, $y\in  V_2$, and $\Psi \in  C^{\infty}(V_ 2\times \rho(V_ 2))$ is almost holomorphic along $\theta = \overline{x}$, with $\Psi(x,\overline{x}) = \varphi(x)$, $x\in V_2$. The amplitude $m(x,y,\theta)$ is $h$--independent and is given by 
\begeq
\label{eq3.1.0.10}
m(x,y,\theta) = \frac{2}{i}\int_0^1 \Psi''_{x\theta} (tx + (1-t)y,\theta)\, dt. 
\endeq
Recalling the expression (\ref{eq1.29.2}) for the area form $\omega$ in $U$, let us rewrite (\ref{eq3.1.0.9}) as follows, 
\begin{multline} 
\label{eq3.1.0.11}
u(x) = \frac{1}{2\pi h} \int\!\!\!\int_{\Gamma_{V_2}} e^{\frac{2}{h}(\Psi(x,\theta) - \Psi(y,\theta) - hF(y,\theta))} \chi(y) c(x,y,\theta;h) u(y)\, d\theta \wedge dy \\ 
+ \mathcal O(h^{\infty}) e^{\varphi(x)/h} \|u\|_{H_{\varphi}(U)}, \quad x\in V_1.
\end{multline} 
Here $F = F(y,\theta)$ is almost holomorphic along $\theta = \overline{y}$ with $F(y,\overline{y}) = f(y)$, and 
\begeq
\label{eq3.1.0.12} 
c(x,y,\theta;h) = m(x,y,\theta) e^{2F(y,\theta)}.
\endeq
Proceeding as in~\cite{BBSj}, we are therefore led to the general problem of eliminating the $y$-dependence in the amplitude of an operator of the form 
\begin{equation}
\label{eq3.1.0.13}
Au(x) = \frac{1}{2 \pi h} \int\!\!\!\int_{\Gamma_{V_2}} e^{\frac{2}{h}(\Psi(x,\theta) - \Psi(y,\theta) - hF(y,\theta))} \chi(y) c(x,y,\theta;h) u(y)\, d\theta \wedge dy, \quad u\in H_{\varphi}(U).
\end{equation}
Here the amplitude $c(x,y,\theta;h)$ is given by 
\[
c(x,y,\theta;h) \sim \sum_{j=0}^{\infty} h^j c_j(x,y,\theta),
\]
where $c_j \in C^{\infty}(V_2 \times V_2 \times \rho(V_2))$ are such that we have for each $N$,
\begin{equation}
\label{eq3.1.0.14}
\overline{\partial}c_j(x,y,\theta) = \mathcal O_{j,N}(\abs{x-\overline{\theta}}^N + \abs{y-\overline{\theta}}^N + \abs{x-y}^N).
\end{equation}
Here one of the three terms in the right hand side of (\ref{eq3.1.0.14}) can be dropped, by the triangle inequality. As in~\cite{BBSj}, we let $\mathcal A$ be the algebra of $C^{\infty}$ functions satisfying (\ref{eq3.1.0.14}), and let us also introduce the ideal $\mathcal I^{\infty}$ consisting of $C^{\infty}$ functions vanishing to infinite order along $x = y = \overline{\theta}$. We claim that there exists a classical symbol $b(x,y,\theta;h) \sim \sum_{j=0}^{\infty} h^j b_j(x,y,\theta)$, with $b_j \in \mathcal A$, $j\geq 0$, such that for some $\widetilde{c}(x,\theta;h)\sim \sum_{j=0}^{\infty} h^j \widetilde{c}_j(x,\theta)$, with $\widetilde{c}_j$ almost holomorphic along $\theta = \overline{x}$, we have
\begin{multline}
\label{eq3.1.0.15}
c(x,y,\theta;h) - \widetilde{c}(x,\theta;h) \\ 
= e^{-\frac{2}{h}(\Psi(x,\theta) - \Psi(y,\theta) - hF(y,\theta))} h\partial_{\theta} \left(e^{\frac{2}{h}(\Psi(x,\theta) - \Psi(y,\theta)-hF(y,\theta))} b(x,y,\theta;h)\right),
\end{multline}
modulo a symbol of the form $\sum_{j=0}^{\infty} h^j f_j$, with $f_j \in \mathcal I^{\infty}$, for all $j\geq 0$. Indeed, rewriting (\ref{eq3.1.0.15}) in the form
\begin{multline*}
c(x,y,\theta;h) - \widetilde{c}(x,\theta;h) \\ 
= 2\left(\Psi'_{\theta}(x,\theta) - \Psi'_{\theta}(y,\theta)\right)b(x,y,\theta;h) - 2h F'_{\theta}(y,\theta) b(x,y,\theta) + h\partial_{\theta}b(x,y,\theta;h),
\end{multline*} 
we get a sequence of division problems,
\begin{equation}
\label{eq3.1.0.16}
c_0(x,y,\theta) - \widetilde{c}_0(x,\theta) = 2\left(\Psi'_{\theta}(x,\theta) - \Psi'_{\theta}(y,\theta)\right)b_0(x,y,\theta), \quad {\rm mod}\,\,\mathcal I^{\infty},
\end{equation}
and for $j\geq 1$,
\begin{multline}
\label{eq3.1.0.17}
c_j(x,y,\theta) - \widetilde{c}_j(x,\theta) - (\partial_{\theta} - 2F'_{\theta}(y,\theta)) b_{j-1}(x,y,\theta)  \\ 
= 2\left(\Psi'_{\theta}(x,\theta) - \Psi'_{\theta}(y,\theta)\right)b_j(x,y,\theta), \quad {\rm mod}\,\,\mathcal I^{\infty}.
\end{multline}
Here in (\ref{eq3.1.0.16}) we take $\widetilde{c}_0(x,\theta) = c_0(x,x,\theta)$, and using that $\Psi''_{y\theta}(y,\overline{y}) = \varphi''_{y\overline{y}}(y) \neq 0$, $y\in V_2$, we see, as in~\cite{BBSj}, that there exists $b_0 \in {\mathcal A}$ such that (\ref{eq3.1.0.16}) holds. Similarly, we find $\widetilde{c}_j = \widetilde{c}_j(x,\theta)$ almost holomorphic along $\theta = \overline{x}$ and $b_j \in \mathcal A$, $j\geq 1$, such that (\ref{eq3.1.0.17}) holds. This establishes the claim in (\ref{eq3.1.0.15}). 

\medskip
\noindent
We next recall from~\cite[Section 2]{BBSj},~\cite[Proposition 2.1]{HiSt} that the classical estimate 
\begin{equation}
\label{eq3.1.0.18}
\varphi(x) + \varphi(y) - 2{\rm Re}\, \Psi(x,\overline{y}) \asymp \abs{x-y}^2, \quad x,y \in V_2, 
\end{equation}
holds, in view of (\ref{eq3.1.0.2}). It follows from (\ref{eq3.1.0.18}) that the contribution of $f = f(x,y,\theta;h) \sim \sum_{j=0}^{\infty} h^j f_j(x,y,\theta)$,  with $f_j \in \mathcal I^{\infty}$, to the amplitude of the operator $A$ in (\ref{eq3.1.0.13}) gives rise to an error term of the form 
\begeq
\label{eq3.1.0.18.1}
\mathcal O(h^{\infty}) e^{\varphi(x)/h} \|u\|_{H_{\varphi}(U)}, \quad x\in V_2.
\endeq
Using (\ref{eq3.1.0.13}) and (\ref{eq3.1.0.15}), we get therefore by Stokes' formula for all $N$, modulo an error term (\ref{eq3.1.0.18.1}), 
\begin{multline}
\label{eq3.1.0.19}
2\pi h Au(x) - \int\!\!\!\int_{\Gamma_{V_2}} e^{\frac{2}{h}(\Psi(x,\theta) - \Psi(y,\theta) - hF(y,\theta))} \chi(y) \widetilde{c}(x,\theta;h) u(y)\, d\theta \wedge dy \\ 
= \int\!\!\!\int_{\Gamma_{V_2}} e^{\frac{2}{h}(\Psi(x,\theta) - \Psi(y,\theta)-hF(y,\theta))} \chi(y) (c(x,y,\theta;h) - \widetilde{c}(x,\theta;h))u(y) d\theta \wedge dy \\
= \int\!\!\!\int_{\Gamma_{V_2}} \chi(y) u(y) hd\left(e^{\frac{2}{h}(\Psi(x,\theta) - \Psi(y,\theta) - hF(y,\theta))} b(x,y,\theta)\right)\wedge dy \\
+ \int\!\!\!\int_{\Gamma_{V_2}} \chi(y) u(y) e^{\frac{2}{h}(\Psi(x,\theta) - \Psi(y,\theta) - hF(y,\theta))} \mathcal O_N(\abs{x-\overline{\theta}}^N)\, d\theta \wedge dy \\
= \int\!\!\!\int_{\Gamma_{V_2}} \chi(y) hd\left( u(y)e^{\frac{2}{h}(\Psi(x,\theta) - \Psi(y,\theta) - hF(y,\theta))} b(x,y,\theta)\, dy\right) \\
+ \int\!\!\!\int_{\Gamma_{V_2}} \chi(y) u(y) e^{\frac{2}{h}(\Psi(x,\theta) - \Psi(y,\theta) - hF(y,\theta))} \mathcal O_N(\abs{x-\overline{\theta}}^N)\, d\theta \wedge dy \\
= -\int\!\!\!\int_{\Gamma_{V_2}} u(y)e^{\frac{2}{h}(\Psi(x,\theta) - \Psi(y,\theta) - hF(y,\theta))} b(x,y,\theta) hd\chi \wedge dy \\ 
+ \int\!\!\!\int_{\Gamma_{V_2}} e^{\frac{2}{h}(\Psi(x,\theta) - \Psi(y,\theta))} \mathcal O_N(\abs{x-\overline{\theta}}^N) u(y)\, d\theta \wedge dy \\
= \mathcal O(h^{\infty})\, e^{\varphi(x)/h} \|u\|_{H_{\varphi}(U)}, \quad x\in V_1.
\end{multline}
Here we have used that $u$ is holomorphic and the last equality follows by an application of (\ref{eq3.1.0.18}), using also that $d\chi$ vanishes in a neighborhood of $\overline{V_1}$. 

\medskip
\noindent 
We get therefore for $x\in V_1$, combining (\ref{eq3.1.0.11}) and (\ref{eq3.1.0.19}), 
\begin{multline} 
\label{eq3.1.0.20}
u(x) = \frac{1}{2\pi h} \int\!\!\!\int_{\Gamma_{V_2}} e^{\frac{2}{h}(\Psi(x,\theta) - \Psi(y,\theta) - hF(y,\theta))} \chi(y) \widetilde{c}(x,\theta;h) u(y)\, d\theta \wedge dy \\ 
+ \mathcal O(h^{\infty}) e^{\varphi(x)/h} \|u\|_{H_{\varphi}(U)} \\
= \frac{1}{h} \int_{V_2} e^{\frac{2}{h} \Psi(x,\overline{y})} \chi(y) a(x,\overline{y};h) u(y) e^{-\frac{2 \varphi(y)}{h}} \,
\omega(y, dy d\overline{y}) 
+ \mathcal O(h^{\infty})\, e^{\varphi(x)/h} \|u\|_{H_{\varphi}(U)}.
\end{multline}
Here $a(x,\theta;h) = (1/2\pi) 2i\, \widetilde{c}(x,\theta;h)$ and we have also used (\ref{eq1.29.2}). It follows from (\ref{eq3.1.0.10}) and (\ref{eq3.1.0.12}) that 
\begin{multline} 
\label{eq3.1.0.21} 
a(x,\theta;h) = \frac{1}{2\pi} 2i\, \widetilde{c}(x,\theta;h) = \frac{i}{\pi} \widetilde{c}_0(x,\theta) + \mathcal O(h) 
= \frac{i}{\pi} {c}_0(x,x,\theta) + \mathcal O(h) \\ 
= \frac{i}{\pi} m(x,x,\theta) e^{2F(x,\theta)} + \mathcal O(h) = \frac{i}{\pi} \frac{2}{i} \Psi''_{x\theta}(x,\theta) e^{2F(x,\theta)} + \mathcal  O(h).
\end{multline}
Taking the restriction to the anti-diagonal $\theta = \overline{x}$ we get therefore 
\[
a_0(x,\overline{x}) = \frac{2}{\pi} \varphi''_{x\overline{x}} (x) e^{2f(x)} = \frac{1}{2\pi}\, 4\varphi''_{x\overline{x}} (x) e^{2f(x)} = \frac{1}{2\pi} \Delta_g \varphi(x).
\]
Here the last equality follows in view of (\ref{eq_metric0}), (\ref{eq_metric}), and (\ref{eq1.29.2}). 

\medskip
\noindent
Using (\ref{eq3.1.0.20}) and following an essentially well known argument presented in detail in~\cite[Section 3]{BBSj},~\cite[Section 5]{HiSt} as it stands, we obtain (\ref{eq3.1.0.6}), uniformly in a small neighborhood $\widetilde{U} \Subset V_1$ of $0$. Indeed, the argument of~\cite[Section 3]{BBSj},~\cite[Section 5]{HiSt} consists of applying the approximate reproducing property (\ref{eq3.1.0.20}),
\begin{multline*}
u(y) = \frac{1}{h} \int_{V_2} e^{\frac{2}{h} \Psi(y,\overline{z})} \chi(z) a(y,\overline{z};h) u(z) e^{-\frac{2 \varphi(z)}{h}} \,
\omega(z, dz d\overline{z}) \\ 
+ \mathcal O(h^{\infty}) e^{\varphi(y)/h} \|u\|_{H_{\varphi}(U)}, \quad y\in V_1,
\end{multline*}
to the function $y \mapsto K(y,x) \in H_{\varphi}(W)$. Using also (\ref{eq3.1.0.5}), we obtain that 
\begeq
\label{eq3.1.0.22}
K(x,y) = \Pi(\widehat{K}(\cdot, \overline{y})\chi)(x) + \mathcal O(h^{\infty}) e^{\frac{\varphi(x) + \varphi(y)}{h}}, \quad x,y \in \widetilde{U},
\endeq
where 
\[
\widehat{K}(z,\overline{y}) = \frac{1}{h} e^{\frac{2}{h} \Psi(z,\overline{y})} \overline{a(y,\overline{z};h)},
\]
and estimating the function 
\[
W \ni x \mapsto \widehat{K}(x,\overline{y})\chi(x) - \Pi(\widehat{K}(\cdot, \overline{y})\chi)(x)
\]
in $L^2_{\varphi}(W,\omega)$ by means of Proposition \ref{exist_dbar}, we complete the proof. 
\end{proof}

\bigskip
\noindent
Continuing to work in the situation of Proposition \ref{Bergman_kernel}, let us set 
\begin{equation}
\label{eq3.1.2}
\widetilde{\Pi} u(x) = \widetilde{\Pi}_{\widetilde{U}} u(x) = \frac{1}{h} \int_{\widetilde{U}} e^{\frac{2}{h} \Psi(x,\overline{y})} a(x,\overline{y};h) u(y) e^{-\frac{2}{h}\varphi(y)} \, \omega(y, dy d\overline{y}), \quad u \in C^{\infty}_0(\widetilde{U}).
\end{equation}
Letting $\chi \in C^{\infty}_0(\widetilde{U})$, we conclude in view of (\ref{eq3.1.0.18}) and Schur's lemma that the globally defined operator $\chi \widetilde{\Pi} \chi$ satisfies 
\begin{equation}
\label{eq3.1.4}
\chi \widetilde{\Pi} \chi = \mathcal O(1): L^2_{\varphi}(M,\omega) \rightarrow L^2_{\varphi}(M,\omega). 
\end{equation}
It also follows from (\ref{eq3.1.0.5}) and Proposition \ref{Bergman_kernel} that 
\begeq
\label{eq3.1.10}
\chi (\Pi - \widetilde{\Pi}) \chi = \mathcal O(h^{\infty}): L^2_{\varphi}(M,\omega) \rightarrow L^2_{\varphi}(M,\omega).
\endeq
Here the orthogonal projection $\Pi$ has been introduced in (\ref{eq3.1.0.4}), (\ref{eq3.1.0.5}). 

\medskip
\noindent
We shall need the following sharpening of (\ref{eq3.1.10}). 
\begin{prop}
\label{prop:trace_class}
Let $\chi \in C^{\infty}_0(\widetilde{U})$, where $(\widetilde{U},x) \Subset W$ is a sufficiently small local holomorphic coordinate neighborhood of a point in $W$. The operator
\[
\chi(\Pi - \widetilde{\Pi})\chi: L^2_{\varphi}(M,\omega) \rightarrow L^2_{\varphi}(M,\omega) 
\]
is of trace class, of trace class norm $\mathcal O_{\chi}(h^{\infty})$.
\end{prop} 
\begin{proof}
We have to show that the operator
\begin{equation}
\label{eq3.1.10.4}
T = e^{-\varphi/h} \chi(\Pi - \widetilde{\Pi})\chi e^{\varphi/h}: L^2(M,\omega) \rightarrow L^2(M,\omega)
\end{equation}
is of trace class, of trace class norm $\mathcal O_{\chi}(h^{\infty})$. When doing so, we let $H^s(M)$ be the standard Sobolev space on $M$, equipped with its natural $h$--dependent norm $\norm{u}_{H^s(M)} = \norm{(1-h^2 \Delta_g)^{\frac{s}{2}} u}_{L^2(M,\omega)}$. Let $N(\lambda)$ be the number of eigenvalues of $-\Delta_g$ in $[0,\lambda]$ and let us write for $s > n = {\rm dim}_{\mathbb R}\, M = 2$, 
\[ 
\|(1 - h^2\Delta)^{-\frac{s}{2}}\|_{{\rm tr}} = \int_0^{\infty} \frac{1}{(1+h^2 \lambda)^{s/2}} dN(\lambda). 
\] 
Here $\|\cdot \|_{{\rm tr}}$ stands for the trace class norm for trace class operators on $L^2(M,\omega)$. Weyl's law for $-\Delta_g$ gives $N(\lambda) = \mathcal O(\lambda^{n/2})$ and we get integrating by parts,  
\begin{multline*}
\|(1 - h^2\Delta)^{-\frac{s}{2}}\|_{{\rm tr}} = \int_0^{\infty} \frac{1}{(1+h^2 \lambda)^{s/2}} dN(\lambda)  
= \mathcal O(h^2) \int_0^{\infty} \frac{N(\lambda)}{(1 + h^2 \lambda)^{\frac{s}{2}+1}}\, d\lambda \\ = \mathcal O(h^2) \int_0^{\infty} \frac{\lambda^{n/2}}{(1 + h^2 \lambda)^{\frac{s}{2}+1}}\, d\lambda = \mathcal O(h^{-n}) \int_0^{\infty} \frac{t^{n/2}}{(1 + t)^{\frac{s}{2}+1}}\, dt = \mathcal O(h^{-n}).
\end{multline*}
It suffices therefore to prove that
\begin{equation}
\label{eq3.10.4}
T = \mathcal O_{\chi}(h^{\infty}): L^2(M,\omega) \rightarrow H^s(M),
\end{equation}
for some $s > 2$. When doing so, we observe that in view of (\ref{eq3.1.10}), we have
\begin{equation}
\label{eq3.10.5}
T = \mathcal O_{\chi}(h^{\infty}):  L^2(M,\omega) \rightarrow L^2(M,\omega),
\end{equation}
and combining (\ref{eq3.10.5}) with the logarithmic convexity of Sobolev norms we see that we only need to show that for each $k\in \mathbb N$ there exists $N_k \in \mathbb N$, such that
\begin{equation}
\label{eq3.10.6}
T = \mathcal O_{k, \chi}(h^{-N_k}): L^2(M,\omega) \rightarrow H^k(M).
\end{equation}
We shall now see that (\ref{eq3.10.6}) holds with $N_k = 0$. First, given $v\in {\rm Hol}(W)$ and $\alpha, \beta \in \mathbb N$, let us consider $L^2$ estimates for the $C^{\infty}_0(\widetilde{U})$--function,
\[
(hD_{\overline{x}})^{\beta}\,(hD_x)^{\alpha} \left(e^{-\varphi/h}\, \chi v\right) = e^{-\varphi/h} \left(hD_{\overline{x}} + i\varphi'_{\overline{x}}\right)^{\beta}\, \left(hD_{x} + i\varphi'_{x}\right)^{\alpha}\left(\chi\, v\right).
\]
We get, using Lemma \ref{MVT} below,
\begin{multline}
\label{eq3.10.7}
\norm{(hD_{\overline{x}})^{\beta}\,(hD_x)^{\alpha} \left(e^{-\varphi/h}\, \chi v\right)}_{L^2(M,\omega)} \leq \mathcal O_{\alpha, \beta, \chi}(1) \sum_{\alpha' \leq \alpha} \norm{e^{-\varphi/h} 1_{{\rm supp}\, \chi}  (hD_x)^{\alpha'}v}_{L^2(M,\omega)} \\
\leq \mathcal O_{\alpha, \beta, \chi}(1) \norm{e^{-\varphi/h} 1_{\Omega} v}_{L^2(M,\omega)}.
\end{multline}
Here $\Omega \Subset \widetilde{U}$ is a neighborhood of ${\rm supp}\, \chi$. Applying (\ref{eq3.10.7}) with $v = \Pi(e^{\varphi/h}\, \chi u)\in {\rm Hol}(W)$, for $u\in L^2(M,\omega)$, we get for each $k\in \mathbb N$,
\begin{equation}
\label{eq3.10.8}
e^{-\varphi/h} \chi\Pi \chi e^{\varphi/h} = \mathcal O_{k, \chi}(1): L^2(M,\omega) \rightarrow H^k(M).
\end{equation} 

\medskip
\noindent
Next, it follows from (\ref{eq3.1.2}) that the Schwartz kernel of the operator $e^{-\varphi/h}\, \chi\, \widetilde{\Pi}\, \chi e^{\varphi/h}$ is given by
\begin{equation}
\label{eq3.10.9}
\frac{1}{h} e^{F(x,y)/h}\, \chi(x)\chi(y) a(x,\overline{y};h),\quad F(x,y) = 2\Psi(x,\overline{y}) - \varphi(x) - \varphi(y). 
\end{equation}
Here the amplitude $b(x,y;h) = \chi(x)\chi(y) a(x,\overline{y};h) \in C^{\infty}_0(\widetilde{U} \times \widetilde{U})$ satisfies 
\[
\nabla_{x}^{\alpha} \nabla_y^{\beta}\, b(x,y;h) = \mathcal O_{\alpha,\beta}(1), \quad \forall \,\, \alpha,\beta \in \mathbb N^2, 
\]
and it follows that
\begin{equation}
\label{eq3.10.9.1}
(hD_{\overline{x}})^{\beta}\,(hD_x)^{\alpha}\left(\frac{1}{h} e^{F(x,y)/h} b(x,y;h)\right) = \frac{1}{h} e^{F(x,y)/h} b_{\alpha,\beta}(x,y;h),
\end{equation}
where $b_{\alpha,\beta}\in C^{\infty}_0(\widetilde{U} \times \widetilde{U})$ satisfies the same estimates as $b$ above. We get, combining (\ref{eq3.10.9.1}), (\ref{eq3.1.0.18}), and Schur's lemma, that for each $k\in \mathbb N$ we have
\begin{equation}
\label{eq3.10.9.2}
e^{-\varphi/h} \chi\widetilde{\Pi} \chi e^{\varphi/h} = \mathcal O_{k, \chi}(1): L^2(M,\omega) \rightarrow H^k(M).
\end{equation}
The result follows, in view of (\ref{eq3.10.5}), (\ref{eq3.10.8}), and (\ref{eq3.10.9.2}). 
\end{proof}

\medskip
\noindent
\begin{lemm}
\label{MVT}
Let $\Omega \subseteq \mathbb C^n$ be open, let $\varphi: \Omega \rightarrow \mathbb R$ be locally Lipschitz continuous, and let $\mu(y, dy d\overline{y})$ be a strictly positive $C^{\infty}$ density of integration in $\Omega$. For each compact set $K \subseteq \Omega$, each open neighborhood $\omega \Subset \Omega$ of $K$, and all $\alpha \in \mathbb N^n$ there are constants $C_{\alpha} > 0$ such that 
\begeq
\label{eq3.10.9.3}
\norm{e^{-\varphi/h} 1_K (hD_x)^{\alpha}f}_{L^2(\Omega, d\mu)} \leq C_{\alpha} \norm{e^{-\varphi/h} 1_{\omega} f}_{L^2(\Omega,d\mu)},
\endeq
for all $f\in {\rm Hol}(\Omega)$. 
\end{lemm}
\begin{proof}
It suffices to show (\ref{eq3.10.9.3}) when $\mu(y, dy d\overline{y}) = L(dy)$ is the Lebesgue measure on $\mathbb C^n$. Let $\chi \in C^{\infty}_0(\mathbb C^n)$ be rotation invariant such that ${\rm supp}\, \chi \subseteq \{y\in \mathbb C^n; \abs{y}< 1\}$ and
$\int \chi(y) L(dy) = 1$. The mean value property gives, for all $f\in {\rm Hol}(\Omega)$ and $h>0$ small enough,
\begin{equation}
\label{eq3.10.10}
f(x) = \frac{1}{h^{2n}}\int f(y) \chi\left(\frac{x-y}{h}\right)\, L(dy), \quad x\in \omega_1,
\end{equation}
where $\omega_1 \Subset \omega$ is an open neighborhood of $K$. The result follows by differentiating (\ref{eq3.10.10}) and applying Schur's lemma. 
\end{proof}

\medskip
\noindent
Let $\chi \in C^{\infty}_0(\widetilde{U})$, where $(\widetilde{U},x) \Subset W$ is a local holomorphic coordinate chart as in Proposition \ref{prop:trace_class}. It follows from (\ref{eq3.1.0.7}) and (\ref{eq3.1.2}) that the operator $\chi \widetilde{\Pi} \chi$ is given by 
\begin{equation}
\label{eq3.1.18}
\chi \widetilde{\Pi} \chi u(x) = \frac{1}{h} \int_{\widetilde{U}} e^{\frac{2}{h} \Psi(x,\overline{y})} \chi(x) \chi(y) \left(a_0(x,\overline{y}) + \mathcal O(h)\right)\, u(y) e^{-\frac{2}{h} \varphi(y)}\, \omega(y,dy d\overline{y}), 
\end{equation}
where the $\mathcal O(h)$ remainder term is in the $C^{\infty}$ sense. In particular, $\chi \widetilde{\Pi} \chi$ is of trace class on $L^2_{\varphi}(M,\omega)$, and we have, see~\cite[Chapter 9]{DiSj},~\cite{Du81}, 
\begin{multline}
\label{eq3.1.19}
\tr \chi  \widetilde{\Pi} \chi = \frac{1}{h} \int \chi^2(x) \left(a_0(x,\overline{x}) + \mathcal O(h)\right)\, \omega(x,dx d\overline{x}) \\
= \frac{1}{2 \pi h} \int \chi^2(x) \Delta_g \varphi(x)\, \omega(x,dx d\overline{x}) + \mathcal O_{\chi}(1) =  
\frac{1}{2\pi h}\int\!\!\!\int_{\Lambda_\varphi} \chi^2(x)  d\xi\wedge dx+ \mO_{\chi}(1).
\end{multline}
Here on the last line we have also used (\ref{eq3.1.0.8}) and (\ref{eq1.31}). Combining Proposition \ref{prop:trace_class} and (\ref{eq3.1.19}) we get that the operator $\chi \Pi \chi$ is of trace class on $L^2_{\varphi}(M,\omega)$ and 
\begeq
\label{eq3.1.20}
\tr \chi  {\Pi} \chi = \tr \chi  \widetilde{\Pi} \chi + \mathcal O_{\chi}(h^{\infty}) = \frac{1}{2\pi h}\int\!\!\!\int_{\Lambda_\varphi} \chi^2(x)  d\xi\wedge dx+ \mO_{\chi}(1).
\endeq
Here in the equality 
\begeq
\label{eq3.1.20.1}
\tr \chi  {\Pi} \chi = \frac{1}{2\pi h}\int\!\!\!\int_{\Lambda_\varphi} \chi^2(x)  d\xi\wedge dx+ \mO_{\chi}(1),
\endeq
we would like to eliminate the assumption that the support of $\chi \in C^{\infty}_0(W)$ is small and to this end, we proceed as follows. Let $T\in \mathcal L(L^2_{\varphi}(M,\omega), L^2_{\varphi}(M,\omega))$ be of trace class and let $0 \leq \chi_j \in C^{\infty}(M)$, $1\leq j \leq N$, be such that 
\begeq
\label{eq3.1.21} 
\sum_{j=1}^N \chi_j^2 =1 \quad \wrtext{in} \,\, M.
\endeq
We then have, using the cyclicity of the trace, 
\begeq
\label{eq3.1.26} 
\tr T = \tr \left(\sum_{j=1}^N \chi_j^2 \right) T = \sum_{j=1}^N \tr \chi_j^2 T = \sum_{j=1}^N \tr \chi_j T \chi_j.
\endeq

\medskip
\noindent
We shall apply (\ref{eq3.1.26}) when $T = \chi \Pi \chi \in \mathcal L(L^2_{\varphi}(M,\omega), L^2_{\varphi}(M,\omega))$, where 
$\chi \in C^{\infty}_0(W)$. Indeed, the operator 
\[
e^{-\varphi/h} \chi \Pi \chi e^{\varphi/h}: L^2(M,\omega) \rightarrow H^s(M)
\]
is continuous for each $s\in \mathbb R$, and it follows as in the proof of Proposition \ref{prop:trace_class} that $\chi \Pi \chi$ is therefore of trace class on $L^2_{\varphi}(M,\omega)$. Given $\chi \in C^{\infty}_0(W)$, there exist finitely many local holomorphic coordinate charts $\widetilde{U}_j$, $1\leq j \leq N$, covering $M$ such that if $\widetilde{U}_j \cap {\rm supp}\, \chi \neq \emptyset$ then $\widetilde{U}_j \Subset W$ and (\ref{eq3.1.20.1}) holds for $\chi \in C^{\infty}_0(\widetilde{U}_j)$. Choose $\chi_j \in C^{\infty}_0(\widetilde{U}_j)$, $1\leq j \leq N$, so that (\ref{eq3.1.21}) holds. We get therefore using (\ref{eq3.1.26}) and (\ref{eq3.1.20.1}),
\begin{multline}
\label{eq3.1.27}
\tr \chi \Pi \chi = \sum_{j=1}^N \tr \chi_j \chi \Pi \chi_j \chi = \sum_{j=1}^N \left(\frac{1}{2\pi h} \int\!\!\!\int_{\Lambda_\varphi} \chi^2_j(x) \chi^2(x)  d\xi\wedge dx+ \mO_{\chi,j}(1)\right) \\ 
= \frac{1}{2\pi h} \int\!\!\!\int_{\Lambda_\varphi} \chi^2(x)  d\xi\wedge dx+ \mO_{\chi}(1).
\end{multline}

\medskip
\noindent
We may summarize the discussion above in the following result. 
\begin{prop} 
\label{prop:trace_Bergman}
Let $W\subseteq M$ be open such that $\varphi\in C^{\infty}(M)$ is strictly subharmonic in a neighborhood of $\overline{W}$, and let 
\[
\Pi: L^2_{\varphi}(W,\omega) \rightarrow H_{\varphi}(W) := {\rm Hol}(W) \cap L^2_{\varphi}(W,\omega)
\]
be the orthogonal projection. For each $\chi \in C^{\infty}_0(W)$, the operator $\chi \Pi \chi$ is of trace class on $L^2_{\varphi}(M,\omega)$, and we have 
\begeq
\label{eq3.1.28}
\tr \chi \Pi \chi = \frac{1}{2\pi h} \int\!\!\!\int_{\Lambda_\varphi} \chi^2(x)  d\xi\wedge dx+ \mO_{\chi}(1).
\endeq
\end{prop}

\bigskip
\noindent
Let us now return to our main discussion concerning upper bounds on the number of singular values of the operator
$h\overline{\partial}: L^2_{\varphi}(M,\omega) \rightarrow L^2_{\varphi}(M,T^*_{0,1} M)$, on the interval $[0,\e^{-\tau/h}]$. To this end, let $0\leq t_1\leq \dots \leq t_N \leq \e^{-\tau/h}$ denote the singular values in $[0,\e^{-\tau/h}]$, counted with multiplicity,
and let $e_1,\dots, e_N \in L^2_{\varphi}(M,\omega)$ stand for a corresponding orthonormal system of singular vectors. An application of the Weyl asymptotics to the semiclassical self-adjoint classically elliptic differential operator $P^*P$, with $P$ defined in (\ref{eq1.3}), gives that
\begin{equation}
\label{eq3.12}
N = N([0,e^{-\tau/h}]) = \mO(h^{-2}),
\end{equation}
see~\cite[Chapter 9]{DiSj}. We may also remark that using \cite[Proposition 4.6]{Sj10}, the estimate (\ref{eq3.12}) can be improved to
\begin{equation*}
N([0,e^{-\tau/h}])\leq N([0,C\sqrt{h}])= \mO(h^{-1}), \quad C \geq 1,
\end{equation*}
but having (\ref{eq3.12}) will suffice for the following argument.

\medskip
\noindent
From (\ref{eq3.1.0.1}) we recall that for each neighborhood $V \subseteq M$ of $M_+(\psi)$ there exists $\delta > 0$ such that
\begin{equation}
\label{eq3.13}
\norm{e^{-\varphi/h}e_j}_{L^2(M\setminus V,\omega)} = \mO(1)\, e^{-\frac{\delta}{h}},
\end{equation}
uniformly with respect to $j$, $1\leq j \leq N$. We also recall from
\eqref{eq1.14} that
\begin{equation}
\label{eq3.14}
\norm{e^{-\varphi/h}h\overline{\partial} e_j}_{L^2(M,T^*_{0,1} M)} = \mO(1)\, e^{-\tau/h},
\end{equation}
uniformly with respect to $j$. Associated to the singular states $e_j$, $1\leq j \leq N$, is the orthogonal projection
\begin{equation}
\label{eq3.15}
\widehat{\Pi}:L^2_{\varphi}(M,\omega) \rightarrow L^2_{\varphi}(M,\omega), \quad \widehat{\Pi}u = \sum_{j=1}^N (u,e_j)_{{L^2_{\varphi}(M,\omega)}} e_j
\end{equation}
onto $\bigoplus_{j=1}^N \CC e_j$. We have
\begin{equation}
\label{eq3.16}
N = \tr \widehat{\Pi} = \| \widehat{\Pi}\|_{\tr},
\end{equation}
where $\| \cdot\|_{\tr}$ denotes the trace class norm for trace class operators on $L^2_{\varphi}(M,\omega)$.

\bigskip
\noindent
Recall that in (\ref{eq3.1.0.0.1}), we have assumed that the compact set $M_+(\psi)$ defined in (\ref{eq2.7_b}) satisfies $M_+(\psi) \subseteq M_+$. Let $W \subseteq M$ be an open set such that $M_+(\psi) \subseteq W \Subset M_+$, and let $\Pi: L^2_{\varphi}(W,\omega) \rightarrow H_{\varphi}(W)$ be the orthogonal projection, as in (\ref{eq3.1.0.4}), (\ref{eq3.1.0.5}). Let also $\chi\in C_0^\infty(W; [0,1])$ be such that $\chi=1$ near $M_+(\psi)$. An application of Proposition \ref{exist_dbar} gives that 
\begin{equation}
\label{eq3.16.2}
\norm{(1 - \Pi) \chi e_j}_{L^2_{\varphi}(W,\omega)} \leq \mathcal O(h^{-1/2}) \norm{h\overline{\partial}(\chi e_j)}_{L^2_{\varphi}(W, T^*_{0,1} W)} \leq \mathcal O_{\chi}(h^{\infty}),
\end{equation}
uniformly with respect to $j$. Here we have also used (\ref{eq3.13}), (\ref{eq3.14}). We obtain that
\begin{equation}
\label{eq3.17}
\chi \Pi \chi e_j = \chi^2 e_j + \mO_{\chi} (h^\infty),
\end{equation}
in $L^2_{\varphi}(M,\omega)$, uniformly with respect to $j$, and combining (\ref{eq3.17}) with (\ref{eq3.13}) we get 
\begin{equation}
\label{eq3.17.1}
\norm{\chi {\Pi} \chi e_j - e_j}_{L^2_{\varphi}(M,\omega)} \leq \mathcal O_{\chi}(h^{\infty}),
\end{equation}
uniformly. Let us write therefore 
\begin{equation}
\label{eq3.17.2}
\chi {\Pi} \chi\, \widehat{\Pi} =\widehat{\Pi} + R,
\end{equation}
where $R = (\chi {\Pi} \chi -1) \widehat{\Pi}$ satisfies in view of (\ref{eq3.17.1}), (\ref{eq3.16}), and (\ref{eq3.12})
\begin{equation}
\label{eq3.18}
\|R\|_{\tr} \leq \norm{(\chi {\Pi} \chi -1) \widehat{\Pi}}_{\mathcal L(L^2_{\varphi}(M,\omega), L^2_{\varphi}(M,\omega))} \norm{\widehat{\Pi}}_{\tr} = \mO_{\chi} (h^\infty).
\end{equation}
It follows that 
\begin{equation*}
N = \norm{\widehat{\Pi}}_{\tr} \leq \norm{\chi {\Pi} \chi\, \widehat{\Pi}}_{\tr} + \mO_{\chi}(h^\infty)
\leq \norm{\chi {\Pi}\chi} _{\tr} \norm{\widehat{\Pi}}_{\mathcal L(L^2_{\varphi}(M,\omega),L^2_{\varphi}(M,\omega))} + \mO_{\chi}(h^\infty).
\end{equation*}
Here $\norm{\widehat{\Pi}}_{\mathcal L(L^2_{\varphi}(M,\omega),L^2_{\varphi}(M,\omega))} = 1$ since $\widehat{\Pi}$ is an
orthogonal projection, and therefore we get 
\begin{equation}
\label{eq3.19}
N \leq \norm{\chi \Pi \chi} _{\tr} + \mO_{\chi}(h^\infty) = {\rm tr} (\chi \Pi \chi) + \mO_{\chi}(h^\infty). 
\end{equation}
Here we have also used that the trace class operator $\chi \Pi \chi$ is self-adjoint positive on $L^2_{\varphi}(M,\omega)$ and therefore
\begin{equation}
\label{eq3.19.0.1}
\norm{\chi \Pi \chi} _{\tr} = {\rm tr} (\chi \Pi \chi).
\end{equation}
Applying Proposition \ref{prop:trace_Bergman} to (\ref{eq3.19}) gives that
\begin{equation}
\label{eq3.20}
N  \leq \frac{1}{2\pi h}\int\!\!\!\int_{\Lambda_\varphi} \chi^2(x)  d\xi\wedge dx+ \mO_{\chi}(1).
\end{equation}
Choosing a decreasing sequence of functions $\chi \in C^{\infty}_0(W;[0,1])$ such that $\chi=1$ near $M_+(\psi)$, tending pointwise to the characteristic function $1_{M_+(\psi)}$ of the compact set $M_+(\psi)$, and using (\ref{eq3.20}), (\ref{eq1.31}), we obtain the statement of Theorem \ref{thm:UpperBd}. 

\subsection{Lower bounds}
Our goal here is to prove Theorem \ref{thm:LowerBd}. Let $0 < \tau < {\rm max}\, \varphi - {\rm min}\, \varphi$ be fixed and let $\psi \in C(M; \mathbb R)$ be a lower bound weight in the sense of Definition \ref{def_lbw_n}. In particular, $\psi$ is subharmonic on the open set
\begin{equation}
\label{eq3.20.1}
\Omega = \{z\in M; \varphi(z) - \tau < \psi(z)\}.
\end{equation}
In what follows, we shall make use of H\"ormander's $L^2$ estimates and existence theorems for the $\overline{\partial}$ operator on $\Omega$ for the exponential weight $\psi$, which is merely continuous subharmonic. Our starting point is therefore the following result, closely related to Proposition \ref{exist_dbar}.

\begin{prop}
\label{prop:L2est}
Let $M$ be a compact Riemann surface equipped with a conformal Riemannian metric $g$ and let $\Omega \subseteq M$ be an open set such that ${\rm ext}(\Omega) \neq \emptyset$. Let $\psi \in C(M;\mathbb R)$ be subharmonic in $\Omega$. There exists $h_0 > 0$ such that for all $h\in (0,h_0]$ and all $\beta \in L^2_{\psi}(\Omega, T^*_{0,1} \Omega) = e^{\psi/h}L^2(\Omega, T^*_{0,1} \Omega)$ there exists a solution 
$u \in L^2_{\psi}(\Omega,\omega) = e^{\psi/h} L^2(\Omega,\omega)$ of the equation $h\overline{\partial}u = \beta$ in $\Omega$ such that
\begin{equation}
\label{eq3.21}
\norm{u}_{L^2_{\psi}(\Omega,\omega)} \leq \mathcal O(h^{-1}) \norm{\beta}_{L^2_{\psi}(\Omega, T^*_{0,1}\Omega)}.
\end{equation}
Here $\omega$ is the area form associated to $g$, introduced in {\rm (\ref{eq:ext1})}. 
\end{prop}
\begin{proof}
We shall first establish the result assuming that $\psi$ is of class $C^2(\Omega)$. To this end, let $z_0 \in {\rm ext}(\Omega)$ and let $0 \leq F \in  C^{\infty}(M\backslash \{z_0\})$ be the non-negative strictly subharmonic function in $M\backslash \{z_0\}$, introduced in the proof of Proposition \ref{propCH}. Let 
\begeq
\label{eq3.21.0}
\psi_{\varepsilon} = \psi + \frac{h}{\varepsilon} F, \quad 0 < h \leq \varepsilon \leq 1,
\endeq
and consider the conjugated operator 
\[ 
P_{\psi_{\varepsilon}} = e^{-\psi_{\varepsilon}/h} \circ h\overline{\partial} \circ e^{\psi_{\varepsilon}/h}.
\] 
Using (\ref{eq1.29.7.1}) we get for all $\mu \in \mathscr D_{0,1}(\Omega)$, 
\begin{equation}
\label{eq3.21.1}
\norm{P^*_{\psi_{\varepsilon}} \mu}_{L^2(M,\omega)}^2 \geq \frac{h}{2} \int_{\Omega} \Delta(\psi_{\varepsilon} + h f)\, \frac{\mu \wedge \overline{\mu}}{2i}, 
\end{equation}
and since $\Delta \psi \geq 0$ in $\Omega$ we obtain from (\ref{eq3.21.0}) and (\ref{eq3.21.1}), 
\begin{equation}
\label{eq3.21.1.1}
\norm{P^*_{\psi_{\varepsilon}} \mu}_{L^2(M,\omega)}^2 \geq \frac{h^2}{2} \int_{\Omega} \left(\frac{\Delta F}{\varepsilon} + \Delta f\right)\, \frac{\mu \wedge \overline{\mu}}{2i}.
\end{equation}
Taking $\varepsilon > 0$ sufficiently small independent of $h$ and using that $\Omega \Subset M\backslash \{z_0\}$, we obtain from (\ref{eq3.21.1.1}), \begin{equation}
\label{eq3.21.1.2}
h \norm{\mu}_{L^2(\Omega, T^*_{0,1} \Omega)} \leq \mathcal O(1) \norm{P^*_{\psi_{\varepsilon}} \mu}_{L^2(M,\omega)}. 
\end{equation}
Combining (\ref{eq3.21.1.2}) with the fact that $P^*_{\psi_{\varepsilon}} = e^{F/\varepsilon} \circ P^*_{\psi} \circ e^{-F/\varepsilon}$, we get 
\begin{equation}
\label{eq3.21.1.3}
h \norm{\mu}_{L^2(\Omega, T^*_{0,1} \Omega)} \leq \mathcal O(1) \norm{P^*_{\psi}\, \mu}_{L^2(M,\omega)},
\end{equation}
for all $\mu \in \mathscr D_{0,1}(\Omega)$. Similarly to the proof of Proposition \ref{exist_dbar}, we conclude, using (\ref{eq3.21.1.3}) together with the Hahn-Banach theorem and the Riesz representation theorem, that if $\psi \in C^2(\Omega)$ is subharmonic in $\Omega$ then for each $\beta \in L^2_{\psi}(\Omega, T^*_{0,1} \Omega)$ we can find a solution $u$ of the equation $h\overline{\partial} u = \beta$ in $\Omega$, satisfying (\ref{eq3.21}). Here the implicit constant in (\ref{eq3.21}) is independent of $\psi$, and we shall next remove the smoothness assumption on $\psi$ by an approximation argument, similar to~\cite[Theorem 4.2.1]{Horm_Conv}. Let $\Omega_1 \Subset \Omega$ be open and let $\psi_j$ be a sequence of strictly subharmonic smooth functions defined in a neighborhood of $\overline{\Omega_1}$, which decreases to $\psi$ on  $\Omega_1$, see~\cite[Proposition 5.1.2]{Br10}. It follows that there exists a sequence $u_j$ satisfying $h\overline{\partial} u_j = \beta$ in $\Omega_1$ such that 
\[ 
\norm{u_j}_{L^2_{\psi_j}(\Omega_1,\omega)} \leq \mathcal O(h^{-1}) \norm{\beta}_{L^2_{\psi_j}(\Omega_1, T^*_{0,1}\Omega_1)} \leq 
\mathcal O(h^{-1}) \norm{\beta}_{L^2_{\psi}(\Omega, T^*_{0,1}\Omega)}, \quad j=1,2, \ldots
\] 
Here we have also used that $\psi_j \geq \psi$. For each $k$, the sequence $(u_j)_{j\geq k}$ is bounded in $L^2_{\psi_k}(\Omega_1,\omega)$, and we can find therefore a (diagonal) subsequence converging weakly in $L^2_{\psi_k}(\Omega_1,\omega)$, for all $k$. The limit $u_{\Omega_1}$ satisfies $h\overline{\partial} u_{\Omega_1} = \beta$ in $\Omega_1$, and 
\[  
\norm{u_{\Omega_1}}_{L^2_{\psi_k}(\Omega_1,\omega)} \leq \mathcal O(h^{-1}) \norm{\beta}_{L^2_{\psi}(\Omega, T^*_{0,1}\Omega)}, \quad k=1,2, \ldots
\] 
Letting $k\rightarrow \infty$ we conclude therefore that 
\[  
\norm{u_{\Omega_1}}_{L^2_{\psi}(\Omega_1,\omega)} \leq \mathcal O(h^{-1}) \norm{\beta}_{L^2_{\psi}(\Omega, T^*_{0,1}\Omega)}. 
\] 
Letting $\Omega_1 \Subset \Omega$ increase to $\Omega$ and taking weak limits of $u_{\Omega_1}$'s as above, we conclude the proof. See also~\cite[Lecture 5]{Br10}. 
\end{proof}

\medskip
\noindent
Following Proposition \ref{prop:L2est}, we shall assume in what follows that the open set $\Omega$ in (\ref{eq3.20.1}) satisfies ${\rm ext}(\Omega) \neq \emptyset \Longleftrightarrow {\rm int}(M_-(\psi)) \neq \emptyset$.

\bigskip
\noindent
Let us recall next the contact set $M_+(\psi) \subseteq \Omega$, defined in (\ref{eq2.7_b}). We shall assume that the interior ${\rm int}(M_+(\psi)) \neq \emptyset$ and that (\ref{eq3.1.0.0.1}) holds. Let $u\in C_0^\infty({\rm int}(M_+(\psi)))$ be an $\mathcal O(h^{\infty})$ -- quasimode of $h\overline{\partial}: L^2_{\varphi}(M,\omega) \rightarrow L^2_{\varphi}(M, T^*_{0,1}M)$ of norm $\leq 1$, so that
\begin{equation}
\label{eq3.23}
\norm{u}_{L^2_{\varphi}(M,\omega)} \leq 1, \quad \norm{h\overline{\partial} u} _{L^2_{\varphi}(M, T^*_{0,1}M)} =\mathcal O(h^\infty).
\end{equation}
It follows from (\ref{eq3.23}) that
\[
\norm{h \overline{\partial} u}_{L^2_{\psi}(\Omega, T^*_{0,1}\Omega)} = \mathcal O(h^{\infty}),
\]
and an application of Proposition \ref{prop:L2est} gives that there exists $v\in L^2_{\psi}(\Omega,\omega)$ such that
\begin{equation}
\label{eq3.24}
\begin{split}
& h\overline{\partial}v = -h\overline{\partial}u \text{ in } \Omega, \\
  & \norm{v}_{L^2_{\psi}(\Omega,\omega)} = \mO(h^\infty){}.
\end{split}
\end{equation}
It follows from (\ref{eq3.23}), (\ref{eq3.24}) that $u+v\in H_\psi(\Omega) = {\rm Hol}(\Omega) \cap L^2_{\psi}(\Omega,\omega)$ with
\begin{equation}
\label{eq3.24.0.1}
\norm{u+v}_{L^2_{\psi}(\Omega,\omega)} \leq 1 + \mO(h^\infty).
\end{equation}

\medskip
\noindent
Let
\[
\omega(\delta) = \sup_{{\rm dist}(x,y) \leq \delta} \abs{\psi(x) - \psi(y)}
\]
be the modulus of continuity of $\psi$. Here ${\rm dist}(\cdot, \cdot)$ is the Riemannian distance function on $M$. The function $0\leq \omega$ is increasing with $\omega(\delta) \rightarrow 0$ as $\delta \rightarrow 0^+$. We also let $\chi \in C^\infty_0(\Omega;[0,1])$ be such that $\chi(x) = 1$ on $\{x\in \Omega;\, \dist(x,\partial \Omega)\geq h\}$. When $x\in {\rm supp}\, (\overline{\partial} \chi)$ we have ${\rm dist}(x,\partial \Omega) \leq h$ and given such an $x$, let $y\in \partial \Omega$ be such that ${\rm dist}(x,\partial \Omega) = {\rm dist}(x,y)$. We get using that the function $\psi - (\varphi -\tau)$ vanishes along $\partial \Omega$,
\begin{equation}
\label{eq3.24.1}
0 \leq \psi(x) - (\varphi(x) - \tau) \leq \abs{\psi(x) - \psi(y)} + \abs{\varphi(x)  - \varphi(y)} \leq \omega(h) + \mathcal O(h).
\end{equation}
We shall next use (\ref{eq3.24.1}) to establish a bound on $h\overline{\partial}(\chi (v+u))$ in $L^2_{\varphi}(M, T^*_{0,1} M)$. When doing so, we write using that $u+v\in {\rm Hol}(\Omega)$ and $u\in C^{\infty}_0({\rm int}(M_+(\psi)))$,
\[
h\overline{\partial}(\chi (v+u)) =(h\overline{\partial}\chi )(u+v) = (h\overline{\partial}\chi )v,
\]
for $h>0$ small enough. We get therefore, in view of (\ref{eq3.24}) and (\ref{eq3.24.1}),
\begin{multline}
\label{eq3.25}
\norm{h\overline{\partial}(\chi (v+u))}_{L^2_{\varphi}(M, T^*_{0,1} M)} = \norm{(h\overline{\partial}\chi )v}_{L^2_{\varphi}(M, T^*_{0,1} M)} \\
= \mathcal O(1)\, e^{-\widetilde{\tau}/h}\norm{(h\overline{\partial}\chi) v}_{L^2_{\psi}(\Omega, T^*_{0,1} \Omega)} = \mathcal O(h^{\infty})\, e^{-\widetilde{\tau}/h}.
\end{multline}
Here $\widetilde{\tau} = \tau - \omega(h)$. Let us also observe that $u - \chi(u+v) = - \chi v$ for $h$ small enough, satisfies
\begin{equation}
\label{eq3.26}
\norm{u - \chi(u+v)}_{L^2_{\varphi}(M,\omega)} = \norm{\chi v}_{L^2_{\varphi}(M,\omega)} \leq \norm{v}_{L^2_{\psi}(\Omega,\omega)} = \mO(h^{\infty}).
\end{equation}
Here we have also used (\ref{eq3.24}) and the fact that $\psi \leq \varphi$. Finally, using (\ref{eq3.23}) and (\ref{eq3.26}), we get
\begin{equation}
\label{eq3.27}
\norm{\chi(u+v)}_{L^2_{\varphi}(M,\omega)} \leq 1 + \mO(h^\infty).
\end{equation}

\medskip
\noindent
To summarize the discussion so far, given an $\mathcal O(h^{\infty})$ -- quasimode $u\in C^{\infty}_0({\rm int}(M_+(\psi)))$ satisfying (\ref{eq3.23}), there exists a function $f = \chi(u+v) \in L^2_{\varphi}(M,\omega)$ such that
\begin{equation}
\label{eq3.27.1}
\norm{f}_{L^2_{\varphi}(M,\omega)} \leq 1 + \mO(h^\infty), \quad
\end{equation}
and
\begin{equation}
\label{eq3.27.2}
\norm{u-f}_{L^2_{\varphi}(M, \omega)}  = \mO(h^\infty), \quad \norm{h\overline{\partial} f}_{L^2_{\varphi}(M, T^*_{0,1} M)} = \mO(h^\infty)\, e^{-\widetilde{\tau}/h}, \quad \widetilde{\tau} = \tau - \omega(h).
\end{equation}
It also follows from (\ref{eq3.24.1}) that when $\psi$ is Lipschitz continuous, (\ref{eq3.27.2}) holds with $\tau$ in place of $\widetilde{\tau}$.

\medskip
\noindent
Let $e_1,\dots,e_{\widetilde{N}} \in L^2_{\varphi}(M,\omega)$ be an orthonormal system of singular states of 
\begeq
\label{eq3.27.3}
h\overline{\partial}: L^2_{\varphi}(M,\omega) \to L^2_{\varphi}(M, T^*_{0,1} M), 
\endeq
corresponding to the singular values $t_1,\dots,t_{\widetilde{N}}$  satisfying $0\leq t_1\leq \dots \leq t_{\widetilde N} \leq \e^{-\widetilde{\tau}/h}$. We shall next approximate the function $f$ in (\ref{eq3.27.1}), (\ref{eq3.27.2}) by a linear combination of $e_1,\dots,e_{\widetilde{N}}$. To this end, let $e_{\widetilde{N}+1},\dots$ and $t_{\widetilde{N}+1},\dots$ be such that $e_1,e_2,\dots , e_{\widetilde{N}}, e_{\widetilde{N}+1} \dots$ form an orthonormal basis of singular states in $L^2_{\varphi}(M,\omega)$, and let $t_1\leq\dots \leq t_{\widetilde{N}} \leq t_{\widetilde{N}+1} \leq \dots $ be the corresponding sequence of singular values.

\medskip
\noindent
Let us write $f=\sum_{j=1}^\infty f_je_j$, where $f_j = (f,e_j)_{L^2_{\varphi}(M,\omega)}$. From the second equation in (\ref{eq3.27.2}) we get that
\begin{equation*}
\sum_{j=1}^\infty t_j^2|f_j|^2 = \mO(h^\infty)\, e^{-2\widetilde{\tau}/h},
\end{equation*}
and in particular,
\begin{equation}
\label{eq3.30.1}
\sum_{j=\widetilde{N}+1}^\infty t_j^2|f_j|^2 \leq \mO(h^\infty)\, e^{-2\widetilde{\tau}/h}.
\end{equation}
Noting that $t_j \geq \e^{-\widetilde{\tau}/h}$ for $j\geq \widetilde{N}+1$, we get from (\ref{eq3.30.1}),
\begin{equation}
\label{eq3.30.2}
\sum_{j = \widetilde{N}+1}^\infty |f_j|^2 \leq \mO(h^\infty).
\end{equation}
Introducing the orthogonal projection
\begin{equation}
\label{eq3.30.3}
\widehat{\Pi} = \sum_{j=1}^{\widetilde{N}}(\cdot,e_j)_{L^2_{\varphi}(M,\omega)} e_j,
\end{equation}
from $L^2_{\varphi}(M,\omega)$ onto $\bigoplus_{j=1}^{\widetilde{N}}\CC e_j$, we obtain using (\ref{eq3.30.2}) that
\begin{equation}
\label{eq3.31}
\| f- \widehat{\Pi}\, f\|_{L^2_{\varphi}(M,\omega)} = \mO(h^\infty).
\end{equation}

\medskip
\noindent
\begin{prop}
\label{prop:LBestimate1}
Assume that $u\in C^{\infty}_0({\rm int}(M_+(\psi)))$ satisfies \eqref{eq3.23}. Then we have
\begin{equation*}
\| u- \widehat{\Pi}\, u\|_{L^2_{\varphi}(M,\omega)}=\mO(h^\infty).
\end{equation*}
\end{prop}
\begin{proof}
Let $f\in L^2_{\varphi}(M,\omega)$ be the function constructed in (\ref{eq3.27.1}), (\ref{eq3.27.2}), satisfying \eqref{eq3.31}. Then
\begin{equation*}
\| u- \widehat{\Pi} f\|_{L^2_{\varphi}(M,\omega)} \leq \| u- f\|_{L^2_{\varphi}(M,\omega)}
+ \| f- \widehat{\Pi} f\|_{L^2_{\varphi}(M,\omega)} \leq \mO(h^\infty),
\end{equation*}
by \eqref{eq3.27.2}, \eqref{eq3.31}. Hence,
\begin{equation*}
\| u- \widehat{\Pi} u\|_{L^2_{\varphi}(M,\omega)} \leq \| u- \widehat{\Pi} f\|_{L^2_{\varphi}(M,\omega)}, 
\leq \mO(h^\infty),
\end{equation*}
since $\mathcal{R}(\widehat{\Pi})\ni v \mapsto \|u-v\|_{L^2_{\varphi}(M,\omega)}$ is minimal for $v=\widehat{\Pi} u$.
\end{proof}

\medskip
\noindent
Recall that we have assumed that $M_+(\psi) \subseteq M_+ = \{x\in M; \Delta \varphi(x) > 0\}$. Let $W \subseteq M$ be an open set such that 
$M_+(\psi) \subseteq W \Subset M_+$, and let us recall the orthogonal projection 
\[ 
\Pi: L^2_{\varphi}(W,\omega) \rightarrow H_{\varphi}(W),
\] 
introduced in (\ref{eq3.1.0.4}), (\ref{eq3.1.0.5}). Let $(U,x) \Subset (V,x) \Subset {\rm int}(M_+(\psi))$ be sufficiently small local holomorphic coordinate neighborhoods of a point in ${\rm int}(M_+(\psi))$ corresponding to $x = 0$. Let $\chi_1 \in C^{\infty}_0(V;[0,1])$ be such that $\chi_1 =1$ near $\overline{U}$. We also let
\[
\widetilde{\Pi} = \widetilde{\Pi}_V: \mathcal O(1): L^2_{\varphi}(V,\omega) \rightarrow L^2_{\varphi}(V,\omega)
\]
be the local asymptotic Bergman projection introduced in (\ref{eq3.1.2}). Letting $\chi \in C^{\infty}_0(U)$ and writing 
\[ 
1_U(\Pi - \widetilde{\Pi})\chi = 1_U \chi_1 (\Pi - \widetilde{\Pi}) \chi_1 \chi,
\] 
we conclude, in view of Proposition \ref{prop:trace_class} that the operator
\begin{equation}
\label{eq3.31.1}
1_U (\Pi - \widetilde{\Pi})\chi: L^2_{\varphi}(M,\omega) \rightarrow L^2_{\varphi}(M,\omega)
\end{equation}
is of trace class, of trace class norm $\mathcal O_{\chi}(h^{\infty})$. We may write therefore for all $v\in L^2_{\varphi}(M,\omega)$ and $x\in M$, using (\ref{eq3.1.2}) and (\ref{eq3.31.1}), 
\begin{multline}
\label{eq3.32}
1_U\, \Pi \,\chi v(x) = 1_U(x) \widetilde{\Pi} \chi v(x) + Rv(x) \\
= 1_U(x) \frac{1}{h} \int_U e^{\frac{2}{h}\Psi(x,\overline{y})} a(x,\overline{y};h) \chi(y) v(y)\, e^{-\frac{2}{h}\varphi(y)}\, \omega(y, dy d\overline{y}) + Rv(x). 
\end{multline}
Here the trace class norm of $R\in {\mathcal L}(L^2_{\varphi}(M,\omega), L^2_{\varphi}(M,\omega))$ is $\mathcal O_{\chi}(h^{\infty})$. 

\medskip
\noindent
When acting on the space of holomorphic functions $H_{\varphi}(W) = {\rm Hol}(W) \cap L^2_{\varphi}(W,\omega)$, the operator $\widetilde{\Pi}\chi$ in (\ref{eq3.32}) can be viewed as an asymptotic Toeplitz operator, and we observe that the integral kernel of $\widetilde{\Pi}\chi$ in (\ref{eq3.32}) is not the canonical one, in the sense that it depends on all the variables $(x,y,\overline{y})$ rather than on $(x,\overline{y})$ only. We shall next eliminate the $y$-dependence in the amplitude, committing only a negligible error. 

\begin{prop}
\label{prop:trace_Toeplitz}
Let $(U,x) \Subset {\rm int}(M_+(\psi))$ be a sufficiently small smoothly bounded local holomorphic coordinate neighborhood of a point in ${\rm int}(M_+(\psi))$ corresponding to $x=0$, and let $\chi \in C^{\infty}_0(U; [0,1])$. There exists a classical symbol $c(x,\theta;h) = c_{\chi}(x,\theta;h) \in C^{\infty}(U\times \rho(U)\times (0,h_0])$ of the form
\begin{equation}
\label{eq3.32.0.0.11}
c(x,\theta;h) \sim \sum_{j=0}^{\infty} h^j c_j(x,\theta), 
\end{equation}
where each $c_j\in C^{\infty}_0(U \times \rho(U))$ is almost holomorphic along $\theta = \overline{x}$, such that the operator 
\begin{equation}
\label{eq3.32.0.0.2}
\widetilde{\Pi}_{\chi}v(x)  = \frac{1}{h} \int_U e^{\frac{2}{h}\Psi(x,\overline{y})} c(x,\overline{y};h) v(y)\, e^{-\frac{2}{h}\varphi(y)}\, \omega(y, dy d\overline{y})
\end{equation}
satisfies the following: the operator $(\widetilde{\Pi} \chi - \widetilde{\Pi}_{\chi})\Pi 1_W: L^2_{\varphi}(M,\omega) \rightarrow L^2_{\varphi}(U,\omega)$ is of trace class, of trace class norm $\mathcal O_{\chi}(h^{\infty})$. The classical symbol $c$ in {\rm (\ref{eq3.32.0.0.11})} satisfies ${\rm supp}\, (c(\cdot,\cdot;h)) \Subset U \times \rho(U)$, uniformly in $h \in (0,h_0]$, and 
\[
c_0(x,\overline{x}) = \chi(x) a_0(x,\overline{x}) = \frac{1}{2\pi} \chi(x) \Delta \varphi(x), \quad x \in U. 
\]
\end{prop}
\begin{proof}
Let $(V,x) \Subset {\rm int}(M_+(\psi))$ be a small local holomorphic coordinate neighborhood of the point in question, satisfying $(U,x) \Subset (V,x)$, such that $\Psi$ and $a$ in (\ref{eq3.32}) are defined in $V \times \rho(V)$. We also let $\chi_1 \in C^{\infty}_0(V;[0,1])$ be such that $\chi_1 = 1$ in a neighborhood of $\overline{U}$. Given $u\in H_{\varphi}(W)$, we may write using (\ref{eq3.1.2}), (\ref{eq1.29.2})
\begin{equation}
\label{eq3.32.0.0.21}
\widetilde{\Pi} \chi u(x) = \frac{1}{2ih} \int\!\!\!\int_{\Gamma_V} e^{\frac{2}{h}(\Psi(x,\theta) - \Psi(y,\theta) - h F(y,\theta))} \chi_1(y) a(x,\theta;h) \widetilde{\chi}(y,\theta) u(y)\, d\theta \wedge dy, \quad x\in U. 
\end{equation}
Here the contour $\Gamma_V$ is given by $\theta = \overline{y}$, $y\in V$, the function $F = F(y,\theta)$ is almost holomorphic along 
$\theta = \overline{y}$ with $F(y,\overline{y}) = f(y)$ in (\ref{eq1.29.2}), and $\widetilde{\chi} \in C^{\infty}_0(U \times \rho(U))$ is almost holomorphic along $\theta = \overline{y}$ with $\widetilde{\chi}(y,\overline{y}) = \chi(y)$. Proceeding as in the proof of Proposition \ref{Bergman_kernel} and using the same notation as there, we shall show that there exists a classical symbol $b(x,y,\theta;h) \sim \sum_{j=0}^{\infty} h^j b_j(x,y,\theta)$, satisfying $b_j \in \mathcal A$ for all $j\geq 0$, such that for some $c(x,\theta;h)\sim \sum_{j=0}^{\infty} h^j {c}_j(x,\theta)$, with ${c}_j \in  C^{\infty}_0(U \times \rho(U))$ almost holomorphic along $\theta = \overline{x}$, we have
\begin{multline}
\label{eq3.32.0.0.23}
a(x,\theta;h) \widetilde{\chi}(y,\theta) - {c}(x,\theta;h) \\ 
= e^{-\frac{2}{h}(\Psi(x,\theta) - \Psi(y,\theta) - hF(y,\theta))} h\partial_{\theta} \left(e^{\frac{2}{h}(\Psi(x,\theta) - \Psi(y,\theta)-hF(y,\theta))} b(x,y,\theta;h)\right),
\end{multline}
modulo a symbol of the form $\sum_{j=0}^{\infty} h^j f_j$, with $f_j \in \mathcal I^{\infty}$, for all $j\geq 0$. Similarly to the proof of Proposition \ref{Bergman_kernel}, we see that (\ref{eq3.32.0.0.23}) leads to a sequence of division problems, 
\begin{equation}
\label{eq3.32.0.0.24}
a_0(x,\theta)\widetilde{\chi}(y,\theta) - c_0(x,\theta) = 2\left(\Psi'_{\theta}(x,\theta) - \Psi'_{\theta}(y,\theta)\right)b_0(x,y,\theta) \quad {\rm mod}\,\,\mathcal I^{\infty},
\end{equation}
and for $j\geq 1$, 
\begin{multline}
\label{eq3.32.0.0.25}
a_j(x,\theta)\widetilde{\chi}(y,\theta) - {c}_j(x,\theta) - (\partial_{\theta} - 2F'_{\theta}(y,\theta)) b_{j-1}(x,y,\theta)  \\ 
= 2\left(\Psi'_{\theta}(x,\theta) - \Psi'_{\theta}(y,\theta)\right)b_j(x,y,\theta) \quad {\rm mod}\,\,\mathcal I^{\infty}.
\end{multline}
Here in (\ref{eq3.32.0.0.24}) we take $c_0(x,\theta) = a_0(x,\theta) \widetilde{\chi}(x,\theta) \in C^{\infty}_0(U\times \rho(U))$, and using that $\Psi''_{y\theta}(y,\overline{y}) = \varphi''_{y\overline{y}}(y) \neq 0$, $y\in {V}$, we see, as in~\cite{BBSj}, that there exists $b_0 \in {\mathcal A}$ of the form 
\begeq
\label{eq3.32.0.0.25.1}
b_0(x,y,\theta) = \zeta(x,y,\theta) \int_0^1 (\partial_x \widetilde{\chi})(tx + (1-t)y,\theta)\, dt, \quad \zeta \in \mathcal A, 
\endeq
such that (\ref{eq3.32.0.0.24}) holds. Continuing in this way and letting 
\[ 
c_j(x,\theta) = a_j(x,\theta) \widetilde{\chi}(x,\theta) - (\partial_{\theta} - 2F'_{\theta}(x,\theta)) b_{j-1}(x,x,\theta), \quad j\geq 1, 
\] 
we obtain $c_j = c_j(x,\theta)$ almost holomorphic along $\theta = \overline{x}$ with ${\rm supp}\, (c_j) \subseteq {\rm supp}(\widetilde{\chi})$, and $b_j \in \mathcal A$, $j\geq 1$, such that (\ref{eq3.32.0.0.25}) holds. We obtain (\ref{eq3.32.0.0.23}) with a classical symbol $c(x,\theta;h)$ such that ${\rm supp}\, (c(\cdot,\cdot;h)) \subseteq {\rm supp}(\widetilde{\chi}) \Subset U \times \rho(U)$.

\medskip
\noindent
It follows next from (\ref{eq3.1.0.18}) and Schur's lemma that the contribution of the amplitude $f = f(x,y,\theta;h) \sim \sum_{j=0}^{\infty} h^j f_j(x,y,\theta)$,  with $f_j \in \mathcal I^{\infty}$, to the operator $\widetilde{\Pi}\chi$ in (\ref{eq3.32.0.0.21}) gives rise to an error term of norm $\mathcal O(h^{\infty}): H_{\varphi}(W) \rightarrow L^2_{\varphi}(U,\omega)$. Using (\ref{eq3.32.0.0.23}), Stokes' formula, and the fact that $\chi_1(y)c(x,\overline{y};h) = c(x,\overline{y};h)$, we get therefore for all $N$ and $x\in U$, modulo such an error term 
\begin{multline}
\label{eq3.32.0.0.26}
2i h \widetilde{\Pi} \chi u(x) - \int\!\!\!\int_{\Gamma_V} e^{\frac{2}{h}(\Psi(x,\theta) - \Psi(y,\theta) - hF(y,\theta))} {c}(x,\theta;h) u(y)\, d\theta \wedge dy \\ 
= \int\!\!\!\int_{\Gamma_V} e^{\frac{2}{h}(\Psi(x,\theta) - \Psi(y,\theta)-hF(y,\theta))} \chi_1(y) (a(x,\theta;h)\widetilde{\chi}(y,\theta) - {c}(x,\theta;h))u(y) d\theta \wedge dy \\
= \int\!\!\!\int_{\Gamma_{V}} \chi_1(y) u(y) hd\left(e^{\frac{2}{h}(\Psi(x,\theta) - \Psi(y,\theta) - hF(y,\theta))} b(x,y,\theta)\right)\wedge dy \\
+ \int\!\!\!\int_{\Gamma_{V}} \chi_1(y) u(y) e^{\frac{2}{h}(\Psi(x,\theta) - \Psi(y,\theta) - hF(y,\theta))} \mathcal O_N(\abs{x-\overline{\theta}}^N)\, d\theta \wedge dy \\
= \int\!\!\!\int_{\Gamma_{V}} \chi_1(y) hd\left( u(y)e^{\frac{2}{h}(\Psi(x,\theta) - \Psi(y,\theta) - hF(y,\theta))} b(x,y,\theta)\, dy\right) \\
+ \int\!\!\!\int_{\Gamma_{V}} \chi_1(y) u(y) e^{\frac{2}{h}(\Psi(x,\theta) - \Psi(y,\theta) - hF(y,\theta))} \mathcal O_N(\abs{x-\overline{\theta}}^N)\, d\theta \wedge dy \\
= - \int\!\!\!\int_{\Gamma_V} u(y)e^{\frac{2}{h}(\Psi(x,\theta) - \Psi(y,\theta) - hF(y,\theta))} b(x,y,\theta) h d\chi_1 \wedge dy \\ 
+ \int\!\!\!\int_{\Gamma_{V}} e^{\frac{2}{h}(\Psi(x,\theta) - \Psi(y,\theta))} \mathcal O_N(\abs{x-\overline{\theta}}^N) u(y)\, d\theta \wedge dy. 
\end{multline}
Here we have also used that $u \in H_{\varphi}(W)$ is holomorphic. As before, we obtain using (\ref{eq3.1.0.18}), (\ref{eq3.32.0.0.26}), and the fact that $d\chi_1$ vanishes in a neighborhood of $\overline{U}$, 
\[ 
\widetilde{\Pi} \chi - \widetilde{\Pi}_{\chi} = \mathcal O(h^{\infty}): H_{\varphi}(W) \rightarrow L^2_{\varphi}(U,\omega).
\] 
In other words, we get 
\begin{equation}
\label{eq3.32.0.0.27}
(\widetilde{\Pi}\chi - \widetilde{\Pi}_{\chi})\Pi 1_W = \mathcal O(h^{\infty}): L^2_{\varphi}(M)\rightarrow L^2_{\varphi}(U,\omega). 
\end{equation}

\medskip
\noindent
When passing from (\ref{eq3.32.0.0.27}) to the trace class estimates, we argue as in the proof of Proposition \ref{prop:trace_class}, observing that for each $k\in \mathbb N$, we have
\begin{equation}
\label{eq3.32.0.0.28}
e^{-\varphi/h}(\widetilde{\Pi}\chi - \widetilde{\Pi}_{\chi})\Pi 1_W\,e^{\varphi/h} = \mathcal O_{k}(1): L^2(M,\omega) \rightarrow H^k(U). 
\end{equation}
Here $H^k(U)$ is the standard Sobolev space on the smoothly bounded open set $U \subseteq M$, equipped with its natural semiclassical norm. We next observe that for $k > {\rm dim}_{\mathbb R} M = 2$, the inclusion map $\iota: H^k(U) \rightarrow L^2(U,\omega)$ is of trace class, of trace class norm $\mathcal O(h^{-2})$. Indeed, this follows by writing $\iota = \rho \iota_M \sigma$, where $\sigma: H^k({U}) \rightarrow H^k(M)$ is a uniformly bounded Seeley extension map~\cite[Section 22]{Es11}, $\iota_M: H^k(M) \rightarrow L^2(M,\omega)$ is the inclusion map, and $\rho: L^2(M,\omega) \rightarrow L^2({U},\omega)$ is the (uniformly bounded) restriction map. It follows from the proof of Proposition \ref{prop:trace_class} that $\iota_M$ is of trace class, of trace class norm $\mathcal O(h^{-2})$, and the trace class property of the inclusion map $\iota: H^k(U) \rightarrow L^2(U,\omega)$ follows. The result of the proposition is now obtained using (\ref{eq3.32.0.0.27}), (\ref{eq3.32.0.0.28}). 
\end{proof}

\bigskip
\noindent
It follows from (\ref{eq3.31.1}) that the operator 
\begeq
\label{eq3.32.0.0.29}
1_U(\Pi - \widetilde{\Pi})\chi \Pi 1_W = 1_U(\Pi - \widetilde{\Pi})\chi 1_U \Pi 1_W: L^2_{\varphi}(M,\omega) \rightarrow L^2_{\varphi}(M,\omega)
\endeq
is of trace class, of trace class norm $\mathcal O_{\chi}(h^{\infty})$, and writing 
\[ 
1_U(\Pi \chi - \widetilde{\Pi}_{\chi}) \Pi 1_W = 1_U(\Pi - \widetilde{\Pi})\chi \Pi 1_W + 1_U (\widetilde{\Pi}\chi - \widetilde{\Pi}_{\chi})\Pi 1_W,
\] 
we obtain, in view of Proposition \ref{prop:trace_Toeplitz}, that the operator 
\begeq
\label{eq3.32.0.0.29.1}
1_U(\Pi \chi - \widetilde{\Pi}_{\chi}) \Pi 1_W = 1_U \Pi \chi \Pi 1_W - \widetilde{\Pi}_{\chi} \Pi 1_W: L^2_{\varphi}(M,\omega) \rightarrow L^2_{\varphi}(M,\omega)
\endeq
is also of trace class, of trace class norm $\mathcal O_{\chi}(h^{\infty})$. Here we have also used that $1_U \widetilde{\Pi}_{\chi} = \widetilde{\Pi}_{\chi}$, since the Schwartz kernel of $\widetilde{\Pi}_{\chi}$ in (\ref{eq3.32.0.0.2}) is compactly supported in $U \times U$. 

\medskip
\noindent
Using the notation of Proposition \ref{prop:trace_Toeplitz} and recalling (\ref{eq3.32.0.0.2}), let us set for $y\in U \Subset {\rm int}(M_+(\psi))$,
\[
u_y(x) = \frac{1}{h} e^{\frac{2}{h}\Psi(x,\overline{y})} c(x,\overline{y};h) e^{-\frac{1}{h}\varphi(y)} \in C^{\infty}_0(U).
\]
The idea now is to apply Proposition \ref{prop:LBestimate1} to the family of "coherent states" $u_y$, and when doing so we observe first that 
\begin{equation}
\label{eq32.0.0.3}
\norm{u_y}_{L^2_{\varphi}(M,\omega)} \leq \mathcal O\left(\frac{1}{h^{1/2}}\right), 
\end{equation}
uniformly for $y \in U$, in view of (\ref{eq3.1.0.18}). We may also remark that the method of stationary phase shows that we have a lower bound on $\norm{u_y}_{L^2_{\varphi}(M,\omega)}$ that is of the same order of magnitude as the upper bound in (\ref{eq32.0.0.3}), for $y\in U \subseteq {\rm int}(M_+(\psi))$ satisfying $\chi(y) \geq \delta > 0$, for $\delta > 0$ fixed. Writing next 
\[
h\overline{\partial} u_y(x) = \mathcal O_N(1) \abs{x-y}^N \frac{1}{h} e^{\frac{2}{h}\Psi(x,\overline{y}) - \frac{1}{h}\varphi(y)}\, d\overline{x}, \quad N=1,2,\ldots , 
\]
and using (\ref{eq3.1.0.18}) again, we see that
\[
\norm{h\overline{\partial} u_y}_{L^2_{\varphi}(M,T^*_{0,1}M)} = \mathcal O(h^{\infty}), 
\]
uniformly in $y\in U$. Applying Proposition \ref{prop:LBestimate1} to the quasimode $h^{1/2} u_y/C$ of $h\overline{\partial}$, where $C>0$ is a sufficiently large constant, we obtain that
\begin{equation}
\label{eq3.32.1}
\norm{u_y - \widehat{\Pi} u_y}_{L^2_{\varphi}(M,\omega)} = \mathcal O(h^{\infty}),
\end{equation}
uniformly for $y \in U$. We get therefore, using (\ref{eq3.32.0.0.2}), (\ref{eq3.32.1}), the Minkowski, and the Cauchy-Schwarz inequalities,
\begin{multline}
\label{eq3.32.2}
\norm{\widetilde{\Pi}_\chi v - \widehat{\Pi} \widetilde{\Pi}_\chi v}_{L^2_{\varphi}(M,\omega)} \leq \int_U \norm{u_y - \widehat{\Pi} u_y}_{L^2_{\varphi}(M,\omega)}\, \abs{v(y)}\, e^{-\frac{1}{h}\varphi(y)}\, \omega(y,dy d\overline{y}) \\
\leq \mathcal O(h^{\infty})\, \norm{v}_{L^2_{\varphi}(M,\omega)},\quad v\in L^2_{\varphi}(M,\omega).
\end{multline}

\begin{prop}
\label{prop:trace_lower}
The operator
\[
\widetilde{\Pi}_\chi - \widehat{\Pi}\, \widetilde{\Pi}_\chi: L^2_{\varphi}(M,\omega) \rightarrow L^2_{\varphi}(M,\omega)
\]
is of trace class, of trace class norm $\mathcal O_{\chi}(h^{\infty})$.
\end{prop} 
\begin{proof}
We proceed as in the proof of Proposition \ref{prop:trace_class}. It follows from (\ref{eq3.32.2}) that
\begin{equation}
\label{eq3.32.3}
\widetilde{\Pi}_\chi - \widehat{\Pi}\, \widetilde{\Pi}_\chi = \mathcal O_{\chi}(h^{\infty}): L^2_{\varphi}(M,\omega) \rightarrow L^2_{\varphi}(M,\omega),
\end{equation}
and it suffices therefore to estimate the operator norm of
\[
e^{-\varphi/h}\left(\widetilde{\Pi}_\chi - \widehat{\Pi}\, \widetilde{\Pi}_\chi\right) e^{\varphi/h}: L^2(M) \rightarrow H^k(M),
\]
for each $k\in \mathbb N$. To this end, using (\ref{eq3.32.0.0.2}) and arguing as in the proof of Proposition \ref{prop:trace_class}, we first see that 
\begin{equation}
\label{eq3.32.4}
e^{-\varphi/h} \widetilde{\Pi}_\chi e^{\varphi/h} = \mathcal O_{k,\chi}(1): L^2(M) \rightarrow H^k(M).
\end{equation}

\medskip
\noindent
It remains for us therefore to estimate the operator norm of
\[
e^{-\varphi/h} \widehat{\Pi} e^{\varphi/h}: H^k(M) \rightarrow H^k(M),
\]
and when doing so we write in view of (\ref{eq3.30.3}),
\begin{equation}
\label{eq3.32.5}
e^{-\varphi/h} \widehat{\Pi} e^{\varphi/h} u = \sum_{j=1}^{\widetilde{N}} (u, g_j)_{L^2(M,\omega)} g_j.
\end{equation}
Here  $g_j = e^{-\varphi/h} e_j \in L^2(M,\omega)$, $1\leq j \leq \widetilde{N}$, is an orthonormal system of singular states of the operator $P: L^2(M,\omega) \rightarrow L^2(M,T^*_{0,1}M)$ in (\ref{eq1.3}) associated to the singular values $t_j \in [0,e^{-\widetilde{\tau}/h}]$, $1\leq j\leq \widetilde{N}$. We have $(P^*P - t_j^2)g_j = 0$, and it follows from the ellipticity of the operator $P^*P$ near fiber infinity in $T^*M$ that we have for each $k\in \mathbb N$,
\begin{equation}
\label{eq3.32.6}
\norm{g_j}_{H^k(M)} \leq \mathcal O_k(1),
\end{equation}
uniformly in $j$. We infer, using (\ref{eq3.32.5}) and (\ref{eq3.32.6}) that
\begin{equation}
\label{eq3.32.7}
e^{-\varphi/h} \widehat{\Pi} e^{\varphi/h} = \mathcal O_k(1) \widetilde{N} = \mathcal O_k(h^{-2}): H^k(M) \rightarrow H^k(M).
\end{equation}
The result now follows from (\ref{eq3.32.3}), (\ref{eq3.32.4}), and (\ref{eq3.32.7}). 
\end{proof}

\medskip
\noindent
Using (\ref{eq3.32.0.0.29.1}) and Proposition \ref{prop:trace_lower} we obtain that
\begeq
\label{eq3.32.7.1}
1_U\, \Pi \,\chi\Pi 1_W = \widetilde{\Pi}_\chi \Pi 1_W + R_2 = \widehat{\Pi}\, \widetilde{\Pi}_\chi \Pi 1_W + R_3,
\endeq
where the trace class norm of $R_2, R_3 \in \mathcal L(L^2_{\varphi}(M,\omega),L^2_{\varphi}(M,\omega))$ is $\mathcal O_{\chi}(h^{\infty})$. Here we recall that $U \Subset {\rm int}(M_+(\psi))$ is a small local holomorphic coordinate chart and $\chi \in C^{\infty}_0(U;[0,1])$. Let $\widetilde{W} \subseteq M$ be open such that $M_+(\psi) \subseteq \widetilde{W} \Subset W$ and let us write, in view of (\ref{eq3.32.7.1}),
\begeq
\label{eq3.32.7.2}
1_U\, \Pi \,\chi\Pi 1_{\widetilde{W}} = \widehat{\Pi}\, \widetilde{\Pi}_\chi \Pi 1_{\widetilde{W}} + R_4,
\endeq
where $R_4 = \mathcal O_{\chi}(h^{\infty})$ in the trace class norm. We would like to compare the operator in the left hand side of (\ref{eq3.32.7.2}) to the self-adjoint operator $1_{\widetilde{W}}\, \Pi \,\chi\Pi 1_{\widetilde{W}}$, and to this end we shall make use of the following essentially well known result, see also \cite{Li01}, \cite{Chr18}, \cite{HeLuXu20} for much stronger off-diagonal decay estimates. 

\begin{prop}
\label{prop:off-diag}
Let $W \subseteq M$ be open such that $W \Subset \{x\in M; \Delta \varphi(x) >0\}$, and let 
\[ 
\Pi: L^2_{\varphi}(W,\omega) \rightarrow H_{\varphi}(W) := {\rm Hol}(W) \cap L^2_{\varphi}(W,\omega)
\] 
be the orthogonal projection. The Schwartz kernel $K(x,y) \in {\rm Hol}(W \times W^{\dagger})$ of $\Pi$ given in {\rm (\ref{eq3.1.0.5})} satisfies, uniformly on compact subsets of $W \times W \setminus \Delta(W \times W)$, 
\begeq
\label{eq3.32.7.3}
K(x,y) = \mathcal O(h^{\infty})\, e^{(\varphi(x) + \varphi(y))/h}.
\endeq
Here $\Delta(W \times W) = \{(x,x); x\in W\}$ is the diagonal of $W \times W$. 
\end{prop} 

\medskip
\noindent
A proof of Proposition \ref{prop:off-diag} is provided in Appendix \ref{sec:Prop3.9}, for completeness and the convenience of the reader.

\medskip
\noindent
It follows from Proposition \ref{prop:off-diag} that 
\begeq
\label{eq3.32.7.4} 
1_{\widetilde{W}\setminus U} \Pi \chi = \mathcal O_{\chi}(h^{\infty}): L^2_{\varphi}(M,\omega) \rightarrow L^2_{\varphi}(M,\omega).
\endeq
In the proof of Proposition \ref{prop:trace_class} we checked that for each $\psi \in C^{\infty}_0(U)$ and each $k\in \mathbb N$, we have 
\begeq
\label{eq3.32.7.5}
e^{-\varphi/h} \psi \Pi 1_{\widetilde{W}} e^{\varphi/h} = \mathcal O_{k,\psi}(1): L^2(M,\omega) \rightarrow H^k(M).
\endeq
It is therefore clear, in view of (\ref{eq3.32.7.4}) and (\ref{eq3.32.7.5}), that the operator 
\[
1_{\widetilde{W}\setminus U} \Pi \chi \Pi 1_{\widetilde{W}}: L^2_{\varphi}(M,\omega) \rightarrow L^2_{\varphi}(M,\omega)
\] 
is of trace class, of trace class norm $\mathcal O_{\chi}(h^{\infty})$. Combining this observation with (\ref{eq3.32.7.2}), we obtain 
\begeq
\label{eq3.32.7.6}
1_{\widetilde{W}}\, \Pi \,\chi\Pi 1_{\widetilde{W}} = \widehat{\Pi}\, \widetilde{\Pi}_\chi \Pi 1_{\widetilde{W}} + R_5,
\endeq
where $R_5 = \mathcal O_{\chi}(h^{\infty})$ in the trace class norm. Now the first equality in (\ref{eq3.32.7.1}) and the discussion above show also that 
\begeq
\label{eq3.32.7.7}
1_{\widetilde{W}}\, \Pi \,\chi\Pi 1_{\widetilde{W}} = \widetilde{\Pi}_\chi \Pi 1_{\widetilde{W}} + R_6,
\endeq
where $R_6 = \mathcal O_{\chi}(h^{\infty})$ in the trace class norm. We have therefore proved that 
\begin{equation}
\label{eq3.32.8}
\norm{(\widehat{\Pi} -1)1_{\widetilde{W}}\, \Pi \,\chi \Pi 1_{\widetilde{W}}}_{\rm tr} = \mathcal O_{\chi}(h^{\infty}).
\end{equation}
Here (\ref{eq3.32.8}) has been established for $\chi \in C^{\infty}_0(U;[0,1])$, where $U \Subset {\rm int}(M_+(\psi))$ is a small local holomorphic coordinate chart, but a partition of unity argument allows us now to eliminate this restriction, and we conclude that (\ref{eq3.32.8}) holds for all 
$\chi \in C^{\infty}_0({\rm int}(M_+(\psi));[0,1])$. 

\medskip
\noindent
We get using (\ref{eq3.32.8}), when $\chi \in C^{\infty}_0({\rm int}(M_+(\psi));[0,1])$, 
\begin{multline}
\label{eq3.32.9}
\norm{1_{\widetilde{W}}\, \Pi \,\chi \Pi 1_{\widetilde{W}}}_{\rm tr} \leq \norm{\widehat{\Pi}\, 1_{\widetilde{W}}\, \Pi \,\chi \Pi 1_{\widetilde{W}}}_{\rm tr} + \mathcal O_{\chi}(h^{\infty}) \\ \leq
\norm{\widehat{\Pi}}_{{\rm tr}}\, \norm{1_{\widetilde{W}}\, \Pi \,\chi \Pi 1_{\widetilde{W}}}_{\mathcal L(L^2_{\varphi}(M,\omega),L^2_{\varphi}(M,\omega))} + \mathcal O_{\chi}(h^{\infty})
\leq \norm{\widehat{\Pi}}_{{\rm tr}} + \mathcal O_{\chi}(h^{\infty}).
\end{multline} 
Here we have also used that $\norm{1_{\widetilde{W}}\, \Pi \,\chi \Pi 1_{\widetilde{W}}}_{\mathcal L(L^2_{\varphi}(M,\omega),L^2_{\varphi}(M,\omega))} \leq 1$. Now the operators $1_{\widetilde{W}}\, \Pi \,\chi \Pi 1_{\widetilde{W}}$ and $\widehat{\Pi}$ are self-adjoint positive on $L^2_{\varphi}(M,\omega)$, and thus 
\begin{equation}
\label{eq3.32.10}
\norm{1_{\widetilde{W}}\, \Pi \,\chi \Pi 1_{\widetilde{W}}}_{\rm tr} = \tr (1_{\widetilde{W}}\, \Pi \,\chi \Pi 1_{\widetilde{W}}), \quad \norm{\widehat{\Pi}}_{{\rm tr}} = \tr \widehat{\Pi} = \widetilde{N}.
\end{equation}
Here $\widetilde{N}$ is the number of singular values in the interval $[0,e^{-\widetilde{\tau}/h}]$, $\widetilde{\tau} = \tau - \omega(h)$ --- see (\ref{eq3.27.3}) and the adjacent discussion. We get therefore using (\ref{eq3.32.9}), (\ref{eq3.32.10}),
\begin{equation}
\label{eq3.33}
\tr (1_{\widetilde{W}}\, \Pi \,\chi \Pi 1_{\widetilde{W}}) \leq \widetilde{N} + \mathcal O_{\chi}(h^{\infty}), \quad \chi \in C^{\infty}_0({\rm int}(M_+(\psi));[0,1]), 
\end{equation}
and in particular, 
\begin{equation}
\label{eq3.33.1}
\tr (1_{\widetilde{W}}\, \Pi \,\chi^2 \Pi 1_{\widetilde{W}}) \leq \widetilde{N} + \mathcal O_{\chi}(h^{\infty}), \quad \chi \in C^{\infty}_0({\rm int}(M_+(\psi));[0,1]). 
\end{equation} 

\begin{prop}
\label{prop:HS}
Let $M_+(\psi) \subseteq \widetilde{W} \Subset W \Subset \{x\in M; \Delta \varphi (x) >0\}$, and let 
\[ 
\Pi: L^2_{\varphi}(W,\omega) \rightarrow H_{\varphi}(W)
\] 
be the orthogonal projection. We have for each $\chi \in C^{\infty}_0({\rm int}(M_+(\psi));[0,1])$, 
\begeq
\label{eq3.33.2}
\tr (1_{\widetilde{W}}\, \Pi \,\chi^2 \Pi 1_{\widetilde{W}}) = \frac{1}{2\pi h}\iint_{\Lambda_{\varphi}} \chi^2(x)  d\xi\wedge dx+ \mO_{\chi,\widetilde{W}}(1).
\endeq
\end{prop}
\begin{proof}
Using a quadratic partition of unity argument, similar to the one employed in the proof of Proposition \ref{prop:trace_Bergman}, we see that it suffices to establish (\ref{eq3.33.2}) for $\chi \in C^{\infty}_0(U;[0,1])$, where $U \Subset {\rm int}(M_+(\psi))$ is a small local holomorphic coordinate chart. In what follows we shall assume therefore $\chi \in C^{\infty}_0(U;[0,1])$. 

\medskip
\noindent
The operator $\chi\, \Pi\, 1_{\widetilde{W}}$ is of Hilbert-Schmidt class on $L^2_{\varphi}(M,\omega)$ and we have 
\begeq
\label{eq3.33.3}
\tr (1_{\widetilde{W}}\, \Pi \,\chi^2 \Pi 1_{\widetilde{W}}) = \|\chi\, \Pi\, 1_{\widetilde{W}}\|^2_{{\rm HS}},
\endeq
in view of the following general observation: let $\mathcal H$ be a complex separable Hilbert space and let $T \in \mathcal L(\mathcal  H, \mathcal H)$ be a Hilbert-Schmidt operator. Then $T^*T$ is of trace class on $\mathcal H$ and ${\rm tr}(T^* T) = \|T\|^2_{{\rm HS}}$. 

\medskip
\noindent
We get, using (\ref{eq3.1.0.5}) and (\ref{eq3.33.3}), 
\begeq
\label{eq3.33.4} 
\tr (1_{\widetilde{W}}\, \Pi \,\chi^2 \Pi 1_{\widetilde{W}}) = \int\!\!\!\int \chi^2(x) 1_{\widetilde{W}}(y) |K(x,y)|^2 e^{-2(\varphi(x) + \varphi(y))/h}\, \omega(y,dy\,d\overline{y})\, \omega(x,dx\,d\overline{x}).
\endeq
Let $\widetilde{\chi} \in C^{\infty}_0(U;[0,1])$ be such that $\widetilde{\chi} = 1$ in a neighborhood of ${\rm supp}\, \chi$, and let us write, using 
(\ref{eq3.33.4}) and Proposition \ref{prop:off-diag}, 
\begin{multline}
\label{eq3.33.5} 
\tr (1_{\widetilde{W}}\, \Pi \,\chi^2 \Pi 1_{\widetilde{W}}) \\ 
= \int\!\!\!\int_{U \times U} \chi^2(x) \widetilde{\chi}^2(y) |K(x,y)|^2 e^{-2(\varphi(x) + \varphi(y))/h}\, \omega(y,dy d\overline{y})\, \omega(x,dx d\overline{x}) + \mathcal O_{\chi,\widetilde{W}}(h^{\infty}). 
\end{multline} 
Applying Proposition \ref{Bergman_kernel} to (\ref{eq3.33.5}) we get, using also (\ref{eq3.1.0.18}), 
\begin{multline}
\label{eq3.33.6} 
\tr (1_{\widetilde{W}}\, \Pi \,\chi^2 \Pi 1_{\widetilde{W}}) \\ 
= \frac{1}{h^2} \int\!\!\!\int_{U \times U} \chi^2(x) \widetilde{\chi}^2(y) e^{\frac{1}{h} \left(4{\rm Re}\, \Psi(x,\overline{y}) - 2\varphi(x) - 2\varphi(y)\right)} 
|a(x,\overline{y};h)|^2\, \omega(y,dy d\overline{y})\, \omega(x,dx d\overline{x}) \\ 
+ \mathcal O_{\chi,\widetilde{W}}(h^{\infty}). 
\end{multline}
We saw in the proof of~\cite[Proposition 2.1]{HiSt} that 
\begeq
\label{eq3.33.7} 
4{\rm Re}\, \Psi(x,\overline{y}) - 2\varphi(x) - 2\varphi(y) = - 2 \varphi''_{x\overline{x}}(x) |y-x|^2 + \mathcal O(|y-x|^3),
\endeq
and in order to understand the double integral in (\ref{eq3.33.6}) we may apply therefore the method of stationary phase to the $y$--integration (Laplace integrals). A straightforward application of the method of stationary phase to (\ref{eq3.33.6}), together with (\ref{eq3.33.7}) and (\ref{eq1.29.2}), gives 
\begin{multline}
\label{eq3.33.8} 
\tr (1_{\widetilde{W}}\, \Pi \,\chi^2 \Pi 1_{\widetilde{W}}) \\ 
= \frac{1}{h^2} \int_U \chi^2(x) \frac{2\pi h}{4 \varphi''_{x\overline{x}}(x)} 
\left(|a_0(x,\overline{x})|^2 + \mathcal O(h)\right) e^{-2f(x)}\, \omega(x,dx\,d\overline{x}) + \mathcal O_{\chi,\widetilde{W}}(h^{\infty}) \\
= \frac{1}{2\pi h} \int_U \chi^2(x) \Delta_g \varphi(x)\, \omega(x,dx\,d\overline{x}) + \mathcal O_{\chi,\widetilde{W}}(1) = \frac{1}{2\pi h}\iint_{\Lambda_{\varphi}} \chi^2(x)  d\xi\wedge dx+ \mO_{\chi,\widetilde{W}}(1).
\end{multline}
Here we have also used (\ref{eq_Lapl}), (\ref{eq3.1.0.8}), and (\ref{eq1.31}). The proof is complete. 
\end{proof}

\medskip
\noindent
We obtain, combining (\ref{eq3.33.1}) and Proposition \ref{prop:HS}, 
\begin{equation}
\label{eq3.37}
\widetilde{N} \geq \frac{1}{2\pi h}\iint_{\Lambda_{\varphi}} \chi^2(x)  d\xi\wedge dx - \mO_\chi(1), \quad \chi \in C^{\infty}_0({\rm int}(M_+(\psi)); [0,1]). 
\end{equation}
Choosing an increasing sequence of functions $\chi \in C^{\infty}_0({\rm int}(M_+(\psi);[0,1])$ tending pointwise to $1_{{\rm int}(M_+(\psi))}$, and recalling (\ref{eq1.31}) again, we complete the proof of Theorem \ref{thm:LowerBd}.

\section{Singular values for small exponential decay rates: thin bands of separation}
\label{tb}

\medskip
\noindent
In this section, we continue to work on a compact Riemann surface $M$ equipped with a conformal Riemannian metric $g$ given in (\ref{eq_metric0}), (\ref{eq_metric1}). The metric $g$ gives rise to a norm on each tangent space $T_z M$, $z\in M$, which will be denoted by 
$|\cdot|_z$. Furthermore, associated to $g$ is the Laplacian $\Delta = \Delta_g \leq 0$, see (\ref{eq_Lapl}), (\ref{eq_Laplace}), (\ref{eq1.4.0.5}).
Our purpose here is to prove Theorem \ref{1csp4_n}. When doing so, we assume that $\varphi \in C^{\infty}(M;\mathbb{ R})$ is such that
\begin{equation}
\label{2.tb0}
d\Delta \varphi \ne 0 \hbox{ on }\gamma :=\{z\in M;\,\Delta \varphi(z)=0\}.
\end{equation}
Until further notice we shall also assume that $\gamma$ is connected, hence a simple closed curve, for which we choose a unit speed parametrization 
\begin{equation}
\label{3.tb0}
\gamma :\mathbb{ T}_\lambda \to M,\ \gamma \in C^\infty ,\ |\dot{\gamma}(x)|_{\gamma(x)} = 1 \hbox{ for all }x\in \mathbb{ T}_\lambda.
\end{equation}
Here $\mathbb{ T}_\lambda =\mathbb{R}/(\lambda \mathbb{ Z})$, for a suitable $\lambda >0$. 

\medskip
\noindent
Let $\widetilde{\mathbb{ T}}=\mathbb{ T}_\lambda+ i(-\alpha ,\alpha)$, $\alpha >0$, be a small open neighborhood of $\mathbb{T}_\lambda$ in $\mathbb{ T}_\lambda+i\mathbb{R}$, and let $\widetilde{\gamma}:\widetilde{\mathbb{ T}} \to M$ be an almost holomorphic extension of $\gamma$ in \eqref{3.tb0}, so that 
\begin{equation}\label{3.tb0.7.5}
\widetilde{\gamma}|_{\mathbb{ T}_\lambda} = \gamma, \quad 
\text{ and }\quad \overline{\partial}\widetilde{\gamma}(w) = \mO(|\Im w|^\infty),\quad w\in \widetilde{\mathbb T}. 
\end{equation}
Here the second equation in (\ref{3.tb0.7.5}) is understood introducing local holomorphic coordinates near $w\in \widetilde{\mathbb T}$ and $\widetilde{\gamma}(w)\in M$, and viewing $\widetilde{\gamma}$ locally as a map: $\mathbb C \rightarrow \mathbb C$. A construction of such an almost holomorphic extension is given in Appendix \ref{sec:almholexten}. 

\medskip
\noindent
Using the inverse function theorem we obtain, after possibly decreasing $\alpha$, that the map $\widetilde{\gamma }$ is a $C^{\infty}$ diffeomorphism from $\widetilde{\mathbb{ T}}$ onto $\Omega :=\widetilde{\gamma }(\widetilde{\mathbb{ T}})$, an open neighborhood of $\gamma$ in $M$. (Here and below we frequently identify curves with their image sets.) 

\medskip
\noindent
When $w\in \widetilde{\mathbb{ T}}$, we write $w = x + iy$, for $x\in \mathbb{T}_\lambda$ and $y\in (-\alpha,\alpha)$. The complex manifold $\widetilde{\mathbb{ T}}$ carries the globally defined smooth real vector fields $\frac{\partial}{\partial x}$, $\frac{\partial}{\partial y}$, that are linearly independent at each point. The vector fields $(\frac{\partial}{\partial x}, \frac{\partial}{\partial y})$ form therefore a global frame for $\widetilde{\mathbb{ T}}$, viewed as a real smooth manifold, and identifying $\widetilde{\mathbb{ T}}$ with $\Omega$ by means of $\widetilde{\gamma}$, we obtain that $(\frac{\partial}{\partial x}, \frac{\partial}{\partial y})$ is a smooth global frame for $\Omega \subseteq M$. Expressing the Riemannian metric $g$ on $\Omega$ with respect to this frame we get 
\begin{equation}
\label{eq:RS_n0.1}
g = g_{xx} dx\otimes dx + g_{xy}\left(dx\otimes dy + dy\otimes dx\right) + g_{yy}dy\otimes dy,
\end{equation}
where $g_{xx}$, $g_{xy}$, $g_{yy} \in C^\infty(\widetilde{\mathbb{ T}})$ are given by 
\[ 
g_{xx} = g\!\left(\frac{\partial}{\partial x},\frac{\partial}{\partial x}\right), \quad 
g_{xy} = g\!\left(\frac{\partial}{\partial x},\frac{\partial}{\partial y}\right), 
\quad g_{yy} = g\!\left(\frac{\partial}{\partial y},\frac{\partial}{\partial y}\right). 
\] 
It follows from (\ref{3.tb0}) that 
\begeq
\label{eq:RS_1}
g_{xx}(x,0) = g\!\left(\frac{\partial}{\partial x}|_{(x,0)},\frac{\partial}{\partial x}|_{(x,0)}\right) = 1, \quad x\in \mathbb{ T}_\lambda,
\endeq
and therefore 
\begin{equation}
\label{eq:RS_2}
g_{xx}(x,y)=1 + \mathcal{O}(y) \quad \text{on }\,\, \widetilde{\mathbb T}.
\end{equation}

\medskip
\noindent
We shall next compare the expressions (\ref{eq:RS_n0.1}), (\ref{eq:RS_2}) for the conformal metric $g$ on $\Omega$ with a local expression implied by (\ref{eq_metric1}). To this end, we work locally near a point $w_0 \in \mathbb{ T}_\lambda \subseteq \widetilde{\mathbb T}$ and let $(U,z)$ be a local holomorphic coordinate neighborhood of $\widetilde{\gamma}(w_0) \in \Omega$. Writing $z = \widetilde{\gamma}(w)$, where $w = x + iy \in \widetilde{\mathbb T}$, $x\in \mathbb{ T}_\lambda$, $y\in (-\alpha,\alpha)$, we get using (\ref{3.tb0.7.5}), 
\begin{equation}
\label{eq:RS_n1}
dz = \frac{\partial \widetilde{\gamma}}{\partial w}(w)\, dw + \frac{\partial \widetilde{\gamma}}{\partial \overline{w}}(w)\, d{\overline{w}}
= \frac{\partial \widetilde{\gamma}}{\partial w}(w)\, dw + \mathcal{ O}(|\Im w|^\infty)\,d{\overline{w}}, 
\end{equation}
\begin{equation}
\label{eq:RS_n2}
d\overline{z} = \frac{\partial \overline{\widetilde{\gamma}}}{\partial w}(w)\, dw + \frac{\partial \overline{\widetilde{\gamma}}}{\partial \overline{w}}(w)\, d{\overline{w}}
= \frac{\partial \overline{\widetilde{\gamma}}}{\partial \overline{w}}(w)\, d\overline{w} + \mathcal{ O}(|\Im w|^\infty)\,dw. 
\end{equation}
Combining (\ref{eq_metric1}), (\ref{eq:RS_n1}), and (\ref{eq:RS_n2}) we obtain therefore
\begin{multline}
\label{eq:RS_n3} 
g = \frac{1}{2} g_U(\widetilde{\gamma}(w)) |\partial_w \widetilde{\gamma}(w)|^2 \left(dw \otimes d\overline{w} + d\overline{w}\otimes dw\right) \\
+ \mathcal{ O}(|\Im w|^\infty)(dw\otimes dw, d{\overline{w}}\otimes dw,dw\otimes d\overline{w},d\overline{w}\otimes d\overline{w}).
\end{multline} 
Writing $w = x + iy$ and using that 
\[
\partial_x \widetilde\gamma(w) = \partial_w \widetilde\gamma(w) + \partial_{\overline{w}} \widetilde\gamma(w) 
    = \partial_w \widetilde\gamma(w) +\mathcal{ O}(|\Im w|^\infty), 
\]
we can express (\ref{eq:RS_n3}) as follows
\begin{multline}
\label{eq:RS_n4}
g = g_U(\widetilde{\gamma}(x+iy))\left|\partial_x \widetilde\gamma(x+iy)\right|^2 (dx\otimes dx + dy\otimes dy) \\
+ \mathcal{ O}(|y|^\infty)(dx\otimes dx,dx\otimes dy, dy\otimes dx, dy\otimes dy).
\end{multline}
Here we find similarly to (\ref{eq:RS_2}), 
\[
g_U(\widetilde{\gamma}(x+iy))\left|\partial_x \widetilde\gamma(x+iy)\right|^2 = 1 + \mathcal O(|y|).
\] 
Comparing (\ref{eq:RS_n0.1}), (\ref{eq:RS_2}), and  (\ref{eq:RS_n4}) we conclude that 
\[ 
g_{yy}(x,y) = g_{xx}(x,y) + \mathcal{ O}(|y|^\infty), \quad g_{xy}(x,y)=\mathcal{ O}(|y|^\infty). 
\] 
The metric $g$ in (\ref{eq:RS_n0.1}) takes therefore the following form, 
\begin{equation}
\label{eq:RS_n0.3}
g = \theta(x,y) (dx\otimes dx +dy\otimes dy) + \mathcal{ O}(|y|^\infty)(dx\otimes dx,dx\otimes dy,dy\otimes dx, dy\otimes dy). 
\end{equation}
Here $\theta \in C^{\infty}(\widetilde{\mathbb T};\mathbb R)$ is such that 
\begin{equation}
\label{eq:RS_n0.2}
\theta(x,y) = 1 + \mathcal{O}(|y|). 
\end{equation}
Let us record for future use that 
\begin{equation}\label{eq:RS_n0.3b}
\det g = \theta(x,y)^2 (1 + \mathcal{ O}(|y|^\infty)). 
\end{equation}

\medskip
\noindent
We shall next express the Laplacian $\Delta$ in the $w = x+ iy$ variables. Writing $(x_1, x_2) = (x,y)$ and using the general formula 
\begin{equation*}
\Delta u = \frac{1}{\sqrt{\det g}}\sum_{j,k=1}^2 \partial_{x_j} \left(g^{jk}\, \sqrt{\det g}\, \partial_{x_k} u\right), \quad (g^{jk}) = (g_{jk})^{-1}, 
\end{equation*}
together with (\ref{eq:RS_n0.3}), (\ref{eq:RS_n0.2}) we get 
\begin{equation}\label{1.tb0,5}
\begin{split}
\Delta &= \frac{1}{\theta}(\partial_x^2 + \partial_y^2) + 
\mathcal{ O}(|y|^\infty)(\partial_x^2,\partial_{x}\partial_{y},\partial_y^2,\partial_{x},\partial_{y})\\
&= (1+\mathcal{O}(y))\Delta_{x,y} + 
\mathcal{ O}(|y|^\infty)(\partial_x^2,\partial_{x}\partial_{y},\partial_y^2,\partial_{x},\partial_{y}), \quad 
x\in \mathbb{ T}_\lambda, \,\, |y|<\alpha .
\end{split}
\end{equation}
We shall also fix an orientation on $\gamma$ so that
\begin{equation}\label{3,5.tb0}
\partial _y \Delta \varphi(w) >0, \hbox{ when }y=0.
\end{equation}

\bigskip
\noindent
Let us define some symbol spaces:
\begin{itemize}
\item $S^0(\mathbb{ T}_\lambda ):=C^\infty (\mathbb{ T}_{\lambda })$. When $u=u(\cdot ,\epsilon )$ belongs to this space and also depends on
the small parameter $\epsilon > 0$, we require $u$ and its derivatives with respect to $x$ to be uniformly bounded with respect to
$\epsilon $.
\item We say that $u=u(x,\epsilon )\in S^0(\mathbb{ T}_{\lambda })$ belongs to $S^0_{\mathrm{cl}}(\mathbb{ T}_\lambda )$ if we have
\begin{equation}
\label{1,5.tb0,5}
u(x,\epsilon )\sim \sum_{k=0}^\infty u_k(x)\epsilon ^k \hbox{ in }C^\infty(\mathbb{ T}_\lambda) .
\end{equation}
\item
We let $S^0_{\mathrm{fcl}}(\mathbb{ T}_\lambda )$ be the space of formal sums as in (\ref{1,5.tb0,5}), with
$u_k\in C^\infty (\mathbb{T}_\lambda )$.
\end{itemize}

\medskip
\noindent
We consider $\Delta$ in a thin band of variable width,
\begin{equation}
\label{2.tb0,5}
\Omega =\Omega _f:\ x\in \mathbb{ T}_\lambda ,\ -\epsilon f_-(x) < y < \epsilon f_+(x),
\end{equation}
where
\begin{equation}\label{2,5.tb0,5}
f_\pm (x)=f_\pm(x,\epsilon )\sim f_\pm^0(x)+\epsilon f^1_\pm(x)+...\in S^0_{\mathrm{fcl}}(\mathbb{ T}_\lambda ),\ \ \ f^0_\pm (x)\asymp 1.
\end{equation}

\medskip
\noindent
Recalling the definition of the open sets $M_\pm$ in (\ref{eq2.4}), that do not depend on $\tau =\epsilon^3\widehat{\tau }$,
$\widehat{\tau} \asymp 1$, we write
\begin{equation}
\label{5.csp2}
\Omega _\pm=M_\pm\setminus \Omega _f^\pm ,
\end{equation}
where
\begin{equation}
\label{6.csp2}
\Omega _f^\pm:=\{(x,y)\in \Omega _f;\, \pm y>0 \}=\{ (x,y);\, 0<\pm y\le \epsilon f_\pm (x,\epsilon ) \}.
\end{equation}
We reduce to the band $\widetilde{\Omega }=\mathbb{ T}_\lambda \times ]-1,1[$ of fixed width by the change of variables
\begin{equation}
\label{3.tb0,5}
\begin{split}
x&=\widetilde{x},\\
    y&=\epsilon
    \left(f_+(\widetilde{x})\frac{1}{2}(1+\widetilde{y})-f_-(\widetilde{x})\frac{1}{2}(1-\widetilde{y}
    )\right)\\
    &=\frac{\epsilon }{2}\left(
      f_+(\widetilde{x})-f_-(\widetilde{x})+(f_+(\widetilde{x})+f_-(\widetilde{x}))\widetilde{y}
    \right).
  \end{split}
\end{equation}
Applying the chain rule,
\[
  \begin{split}
\partial _{\widetilde{x}}&=\frac{\partial x}{\partial
  \widetilde{x}}\partial _x+\frac{\partial y}{\partial
  \widetilde{x}}\partial _y
=\partial _x+\frac{\epsilon
}{2}\left((f_+'(\widetilde{x})-f_-'(\widetilde{x}))
+(f_+'(\widetilde{x})+f_-'(\widetilde{x}))\widetilde{y}
\right)\partial _y,
\\
\partial _{\widetilde{y}}&=\frac{\partial x}{\partial
  \widetilde{y}}\partial _x+\frac{\partial y}{\partial
  \widetilde{y}}\partial _y
=\frac{\epsilon }{2}\left(f_+(\widetilde{x})+f_-(\widetilde{x})\right)\partial _y,
\end{split}
\]
we get
\begin{equation}\label{0.tb1}
  \begin{cases}\begin{displaystyle}
\partial _y=\frac{2}{\epsilon
  (f_++f_-)}\partial _{\widetilde{y}},
\end{displaystyle}
\\
\begin{displaystyle}
\partial _x=\partial
_{\widetilde{x}}-\frac{f_+'-f_-'+(f_+'+f_-')\widetilde{y}}{f_++f_-}\partial _{\widetilde{y}},
\end{displaystyle}
\end{cases}
\end{equation}
where $f_\pm=f_\pm (\widetilde{x})$, $\ f_\pm'=f_\pm'(\widetilde{x})$.
This gives
$$
\Delta _{x,y}=\left(\partial
  _{\widetilde{x}}-\frac{f_+'-f_-'+(f_+'+f_-')\widetilde{y}}{f_++f_-}\partial
  _{\widetilde{y}}\right)^2
+\frac{4}{\epsilon ^2(f_++f_-)^2}\partial _{\widetilde{y}}^2,
$$
\[
  \begin{split}
&\frac{\epsilon ^2(f_++f_-)^2}{4}\Delta _{x,y}=\frac{\epsilon ^2(f_++f_-)^2}{4}\left(\partial
  _{\widetilde{x}}-\frac{f_+'-f_-'+(f_+'+f_-')\widetilde{y}}{f_++f_-}\partial
  _{\widetilde{y}}\right)^2
+\partial _{\widetilde{y}}^2\\
&=\frac{\epsilon ^2(f_++f_-)^2}{4}
\left(\partial
  _{\widetilde{x}}^2-\partial _{\widetilde{x}}\circ
  \frac{ f_+'-f_-'+\widetilde{y} (f_+'+f_-')}{f_++f_-}\partial _{\widetilde{y}}
  -\frac{f_+'-f_-'+\widetilde{y} (f_+'+f_-')}{f_++f_-}\partial
  _{\widetilde{y}}\circ \partial
  _{\widetilde{x}}\right.\\&\left.\hskip 5mm
+\left(\frac{f_+'-f_-'+\widetilde{y}(f_+'+f_-')}{f_++f_-}\partial _{\widetilde{y}} \right)^2
\right)+\partial _{\widetilde{y}}^2\\
& =\frac{(f_++f_-)^2}{4}(\epsilon \partial _{\widetilde{x}})^2+\partial
_{\widetilde{y}}^2+\epsilon
\widetilde{a}(\widetilde{x},\widetilde{y})\epsilon \partial
_{\widetilde{x}}\partial _{\widetilde{y}}+\epsilon
^2\widetilde{b}(\widetilde{x},\widetilde{y})\partial
_{\widetilde{y}}+\epsilon
^2\widetilde{c}(\widetilde{x},\widetilde{y})\partial _{\widetilde{y}}^2,
\end{split}
\]
where $\widetilde{a},\, \widetilde{b},\, \widetilde{c}\in S^0_{\mathrm{fcl}}(\overline{\widetilde{\Omega }})$, with the natural
definition of $S^0_{\mathrm{fcl}}(\overline{\widetilde{\Omega}})$. The formulae (\ref{1.tb0,5}) and (\ref{0.tb1}) give that
$$
\epsilon^2 \Delta = \epsilon ^2\Delta _{x,y}
+\mathcal{ O}(\epsilon)
\left( (\epsilon \partial _{\widetilde{x}})^2, \partial
  _{\widetilde{y }}^2, \epsilon \partial _{\widetilde{x}}\partial
  _{\widetilde{y}},\epsilon \partial_{\widetilde{x}}, \partial _{\widetilde{y}} \right),
$$
where the remainder is an abbreviation for
\[
\mathcal{ O}(\epsilon)(\epsilon \partial _{\widetilde{x}})^2+ \mathcal{ O}(\epsilon)\partial
  _{\widetilde{y }}^2+\mathcal{ O}(\epsilon) \epsilon \partial _{\widetilde{x}} \partial
  _{\widetilde{y}} + \mathcal O(\epsilon) (\epsilon \partial_{\widetilde{x}}) + \mathcal{ O}(\epsilon)\partial _{\widetilde{y}},
\]
and $\mathcal{ O}(\epsilon )$ indicates a smooth function which is $\mathcal{ O}(\epsilon )$ with all its derivatives. We obtain that
\begin{equation}\label{1.tb1}
\begin{split}
    P:=&-\frac{\epsilon ^2(f_++f_-)^2}{4} \Delta =
    -\frac{\epsilon ^2(f_++f_-)^2}{4}\Delta _{x,y}+\mathcal{
      O}(\epsilon )
\left( (\epsilon \partial
_{\widetilde{x}})^2,\epsilon \partial_{\widetilde{x}}\partial
_{\widetilde{y}}, \partial _{\widetilde{y}}^2,\epsilon \partial _{\widetilde{x}},\partial _{\widetilde{y}}\right)\\
=&\frac{(f_++f_-)^2}{4}(\epsilon
D_{\widetilde{x}})^2+D_{\widetilde{y}}^2 \\ &+
\epsilon \left(a(\widetilde{x },\widetilde{y})
  \epsilon D_{\widetilde{x}}D_{\widetilde{y}}+b(\widetilde{x},\widetilde{y})D_{\widetilde{y}}
  +c(\widetilde{x},\widetilde{y})D_{\widetilde{y}}^2+d(\widetilde{x},\widetilde{y})(\epsilon
  D_{\widetilde{x}})^2
+e(\widetilde{x },\widetilde{y})\epsilon D_{\widetilde{x}}
\right),
\end{split}
\end{equation}
where $a,\,b,\,c,\, d,\ e\in S^0_{\mathrm{fcl}}(\overline{\widetilde{\Omega }})$. In our asymptotic analysis of the operator $P$, we shall follow the general ideas for treating spectral problems with a mixture of fast and slow quantities, appearing in the Born-Oppenheimer
approximation, thin band problems, and elsewhere, by means of pseudodifferential operators in the slow variables, with symbols that
are operator valued. See \cite{GeMaSj91}, \cite[Chapter 13]{DiSj}.

\medskip
\noindent
Here we consider
\[
P:\, H^2(\widetilde{\Omega })\cap H_0^1(\widetilde{\Omega } )\to H^0(\widetilde{\Omega }),
\]
and we shall view $P$ as an $\epsilon $-pseudodifferential operator on $\mathbb{T}_\lambda $ with the operator valued symbol given by
\begin{equation}\label{3.tb2}
\begin{split}
P(\widetilde{x},\widetilde{\xi };\epsilon )
=&\frac{(f_++f_-)^2}{4}\widetilde{\xi }^2+D_{\widetilde{y}}^2\\ &+
\epsilon \left(a(\widetilde{x },\widetilde{y})
\widetilde{\xi
}D_{\widetilde{y}}+b(\widetilde{x},\widetilde{y})D_{\widetilde{y}}+c(\widetilde{x},\widetilde{y})D_{\widetilde{y}}^2+d(\widetilde{x},\widetilde{y})\widetilde{\xi
}^2+e(\widetilde{x},\widetilde{y})\widetilde{\xi }\right):\\
&\hskip 1 cm (H_{\widetilde{\xi }}^2\cap H_0^1)(]-1,1[)\to H^0(]-1,1[).
\end{split}\end{equation}
The subscript $\widetilde{\xi } $ here indicates that the space $H^2(]-1,1[)$ is equipped with the $\widetilde{\xi }
$-dependent norm
\begin{equation}\label{4.tb2}
\|u\|_{H^2_{\widetilde{\xi } }}=(\langle \widetilde{\xi }\rangle^2\|u\|)^2+(\langle
\widetilde{\xi }\rangle \|D_{\widetilde{y}}u\|)^2+\|D_{\widetilde{y}}^2u\|^2,
\end{equation}
where $\|v\|=\|v\|_{L^2(]-1,1[)}$. It follows from (\ref{3.tb2}) that $P(\widetilde{x},\widetilde{\xi
};\epsilon )$ is a formal classical symbol on $\mathbb{ T}_\lambda \times \mathbb{
  R}_{\widetilde{\xi }}$ of the class
\begin{equation}\label{5.tb2}
  S^{0}_{\mathrm{fcl}}\left(\mathbb{ T}_\lambda \times \mathbb{ R}_{\widetilde{\xi }};
  \mathcal{ L}
 \left((H_{\widetilde{\xi }}^2\cap H_0^1)(]-1,1[),H^0(]-1,1[)\right)\right),
\end{equation}
based on
$S^0(\mathbb{ T}_\lambda \times \mathbb{ R}_{\widetilde{\xi }},...)  $, or equivalently, on the classical symbol space $S^0_{1,0}(...)$) with
the required symbol estimates $\| \partial _{\widetilde x}^j\partial _{\widetilde{\xi }}^ka\|\le
\mathcal{ O}(\langle \widetilde{\xi }\rangle^{-k})$. We next observe that the symbol $P(\widetilde{x},\widetilde{\xi};\epsilon)$ in (\ref{3.tb2}) is elliptic for $\epsilon >0$ is small enough, i.e.\ $P(\widetilde{x},\widetilde{\xi };\epsilon )$ has a two-sided inverse satisfying
\begin{equation}\label{6.tb2}
P(\widetilde{x},\widetilde{\xi };\epsilon )^{-1}\in
S^{0}_{\mathrm{cl}}(\mathbb{ T}_\lambda \times \mathbb{ R}_{\widetilde{\xi }};
\mathcal{ L} \left(H^0(]-1,1[) ,(H_{\widetilde{\xi }}^2\cap H_0^1)(]-1,1[) \right),
\end{equation}
uniformly for $0<\epsilon \le \epsilon _0$, for some $\epsilon _0$ with $0<\epsilon _0\ll 1$.\footnote {More precisely, we take
realizations of the formal classical symbols $f_\pm$, $a$, $b$, $c$, $d$, $e$ and get a true symbol denoted by the same letter, whose inverse
satisfies (\ref{6.tb2}).}

\medskip
\noindent
Let us verify the ellipticity of $P(\widetilde{x},\widetilde{\xi};\epsilon)$ in (\ref{3.tb2}): we have, with the $L^2$ norms and scalar products, for $u\in (H^2_{\widetilde{\xi}}\cap H^1_0)(]-1,1[)$,
\begin{multline*}
\left|\left|\left(\frac{(f_+ + f_-)^2}{4}\widetilde{\xi}^2 + D_{\widetilde{y}}^2\right)u\right|\right|^2 \\
= \frac{(f_+ + f_-)^4}{16}\widetilde{\xi}^4 \|u\|^2 +
\frac{(f_+ + f_-)^2}{4}\widetilde{\xi}^2\, 2{\rm Re}\, (u,D_{\widetilde{y}}^2 u) + \|D_{\widetilde{y}}^2 u\|^2 \\
\geq \frac{1}{C} \left(\widetilde{\xi}^4 \|u\|^2 + \widetilde{\xi}^2 \|D_{\widetilde{y}}u\|^2 + \|D_{\widetilde{y}}^2 u\|^2\right) \geq \frac{1}{C} \|u\|_{H^2_{\widetilde{\xi}}}^2,\quad |\widetilde{\xi}| \geq 1,
\end{multline*}
which gives the ellipticity for $\widetilde{\xi}$ away from zero. When $|\widetilde{\xi}| \leq 1$, we write for
$u\in (H^2_{\widetilde{\xi}}\cap H^1_0)(]-1,1[)$,
\[
\left(\left(D_y^2 + \frac{(f_+ + f_-)^2}{4}\widetilde{\xi}^2\right)u,u\right) \geq \|D_{\widetilde{y}}u\|^2 \geq c \|D_{\widetilde{y}}u\|\, \|u\|.
\]
Here we have also used the Poincar\'e inequality, since $u\in H^1_0(]-1,1[)$. It follows that
\[
\norm{D_{\widetilde{y}}u} \leq C \left|\left|\left(\frac{(f_+ + f_-)^2}{4}\widetilde{\xi}^2 + D_{\widetilde{y}}^2\right)u\right|\right|,
\]
and thus, also,
\begin{multline*}
\norm{D^2_{\widetilde{y}} u} \leq \left|\left|\left(\frac{(f_+ + f_-)^2}{4}\widetilde{\xi}^2 + D_{\widetilde{y}}^2\right)u\right|\right| + C \norm{u}\\
\leq \left|\left|\left(\frac{(f_+ + f_-)^2}{4}\widetilde{\xi}^2 + D_{\widetilde{y}}^2\right)u\right|\right| + C \norm{D_{\widetilde{y}}u} \leq
C \left|\left|\left(\frac{(f_+ + f_-)^2}{4}\widetilde{\xi}^2 + D_{\widetilde{y}}^2\right)u\right|\right|.
\end{multline*}
We have now shown the ellipticity of $P(\widetilde{x},\widetilde{\xi};\epsilon)$, for $\epsilon > 0$ small enough.

\medskip
\noindent
It follows from this that $P(\widetilde{x},\widetilde{\xi };\epsilon)$ is Fredholm and hence bijective for $\epsilon $ small.
We notice that more generally $P(\widetilde{x},\widetilde{\xi};\epsilon)$ in (\ref{3.tb2}) is a symbol of the class
\begin{equation}
\label{5gen.tb2}
S^{0}_{\mathrm{fcl}}\left(\mathbb{ T}_\lambda \times \mathbb{ R}_{\widetilde{\xi }};
\mathcal{ L} \left((H^{k+2}_{\widetilde{\xi }}\cap H_0^1)(]-1,1[),H^k(]-1,[)\right)\right),
\end{equation}
for every $k\in \mathbb{ N}$, and that (\ref{6.tb2}) extends to
\begin{equation}
\label{6gen.tb2}
P(\widetilde{x},\widetilde{\xi };\epsilon )^{-1}\in
S^{0}_{\mathrm{cl}}(\mathbb{ T}_\lambda \times \mathbb{ R}_{\widetilde{\xi }};
\mathcal{ L}\left(H^k(]-1,1[) ,(H_{\widetilde{\xi }}^{k+2}\cap H_0^1)(]-1,1[) \right).
\end{equation}

\medskip
\noindent
From the standard semiclassical pseudodifferential calculus, we conclude that for $\epsilon >0$ small, (any realization of) $P:\, (H^2\cap
H^1_0)(\widetilde{\Omega })\to L^2(\widetilde{\Omega })$ has a bounded inverse $Q$ such that
\begin{equation}\label{1.tb2,5}
Q,\ D_{\widetilde{y}}Q,\ D_{\widetilde{y}}^2Q,\ \epsilon
D_{\widetilde{x}}Q,\ (\epsilon D_{\widetilde{x}})^2Q,\
D_{\widetilde{y}}\epsilon D_{\widetilde{x}}Q\ =\mathcal{ O}(1):\, L^2(\widetilde{\Omega}) \to L^2(\widetilde{\Omega}),
\end{equation}
uniformly with respect to $\epsilon >0$.

\medskip
\noindent
Consider the Dirichlet problem
\begin{equation}\label{1.tb2,7}
\begin{cases}
\Delta u=0\hbox{ in }\Omega ,\\
{{u}_\vert}_{\partial \Omega }=v,
\end{cases}
\hbox{ where }v\in C^\infty (\partial \Omega ),\ \Omega =\Omega _f.
\end{equation}
The boundary condition in (\ref{1.tb2,7}) can be written more explicitly,
\begin{equation}\label{2.tb2,7}
v(x,\pm\epsilon f_\pm(x))=:v_\pm (x),
\end{equation}
where $v_\pm=v_\pm (x,\epsilon )\in S^0_{\mathrm{fcl}}(\mathbb{ T}_\lambda
)$.
The well-posedness of this problem in $C^{\infty}(\overline{\Omega})$ is of course classical see e.g.~\cite[Chapter 4]{Es11}, and we shall consider solutions that are formal power series in $\epsilon $.

\medskip
\noindent
Expressing (\ref{1.tb2,7}) in the $\widetilde{x},\, \widetilde{y}$ variables, we get
\begin{equation}\label{3.tb2,7}
Pu=0\hbox{ in }\widetilde{\Omega },\ u(\widetilde{x},\pm 1)=v_\pm (\widetilde{x}).
\end{equation}
As an approximate solution, let us take
\begin{equation}
\label{4.tb2,7}
u_0(\widetilde{x},\widetilde{y})=v_+(\widetilde{x})\frac{1+\widetilde{y}}{2}+v_-(\widetilde{x})\frac{1-\widetilde{y}}{2}\in
C^\infty (\mathbb{ T}_\lambda \times [-1,1]).
\end{equation}
Then $u_0(\widetilde{x},\pm 1)=v_\pm (\widetilde{x})$, and by (\ref{1.tb1}), we have
$Pu_0=\mathcal{O}(\epsilon (\|v_+\|_{C^2}+\|v_-\|_{C^2})$ with all its derivatives. The solution to the problem (\ref{1.tb2,7}) is given by
\begin{equation}\label{5.tb2,7}
u=u_0-QPu_0.
\end{equation}

\medskip
\noindent
We next study an asymptotic expansion of $u$ in powers of $\epsilon $, and recall that $v_\pm(x)=v_\pm (x,\epsilon )\in S^0_{\mathrm{fcl}}$ are
smooth. By (\ref{4.tb2,7}), $u_0$ is smooth on $\overline{\widetilde{\Omega}}=\mathbb{T}_\lambda \times [-1,1]$. It follows from this and (\ref{1.tb1}) that
\begin{equation}\label{6.tb2,7}
Pu_0=\epsilon v_2(\widetilde{x},\widetilde{y}),\ v_2\in
S^0_{\mathrm{fcl}}(\overline{\widetilde{\Omega }}).
\end{equation}
Here we recall (\ref{2,5.tb0,5}). We have in view of (\ref{1.tb1}),
\begin{equation}
\label{1.tb3}
P=D_{\widetilde{y}}^2+\epsilon R,
\end{equation}
where we shall only use that $R:\, C^\infty (\overline{\widetilde{\Omega }})\to C^\infty (\overline{\widetilde{\Omega }})$ is a linear operator with an asymptotic expansion $R\sim R_0+\epsilon R_1+ \ldots ,$ and that
$D_{\widetilde{y}}^2:\, H^2\cap H_0^1\to L^2$ on $]-1,1[$ is bijective, and similarly for the higher order Sobolev spaces. Asymptotically in
powers of $\epsilon $, the inverse $Q=P^{-1}$ has the expansion
\begin{equation}\label{2.tb3}\begin{split}
  Q&=\left(1+\epsilon \left(D_{\widetilde{
    y }}^2 \right)^{-1}R
\right)^{-1}(D_{\widetilde{y}}^2)^{-1}\\
&=(D_{\widetilde{y}}^2)^{-1}-\epsilon
(D_{\widetilde{y}}^2)^{-1}R (D_{\widetilde{y}}^2)^{-1}+\epsilon
^2\left( (D_{\widetilde{y}}^2)^{-1}R\right)^2 (D_{\widetilde{y}}^2)^{-1} - \ldots ,
\end{split}\end{equation}
and applying this to (\ref{6.tb2,7}), (\ref{5.tb2,7}), (\ref{1.tb3}), we get
\begin{equation}\label{3.tb3}
u\sim u_0+\epsilon u_1+\epsilon ^2u_2+ \ldots
\hbox{ in }C^\infty (\overline{\widetilde{\Omega }}),
\end{equation}
where $u_0$ is given by (\ref{4.tb2,7}) and
\begin{equation}\label{4.tb3}
u_j\in C^\infty
(\overline{\widetilde{\Omega }})\hbox{ for }j\ge 1.
\end{equation}

\medskip
\noindent
The asymptotic expansion in (\ref{3.tb3}) holds in the $C^\infty$ sense, so
\begin{equation}\label{5.tb3}
\partial _{\widetilde{y}}u\sim \partial _{\widetilde{y}}u_0+\partial_{\widetilde{y}}\epsilon u_1+ \partial _{\widetilde{y}}\epsilon ^2u_2 + \ldots
\end{equation}
In particular this holds for $\widetilde{y}=\pm 1$. Using (\ref{0.tb1}), we get the following result for the solution $u$ of the original problem (\ref{1.tb2,7}):

\begin{prop}
\label{1tb3,5}
Let $\mathbb{ T}_\lambda :=\mathbb{ R}/(\lambda \mathbb{ Z})$, let $\gamma:\mathbb{ T}_\lambda\to M$ be a simple closed curve with
$|\dot{\gamma}(x)|_{\gamma(x)}=1$, $x\in \mathbb{ T}_\lambda$, and let $\widetilde{\gamma }:\widetilde{\mathbb{ T}}_\lambda +i(-\alpha,\alpha) \to M$ be an almost holomorphic extension, where $\alpha >0$ is small. Write $\widetilde{\gamma }(w)=\widetilde{\gamma}(x+iy)$. In the $(x,y)$ coordinates, the Laplacian on $M$ takes the form {\rm (\ref{1.tb0,5})}. We consider $\Delta$ in a thin band $\Omega _f$ as in {\rm (\ref{2.tb0,5})}, with $f$ as in {\rm (\ref{2,5.tb0,5})} and introduce the variables $\widetilde{x}$, $\widetilde{y}$ in {\rm (\ref{3.tb0,5})}, transforming the thin band $\Omega _f$ to $\widetilde{\Omega }=\mathbb{ T}_\lambda \times ]-1,1[$. We view
$$
P:=-\frac{\epsilon ^2(f_++f_-)^2}{4}\Delta 
$$
as an operator on $\widetilde{\Omega }$ and we have {\rm (\ref{1.tb1})}. For $\epsilon >0$ small enough, the operator $P:(H_0^1\cap H^2)(\widetilde{\Omega })\to L^2(\widetilde{\Omega }) $ has a bounded inverse $Q$ satisfying {\rm (\ref{1.tb2,5})}.

\noindent
The Dirichlet problem {\rm (\ref{1.tb2,7})} has a unique solution $u\in C^\infty (\overline{\Omega }_f)$. Here the boundary condition is more
explicit in {\rm (\ref{2.tb2,7})}. For the equivalent formulation {\rm (\ref{3.tb2,7})}, the solution $u$ has the asymptotic expansion
{\rm (\ref{3.tb3})} in $C^\infty (\overline{\widetilde{\Omega }})$, where $u_0(\widetilde{x},\widetilde{y})$ is given in {\rm (\ref{4.tb2,7})}.
\end{prop}

\medskip
\noindent
We remark that even though we can get exact solutions to our Dirichlet problem, formal asymptotic ones with $\mathcal{O}(\epsilon ^\infty )$ errors in $C^\infty$ will suffice for our purposes. See also \cite[Section 12.2]{Sj19}.

\medskip
\noindent
Let $\varphi \in C^\infty (M;\mathbb{ R})$ satisfy (\ref{2.tb0}), still with $\gamma $ connected, and choose coordinates
$w=x+iy$ as above. We shall study the Dirichlet problem (\ref{1.tb2,7}), now with the functions $f_\pm(x)=f_\pm(x;\epsilon )\in S^0_{\mathrm{fcl}}(\mathbb{T}_\lambda )$ to be determined and with $v_\pm (x)=\varphi (x,\pm \epsilon f_\pm (x))\pm \tau /2$:
\begin{equation}\label{1.tb4}
\Delta u=0\hbox{ in }\Omega _f,\ \ u(x,\pm \epsilon f_\pm (x))=\varphi
(x,\pm \epsilon f_\pm (x))\pm \tau /2,\ \ u\in C^\infty (\overline{\Omega }_f).
\end{equation}
Here $\tau >0$ will be of the order of magnitude $\epsilon ^3$.

\medskip
\noindent
Our problem is to find $f_\pm =f_\pm(x,\epsilon )\asymp 1$, such that in addition to (\ref{1.tb4}), we have
\begin{equation}\label{2.tb4}
\epsilon \partial _yu(x,\pm \epsilon f_\pm (x))=\epsilon \partial
_y\varphi (x,\pm \epsilon f_\pm (x))+\mathcal{ O}(\epsilon ^\infty )\hbox{ in
}C^\infty .
\end{equation}
We look for $f_\pm$ with an asymptotic expansion
$$
f_\pm (x)=f_\pm (x,\epsilon )\sim f_\pm^0(x)+\epsilon
f_\pm^0(x)+ \ldots \hbox{ in }C^\infty .
$$
The right hand side in (\ref{1.tb4}) will be $\mathcal{ O}(\epsilon^3)$, as we shall see, and
we shall look for $u$ of the form (\ref{7.tb4}) below with an additional factor $\epsilon ^3$ in front, compared to (\ref{3.tb3}).

\medskip
\noindent
We may assume that
\begin{equation}\label{3+.tb4}
\varphi =\mathcal{ O}(\mathrm{dist\,}(\cdot ,\gamma )^3).
\end{equation}
In fact, recall that $\Delta \varphi =0$ on $\gamma $ by (\ref{2.tb0}), and let $\widetilde{\varphi }\in C^\infty (\Omega )$ satisfy
\begin{equation}\label{4+,tb4}
\Delta \widetilde{\varphi}=\Delta \varphi +\mathcal{ O}(\mathrm{dist\,}(\cdot
,\gamma )^\infty ),\ \ \widetilde{\varphi }=\mathcal{
  O}(\mathrm{dist\,}(\cdot ,\gamma )^3).
\end{equation}
Assume that $\widetilde{u}\in C^\infty (\overline{\Omega }_f)$ solves the problem (\ref{1.tb4}), (\ref{2.tb4}), with $\varphi$ replaced by
$\widetilde{\varphi }$. Let $u=\widetilde{u}+\varphi -\widetilde{\varphi}$. Then $\Delta u =\mathcal{ O}(\epsilon ^\infty )$ in $\Omega _f$ and
similarly (\ref{2.tb4}) holds. Modifying $u$ by a term $\mathcal{ O}(\epsilon ^\infty )$ in $C^\infty $, we can get an exact solution
of (\ref{1.tb4}), (\ref{2.tb4}).

\medskip
\noindent
From now on assume (\ref{3+.tb4}), so that
\begin{equation}\label{6+.tb4}
\varphi =\mathcal{ O}(y^3)
\end{equation}
in the coordinates $(x,y)$ on $\Omega $, with $w=x+iy$.

\medskip
\noindent
Let us rewrite the problem in the $(\widetilde{x},\widetilde{y})$ variables, using the same symbol for the function $u$. We expect $\tau
=\mathcal{ O}(\epsilon ^3)$ and put
\begin{equation}\label{3.tb4}
\tau =\epsilon ^3\widehat{\tau },\ \ \tau > 0,
\end{equation}
expecting to have $\widehat{\tau } = \mathcal O(1)$. We look for $u=u(\widetilde{x},\widetilde{y};\epsilon )\in C^\infty (\overline{\widetilde{\Omega}})$ such that (cf.\ (\ref{0.tb1}))
\begin{equation}\label{4.tb4}
Pu=\mathcal{O}(\epsilon ^\infty )\hbox{ in }C^\infty .
\end{equation}
\begin{equation}\label{5.tb4}
u(\widetilde{x},\pm 1)=\varphi (\widetilde{x},\pm \epsilon f_\pm (\widetilde{x},\epsilon ))\pm \frac{\tau }{2},
\end{equation}
\begin{equation}\label{6.tb4}
\partial _{\widetilde{y}}u(\widetilde{x},\pm 1)=\frac{\epsilon
  (f_+(\widetilde{x},\epsilon )+f_-(\widetilde{x},\epsilon
  ))}{2}(\partial _y\varphi )(\widetilde{x},\pm \epsilon f_\pm
(\widetilde{x},\epsilon ))+\mathcal{ O}(\epsilon ^\infty )\hbox{ in
}C^\infty .
\end{equation}
We try $u\in \epsilon ^3S^0_{\mathrm{fcl}}(\overline{\widetilde{\Omega}})$,
\begin{equation}\label{7.tb4}
u(\widetilde{x},\widetilde{y};\epsilon )\sim \epsilon
^3(u_0(\widetilde{x},\widetilde{y})+\epsilon
u_1(\widetilde{x},\widetilde{y})+ \ldots)\hbox{ in }\epsilon ^3C^\infty
(\overline{\widetilde{\Omega }}),
\end{equation}
and use that
\begin{equation}\label{8.tb4}
\partial _{\widetilde{y}}u(\widetilde{x},\widetilde{y};\epsilon )\sim \epsilon
^3(\partial _{\widetilde{y}}u_0(\widetilde{x},\widetilde{y})+\epsilon
\partial _{\widetilde{y}}u_1(\widetilde{x},\widetilde{y})+ \ldots).
\end{equation}
We look for $f_\pm (\widetilde{x},\epsilon )\in
S^0_{\mathrm{fcl}}(\mathbb{ T}_\lambda )$,
\begin{equation}\label{9.tb4}
f_\pm (\widetilde{x},\epsilon )\sim f^0_{\pm}(\widetilde{x})+\epsilon
f^1_\pm (\widetilde{x})+ \ldots
\end{equation}
Recall that by (\ref{4.tb4}), (\ref{1.tb3}),
\begin{equation}\label{10.tb4}
0=(D_{\widetilde{y}}^2+\epsilon R)(u_0+\epsilon
u_1+ \ldots)=D_{\widetilde{y}}^2u_0+\mathcal{ O}(\epsilon ),\ \ D_{\widetilde{y}}^2u_0=0.
\end{equation}
We have, by a Taylor expansion in $y$,
\begin{equation}\label{11.tb4}
\varphi (\widetilde{x},\pm \epsilon f_\pm(\widetilde{x},\epsilon ))=\pm
\frac{\epsilon ^3}{6}f^0_{\pm}(\widetilde{x})^3\partial _y^3\varphi
(\widetilde{x},0)+\mathcal{ O}(\epsilon ^4),
\end{equation}
\begin{equation}\label{12.tb4}
\partial _y\varphi (\widetilde{x},\pm \epsilon f_\pm(\widetilde{x},\epsilon ))=
\frac{\epsilon ^2}{2}f^0_{\pm}(\widetilde{x})^2\partial _y^3\varphi
(\widetilde{x},0)+\mathcal{ O}(\epsilon ^3).
\end{equation}
To the leading order in $\epsilon $ we get from (\ref{3.tb4}), (\ref{4.tb4}), (\ref{5.tb4}), and (\ref{11.tb4}),
\begin{equation}
\label{1.tb5}
D_{\widetilde{y}}^2u_0=0,
\end{equation}
\begin{equation}\label{2.tb5}
u_0(\widetilde{x},\pm 1)=\pm
\left(\frac{1}{6}f_\pm^0(\widetilde{x})^3\varphi
  ^{(3)}(\widetilde{x})+\frac{\widehat{\tau }}{2} \right),\ \ \varphi
^{(3)}(\widetilde{x}):=\partial _y^3\varphi (\widetilde{x},0).
\end{equation}
It follows from (\ref{1.tb5}) that $u_0(\widetilde{x},\widetilde{y})$ is affine in $\widetilde{y}$, so (\ref{2.tb5}) gives that
$$
u_0(\widetilde{x},\widetilde{y})=\left(\frac{\varphi
    ^{(3)}(\widetilde{x})}{6}f^0_+(\widetilde{x})^3+\frac{\widehat{\tau } }{2}
\right)\frac{1+\widetilde{y}}{2}
-
\left(\frac{\varphi
    ^{(3)}(\widetilde{x})}{6}f^0_-(\widetilde{x})^3+\frac{\widehat{\tau
    }}{2} \right)\frac{1-\widetilde{y}}{2},
$$
\begin{equation}\label{3.tb5}
  u_0(\widetilde{x},\widetilde{y})=
  \left(\frac{\varphi
      ^{(3)}(\widetilde{x})}{12}\left(f_+^0(\widetilde{x})^3-f_-^0(\widetilde{x})^3
    \right) \right)
  +
  \left(\frac{\varphi
      ^{(3)}(\widetilde{x})}{12}\left(f_+^0(\widetilde{x})^3+f_-^0(\widetilde{x})^3
    \right) +\frac{\widehat{\tau }}{2} \right) \widetilde{y}.
\end{equation}
Hence
\begin{equation}\label{4.tb5}
  \partial _{\widetilde{y}}u_0(\widetilde{x},\widetilde{y})=
  \frac{\varphi
      ^{(3)}(\widetilde{x})}{12}\left(f_+^0(\widetilde{x})^3+f_-^0(\widetilde{x})^3
    \right) +\frac{\widehat{\tau }}{2}
\end{equation}

\medskip
\noindent
To the leading order ($\epsilon ^3$), (\ref{4.tb5}), (\ref{12.tb4}) and the condition (\ref{6.tb4}) give
$$
\frac{\varphi^{(3)}(\widetilde{x})}{12}\left(f_+^0(\widetilde{x})^3+f_-^0(\widetilde{x})^3
\right) +\frac{\widehat{\tau }}{2} = \frac{1}{4}(f_+^0+f_-^0)(f_\pm^0)^2\varphi ^{(3)}(\widetilde{x}),
$$
which we can write as follows,
\begin{equation}
\label{5.tb5}
\frac{1}{6}((f_+^0)^3+(f_-^0)^3)+\frac{\widehat{\tau }}{\varphi ^{(3)}} = \frac{1}{2}(f_+^0+f_-^0)(f_\pm^0)^2.
\end{equation}
The left hand side in (\ref{5.tb5}) is independent of the choice of $f_+^0$ or $f_-^0$ in the right hand side, so we necessarily have
\begin{equation}
\label{6.tb5}
f_+^0=f_-^0=:f_0,
\end{equation}
and (\ref{5.tb5}) reduces to
$$
\frac{1}{3}f_0^3+\frac{\widehat{\tau }}{\varphi ^{(3)}}=f_0^3,\,\,\hbox{
  i.e.\ }\, \frac{2}{3}f_0^3=\frac{\widehat{\tau }}{\varphi ^{(3)}},
$$
\begin{equation}\label{7.tb5}
f_\pm^0(\widetilde{x})=\left(\frac{3}{2}\frac{\widehat{\tau
    }}{\varphi ^{(3)}(\widetilde{x})} \right)^{1/3},
\end{equation}
where we recall that $\tau =\epsilon ^3\widehat{\tau }$. By
(\ref{3,5.tb0}) we have
$$
\varphi ^{(3)}(\widetilde{x})>0.
$$

\medskip
\noindent
Let us summarize the discussion so far:
\begin{prop}
\label{1tb5}
Let $f_\pm\in S^0_{\mathrm{fcl}}(\mathbb{ T}_\lambda )$ have the asymptotic expansion {\rm (\ref{9.tb4})}, $\tau = \epsilon ^3\widehat{\tau}$,
$\widehat{\tau } = \mathcal O(1)$, $f_+^0,\, f_-^0>0$. For any such $f_+$, $f_-$, the problem {\rm (\ref{4.tb4})}, {\rm (\ref{5.tb4})} has a unique solution $\mathrm{mod\,}\mathcal{ O}(\epsilon^\infty)$ of the form {\rm (\ref{7.tb4})}, where $u_0$ is affine in $\widetilde{y}$ and has the form {\rm (\ref{3.tb5})}. Furthermore, for such $u,\, f_+,\, f_-$, we have {\rm (\ref{6.tb4})} with an error $\mathcal{ O}(\epsilon ^4)$,
\begin{equation}
\label{1.tb6}
\partial _{\widetilde{y}}u(\widetilde{x},\pm 1)=\frac{\epsilon (f_+(\widetilde{x},\epsilon )+f_-(\widetilde{x},\epsilon
  ))}{2}(\partial _y\varphi )(\widetilde{x},\pm \epsilon  f_\pm
(\widetilde{x},\epsilon ))+\mathcal{ O}(\epsilon ^4 )
\end{equation}
if and only if $f_+^0=f_-^0=:f_0$ is given by {\rm (\ref{7.tb5})}.
\end{prop}

\bigskip
\noindent
We next consider higher order corrections in the problem
(\ref{4.tb4}), (\ref{5.tb4}), (\ref{6.tb4}). Assume that we have found
$f_+$, $f_-$ as in Proposition \ref{1tb5} such that with $u$ solving
(\ref{4.tb4}), (\ref{5.tb4}) up to $\mathcal{ O}(\epsilon ^\infty $) we
have
\begin{equation}\label{2.tb6}
\partial _{\widetilde{y}}u(\widetilde{x},\pm 1)=\frac{\epsilon }{2}
(f_+(\widetilde{x},\epsilon )+f_-(\widetilde{x},\epsilon ))(\partial
_y\varphi ) (\widetilde{x},\pm \epsilon f_\pm (\widetilde{x},\epsilon ))+a_\pm
(\widetilde{x})\epsilon ^{3+N}+\mathcal{ O}(\epsilon ^{4+N}),
\end{equation}
for some $N\in \mathbb{ N}\cap [1,+\infty [$ and for some smooth real
functions $a_+$, $a_-$.

\medskip
\noindent
We modify $f_\pm$ by putting
\begin{equation}
\label{3.tb6}
\widehat{f }_\pm (\widetilde{x},\epsilon )=f_\pm (\widetilde{x},\epsilon)+g_\pm (\widetilde{x})\epsilon ^N,\ \ g_\pm(\widetilde{x})\hbox{ smooth}.
\end{equation}
Then we have (cf.\ (\ref{12.tb4})),
\begin{equation}
\label{4.tb6}
\begin{split}
& \varphi (\widetilde{x},\pm \epsilon \widehat{f}_\pm)-
\varphi (\widetilde{x},\pm \epsilon f_\pm) \\
& = \pm \partial _y\varphi (\widetilde{x},\pm \epsilon f_\pm )g_\pm
   (\widetilde{x})\epsilon ^{N+1}+\mathcal{ O}(\epsilon ^{2N+3})\\
 & =\pm \partial_y \varphi (\widetilde{x},\pm \epsilon f_0)g_\pm
 (\widetilde{x})\epsilon ^{N+1}+\mathcal{ O}(\epsilon ^{N+4}) \\
 & = \pm \frac{1}{2}\partial _y^3\varphi (\widetilde{x},0)f_0 ^2
 g_\pm \epsilon ^{N+3}+\mathcal{ O}(\epsilon ^{N+4}),
\end{split}
\end{equation}
\begin{equation}\label{5.tb6}
  \begin{split}
   & \partial _y\varphi (\widetilde{x},\pm \epsilon \widehat{f}_\pm)-
   \partial _y\varphi (\widetilde{x},\pm \epsilon f_\pm)\\
   &=\pm \partial _y^2\varphi (\widetilde{x},\pm \epsilon f_\pm )g_\pm
   (\widetilde{x})\epsilon ^{N+1}+\mathcal{ O}(\epsilon ^{2N+2})\\
 &=\pm \partial _y^2\varphi (\widetilde{x},\pm \epsilon f_0 )g_\pm
 (\widetilde{x})\epsilon ^{N+1}+\mathcal{ O}(\epsilon ^{N+3})\\
 &= \partial _y^3\varphi (\widetilde{x},0)f_0
 g_\pm \epsilon ^{N+2}+\mathcal{ O}(\epsilon ^{N+3}).
 \end{split}
\end{equation}

\medskip
\noindent
Let $\widehat{u}$ be a solution ($\mathrm{mod\,}\mathcal{ O}(\epsilon^\infty ))$ to the problem (\ref{4.tb4}), (\ref{5.tb4}), with $f_\pm$
replaced by $\widehat{f}_\pm$ in (\ref{3.tb6}). Then
\begin{equation}
\label{1.tb7}
P(\widehat{u}-u)=\mathcal{ O}(\epsilon ^\infty ) \hbox{ in }C^\infty .
\end{equation}
\begin{equation}\label{2.tb7}
(\widehat{u}-u)(\widetilde{x},\pm 1)=\varphi (\widetilde{x},\pm \epsilon
\widehat{f}_\pm (\widetilde{x}))-\varphi (\widetilde{x},\pm \epsilon
f_\pm (\widetilde{x})).
\end{equation}
Here we use (\ref{4.tb6}) to see that
\begin{equation}\label{3.tb7}
(\widehat{u}-u)(\widetilde{x},\pm 1)=\pm \frac{1}{2}\partial _y^3 \varphi
(\widetilde{x},0)f_0^2g_\pm \epsilon ^{N+3}+\mathcal{ O}(\epsilon ^{N+4}).
\end{equation}
We conclude that
\begin{equation}\label{4.tb7}
\widehat{u}-u=\epsilon ^{N+3}v+\mathcal{ O}(\epsilon ^{N+4}),
\end{equation}
where $v$ is the function, affine in $y$, such that
$$
v(\widetilde{x},\pm 1)=\pm \frac{1}{2}\partial _y^3\varphi
(\widetilde{x},0)f_0^2g_\pm,
$$
i.e.\
\begin{equation}\label{5.tb7}
  \begin{split}
v(\widetilde{x},\widetilde{y})&=\frac{1}{2}\partial _y^3\varphi
(\widetilde{x},0)f_0^2\left( g_+ \frac{1+\widetilde{y}}{2}
-g_- \frac{1-\widetilde{y}}{2}
\right)\\
&=
\frac{1}{4}\partial _y^3\varphi
(\widetilde{x},0)f_0^2\left(
g_+-g_-+\left(g_++g_-\right)\widetilde{y}
\right).
  \end{split}
\end{equation}
In particular,
\begin{equation}\label{6.tb7}
  \partial _{\widetilde{y}}v(\widetilde{x},\widetilde{y})=
  \frac{1}{4}\partial _y^3\varphi
(\widetilde{x},0)f_0^2\left(g_++g_-\right) ,
\end{equation}
for all $\widetilde{y}$.
Equation (\ref{4.tb7}) is valid in the $C^\infty $ sense so
$$
\partial _{\widetilde{y}}\widehat{u}-\partial
_{\widetilde{y}}u=\epsilon ^{N+3}\partial _{\widetilde{y}}v+\mathcal{
  O}(\epsilon ^{N+4}),
$$
and hence by (\ref{6.tb7}),
\begin{equation}
\label{7.tb7}
\partial _{\widetilde{y}}\widehat{u}-\partial_{\widetilde{y}}u = \frac{1}{4}\partial _y^3\varphi
(\widetilde{x},0)f_0^2\left(g_++g_-\right)\epsilon
^{N+3}+\mathcal{ O}(\epsilon ^{N+4}).
\end{equation}

\medskip
\noindent
We now return to (\ref{2.tb6}). Choosing $\widehat{f}_\pm$ in
(\ref{3.tb6}), we want to modify the term
$a_\pm (\widetilde{x})\epsilon ^{N+3}$ in (\ref{2.tb6}) arbitrarily by replacing
$f_+,\,f_-,\,u$ by $\widehat{f}_+,\,\widehat{f}_-,\,\widehat{u}$. We
have already obtained an approximation to $\partial_{\widetilde{y}} \widehat{u}-\partial_{\widetilde{y}} u$ in (\ref{7.tb7}), and we next look at
\begin{equation}\label{1.tb8}
\epsilon \frac{\widehat{f}_++\widehat{f}_-}{2}(\partial _y\varphi )(\widetilde{x},\pm
\epsilon \widehat{f}_\pm)
- \epsilon \frac{f_++f_-}{2}(\partial _y \varphi )(\widetilde{x},\pm
\epsilon f_\pm)
= \mathrm{I}_\pm+\mathrm{II}_\pm,
\end{equation}
where
$$
\mathrm{I}_\pm=
\epsilon \frac{\widehat{f}_++\widehat{f}_-}{2}\left((\partial _y\varphi )(\widetilde{x},\pm
  \epsilon \widehat{f}_\pm)
  -(\partial_y \varphi )(\widetilde{x},\pm \epsilon f_\pm),
   \right)
   $$
   $$
\mathrm{II}_\pm=\frac{\epsilon
}{2}(\widehat{f}_+-f_++\widehat{f}_--f_-)(\partial_y \varphi
)(\widetilde{x},\pm \epsilon f_\pm ).
   $$
By (\ref{5.tb6}),
\begin{equation}
\label{2.tb8}
\begin{split}
\mathrm{I}_\pm&=\frac{\epsilon}{2}(\widehat{f}_++\widehat{f}_-)(\partial _y^3\varphi
  (\widetilde{x},0)f_0g_\pm \epsilon ^{N+2}+\mathcal{ O}(\epsilon^{N+4})\\
  &
=\epsilon ^{N+3}\partial _y^3\varphi (\widetilde{x},0)f_0^2g_\pm+\mathcal{
  O}(\epsilon ^{N+4}).
\end{split}
\end{equation}
By (\ref{3.tb6}),
$$
\mathrm{II}_\pm=\epsilon ^{N+1}\frac{1}{2}(g_++g_-)\partial _y\varphi
(\widetilde{x},\pm \epsilon f_\pm)
$$ and by (\ref{12.tb4})
\begin{equation}\label{3.tb8}
  \mathrm{II}_\pm=\frac{1}{4}\epsilon ^{N+3}(g_++g_-)(f_\pm^0)^2\partial
  _y^3\varphi (\widetilde{x},0)+\mathcal{ O}(\epsilon ^{N+4}).
\end{equation}
Thus by (\ref{1.tb8}),
\begin{equation}\label{4.tb8}\begin{split}
\epsilon& \frac{\widehat{f}_++\widehat{f}_-}{2}(\partial _y
\varphi )(\widetilde{x},\pm \epsilon \widehat{f}_\pm)-
\epsilon \frac{f_++f_-}{2}(\partial _y
\varphi )(\widetilde{x},\pm \epsilon f_\pm)\\
&=\epsilon ^{N+3}
  (\partial _y^3\varphi
  (\widetilde{x},0)f_0^2g_\pm+\frac{1}{4}\epsilon ^{N+3}(g_++g_-)f_0^2\partial _y^3\varphi (\widetilde{x},0)
  +\mathcal{ O}(\epsilon ^{N+4}).
\end{split}
\end{equation}

\medskip
\noindent
We have the
following analogue of (\ref{2.tb6}):
\begin{equation}\label{1.tb9}
\partial _{\widetilde{y}}\widehat{u}(\widetilde{x},\pm 1) = \epsilon \frac{\widehat{f}_+(\widetilde{x},\epsilon
    )+\widehat{f}_-(\widetilde{x},\epsilon )}{2}(\partial _y\varphi
  )(\widetilde{x},\pm \epsilon \widehat{f}_\pm (\widetilde{x},\epsilon
  ))+\widehat{a}_\pm(\widetilde{x})\epsilon ^{3+N}+\mathcal{ O}(\epsilon ^{4+N}),
\end{equation}
where $\widehat{a}_\pm (\widetilde{x})$ are smooth. Taking the difference (\ref{1.tb9})$-$(\ref{2.tb6}) and using (\ref{4.tb8}),
(\ref{7.tb7}) gives
\begin{multline*}
\frac{1}{4} \partial _y^3\varphi (\widetilde{x},0)f_0^2(g_++g_-)\epsilon
^{N+3}=
\left(\widehat{a}_\pm(\widetilde{x})-a_\pm(\widetilde{x})
\right)\epsilon ^{N+3}+\mathcal{ O}(\epsilon ^{N+4})\\
+\epsilon ^{N+3}\left(\partial _y^3\varphi
  (\widetilde{x},0) \right)
f_0^2g_\pm+\frac{1}{4}\epsilon ^{N+3}(g_++g_-)f_0 ^2 \partial
_y^3\varphi (\widetilde{x},0),
\end{multline*}
which amounts to the following identity for the leading order terms,
$$
\widehat{a}_\pm-a_\pm = - \partial _y^3\varphi (\widetilde{x},0)f_0^2 g_\pm.
$$

\medskip
\noindent
In conclusion we have the following result.
\begin{prop}\label{1tb11}
Let $f_\pm$ be as in Proposition {\rm \ref{1tb5}}, such that with
$u\in S^0_{\mathrm{fcl}}(\overline{\widetilde{\Omega }})$
solving {\rm (\ref{4.tb4})}, {\rm (\ref{5.tb4})}, we have {\rm (\ref{2.tb6})},
$$
\partial _{\widetilde{y}}u(\widetilde{x},\pm 1)=\frac{\epsilon
}{2}(f_+(\widetilde{x},\epsilon )+f_-(\widetilde{x},\epsilon ))
(\partial _y\varphi )(\widetilde{x},\pm \epsilon f_\pm
(\widetilde{x},\epsilon ))+a_\pm (\widetilde{x})\epsilon ^{3+N}+\mathcal{
  O}(\epsilon ^{4+N}),
$$
for some $N\in \mathbb{ N}\cap [1,+\infty [$ and some smooth real
functions $a_\pm$. Let $\widehat{a}_\pm (\widetilde{x})$ be smooth
real functions independent of $\epsilon $. Then there exist unique
smooth real functions $g_\pm$, independent of $\epsilon$, such that if
$\widehat{f}_\pm$ are defined by {\rm (\ref{3.tb6})} and $\widehat{u}\in
S^0_{\mathrm{fcl}}(\overline{\widetilde{\Omega }})$ is the
solution of {\rm (\ref{4.tb4})}, {\rm (\ref{5.tb4})} with $f_\pm$ replaced by
$\widehat{f}_\pm$, then
\begin{equation}\label{2.tb11}
\partial _{\widetilde{y}}\widehat{u}(\widetilde{x},\pm 1)=\frac{\epsilon
}{2}(\widehat{f}_+(\widetilde{x},\epsilon )+\widehat{f}_-(\widetilde{x},\epsilon ))
(\partial _y\varphi )(\widetilde{x},\pm \epsilon \widehat{f}_\pm
(\widetilde{x},\epsilon ))+\widehat{a}_\pm (\widetilde{x})\epsilon ^{3+N}+\mathcal{
  O}(\epsilon ^{4+N}).
\end{equation}
\end{prop}

\medskip
\noindent
If we choose $\widehat{a}_\pm =0$ in (\ref{2.tb11}) and write the remainder as $\check{a}_\pm(\widetilde{x})\epsilon ^{N+4}+\mathcal{
O}(\epsilon ^{N+5})$, we get
$$
\partial _{\widetilde{y}}\widehat{u}(\widetilde{x},\pm 1)=\frac{\epsilon
}{2}(\widehat{f}_+(\widetilde{x},\epsilon )+\widehat{f}_-(\widetilde{x},\epsilon ))
(\partial _y\varphi )(\widetilde{x},\pm \epsilon \widehat{f}_\pm
(\widetilde{x},\epsilon ))+\check{a}_\pm (\widetilde{x})\epsilon ^{4+N}+\mathcal{
  O}(\epsilon ^{5+N}).
$$
Also recall from (\ref{3.tb6}), (\ref{4.tb7}) that
$$
\widehat{f}_\pm -f_\pm=\mathcal{ O}(\epsilon ^N),\ \ \widehat{u}-u=\mathcal{O}(\epsilon ^{N+3}).
$$
If we choose $\check{a}_\pm=0$, the procedure can be repeated, and in the limit we get $f_\pm^\infty $, $u^\infty $, solving (\ref{4.tb4}),
(\ref{5.tb4}), and also (\ref{6.tb4}): 
\begin{equation}
\label{3.tb11}
\partial _{\widetilde{y}}u^\infty (\widetilde{x},\pm 1) = \frac{\epsilon
}{2}(f^\infty _+(\widetilde{x},\epsilon )+f_-^\infty (\widetilde{x},\epsilon ))
(\partial _y\varphi )(\widetilde{x},\pm \epsilon f_\pm^\infty
(\widetilde{x},\epsilon ))+\mathcal{ O}(\epsilon ^\infty ).
\end{equation}

\medskip
\noindent
We now return to the coordinates $(x,y)$ for which $\Omega _f$ (still assumed to be connected) is given by (\ref{2.tb0,5}). Here we
choose $f_\pm =f_\pm^\infty $ in (\ref{3.tb11}). Abusing the notation, we denote by $u^\infty $ the same function as a function of
$(x,y)$. Thus we have
\begin{equation}\label{1.tb12}
\Delta u^\infty = \mathcal{ O}(\epsilon ^\infty ) \hbox{ in }\Omega _f,
\end{equation}
\begin{equation}\label{2.tb12}
u^\infty (x,\pm\epsilon f_\pm^\infty (x))=\varphi (x,\pm f^\infty
_\pm(x))\pm \tau /2,\ \ x\in \mathbb{ T}_\lambda ,
\end{equation}
\begin{equation}\label{3.tb12}
(\partial _y u^\infty) (x,\pm \epsilon f^\infty _\pm (x))=(\partial _y\varphi
)(x,\pm f_\pm^\infty (x,\epsilon ))+\mathcal{ O}(\epsilon ^\infty ).
\end{equation}

\medskip
\noindent
Let $f_\pm=f^\infty _\pm$, $\partial _\pm\Omega_f=\{(x,\pm\epsilon f_\pm(x) ); x\in \mathbb T_{\lambda}\}$, so that $\partial \Omega _f=\partial
_+\Omega _f\cup \partial _-\Omega _f$. Let $n$ be the unit normal to $\partial \Omega_f $, oriented in the direction of increasing
$\Delta \varphi $: $\langle n, d\Delta \varphi \rangle >0$. (This is the
$y$-direction by (\ref{3,5.tb0}).) We now break the symmetry and
define
$$
u=u^\infty -\tau /2,
$$
so that (\ref{1.tb12})--(\ref{3.tb12}) become
\begin{equation}
\label{4.tb12}
\Delta u=\mathcal{ O}(\epsilon ^\infty )\hbox{ in }\Omega _f,
\end{equation}
\begin{equation}\label{5.tb12}
  u=\begin{cases}
    \varphi \hbox{ on }\partial _+\Omega _f\\
    \varphi -\tau \hbox{ on }\partial _-\Omega _f
  \end{cases},
\end{equation}
\begin{equation}\label{6.tb12}
\partial _nu=\partial _n\varphi +\mathcal{ O}(\epsilon ^\infty )\hbox{ on
}\partial \Omega _f.
\end{equation}
Here we recall that we can drop the assumption (\ref{3+.tb4}), as explained after that equation.

\medskip
\noindent
Modifying $u$ by adding a term which is $\mathcal{ O}(\epsilon ^\infty )$ with all its derivatives, we may assume that (\ref{6.tb12}) holds
exactly, i.e. without the remainder $\mathcal{ O}(\epsilon ^\infty)$. Using this function $u$ we shall construct upper bound and
lower bound weights, as introduced in Definitions \ref{def_ubw_n}, \ref{def_lbw_n}. When doing so, we work in the coordinates $\widetilde{x},\widetilde{y}$. Using the same notation for functions of $(x,y)$ and of $(\widetilde{x},\widetilde{y})$, $\widetilde{y}=\mathcal{ O}(1)$, we have in view of (\ref{3.tb0,5}),
\begin{equation}\label{1.prad1}
\varphi (x,y)=\varphi\left(\widetilde{x},\frac{\epsilon
}{2}(f_+-f_-+(f_++f_-)\widetilde{y}\right)\sim \epsilon ^3(\varphi
_0(\widetilde{x},\widetilde{y})+\epsilon
\varphi_1(\widetilde{x},\widetilde{y})+ \ldots),
\end{equation}
in $C^\infty (\mathbb{T}_\lambda \times [-1,1])$. Here $f_\pm=f_\pm (\widetilde{x},\epsilon )$ and
\begin{equation}
\label{2.prad1}
\varphi _0(\widetilde{x},\widetilde{y})=f_0^3 \partial _y^3\varphi (x,0)\frac{\widetilde{y}^3}{6}.
\end{equation}
Recall the similar expansion (\ref{7.tb4}) for $u(\widetilde{x},\widetilde{y})$, where by (\ref{3.tb5}), and the fact
that $f_\pm^0=f_0$, we have
\begin{equation}\label{3.prad1}
u_0(\widetilde{x},\widetilde{y}) = \left(\frac{1}{6}\varphi^{(3)}(\widetilde{x})f_0^3+\frac{\widehat{\tau }}{2} \right)
\widetilde{y}-\frac{\widehat{\tau }}{2}
\end{equation}
is affine in $\widetilde{y}$ and we have incorporated the subtraction of $\tau /2$ prior to (\ref{4.tb12}), (\ref{5.tb12}),
(\ref{6.tb12}). As observed above, we may assume that (\ref{6.tb12}) holds with no error term. It is clear that
\begin{equation}
\label{4.prad1}
\varphi - u \asymp \epsilon ^3(\widetilde{y}-1)^2,\ \ u-(\varphi -\tau )\asymp
\epsilon ^3 (\widetilde{y}+1)^2 \hbox{ on }\mathbb{ T}_\lambda \times [-1,1].
\end{equation}
It is clear from (\ref{5.tb12}), (\ref{6.tb12}), that
\begin{equation}\label{5.prad1}
\varphi -\tau \le u\le \varphi \hbox{ on }\mathbb{ T}_\lambda \times [-1,1],
\end{equation}
or  on $\overline{\Omega } _f$ in the
$(x,y)$-variables.

\medskip
\noindent
We start with the construction of lower bound weights. Deform $u$ into a slightly subharmonic weight $u_{\mathrm{sub}}$ on $\mathbb{ T}_\lambda
\times [-1,1]$ such that for some $N\in [4,+\infty [$,
\begin{equation}
\label{ulb.1}
u_{\mathrm{sub}}-u= \epsilon ^N(\widetilde{y}-1)^2,
\end{equation}
\begin{equation}\label{ulb.2}
\Delta u_{\mathrm{sub}}\asymp \epsilon ^{N-2},
\end{equation}
and notice that
\begin{equation}\label{ulb.3}
\begin{cases}
u_{\mathrm{sub}}=u,\\
\partial _{\widetilde{y}}u_{\mathrm{sub}}=\partial _{\widetilde{y}}u,
\end{cases}\hbox{ when }\widetilde{y}=1.
\end{equation}
If $C>0$ is large enough, we get from (\ref{4.prad1}), (\ref{ulb.1}),
\begin{equation}\label{ulb.3,5}
  \begin{cases}
  u_{\mathrm{sub}}\le \varphi -\tau + C\epsilon ^N,\hbox{ when }\widetilde{y}=-1,\\
u_{\mathrm{sub}} \ge \varphi -\tau +C\epsilon ^N,\hbox{ when }-1+D\,\epsilon ^{(N-3)/2}\le \widetilde{y}\le 1,
\end{cases}.
\end{equation}
for some suitable $D>0$.

\medskip
\noindent
Let $\widetilde{\tau }=\tau -C\epsilon ^N$ and define the Lipschitz function $\psi _{lb}$ on $\mathbb T_{\lambda }\times
]-2,2[$ by
\begin{equation}
\label{ulb.4}
\psi _{lb}=\begin{cases}\varphi ,\ &\widetilde{y}\ge 1,\\
  \max (u_{\mathrm{sub}},\varphi -\widetilde{\tau}),\ & -1\le
  \widetilde{y}\le 1,\\
\varphi -\widetilde{\tau },\ &\widetilde{y}\le -1.\end{cases}
\end{equation}
Then $\varphi -\widetilde{\tau } \le \psi _{lb}\le \varphi $ and $\psi_{lb}$ is subharmonic in the region where
$\psi _{lb}>\varphi -\widetilde{\tau }$.

\medskip
\noindent
Now we drop the assumption that $\gamma$ is connected, see the discussion after \eqref{2.tb0}. By \eqref{2.tb0}, it follows that $\gamma$ is then the disjoint union of finitely many simple closed curves $\gamma_j$, $j=1,\dots,N$, each satisfying \eqref{2.tb0}. For each $\gamma_j$, we define the thin bands $\Omega_f^j$ as in \eqref{2.tb0,5}, the sets $\Omega_f^{\pm,j}$ as in \eqref{6.csp2} and similarly to \eqref{5.csp2} we put 
\begin{equation}
\label{5.csp2_mod}
\Omega _\pm^j=M_\pm\setminus \Omega_f^{\pm,j} , \quad j=1,\dots, N. 
\end{equation}
We then proceed as in (\ref{3.tb0,5})--(\ref{ulb.4}) for each $\Omega_f^j$ and get, for $\varepsilon>0$ small enough, functions $u^j$ as in 
(\ref{5.tb12})--(\ref{5.prad1}) and $u_{\mathrm{sub}}^j$, $\psi _{lb}^j$, $j=1,\dots,N$, as in \eqref{ulb.4}. Notice that the 
constant $C>0$ in \eqref{ulb.3,5} and in the definition of $\widetilde{\tau}$ 
above \eqref{ulb.4} can be chosen independently of $j=1,\dots,N$.

\medskip
\noindent
We may combine the functions $\psi_{lb}^j$ to a lower bound weight on all of $M$ by letting
\begin{equation}
\label{ulb4,5}
\psi _{lb}=\begin{cases}\varphi \hbox{ on } \bigcap_1^N\Omega _+^j,\\
  \max (u_{\mathrm{sub}}^j,\varphi -\widetilde{\tau })\hbox{ on }\Omega
  _f^j, ~~j=1,\dots,N,\\
\varphi -\widetilde{\tau }\hbox{ on }\bigcap_1^N\Omega _-^j.\end{cases}
\end{equation}

\medskip
\noindent
We obtain therefore the following result.
\begin{prop}\label{ulb1}
The function $\psi _{lb}$ defined in {\rm (\ref{ulb4,5})} is a lower bound weight in the sense of Definition {\rm \ref{def_lbw_n}}, with $\tau $ 
replaced by $\widetilde{\tau}=\tau -C\epsilon ^N$, when $C>0$ is large enough, and $\epsilon $ is small enough. Moreover, the contact set $M_+(\psi _{lb})$ in {\rm (\ref{eq2.7_b})} satisfies $M_+(\psi_{lb}) = \overline{\bigcap_1^N \Omega_+^j}$. 
\end{prop}

\medskip
\noindent
To get an upper bound weight, we deform $u^j$ into a slightly superharmonic weight, $u^j_{\mathrm{sup}}$, such that
\begin{equation}\label{ulb.5}
u_{\mathrm{sup}}^j-u^j= -\epsilon ^N(\widetilde{y}+1)^2,
\end{equation}
\begin{equation}\label{ulb.6}
\Delta u_{\mathrm{sup}}^j\asymp -\epsilon^{N-2},
\end{equation}
and notice that
\begin{equation}\label{ulb.7}
\begin{cases}
u_{\mathrm{sup}}^j=u^j\\
\partial _{\widetilde{y}}u_{\mathrm{sup}}^j=\partial _{\widetilde{y}}u^j,
\end{cases} \hbox{ when }\widetilde{y}=-1.
\end{equation}
We have $u_{\mathrm{sup}}^j-(\varphi -\tau )\asymp\epsilon ^3(\widetilde{y}+1)^2$ near $\widetilde{y}=-1$. We also have
$u_{\mathrm{sup}}^j\le \varphi $ for $-1\le \widetilde{y}\le 1$ and $u_{\mathrm{sup}}^j-\varphi \asymp -\epsilon ^N$ in an $\epsilon $-dependent 
neighborhood of $\widetilde{y}=1$.

\medskip
\noindent
If $C>0$ is sufficiently large but independent of $\epsilon $, we get
\begin{equation}\label{ulb.8}
u_{\mathrm{sup}}^j+C\epsilon ^N\ge \varphi ,\hbox{ when }\,\, \widetilde{y}=1,
\end{equation}
and
\begin{equation}\label{ulb.9}
u_{\mathrm{sup}}^j+C\epsilon ^N\le \varphi , \hbox{ when}\,\, |\widetilde{y}-1|\gg \epsilon ^{\frac{N-3}{2}}.
\end{equation}
Notice that we can and will choose $C>0$ independently of $j=1,\dots,N$. Put
\begin{equation}
\label{ulb.10}
\widetilde{\tau } = \tau -C\epsilon ^N
\end{equation}
and define
\begin{equation}\label{ulb.11}
\psi _{ub}=\begin{cases}
\varphi -\widetilde{\tau }, \ &\hbox{ in }\bigcap_1^N\Omega _-^j\\
    \min (u_{\mathrm{sup}}^j+C\epsilon ^N,\varphi ),\ & \hbox{ in }\Omega
    _f^j, ~~j=1,\dots,N,\\
    \varphi & \hbox{ in }\bigcap_1^N\Omega _+^j.
\end{cases}
\end{equation}
Then $\psi _{ub}\le \varphi$, with the contact set $M_+(\psi_{ub})$ satisfying
\begin{equation}
\label{ulb.12}
\bigcap_1^N \overline{\Omega}_+^j \subseteq M_+(\psi _{ub})\subseteq 
\bigcap_1^N\overline{\Omega}_+^j 
\cup \{ (\widetilde{x},\widetilde{y}) \in \widetilde{\Omega}_f^j; 
      \, 1-\mathcal{ O}(\epsilon^{(N-3)/2})\le \widetilde{y}\le 1 \},
\end{equation}
\begin{equation}\label{ulb.13}
\psi _{ub}\ge \varphi -\widetilde{\tau },
\end{equation}
and $\psi _{ub}$ is superharmonic outside $M_+(\psi _{ub})$. In \eqref{ulb.12} 
$\widetilde{\Omega}_f^j$ denotes the rescaled set $\Omega_f^j$ as in 
\eqref{3.tb0,5}. In conclusion, we obtain the following result.

\begin{prop}
\label{ulb2}
The function $\psi _{ub}$ defined in {\rm (\ref{ulb.11})} is an upper bound weight in the sense of Definition {\rm \ref{def_ubw_n}},
with $\tau $ replaced by $\widetilde{\tau } = \tau - C \epsilon^N$, when $C>0$ is large enough, and $\epsilon $ is small enough.
The contact set $M_+(\psi _{ub})$ satisfies {\rm (\ref{ulb.12})}.
\end{prop}

\medskip
\noindent
Let us also notice that we can choose the same constant $C$ and $\widetilde{\tau }$ for the weights $\psi _{ub}$ and $\psi _{lb}$. Combining Theorems \ref{thm:UpperBd}, \ref{thm:LowerBd} and Propositions \ref{ulb1}, \ref{ulb2}, we get
\begin{prop}
\label{1prad5}
Let $0 < \tau = \epsilon^3 \widehat{\tau}$, with $\widehat{\tau} = \mathcal O(1)$, and set $\widetilde{\tau} = \tau - C \epsilon^N$, for some fixed $4\leq N \in \mathbb N$. The number $N([0,e^{-\widetilde{\tau }/h}])$ of singular values of $P$ in the interval $[0,e^{-\widetilde{\tau }/h}]$ satisfies
\begin{equation}\label{3.prad5}
N([0,e^{-\widetilde{\tau }/h}]) \begin{cases} \le
  \frac{1}{2\pi h}\int_{M_+(\psi_{\mathrm{ub}})}\, \Delta \varphi(z)\, \omega(z, dz d\overline{z}) +\frac{o(1)}{h},\\
\ge \frac{1}{2\pi h}\int_{M_+(\psi_{\mathrm{lb}})}\Delta \varphi(z)\, \omega(z, dz d\overline{z}) - \frac{o(1)}{h}
\end{cases},\ h\to 0^+.
\end{equation}
Here $M_+(\psi _{lb})=\overline{\bigcap_1^N \Omega_+^j}$ and $M_+(\psi _{ub})$ satisfies {\rm (\ref{ulb.12})}. The estimates are not necessarily uniform for all small $\epsilon >0$ but they are uniform for $\epsilon \in [\beta ,2\beta ]$, for each fixed sufficiently small $\beta >0$.
\end{prop}

\medskip
\noindent
We now come to prove Theorem \ref{1csp4_n}. When doing so, recall the set $\gamma$ in \eqref{2.tb0} and assume again until further notice that it is connected, hence a simple closed curve. Let us also recall the unit speed parametrization of $\gamma$ given in (\ref{3.tb0}), its almost holomorphic extension in (\ref{3.tb0.7.5}), as well as the notation $z=\widetilde{\gamma }(w)$, $w=x+iy$, $x\in \mathbb{T}_{\lambda}$, $y\in (-\alpha,\alpha)$. It follows from (\ref{eq_metric}), \eqref{eq:RS_n1}, \eqref{eq:RS_n2}, 
(\ref{eq:RS_n0.3b}) that when expressed in terms of the $(x,y)$ coordinates, the area form $\omega$ is given by 
\begin{equation}
\label{2.csp1.5}
\omega = \sqrt{\det g(x,y)}dxdy = \theta(x,y)(1+\mathcal{O}(|y|^{\infty}))\,dxdy.
\end{equation}
Here we identify the $(1,1)$-form $dx\wedge dy$ with the Lebesgue measure $dx dy$. We get using (\ref{1.tb0,5}) and \eqref{2.csp1.5}, 
\begin{equation}
\label{5.csp1}
(\Delta u)\omega=\left(\Delta _{x,y}u +\mathcal{ O}(|y|^\infty )(\partial_x^2u,\partial_{x}\partial_{y}u,\partial_y^2u,\partial_{x}u,\partial_{y}u)
\right)dxdy, \quad u \in C^{\infty}(M).
\end{equation}

\medskip
\noindent
Recall that $\varphi \in C^{\infty}(M; \mathbb R)$ is such that $\Delta \varphi $ satisfies (\ref{2.tb0}), and that we have equipped $\gamma$ with an orientation such that $\partial_y \Delta \varphi >0$ on $\gamma $, see (\ref{3,5.tb0}). On $\gamma $ we have $\partial _y u=\partial _n u$
when $u$ is smooth, where $n$ is the unit normal to $\gamma $, oriented in the positive $y$-direction. Recall also that $\widetilde{\varphi }$ is a real smooth function defined near $\gamma $, unique modulo $\mathcal{ O}(y^\infty )$, such that
(\ref{4+,tb4}) holds. We get using also \eqref{5.csp1},
\begin{equation}
\label{6.csp1}
\begin{cases}
\Delta _{x,y}\widetilde{\varphi }=\Delta _{x,y}\varphi +\mathcal{ O}(y^\infty)=d(x)y+\mathcal{ O}(y^2),\\
(\partial _y^j\widetilde{\varphi })(x,0)=0,\hbox{ for }0\le j\le 2,
\end{cases}
\end{equation}
where
\begin{equation}\label{1.csp2}
d(x)={{\partial _y\Delta _{x,y}\varphi }_\vert}_{y=0}={{\partial _n \Delta_{x,y}\varphi }_\vert}_{\gamma }.
\end{equation}
It follows that
\begin{equation}\label{2.csp2}
\widetilde{\varphi }(x,y)=\frac{d(x)y^3}{6}+\mathcal{ O}(y^4).
\end{equation}

\medskip
\noindent
We next recall the domain $\Omega _f$ in (\ref{2.tb0,5}), (\ref{2,5.tb0,5}), for which there is a solution $u$ to
(\ref{4.tb4})--(\ref{6.tb4}), now with the original $\varphi $ replaced by $\widetilde{\varphi }=\mathcal{ O}(y^3)$ in Proposition \ref{1tb5},
\ref{1tb11}. The functions $f_\pm$ have the asymptotic expansions (\ref{2,5.tb0,5}), where
$f^0_\pm(x)=f_0(x)$ are given by
\begin{equation}\label{3.csp2}
f_0(x)=\left(\frac{3}{2}\frac{\widehat{\tau }}{\partial
    _y^3\widetilde{\varphi }(x,0)}
\right)^{1/3}=\left(\frac{3}{2}\frac{\widehat{\tau }}{d(x)} \right)^{1/3}.
\end{equation}
(Cf. (\ref{6.tb5}), (\ref{7.tb5}).)

\medskip
\noindent
With these preliminaries in place, we can study the integral 
\begin{equation}
\label{4.csp2}
\iint_{\Lambda _\varphi }1_{\Omega _+}(x) \sigma|_{\Lambda _\varphi }
=\int_{\Omega_+}\Delta \varphi\, \omega.
\end{equation}
Here we have also used {\eqref{eq1.31}}. When $W\subseteq M_+$ is measurable, we write
$$
\mathrm{Vol}(W)=\int_W \Delta \varphi\, \omega,
$$
so that by (\ref{5.csp2}),
\begin{equation}
\label{7.csp2}
\mathrm{Vol}(\Omega _+)=\mathrm{Vol}(M_+)-\mathrm{Vol}(\Omega _f^+),
\end{equation}
and here $\mathrm{Vol}(M_+)$ is independent of $\tau $.

\medskip
\noindent
By (\ref{5.csp1}) we have
\begin{equation*}
\mathrm{Vol}(\Omega _f^+) = \int_{\mathbb{ T}_\lambda
}\left(\int_0^{\epsilon f_+(x,\epsilon )} \Delta _{x,y}\varphi\,dy\right)dx +\mathcal{ O}(\epsilon ^\infty),
\end{equation*}
and therefore by (\ref{6.csp1}),
\begin{equation*}
\begin{split}
\mathrm{Vol}(\Omega _f^+) &= \int_{\mathbb{
    T}_\lambda }\left(\int_0^{\epsilon f_+(x,\epsilon )}(d(x)y+\mathcal{
    O}(y^2)) dy \right) dx + \mathcal{ O}(\epsilon ^\infty) \\
&= \int_{\mathbb{T}_\lambda }\left(
\frac{1}{2}d(x)f_0(x)^2\epsilon ^2+\mathcal{ O}(\epsilon ^3)
\right) dx + \mathcal{ O}(\epsilon ^\infty) \\
&=\epsilon ^2\int_{\mathbb{ T}_\lambda }\frac{1}{2}d(x)f_0(x)^2 dx+\mathcal{
  O}(\epsilon ^3).
\end{split}
\end{equation*}
Inserting (\ref{3.csp2}) gives
$$
\mathrm{Vol}(\Omega _f^+)=\epsilon ^2 \frac{1}{2}\left(\frac{3}{2}
\right)^{2/3}\widehat{\tau }^{2/3}\int_{\mathbb{ T}_\lambda } d(x)^{1/3}dx + \mathcal{ O}(\epsilon ^3),
$$
and since $\tau=\epsilon^3 \widehat{\tau } $ by (\ref{3.tb4}), we get
\begin{equation}\label{1.csp3}
\mathrm{Vol}(\Omega _f^+)=\tau ^{2/3}\frac{1}{2}\left(\frac{3}{2}
\right)^{2/3}\int_\gamma (\partial _n\Delta \varphi )^{1/3} dx +\mathcal{
  O}(\epsilon ^3),
\end{equation}
recalling that $d(x)=\partial _n\Delta \varphi $ on $\gamma $ and that $dx$ is the arc length element along $\gamma $.

\medskip
\noindent
Since $d={{\partial _n\Delta \phi }_\vert}_{\gamma }\asymp 1$, we infer from (\ref{1.csp3}) that
\begin{equation}
\label{2.csp3}
\mathrm{Vol}(\Omega _f^+)\asymp\mathcal{ O}(\epsilon ^2),
\end{equation}
and with this in mind we can estimate the volume of other thin regions.

\medskip
\noindent
Comparing (\ref{1.csp3}) (including its proof) and (\ref{ulb.12}), we get
\begin{equation}\label{3.csp3}
\mathrm{Vol}(M_+(\psi_{\mathrm{ub}})\setminus \Omega_+)=\mathcal{ O}(\epsilon ^{2+(N-3)/2})=\mathcal{ O}(\epsilon ^{(N+1)/2}),
\end{equation}
and combining (\ref{3.csp3}) with Proposition \ref{1prad5}, we obtain that
\begin{equation}\label{5.csp3}
N([0,e^{-\widetilde{\tau }/h}])= \frac{1}{2\pi h}\left(\mathrm{Vol}(\Omega _+)+\mathcal{ O}(\epsilon ^{(N+1)/2})+o(1)
\right),\ h\to 0^+,
\end{equation}
where the "$\mathcal{ O}(\epsilon ^{(N+1)/2})$" is uniform in $\epsilon ,h$ for each $N\ge 4$ and the ``$o(1)$'' is not necessarily uniform in
$\epsilon$ when $\epsilon \to 0 $, but we do have uniformity for $\epsilon $ varying in a small interval $[\beta ,2\beta ]$, for each
small fixed $\beta >0$. We also recall that $\Omega _+=\Omega _+(\tau)$ depends on $\tau $.

\medskip
\noindent
With $\widehat{\tau }\asymp 1$ fixed and $\tau =\epsilon^3\widehat{\tau }$, we get
$$
\frac{d\tau }{\tau }=3\frac{d\epsilon }{\epsilon }.
$$
If $\widetilde{\epsilon }$ satisfies $|\widetilde{\epsilon }-\epsilon|\ll \epsilon $ so that
$\widetilde{\tau }=\widetilde{\epsilon}^3\widehat{\tau }$ satifies $|\widetilde{\tau }-\tau |\ll \tau $,
then
$$
\frac{\widetilde{\tau }-\tau }{\tau }\approx 3\frac{\widetilde{\epsilon }-\epsilon }{\epsilon }.
$$
We apply this to (\ref{5.csp3}) with $\widetilde{\tau }=\tau -\sigma $,
$\sigma =C\epsilon ^N$, by (\ref{ulb.10}), and get
$$
\frac{\widetilde{\epsilon }-\epsilon }{\epsilon }\approx
-\frac{1}{3}\frac{\sigma }{\tau }=-\frac{C}{3 \widehat{\tau }}\epsilon
^{N-3},\ \ \widetilde{\epsilon }-\epsilon \approx
-\frac{C}{3\widehat{\tau }}\epsilon ^{N-2}.
$$
Since $\mathrm{Vol}(\Omega _f^+(\tau ))\approx \mathrm{Const.}\epsilon
^2$, we see that
$$
\mathrm{Vol}(\Omega _f^+(\widetilde{\tau }))-\mathrm{Vol}(\Omega
_f^+(\tau ))\approx \mathrm{Const.}\epsilon (\widetilde{\epsilon
}-\epsilon )\approx \mathrm{Const.}\epsilon ^{N-1}.
$$
Since $N-1\ge (N+1)/2$ we can replace $\Omega _+(\tau )$ by $\Omega_+(\widetilde{\tau })$ in (\ref{5.csp3}) and writing $\tau $ instead of
$\widetilde{\tau }$, we get
\begin{equation}
\label{eq:tb100}
N([0,e^{-\tau /h}])= \frac{1}{2\pi h}\left(\mathrm{Vol}(\Omega
  _+)+\mathcal{ O}(\epsilon ^{(N+1)/2})+o(1) \right),\ h\to 0,
\end{equation}
with $\Omega _+=\Omega _+(\tau )$. 

\medskip
\noindent
We shall finally drop the assumption that $\gamma$ is connected and work with the assumptions and notation as in the paragraph above \eqref {5.csp2_mod}. Then we get in place of \eqref{7.csp2}, 
\begin{equation}\label{7.csp2b}
\mathrm{Vol}\left(\bigcap_1^N\Omega _+^j\right)=\mathrm{Vol}(M_+)-\sum_{j=1}^N\mathrm{Vol}(\Omega _f^{+,j}).
\end{equation}
We then apply the arguments in (\ref{7.csp2})--(\ref{eq:tb100}) for each $j$. In particular 
\eqref{1.csp3} becomes
\begin{equation}\label{1.csp3b}
\sum_{j=1}^N\mathrm{Vol}(\Omega _f^{+,j})=\tau ^{2/3}\frac{1}{2}\left(\frac{3}{2}
\right)^{2/3}\int_\gamma (\partial _n\Delta \varphi )^{1/3} dx +\mathcal{O}(\epsilon ^3), \quad \gamma = \bigcup_{j=1}^N \gamma_j.
\end{equation}
Combining (\ref{7.csp2b}), (\ref{1.csp3b}), and (\ref{eq:tb100}), with $\Omega_+$ replaced by $\bigcap_1^N\Omega_+^j$, we complete the proof of Theorem \ref{1csp4_n}.

\section{Optimal weights via a free boundary problem}
\label{sec:FB}

\medskip
\noindent
The purpose of this section is to prove Theorem \ref{thm:DPPex2}. In fact, when doing so, we shall proceed in greater generality, considering quite general double obstacle problems on compact Riemannian manifolds. Thus, let $M$ be a compact connected smooth Riemannian manifold of dimension $d\geq1$ without boundary. Let $|\cdot|_g$ denote the norm on $T_xM$ induced by the metric $g$ of $M$, let $r_{\mathrm{inj}}>0$ denote the injectivity radius of $M$, and let $B(x_0,r)$ denote the open geodesic ball of radius $0<r\leq r_{\mathrm{inj}}$ and centered at $x_0\in M$. We will denote by $dx$ the Riemannian density of integration and by $\vol(A)$ the corresponding Riemannian volume
of a measurable set $A\subset M$. Moreover, let $-\Delta$ denote the Laplace-Beltrami operator on $M$, and let $W^{s,p}(M):=(1-\Delta)^{-s/2}L^p(M)$, $1 < p < +\infty$ and $s\in \RR$, denote the Sobolev spaces on $M$. We can equip
$W^{s,p}(M)$ with the norm $\|u\|_{W^{s,p}}:=\|(1-\Delta)^{s/2}u \|_{L^p}$. Moreover, note that these spaces have the duality property $(W^{s,p})^* = W^{-s,p'}$, with $p^{-1}+(p')^{-1} =1$. When $p=2$ we will also write $H^{s}=W^{s,2}$.
In Appendix \ref{sec:appHoSoSpace} we review some basic notions of
Sobolev and Hölder spaces on smooth compact manifolds.

\medskip
\noindent
We will also denote by $\{\psi_j\}_{j\in \mathbb N}$ an orthonormal basis of $L^2$ composed of real valued eigenfunctions of the selfadjoint operator $-\Delta:L^2\to L^2$, equipped with the domain $H^2$, associated with the increasing sequence of eigenvalues
\begin{equation}
\label{eq:fbp0}
0=\lambda_0 < \lambda_1 \leq \dots
\end{equation}
In particular, we may choose $\psi_0 = 1/\sqrt{\vol(M)}$.

\medskip
\noindent
Let
\begin{equation}
\label{eq:fbp1.0}
\phi_\pm:M\to \RR
\end{equation}
be non-constant $C^2$ functions and suppose that there exists a $\tau>0$ such that for all $x\in M$, we have
\begin{equation}
\label{eq:fbp1}
\phi_+(x)-\phi_-(x)\geq \tau.
\end{equation}
For a continuous function $u\in H^1(M)$, consider the following differential inequality problem
\begin{equation}
\label{eq:din1}
\begin{cases}
\Delta u \geq 0 \quad \text{on } \{ u > \phi_-\}, \\
\Delta u \leq 0 \quad \text{on } \{ u<\phi_+\}, \\
\phi_- \leq u \leq \phi_+ \text{ on } M.
\end{cases}
\end{equation}
Such a problem is also sometimes called a \emph{double obstacle problem}
or a \emph{free boundary problem}.

\begin{theo}\label{thm:DPPex}
For $d< p<\infty$, the double obstacle problem \eqref{eq:din1} has a solution $u\in W^{2,p}(M)$. If $\max_M \phi_- \geq \min_M \phi_+$,
then it is unique. In particular, when $\max_M \phi_- = \min_M \phi_+$, then the unique solution to \eqref{eq:din1} is given by $u\equiv\max_M \phi_-$.
\end{theo}

\medskip
\noindent
Before we present the proof of this theorem in Section \ref{sec:ExReUnDOP} below, we will make some remarks. Using the Sobolev embedding
\eqref{eq:MorreyInequality2a} we know from Theorem \ref{thm:DPPex} that when $\max_M \phi_- \geq \min_M \phi_+$ the unique solution to \eqref{eq:din1} satisfies
\begin{equation}\label{eq:DOPreg}
u \in C^{1,\alpha}(M), \quad \forall ~0<\alpha<1.
\end{equation}
Notice that when $\max_M \phi_- < \min_M \phi_+$ we lose uniqueness of the solution because any constant $u\in (\max_M \phi_-,\min_M \phi_+)$
is a solution. 
\begin{rem}\label{rem:OptReg} When $\max_M \phi_- \geq \min_M \phi_+$, the $C^{1,\alpha}$ regularity of the unique solution $u$ to the double obstacle problem \eqref{eq:din1} is not optimal. Indeed, when $\max_M \phi_- = \min_M \phi_+$, then  $u\equiv\max_M \phi_-$ and so it is analytic. When $\max_M \phi_- > \min_M \phi_+$ one can show that $u\in C^{1,1}(M)$ which is optimal since $\Delta u$ exhibits a discontinuity at the free boundary $\partial\{\phi_-<u<\phi_+\}$. Since we do not require the optimal regularity of $u$ for our main result, we refrain from giving the proof but instead simply outline the argument: First note that $u$ is harmonic in the open set $\{\phi_- < u < \phi_+\}$ and therefore $C^\infty$ there. In the contact sets $\{u=\phi_\pm\}$, it follows from \eqref{eq:cs1.0} that, in local coordinates, all second order derivatives of $u$ are bounded in the $L^\infty$ norm. It remains to control the second order derivatives of $u$ near the boundaries $\partial \{u=\phi_\pm\}$. To do this, one can follow the well established approach for single obstacle problems, see for instance \cite[Chapter 2]{PSU12} and \cite{Fr72,Ca77,Ca98}. More precisely, one can use the maximum principle and Harnack's inequality to show that there exist constants $C_0,C>0$ depending only on $M$ and $g$, such that if $\dist(x,\partial\{u=\phi_-\}) \leq 1/C_0$ then
  \begin{equation*}
    0 \leq u(x)-\phi_-(x) \leq C \|\Delta \phi_-\|_{L^\infty}
    \dist(x,\partial \{u=\phi_-\})^2,
  \end{equation*}
  and if $\dist(x,\partial \{u=\phi_+\}) \leq 1/C_0$ then
  \begin{equation*}
    0 \leq \phi_+(x)-u(x) \leq C \|\Delta \phi_+\|_{L^\infty}
    \dist(x,\partial \{u=\phi_+\})^2.
  \end{equation*}
This together with the Schauder interior estimates \cite[Theorem 6.2]{GiTr01} can then be used to show that $u\in W^{2,\infty}(M)$ which gives that $u\in C^{1,1}(M)$. 
\end{rem}

\subsection{Existence, regularity and uniqueness of solutions to the double
obstacle problem}\label{sec:ExReUnDOP}
This section is devoted to the proof of Theorem \ref{thm:DPPex}. To prove existence and $W^{2,p}$ regularity of solutions to \eqref{eq:din1}
we will use the method of penalization.

\medskip
\noindent
Let $\varepsilon>0$ and define the following Lipschitz continuous approximation of the Heaviside function,
\begin{equation*}
  H_\varepsilon(t) := 1_{[0,+\infty[}(t)
  + (1+t/\varepsilon)1_{[-\varepsilon,0[}(t).
\end{equation*}
Using this, we define for $x\in M$ and $t\in\RR$,
\begin{equation}\label{eq:pen1a}
  f_\varepsilon(t,x) := (\Delta\phi_-(x))^-H_\varepsilon( \phi_-(x)-t)
  -(\Delta\phi_+(x))^+H_\varepsilon(t- \phi_+(x)).
\end{equation}
Here, we use the standard notation for the positive and negative part
of a function, i.e. $(\Delta\phi_\pm(x))^\pm=\max(\pm\Delta\phi_\pm(x),0)$.
Recall the eigenfunctions $\psi_j$ introduced in the
paragraph above \eqref{eq:fbp0} and write $\phi_\pm^0:=(\phi_\pm|\psi_0)$.
By \eqref{eq:fbp1},
\begin{equation}\label{eq:pen1a.0}
  \phi_-^0 < \phi_+^0.
\end{equation}
Put $C_0:=2\max(\|\Delta\phi_-\|_{L^\infty},\|\Delta\phi_+\|_{L^\infty})$.
For $\varepsilon>0$ sufficiently small, let $g_{\varepsilon}:\RR\to \RR$ be
the strictly decreasing continuous piecewise affine function such that
$\{\phi_-^0,\phi_-^0+\varepsilon(1-\varepsilon), \phi_+^0-\varepsilon(1-\varepsilon),\phi_+^0\}$
are the only breakpoints between the affine segments with values
\begin{equation}\label{eq:pen1b0}
  \begin{split}
    &g_{\varepsilon}(\phi_\pm^0)=\mp C_0 \\
    &g_{\varepsilon}(\phi_\pm^0+\varepsilon(1-\varepsilon))=\mp C_0\varepsilon,
   \end{split}
\end{equation}
and such that $g_{\varepsilon}'(t) = -C_0$ when $t<\phi_-^0$ or
$t>\phi_+^0$. See Figure \ref{Fig1} for an illustration. Notice
that $g_{\varepsilon}:\RR\to \RR$ is a homeomorphism,
$g_{\varepsilon}: [\phi_-^0,\phi_+^0]\to [-C_0,C_0]$ is a homeomorphism, and
\begin{equation}\label{eq:pen1c}
\begin{split}
  &g_{\varepsilon}(t)>C_0, \quad t<\phi_-^0, \\
  &g_{\varepsilon}(t)<-C_0, \quad t>\phi_+^0, \\
  &|g_{\varepsilon}(t)| \leq C_0 \varepsilon,
   \quad t\in [\phi_-^0+\varepsilon(1-\varepsilon),\phi_+^0-\varepsilon(1-\varepsilon)], \\
  & g_\varepsilon\!\left(\frac{\phi_+^0+\phi_-^0}{2}\right) = 0.
\end{split}
\end{equation}

\begin{figure}[h]
  \begin{center}
  \begin{tikzpicture}
  \draw[->] (0,0) -- (12,0);
  \draw (12,-0.3) node{{\small$t\in \RR$}};
  \draw[->] (1,-3) -- (1,3);
  \draw (1.5,3) node{{\small$g_\varepsilon(t)$}};
  \draw[line width=0.25mm,-] (0,3) -- (2,2);
  \draw[line width=0.25mm,-] (2,2) -- (3,0.5);
  \draw[line width=0.25mm,-] (3,0.5) -- (9,-0.5);
  \draw[line width=0.25mm,-] (9,-0.5) -- (10,-2);
  \draw[line width=0.25mm,-] (10,-2) -- (12,-3);
  \draw[dotted] (2,2) -- (2,0);
  \draw[dotted] (2,2) -- (1,2);
  \draw (0.6,2) node{$C_0$};
  \draw[-] (2,0.1) -- (2,-0.1);
  \draw (1.8,-0.3) node{{\small$\phi_-^0$}};
  \draw[dotted] (3,0.5) -- (3,0);
  \draw[dotted] (3,0.5) -- (1,0.5);
  \draw (0.6,0.5) node{{\small $C_0\varepsilon$}};
  \draw[-] (3,0.1) -- (3,-0.1);
  \draw (3.7,-0.3) node{{\small$\phi_-^0+\varepsilon(1-\varepsilon)$}};
  \draw[dotted] (10,-2) -- (10,0);
  \draw[dotted] (10,-2) -- (1,-2);
  \draw (0.5,-2) node{$-C_0$};
  \draw[-] (10,0.1) -- (10,-0.1);
  \draw (10.4,0.4) node{{\small$\phi_+^0$}};
  \draw[dotted] (9,0) -- (9,-0.5);
  \draw[dotted] (9,-0.5) -- (1,-0.5);
  \draw (0.4,-0.5) node{{\small$-C_0\varepsilon$}};
  \draw[-] (9,0.1) -- (9,-0.1);
  \draw (8.5,0.4) node{{\small$\phi_+^0-\varepsilon(1-\varepsilon)$}};
  \end{tikzpicture}
  \end{center}
  \caption{An illustration of the penalty function $g_\varepsilon(t)$
  above \eqref{eq:pen1b0}.}\label{Fig1}
  \end{figure}
\begin{prop}\label{prop:NLeq}
Let $\varepsilon>0$ be small enough and recall \eqref{eq:pen1a}, \eqref{eq:pen1c}. For $d < p < \infty$ there exists a solution
$u^\varepsilon \in W^{2,p}(M)$ to the non-linear equation
\begin{equation}\label{eq:pen2nl}
-\Delta u^\varepsilon = f_\varepsilon(u^\varepsilon,\cdot) + g_\varepsilon(( u^\varepsilon | \psi_0 )),
\end{equation}
and it satisfies
\begin{equation}\label{eq:pen2dnl}
\begin{split}
  &( u^\varepsilon | \psi_0 )\in [\phi_-^0,\phi_+^0], \\
  &\| u^\varepsilon \|_{W^{2,p}}
  \leq  \mO_p(1)(C_0+|\phi_-^0|+|\phi_+^0|),
\end{split}
\end{equation}
uniformly in $\varepsilon>0$.
\end{prop}

\medskip
\noindent
{\it Remark}. The method of penalization is a by now classical approach to proving regularity of obstacle problems going back at least
to \cite{BrSt68,LeSt69}, see for instance \cite{Ro87,PSU12} and the references therein. Here, in \eqref{eq:pen2nl}, we used the penalization
function $f_\varepsilon$ which (up to a minor modification) was introduced in \cite{LePa23}. What is novel here is that
since we work on a compact manifold without boundary we need to introduce the additional penalization term $g_\varepsilon$.

\medskip
\noindent
We postpone the proof of Proposition \ref{prop:NLeq} till the end of this section and continue directly with the
\begin{proof}[Proof of Theorem {\rm \ref{thm:DPPex}}]
1. Since the last bound in \eqref{eq:pen2dnl} is independent of $\varepsilon$, it follows that there exists a $u\in W^{2,p}$ and
a subsequence of $\varepsilon \to 0$ such that
\begin{equation}\label{eq:pen7}
u^{\varepsilon} \rightharpoonup u \quad \text{weakly in } W^{2,p},
\end{equation}
i.e. $\langle u^\varepsilon, v\rangle \to \langle u, v\rangle$ for all $v\in W^{-2,p'}$, $p^{-1}+(p')^{-1}=1$. Since $d<p$, the Sobolev embedding
\eqref{eq:MorreyInequality2a} gives that $u\in C^{1,1-d/p}(M)$. Moreover, $W^{2,p}(M)$ is compactly embedded in $C(M)$, see \eqref{eq:CompImbed}. So, additionally to \eqref{eq:pen7}, we have
\begin{equation}\label{eq:pen8}
u^{\varepsilon} \to u \quad \text{uniformly.}
\end{equation}
So, by \eqref{eq:pen2dnl}
\begin{equation}\label{eq:pen7.1}
( u| \psi_0 )\in [\phi_-^0,\phi_+^0], \\
\end{equation}
Next, we will show that the limit $u$ in \eqref{eq:pen7}, \eqref{eq:pen8} is a solution to \eqref{eq:din1}.

\medskip
\noindent
2. Suppose that $u\geq\phi_-$ on $M$. If $( u | \psi_0 ) = \phi_-^0$, then $u=\phi_-$ on $M$. By \eqref{eq:fbp1} and \eqref{eq:pen8},
for $\varepsilon>0$ small enough, we have that $u^\varepsilon \leq (\phi_-+\phi_+)/2$ on $M$. So by \eqref{eq:pen1a}, we have that $f_\varepsilon(u^\varepsilon(x),x)\geq 0$. Thus, by \eqref{eq:pen2nl} and the last property in \eqref{eq:pen1c},
\begin{equation}\label{eq:pen8.1}
    -\Delta u^\varepsilon \geq g_\varepsilon(( u^\varepsilon | \psi_0 ))
    \geq 0, \quad \text{a.e. on } M,
\end{equation}
for $\varepsilon >0$ small enough. By \eqref{eq:pen7}, it follows that $-\Delta u \geq 0$ on $M$ in the sense of distributions, and a.e. on $M$
since $u\in W^{2,p}$. Hence $u$ is constant. Indeed, put $\widetilde{u}= u - \min_M u\geq 0$. So $\Delta \widetilde{u} =
\Delta u \geq 0$ almost everywhere. Since $\widetilde{u}\in W^{2,p}$, we get
\begin{equation*}
0\leq \int_M \Delta \widetilde {u} \, \widetilde {u} dx = -\int_M  |d\widetilde {u}|^2_g dx \leq 0.
\end{equation*}
It follows that $\widetilde{u}$ and, therefore, $u$ are constant almost everywhere and thus everywhere by continuity.
This is however a contradiction, since $u=\phi_-$ is not constant (see the paragraph above \eqref{eq:fbp1}).
Thus,
\begin{equation}\label{eq:pen8.1.1}
    ( u | \psi_0 )  > \phi_-^0 \text{ provided that }u\geq\phi_- \text{ on }M.
\end{equation}
A similar argument shows that $( u | \psi_0 )  < \phi_+^0$, when $u\leq\phi_+$ on $M$.

\medskip
\noindent
3. Next we will show that $u\geq \phi_-$ on $M$ by contradiction. Suppose that $\{u<\phi_-\} \subset M$ is non-empty. By continuity
it is an open set and it cannot be equal to $M$ in view of \eqref{eq:pen7.1}. Let $x_0\in\{u<\phi_-\} \subset M$. Again by continuity, there exist
$r_0,\delta_0>0$, such that $u(x)-\phi_-(x)\leq -\delta_0$ for all $x\in B(x_0,r_0)$. The latter denotes the open geodesic ball of radius
$r_0$ and centered at $x_0$. Taking $\varepsilon>0$ small enough, it follows from \eqref{eq:pen8} that
$u^\varepsilon(x)-\phi_-(x)\leq -\delta_0/2$ for all $x\in B(x_0,r_0)$. By \eqref{eq:pen1a}, \eqref{eq:pen2nl}, we find that
\begin{equation}\label{eq:pen9}
    -\Delta u^\varepsilon \geq -\Delta\phi_-
        + g_\varepsilon(( u^\varepsilon | \psi_0 )) \quad \text{a.e. on }B(x_0,r_0).
\end{equation}

\medskip
\noindent
3a. Suppose first that $( u | \psi_0 ) = \phi_-^0$. Taking $\varepsilon>0$ small enough we get that $(u^\varepsilon| \psi_0 ) \leq (\phi_-^0+\phi_+^0)/2$ by \eqref{eq:pen1a.0}. The last property in \eqref{eq:pen1c} then gives that $g_\varepsilon(( u^\varepsilon | \psi_0 ))\geq 0$, which implies that
\begin{equation*}
  -\Delta (u^\varepsilon -\phi_-) \geq 0 \quad \text{a.e. on }B(x_0,r_0).
\end{equation*}
The weak convergence \eqref{eq:pen7} shows that $-\Delta (u -\phi_-)\geq 0$ on $B(x_0,r_0)$ in the sense of distributions, and therefore almost everywhere since $u\in W^{2,p}$. Since $x_0$ was chosen arbitrarily, it follows that
\begin{equation*}
    -\Delta (u -\phi_-) \geq 0 \quad \text{a.e. on } \{u<\phi_-\} .
\end{equation*}
The maximum principle \cite[Theorem 8.1]{GiTr01}\footnote{In fact \cite[Theorem 8.1]{GiTr01} applies to functions on bounded domains in $\RR^d$, however, by following the exact same arguments of the proof, it can be applied to our case.}, implies that $u \geq \phi_-$ on $\{u<\phi_-\}$, a contradiction. Therefore,
\begin{equation}\label{eq:AvLB}
    ( u | \psi_0 ) > \phi_-^0
\end{equation}

\medskip
\noindent
3b. Suppose next that $( u | \psi_0 ) = \phi_+^0$. Since $\{u<\phi_-\}$ has strictly positive measure, the same must hold for $\{u>\phi_+\}$. Otherwise $( u | \psi_0 ) < \phi_+^0$ by \eqref{eq:pen7.1}. By \eqref{eq:pen8}, we find that $(u^\varepsilon| \psi_0 ) \geq (\phi_-^0+\phi_+^0)/2$ for $\varepsilon>0$ small enough, so $g_\varepsilon(( u^\varepsilon | \psi_0 ))\leq 0$ by the last property in \eqref{eq:pen1c}.

\medskip
\noindent
By continuity, the set $\{u>\phi_+\}$ is open and for any $x_1\in \{u>\phi_+\}$, there exist $r_1,\delta_1>0$ such that $u-\phi_+\geq \delta_1$ on $B(x_1,r_1)$. Taking $\varepsilon>0$ small enough, it follows from \eqref{eq:pen8} that $u_\varepsilon(x)-\phi_+(x)\geq \delta_1/2$ for all $x\in B(x_1,r_1)$. Using \eqref{eq:pen1a}, \eqref{eq:pen2nl}, we find that
\begin{equation}\label{eq:pen9.1}
  \begin{split}
    -\Delta u^\varepsilon &\leq -\Delta\phi_+
        + g_\varepsilon(( u^\varepsilon | \psi_0 )) \\
        &\leq -\Delta\phi_+
        \quad \text{a.e. on }B(x_1,r_1).
  \end{split}
\end{equation}
The weak convergence \eqref{eq:pen7} shows that $\Delta (u -\phi_+)\geq 0$
  on $B(x_1,r_1)$ in the sense of distributions, and therefore almost everywhere
  since $u\in W^{2,p}$. Since $x_1$ was chosen arbitrarily, it follows
  that
\begin{equation*}
    \Delta (u -\phi_+) \geq 0 \quad \text{a.e. on } \{u>\phi_+\} .
  \end{equation*}
The maximum principle, implies that $u \leq \phi_+$ on $\{u>\phi_+\}$, a contradiction. Therefore,
\begin{equation}\label{eq:AvLB2}
    ( u | \psi_0 ) < \phi_+^0
\end{equation}

\medskip
\noindent
3c. Combining \eqref{eq:AvLB}, \eqref{eq:AvLB2} and \eqref{eq:pen8}, there exists an $\eta>0$ such that for $\varepsilon >0$ small enough
\begin{equation}\label{eq:pen9.0}
    \phi_-^0+ \eta \leq ( u^\varepsilon | \psi_0 ) \leq \phi_+^0- \eta.
\end{equation}
So, by \eqref{eq:pen1c}, $|g_\varepsilon(( u^\varepsilon | \psi_0 ))|\leq C_0\varepsilon$. Combining this with \eqref{eq:pen9} we get
\begin{equation*}
    -\Delta u^\varepsilon \geq -\Delta\phi_-
        + \mO(\varepsilon) \quad \text{ on }B(x_0,r_0).
\end{equation*}
By \eqref{eq:pen7}, it follows that $-\Delta (u- \phi_-) \geq 0$ a.e. on $B(x_0,r_0)$, and therefore on $\{u<\phi_-\}$ since $x_0$ was
chosen arbitrarily. The maximum principle yields that $u\geq \phi_-$ on $\{u<\phi_-\}$, a contradiction. Hence, $u\geq \phi_-$ on $M$. A similar argument shows that $u\leq \phi_+$ on $M$.

\medskip
\noindent
4. By Step 2 and 3, we know that $\phi_- \leq u \leq \phi_+$ on $M$ and that $\phi_-^0 < ( u | \psi_0 ) < \phi_+^0$. Arguing as in
\eqref{eq:pen9.0}, the last condition implies that
\begin{equation}\label{eq:pen10}
    |g_\varepsilon(( u^\varepsilon | \psi_0 ))|\leq C_0\varepsilon
\end{equation}
for $\varepsilon >0$ small enough.

\medskip
\noindent
By the same argument as in Step 2 we find that $\{u>\phi_-\}$ is not empty. Continuity implies that it is open. Let $x_0\in \{u>\phi_-\}$. Thus, there exist $r,\delta>0$, such that $u(x)-\phi_-(x)\geq \delta$ for all $x\in B(x_0,r)$. By \eqref{eq:pen8}, it follows that for $\varepsilon>0$ small enough, $u^\varepsilon(x)-\phi_-(x)\geq \delta/2$ for all $x\in B(x_0,r)$. By \eqref{eq:pen1a}, \eqref{eq:pen2nl},
\begin{equation}\label{eq:pen11}
    -\Delta u^\varepsilon \leq g_\varepsilon(( u^\varepsilon | \psi_0 )) \quad
    \text{ a.e. on }B(x_0,r).
\end{equation}
By \eqref{eq:pen10}, \eqref{eq:pen7}, it follows that $\Delta u\geq 0$ a.e. on $B(x_0,r)$, and therefore on $\{u>\phi_-\}$ since $x_0$ was chosen arbitrarily. A similar argument shows that $\Delta u\leq 0$ on $\{u<\phi_+\}$. This shows that $u\in W^{2,p}$ is a solution to \eqref{eq:din1}.

\medskip
\noindent
4. To prove uniqueness, let $u_1,u_2$ be two solutions to \eqref{eq:din1}. First, assume that $\max_M \phi_- > \min_M \phi_+$.
Notice that this condition implies that any solution to \eqref{eq:din1} is non-constant. Suppose that the set $S:=\{u_1<u_2\}\subset M$ is not
empty. Note that $S$ cannot be equal to $M$. Indeed, if it were, then $u_1< \phi_+$ and $u_2> \phi_-$ on $M$. So by \eqref{eq:din1},
$\Delta u_1 \leq 0 \leq \Delta u_2 $ a.e. on $M$. This implies the $u_1,u_2$ are constant which is a contradiction, so $S\neq M$.

\medskip
\noindent
Continuing, we know by the continuity of $u_1,u_2$ that $S$ is open. Notice that $u_1< \phi_+$ and $u_2> \phi_-$ on $S$. Thus,
by \eqref{eq:din1}, we find that $\Delta(u_1-u_2)\leq 0$ a.e. on $S$. The maximum principle implies that $u_1-u_2\geq 0$ on $S$,
which is a contradiction.

\medskip
\noindent
Secondly, consider the case $\max_M \phi_- = \min_M \phi_+$. Arguing as above, we see that the set $S$ is either equal to $M$, in which
case both solutions are constant and thus equal to $\max_M \phi_- = \min_M \phi_+$, a contradiction. Or, $S\neq M$, in which case,
arguing as above, the maximum principle yields a contradiction.

\medskip
\noindent
Notice that this argument shows that in this case, the unique solution to \eqref{eq:din1} is given by $u\equiv\max_M \phi_-$.
\end{proof}

\begin{proof}[Proof of Proposition \ref{prop:NLeq}]
1. First, we recall that $-\Delta:H^2(M)\to L^2(M)$ is Fredholm of index $0$,
selfadjoint on $L^2(M)$ and has the kernel $\mathcal{N}(-\Delta) = \CC1$.
By the self-adjointness, $\mathcal R(-\Delta)^{\perp}= \CC1$.
Recall the discussion above \eqref{eq:fbp0} and denote by $\pi_N = (\cdot|\psi_0)\psi_0$
the $L^2$ orthogonal projection onto $\mathcal{N}(-\Delta)$.

\medskip
\noindent
Given a real-valued $v\in L^{\infty}(M)$, and a $v_+\in\RR$, we
wish to obtain a solution $(u^\varepsilon,u^\varepsilon_{-})\in H^2(M)\times \RR$
to the non-linear Grushin problem\footnote{This approach is also called
Lyapunov-Schmidt reduction.}
\begin{equation}\label{eq:pen2}
\begin{cases}
  -\Delta u^\varepsilon(x) = f_\varepsilon(v(x),x)
            + g_\varepsilon(u^\varepsilon_{-}), \\
  \pi_N u^\varepsilon = v_+\psi_0 .
\end{cases}
\end{equation}
Notice that $-\Delta: \mathcal{N}(-\Delta)^{\perp}\cap H^2 \to \mathcal{N}(-\Delta)^{\perp}$ is bijective and denote its inverse by $Q_0$. We extend $Q_0$ to an operator $Q:L^2\to L^2$ by setting $Q|_{\mathcal{N}(-\Delta)^{\perp}} = Q_0$ and $Qw=0$ if $w\in \mathcal{N}(-\Delta) =\CC 1$. Notice that
\begin{equation}\label{eq:pen2.1}
  -\Delta Q = 1 - \pi_N \quad \text{and} \quad
  Q(-\Delta) = 1 - \pi_N.
\end{equation}
For \eqref{eq:pen2} to have a solution, it is necessary that
\begin{equation}\label{eq:pen5.0}
  g_\varepsilon(u^\varepsilon_{-})
= - \pi_N f_\varepsilon(v(\cdot),\cdot).
\end{equation}
Since $f_\varepsilon(v(\cdot),\cdot)$ is in $L^\infty$, and
\begin{equation}\label{eq:pen4}
  \| f_\varepsilon(v(\cdot),\cdot) \|_{L^\infty} \leq  C_0,
\end{equation}
it follows from the fact that $g_{\varepsilon}$ is a homeomorphism and
that $g^{-1}_\varepsilon([-C_0,C_0])=[\phi_-^0,\phi_+^0]$ (see the paragraph above
\eqref{eq:pen1c}) that there exists a unique $u^\varepsilon_-$ satisfying
\eqref{eq:pen5.0}. Moreover,
\begin{equation}\label{eq:pen5}
  u^\varepsilon_- \in [\phi_-^0,\phi_+^0].
\end{equation}
Consequently, \eqref{eq:pen2} has a unique solution
$(u^\varepsilon,u^\varepsilon_-)\in H^{2} \times \RR$
with $u^\varepsilon_-$ as in \eqref{eq:pen5.0}, \eqref{eq:pen5} and
\begin{equation}\label{eq:pen2b}
  u^\varepsilon = v_+ \psi_0  + Q f_\varepsilon(v(\cdot),\cdot).
\end{equation}
In particular, the solution $u^\varepsilon$ is real-valued as this is true for all
coefficients in the equation \eqref{eq:pen2} and for $v_+$.

\medskip
\noindent
2. By ellipticity, it follows that $Q\in \Psi^{-2}(M)$ --- see Appendix \ref{sec:appHoSoSpace} for the definition of the space of pseudodifferential operators $\Psi^m(M)$. Using \eqref{eq:pen2.1}, \eqref{eq:pen2}, and applying $(1-\Delta)$
to \eqref{eq:pen2b} gives
\begin{equation}\label{eq:bddu1}
  (1-\Delta) u^\varepsilon =
  (1-\Delta) Q (-\Delta)u^\varepsilon + v_+\psi_0,
\end{equation}
where $(1-\Delta) Q \in \Psi^0(M)$. For any
$1 < p <+\infty$, a pseudodifferential operator of class $\Psi^{0}(M)$
maps $L^p\to L^p$ continuously (see \eqref{eq:PseudoMapping2}). It follows
by \eqref{eq:pen4}, \eqref{eq:pen2b} that $u^\varepsilon\in W^{2,p}$
for any $1 < p <+\infty$ and
\begin{equation}\label{eq:pen2d}
\begin{split}
  \| u^\varepsilon \|_{W^{2,p}}
     &\leq \mO(1)(\| \Delta u^\varepsilon\|_{L^p} + |v_+|) \\
     &\leq \mO(1)(\| (1-\pi_N) f_\varepsilon(v(\cdot),\cdot)\|_{L^p} + |v_+|) \\
     &\leq \mO(1)(C_0 + |v_+|).
\end{split}
\end{equation}
Notice that the estimate is uniform in $\varepsilon >0$.
%
For $d < p < \infty$, the Sobolev space $W^{2,p}$ is compactly
embedded in $C(M)$, (see \eqref{eq:CompImbed}) and we have
\begin{equation}\label{eq:pen2e.4}
  \|w\|_{C(M)} \leq \mO_{p}(1) \|w\|_{W^{2,p}}, \quad w\in W^{2,p}.
\end{equation}
Thus, by \eqref{eq:pen2d}, there exists a $C_{p}>0$ such that
\begin{equation}\label{eq:pen2e.2}
  \|u^\varepsilon\|_{C(M)}
  \leq C_{p}( |v_+| +C_0)
\end{equation}
uniformly in $\varepsilon >0$.

\begin{lemm}\label{lem:FP}
  Let $d < p < \infty$, and equip the real Banach space $C(M;\RR) \oplus \RR$
  with the norm $\|(v,v_+)\| = \|v\|_{C(M)} + |v_+|$. Let $C_p,C_0>0$ be as in
  \eqref{eq:pen2e.2}, and put $K:= C_p(\max(|\phi_-^0|,|\phi_+^0|) +C_0)$. Let
  $\overline{B(0,K)}\subseteq C(M)$ be the closed ball of radius $K$
  and centered at $0$. Then,
  \begin{equation*}
    \begin{split}
      G:\overline{B(0,K)} \oplus [\phi_-^0,\phi_+^0] &\longrightarrow \overline{B(0,K)} \oplus[\phi_-^0,\phi_+^0] \\
      (v,v_+) &\longmapsto (u^\varepsilon,u_-^\varepsilon),
    \end{split}
  \end{equation*}
  mapping $(v,v_+)$ onto the unique solution $(u^\varepsilon,u_-^\varepsilon)$
  of \eqref{eq:pen2}, admits a fixed point.
\end{lemm}

\medskip
\noindent
We present the proof of this Lemma at the end of this section. Continuing,
it follows that the fixed point of $G$ obtained in Lemma \ref{lem:FP}
solves
\begin{equation*}
  \begin{cases}
    -\Delta u^\varepsilon(x) = f_\varepsilon(u^\varepsilon(x),x)
              + g_\varepsilon(u^\varepsilon_{-}), \\
    ( u^\varepsilon| \psi_0) = u^\varepsilon_{-} \in [\phi_-^0,\phi_+^0].
  \end{cases}
\end{equation*}
Thus, it solves the non-linear equation
\begin{equation*}
    -\Delta u^\varepsilon = f_\varepsilon(u^\varepsilon,\cdot)
        + g_\varepsilon(( u^\varepsilon | \psi_0 )).
\end{equation*}
By \eqref{eq:pen2d}, we have that $u^\varepsilon \in W^{2,p}$ and
\begin{equation*}
  \| u^\varepsilon \|_{W^{2,p}}
  \leq  \mO_p(1)(C_0+|\phi_-^0|+|\phi_+^0|)
\end{equation*}
uniformly in $\varepsilon >0$.
\end{proof}
\begin{proof}[Proof of Lemma \ref{lem:FP}]
1. Since the embedding of $W^{2,p}(M;\RR)$ into $C(M;\RR)$ is  compact, it follows by \eqref{eq:pen5}, \eqref{eq:pen2e.2} that
$G( \overline{B(0,K)}\oplus[\phi_-^0,\phi_+^0]) \subset \overline{B(0,K)}\oplus[\phi_-^0,\phi_+^0]$ is relatively compact.
Moreover, note that $\overline{B(0,K)}\oplus[\phi_-^0,\phi_+^0]$ is a closed convex subset of $C(M;\RR) \oplus \RR$.

\medskip
\noindent
2. Since $H_\varepsilon(t)$ is Lipschitz continuous, it follows that
\begin{equation}\label{eq:pen2.0}
    \| f_\varepsilon(v_1,\cdot)-f_\varepsilon(v_2,\cdot))\|_{L^\infty}
    \leq \mO(1/\varepsilon)\| v_1 - v_2 \|_{L^\infty}.
  \end{equation}
Since $g_\varepsilon : [\phi_-^0,\phi_+^0]\to [-C_0,C_0]$ is a homeomorphism,
 \eqref{eq:pen5.0} implies that
  \begin{equation}\label{eq:pen2.6}
    \overline{B(0,K)}\times[\phi_-^0,\phi_+^0]  \ni (v,v_+) \mapsto u_-^\varepsilon \in
    [\phi_-^0,\phi_+^0] \text{ is continuous.}
  \end{equation}

\medskip
\noindent
Insisting on the dependency on $v,v_+$, we write the solution to \eqref{eq:pen2} as $u^\varepsilon(v,v_+)$. By \eqref{eq:pen2.0}, \eqref{eq:pen2b}
and an ellipticity estimate as in \eqref{eq:pen2d}, we get that
\begin{equation*}
  \begin{split}
    \| u^\varepsilon(v,v_+) &- u^\varepsilon(\widetilde{v}(v,\widetilde{v}_+)) \|_{W^{2,p}}\\
    &\leq  \mO_p(1)
    \big(\|(1-\pi_N)(f(u^\varepsilon(v,v_+)(\cdot),\cdot)
      - f(u^\varepsilon(\widetilde{v},\widetilde{v}_+)(\cdot),\cdot))\|_{L^p}
    +|v_+- \widetilde{v}_+|\big) \\
    &\leq  \mO_p(1/\varepsilon) \big(\|v-\widetilde{v}\|_{L^\infty}
        +|v_+- \widetilde{v}_+|\big). \\
  \end{split}
  \end{equation*}
  Using \eqref{eq:pen2e.4}, we get
  \begin{equation*}
      \| u(v,v_+) - u(\widetilde{v}(v,\widetilde{v}_+)) \|_{C(M)}
      \leq  \mO_p(1/\varepsilon) \big(\|v-\widetilde{v}\|_{L^\infty}
          +|v_+- \widetilde{v}_+|\big).
    \end{equation*}
Combining this with \eqref{eq:pen2.6}, it follows that $G$ is continuous. The statement then follows from the Schauder fixed-point theorem,
see for instance \cite[Corollary 11.2]{GiTr01}.
\end{proof}
\subsection{The contact sets and the free boundary}\label{sec:ContactSetFB}
Let $u$ be a solution to \eqref{eq:din1}. In this section we collect some
properties of the \emph{contact sets}
\begin{equation}\label{eq:fbp9}
M_\pm(u):=\{x\in M; u(x) = \phi_\pm(x) \}
\end{equation}
Note that these sets are closed. Their boundaries
\begin{equation}\label{eq:fbp9.0}
  \partial M_\pm(u)
\end{equation}
are called \emph{free boundary}.
\begin{prop}\label{prop:cs1} Let $u$ be the unique solution to
  \eqref{eq:din1}. If $\max_M \phi_- > \min_M \phi_+$, then
  \begin{equation}\label{prop:cs1.1}
  \vol(M_\pm(u)) > 0.
  \end{equation}
  If $\max_M \phi_- = \min_M \phi_+$, then
  \begin{equation}\label{prop:cs1.2}
    M_+(u) = \phi_+^{-1}(\min_M \phi_+), \quad
    M_-(u) = \phi_-^{-1}(\max_M \phi_-).
  \end{equation}
\end{prop}
\begin{proof}
1. When $ \max_M \phi_- > \min_M \phi_+$ then $u$ cannot be
constant since $\phi_-\leq u \leq \phi_+$. If $\vol( M_-(u)) = 0$,
then $\Delta u \geq 0$ almost everywhere on $M$ which implies
that $u$ is constant, see the argument in the paragraph after
\eqref{eq:pen8.1}. This is a contradiction, so $\vol(M_-(u))>0$.
Following a similar argument shows that $\vol(M_+(u)) > 0$.

\medskip
\noindent
2. If $\max_M \phi_- = \min_M \phi_+$, then we conclude \eqref{prop:cs1.2}
directly from Theorem \ref{thm:DPPex}.
\end{proof}

\medskip
\noindent
By Stokes' theorem $\langle \Delta u, 1 \rangle =0$, so
\begin{equation*}
  \langle \Delta u, 1_{M_+(u)} \rangle =
  \langle -\Delta u, 1_{M_-(u)} \rangle.
\end{equation*}
For $f\in H^1(M)$, we have that $\nabla f = 0$ almost everywhere
in any set where $f$ is constant, this follows for instance from
\cite[Lemma 7.7]{GiTr01}. Since $u\in W^{2,p}(M)$, it follows that
for any two vector fields $X_1,X_2\in C^\infty(M;TM)$ we have that
\begin{equation}\label{eq:cs1.0}
  X_1 X_2 (u-\phi_\pm) = 0 \quad \text{a.e. in } \{u=\phi_\pm\}.
\end{equation}
In particular,
\begin{equation}\label{eq:cs1.00}
  \Delta u = \Delta \phi_\pm \quad \text{a.e. in } \{u=\phi_\pm\}.
\end{equation}
So
\begin{equation}\label{eq:cs1.1}
  \int_{M_+(u)}\Delta \phi_+ dx
  =  -\int_{M_-(u)}\Delta \phi_- dx.
\end{equation}
Next, recall \eqref{eq:din1}. Since $\Delta u \geq 0$ on $\{u>\phi_-\}$,
and $\Delta u \leq 0$ in $\{u<\phi_+\}$, it follows that
\begin{equation}\label{eq:cs1}
M_\pm(u) \subseteq \{\pm\Delta\phi_\pm \geq 0\}.
\end{equation}
Since $u$ is continuous, \eqref{eq:fbp1}
implies that the contact sets are disjoint, i.e.
\begin{equation}\label{eq:cs3}
M_-(u)\cap M_+(u) = \emptyset.
\end{equation}
\begin{prop}\label{prop:cs2}
Let $\phi_\pm\in C^4(M;\RR)$ be non constant, satisfy
\eqref{eq:fbp1} and
\begin{equation}\label{eq:cs4}
  d \Delta\phi_\pm \neq 0 \text{ on } \Gamma_\pm:=(\Delta\phi_\pm)^{-1}(0).
\end{equation}
Then, $M_\pm(u) \cap (\Delta\phi_\pm)^{-1}(0)=\emptyset$.
\end{prop}

\medskip
\noindent
Before presenting the proof we remark that an analogous result in the
case of single obstacle problems on Euclidean domains goes back to
\cite{CaRi76}.
\begin{proof}
We will consider only the case of $\phi_-$ since the second case is similar. By \eqref{eq:cs4}, it follows that $\Gamma_-$ is a $C^2$-submanifold of $M$ of dimension $d-1$.

\medskip
\noindent
Let $x_0\in \Gamma_-$ and suppose that $x_0\in M_-(u)$. By
  \eqref{eq:cs4}, we see that $\Gamma_-$ is the boundary of
  $\{\Delta\phi_- < 0\}$ and that for any neighborhood
  $U\subset M$ of $x_0$
  \begin{equation}\label{eq:cs5}
    U\cap \{ \Delta\phi_- > 0\}\neq \emptyset \text{ is open.}
  \end{equation}
Let $B\subset U\cap \{ \Delta\phi_- > 0\}$ be a small open set such that $\partial B$ is tangent to $\Gamma_-$ at $x_0$. This is
possible since $\Gamma_-$ is $C^2$. By \eqref{eq:cs3}, we see that upon potentially shrinking $B$ we may assume that $B\subset \{\phi_- < u < \phi_+\}$. By \eqref{eq:din1} we know that
\begin{equation*}
    \begin{split}
      &\Delta(\phi_--u) = \Delta \phi_- \geq 0 \text{ in } B,\\
      & \phi_--u <0\,\, \text{in}\,\, B, \\
      & (\phi_- - u)(x_0) = 0.
\end{split}
\end{equation*}
Since $u$ is globally $C^{1,\alpha}$, see \eqref{eq:DOPreg}, and harmonic in $\{\phi_- < u < \phi_+\}$, we see that $u \in C^2(B)\cap C^{1}(\overline{B})$. Hopf's Lemma\footnote{Sometimes also called Zaremba's principle} \cite[Section 5, Proposition 2.2]{Ta10} implies that
  \begin{equation*}
    \partial_{\nu}(\phi_- - u)(x_0) >0,
  \end{equation*}
  where $\nu$ denotes the outward-pointing normal to $\partial B$ at $x_0$.
  Thus, by Taylor expansion, we see that near $x_0$, $(\phi_- - u)$
  is strictly increasing in the direction of $\nu$. This is a
  contradiction since $(\phi_- - u)(x_0)=0$ and $u\geq \phi_-$ on $M$.
\end{proof}
\subsection{Regularity of the free boundary via reduction to a one obstacle problem}
Let $u\in W^{2,p}(M)$, $d< p<\infty$, be the unique solution to the double obstacle problem \eqref{eq:din1}, see Theorem \ref{thm:DPPex}. Recall also that $u \in C^{1,\alpha}(M)$ for any $0<\alpha<1$, see \eqref{eq:DOPreg}. The aim of this section is to study the free boundaries $\partial M_{\pm}(u)$. We will work under the assumptions of Proposition \ref{prop:cs2}, and we suppose that 
\begin{equation}\label{eq:shifteq0.1}
  \max_M \phi_- > \min_M \phi_+.
\end{equation}
We will focus on the free boundary $\partial M_{-}(u)$ since treating $\partial M_{+}(u)$ is similar. 
Put
\begin{equation}\label{eq:shifteq0}
v_-:= u - \phi_- \geq 0
\end{equation}
and note that $v_-$ has the same regularity as $u$. Since $u\in W^{2,p}(M)$, $d< p<\infty$, it follows from H\"older's inequality that $u\in H^2(M)$. Moreover, by \eqref{eq:din1}, 
\begin{equation}\label{eq:shifteq1}
  \begin{cases}
  \Delta v_- = -\Delta\phi_- \text{ in } M\backslash (M_-(u)\cup M_+(u)) \\
  \Delta v_- = 0 \text{ a.e. in } M_-(u),\\
  v_-\geq 0.
  \end{cases}
\end{equation}
As discussed in beginning of Section \ref{sec:ContactSetFB}, the contact set $M_-(u)$ is closed and we recall \eqref{eq:cs1}. It follows from Proposition \ref{prop:cs2} that $\dist(M_-(u),(\Delta\phi_-)^{-1}(0))>0$. Put 
\begin{equation*}
  r_0:= \min(r_{\mathrm{inj}},\dist(M_-(u),(\Delta\phi_-)^{-1}(0)))>0.
\end{equation*}
Here $r_{\mathrm{inj}}>0$ denotes the injectivity radius of $M$. For $x_0\in\partial M_-(u)$ consider the open geodesic ball $B(x_0,r)$, $0< r < r_0$, and note that by \eqref{eq:cs1} it is contained in $\{-\Delta\phi_-\geq 0\}$. We will work in local geodesic coordinates $x=(x_1,\dots,x_d)\in B(0,r)\subset\RR^d$ on $B(x_0,r)$ so that $x_0$ corresponds to $x=0$. Let $\Omega:=B(0,r/2)$ and notice that
\begin{equation}\label{eq:fct0.1}
  v_-|_{\Omega} \in H^2(\Omega) 
  \quad \text{and} \quad   
  0\leq v_-|_{\partial\Omega} = (u - \phi_-)|_{\partial\Omega}=:b \in H^{3/2}(\partial\Omega).
\end{equation}
Recall the general expression for the Laplace-Beltrami operator in local coordinates
\begin{equation*}
\Delta f = \frac{1}{\sqrt{\det g}}\sum_{j,k=1}^d \partial_{x_j} \left(g^{jk}\, \sqrt{\det g}\, \partial_{x_k} f\right), 
\quad (g^{jk}) = (g_{jk})^{-1}, 
\end{equation*}
for $f\in H^2(\Omega)$. Here the $g_{jk}$ depend smoothly on $x$ and are the coefficients of the Riemannian metric $g$ in local coordinates. Moreover, the real symmetric positive definite matrix $(g_{ij})_{i,j}$ is such that there exists $0<\lambda < \Lambda$ such that for all $\xi\in\RR^d$ and all $x\in\overline{\Omega}$ 
\begin{equation*}
  \lambda|\xi|^2 \leq \sum_{i,j} g_{ij}(x)\xi^i\xi^j \leq \Lambda|\xi|^2.
\end{equation*}
The same holds true for its inverse $(g^{jk})_{j,k}$. The latter defines a scalar product on the fibers of $T^*\Omega\simeq \Omega\times \RR^d$ which we will denote by $\langle \xi|\eta\rangle_g = \sum_{j,k} g^{i,j}\xi_i\, \eta_k$, $\xi,\eta\in\RR^d$, and it varies smoothly in $x$. The induced fiber-wise norm will be denoted by $|\xi|_g$.

Next, put $F := - 2\Delta \phi_- \geq 0$ and consider the functional
\begin{equation}\label{eq:fct0}
  J(v) = \int_{\Omega} (|\nabla v|_g^2  +  F v )dx,
\end{equation}
for $v\in K = \{v \in H^1(\Omega); v\geq 0, v|_{\partial\Omega}=b \}$. We remind the reader here that, per our notational convention introduced at the beginning of Section \ref{sec:FB}, $dx$ stands for the Riemannian density of integration. This functional admits a unique minimizer among all functions in $K$. The argument to prove this is standard and we will outline it here briefly for the reader's convenience: Since $J(v)$ is positive and $v_-|_{\Omega}\in K$, the infimum of $J(v)$ over all $v\in K$ is positive and finite, and we can find a minimizing sequence $\{v_k\}_{k\in\NN}$ of functions in $K$. One can easily check that this sequence is bounded in $H^1(\Omega)$, so we can extract a subsequence $v_{k_j}$ which converges weakly in $H^1(\Omega)$ to a function $v_0\in H^1(\Omega)$ since the closed unit ball in $H^1(\Omega)$ is weakly sequentially compact. The set $K$ is convex and closed in the strong $H^1$ topology, so by the Hahn-Banach separation theorem it is also weakly closed. Hence $v_0\in K$. Using that the functional $J$ is lower semi-continuous under weak convergence it follows that $v_0$ is indeed a minimizer of $J$. 
\par
If there are two different minimizers $v_1,v_2\in K$, then by the Hölder inequality 
\begin{equation}\label{eq:fct1}
  J\!\left(\frac{v_1+v_2}{2}\right) \leq \frac{1}{2}(J(v_1)+J(v_2)).
\end{equation}
Equality must hold since otherwise the $v_j$ would not be minimizers. A straightforward computation then shows that 
\begin{equation*}
  \int_{\Omega} |\nabla(v_1-v_2)|_g^2 dx=0.
\end{equation*}
Hence, $\nabla v_1=\nabla v_2$ almost everywhere in $\Omega$. Since also $v_1-v_2=0$ on $\partial\Omega$, the Poincaré inequality shows that $v_1=v_2$ almost everywhere in $\Omega$. In conclusion, we have shown that the functional $J(v)$ in \eqref{eq:fct0} admits a unique minimizer among all functions $v\in K$. 

Next, using classical variational calculus, one can easily check that $w\in K$ is a minimizer of \eqref{eq:fct0} if and only if for all $v\in K$ 
\begin{equation}\label{eq:fct2}
  \int_{\Omega} \big[2\langle \nabla w | \nabla(v-w)\rangle_g + F(v-w)\big] dx \geq 0.
\end{equation}
Indeed, to see sufficiency, notice that $w+\varepsilon(v-w)\in K$, $\varepsilon\in (0,1)$, for any $v\in K$ due to the convexity of $K$. Since $w$ is a minimizer of \eqref{eq:fct0}, it follows that $J(w)\leq J(w+\varepsilon(v-w))$. 
Hence, 
\begin{equation*}
  J(w) \leq J(w) + \varepsilon \int_\Omega \big[2\langle \nabla w | \nabla(v-w)\rangle_g + F(v-w)\big] dx + \varepsilon^2 \int_\Omega |\nabla(v-w)|^2_g\, dx.
\end{equation*}
Subtracting from both sides $J(w)$, dividing by $\varepsilon$ and letting $\varepsilon\to 0$, shows \eqref{eq:fct2}. To see the necessity of \eqref{eq:fct2}, notice that it is equivalent to 
\begin{equation*}
  \int_{\Omega} 2\langle \nabla w | \nabla v\rangle_g + Fv\big] dx \geq 
  J(w) + \int_{\Omega} |\nabla w|_g^2 dx.
\end{equation*}
Since $2\langle \nabla w | \nabla v\rangle_g  \leq |\nabla w|_g^2 + |\nabla v|_g^2$, it follows that $J(w)\leq J(v)$ for all $v\in K$ which shows that $w$ is a minimizer of \eqref{eq:fct0}.

\medskip
By \eqref{eq:shifteq1}, \eqref{eq:fct0.1} we know that $v_-|_{\Omega}\in K$. To ease the notation we will drop the restriction to $\Omega$ and simply write $v_-$. Notice that $v-v_- \in H^1_0(\Omega)$ for any $v\in K$. Integration by parts and the definition of $F$ in \eqref{eq:fct0} give 
\begin{equation}\label{eq:fct3}
  \int_{\Omega} \big[2\langle \nabla v_- | \nabla(v-v_-)\rangle_g + F(v-v_-)\big] dx
  =-2\int_{\Omega} (\Delta v_- +\Delta\phi_-) (v-v_-) dx.
\end{equation}
By \eqref{eq:shifteq0}, \eqref{eq:shifteq1} we see that $\Delta v_- = -\Delta\phi_-$ 
on $\{v_- > 0\}\cap \Omega$ which gives that the integrand on the right hand side of \eqref{eq:fct3} is equal to $0$ there. Similarly, we see that $\Delta v_- =0$ almost everywhere in $\{v_- = 0\}\cap \Omega$. Thus, using also that any $v\in K$ is non-negative and that $-\Delta\phi_- \geq 0$ in $\Omega$, 
\begin{equation*}
  -2\int_{\Omega} (\Delta v_- +\Delta\phi_-) (v-v_-)dx
  \geq 0.
\end{equation*}
It follows that $v_-$ is the unique minimizer of \eqref{eq:fct0}. By \cite[Theorem 1.1]{FoGeSp15}, we know that the free boundary $\partial\{v_- = 0\}\cap\Omega$ is given by the disjoint union $\mathrm{Reg}(v_-)\cup\mathrm{Sing}(v_-)$ of regular and singular points of the free boundary. More precisely, we know that the singular points $\mathrm{Sing}(v_-) \subset \bigcup_{k=0}^{d-1}S_k$, with the $S_k$ contained in the union of at most countably many $C^1$-submanifolds of dimension $k$. Moreover, the set of regular points $\mathrm{Reg}(u)$ is relatively open in $\partial\{v_- = 0\}\cap\Omega$ and that for every point $y_0\in\mathrm{Reg}(u)\cap\Omega$ there exist $s = s(y_0) > 0$ and $\beta= \beta(y_0)\in(0, 1)$ such that $\mathrm{Reg}(u)\cap B(y_0,s)$ is a $C^{1,\beta}$-submanifold of dimension $d-1$. Hence, in view of \eqref{eq:shifteq0}, covering the compact set $\partial M_-(u)$ with finitely many geodesic balls $B(p_j,r_j)$, $0<r_j<r_0$, $p_j\in \partial M_-(u)$, it follows that the free boundary $\partial M_-(u)$ is given by the disjoint union $\mathrm{Reg}_-(u)\cup\mathrm{Sing}_-(u)$ of regular and singular points such that 
\begin{itemize}
  \item $\mathrm{Sing}_-(u) \subset \bigcup_{k=0}^{d-1}S_k$, with the $S_k$ contained in the union of at most countably many $C^1$-submanifolds of dimension $k$;
  \item $\mathrm{Reg}_-(u)$ is relatively open in $\partial M_-(u)$ and that for every point $p\in\mathrm{Reg}_-(u)$ there exist $s = s(p) > 0$ and $\beta= \beta(p)\in(0, 1)$ such that $\mathrm{Reg}_-(u)\cap B(p,s)$ is a $C^{1,\beta}$-submanifold of dimension $d-1$.
\end{itemize}
These properties together with the compactness $\partial M_-(u)$ directly imply that 
\begin{equation}\label{eq:FBlowerVol0}
  \mathrm{vol}(\partial\{u= \phi_-\})=0.
\end{equation}

Performing the exact same steps with $v_+=\phi_+-u$ gives similar results for the free boundary $\partial M_+(u)$. In conclusion we obtain 
\begin{theo}
\label{thm:FreeBoundary}
Let $\phi_\pm\in C^4(M;\RR)$ be non constant, satisfy \eqref{eq:fbp1}, and
\begin{equation*}
d \Delta\phi_\pm \neq 0 \text{ on } \Gamma_\pm:=(\Delta\phi_\pm)^{-1}(0).
\end{equation*}
Then, the free boundary $\partial M_\pm(u)$ are given by the disjoint union $\mathrm{Reg}_\pm(u)\cup\mathrm{Sing}_\pm(u)$ of regular and singular points such that 
\begin{itemize}
  \item $\mathrm{Sing}_\pm(u) \subset \bigcup_{k=0}^{d-1}S_k$, with the $S_k^\pm$ contained in the union of at most countably many $C^1$-submanifolds of dimension $k$;
  \item $\mathrm{Reg}_\pm(u)$ is relatively open in $\partial M_\pm(u)$ and that for every point $p\in\mathrm{Reg}_-(u)$ there exist $s = s(p) > 0$ and $\beta= \beta(p)\in(0, 1)$ such that $\mathrm{Reg}_\pm(u)\cap B(p,s)$ is a $C^{1,\beta}$-submanifold of dimension $d-1$.
\end{itemize}
In particular, 
\begin{equation*}
    \vol(\partial M_\pm(u))=0.
\end{equation*}
\end{theo}
\subsection{Optimal weights}
Let us return to the case when $M$ is a compact Riemann surface equipped with a conformal Riemannian metric, as in Section \ref{sec:Setting}. We view $M$ as a real $2$-dimensional Riemannian manifold, and we let $\varphi:M\to \RR$ be a non-constant $C^4$ function such that
\begin{equation}
\label{eq:cs4_2}
d \Delta \varphi \neq 0 \text{ on } (\Delta \varphi)^{-1}(0),
\end{equation}
cf. (\ref{2.tb0}). Let $\tau >0$ such that
\begin{equation}
\label{eq:opw2_2}
\max_M \varphi - \min_M \varphi \geq \tau >0.
\end{equation}

\medskip
\noindent
Consider the following special case of the double obstacle problem \blue{(\ref{eq:din1})},
\begin{equation}\label{eq:din1_2}
\begin{cases}
\Delta u \geq 0 \quad \text{on } \{ u > \varphi-\tau\}, \\
\Delta u \leq 0 \quad \text{on } \{ u<\varphi\}, \\
\varphi-\tau \leq u \leq \varphi \text{ a.e on } M.
\end{cases}
\end{equation}
Putting $\phi_-:=\varphi-\tau$ and $\phi_+:=\varphi$, we see that the condition \eqref{eq:fbp1}, the assumptions of Theorems \ref{thm:DPPex} and \ref{thm:FreeBoundary}, and of Propositions \ref{prop:cs1} and \ref{prop:cs2} are satisfied. Combining these results with \eqref{eq:cs1.1} we obtain Theorem \ref{thm:DPPex2}.
\appendix

\section{Almost holomorphic extensions}
\label{sec:almholexten}
\subsection{Almost holomorphic extensions of vector fields}
We shall first work locally on $\mathbb R^n$ and its complexification $\CC^n$. Let 
\begin{equation}
\label{A1}
\nu = \sum_{j=1}^n a_j(x)\partial_{x_j}
\end{equation}
be a smooth vector field on a domain in $\RR^n$. We allow the coefficients $a_j$ to be complex valued, but assume that they are smooth. Let $\widetilde{a}_j$ stand for almost holomorphic extensions of $a_j$ on a suitable domain in $\CC^n$, see~\cite{Ho69,MeSj74}. We denote the coordinates on $\mathbb C^n$ by $z_1,\dots,z_n$, $z_j=x_j+iy_j$, so that the almost holomorphic extensions $\widetilde{a}_j$ become functions of $z$. We have 
\begin{equation}
\widetilde{a}_j\big|_{\RR^n} = {a}_j, \quad \partial_{\overline{z}}\widetilde{a}_j(z) = \mathcal{O}(|\Im z|^\infty).
\end{equation}
An almost holomorphic extension of the vector field $\nu$ in (\ref{A1}) is then the $(1,0)$ vector field on $\mathbb C^n$ given by 
\begin{equation}
\widetilde{\nu} = \sum_{j=1}^n \widetilde{a}_j(z)\partial_{z_j}, \quad \partial_{z_j} = \frac{1}{2}\left(\partial_{x_j} - i \partial_{y_j} \right).
\end{equation}
Later we may drop the tildes when there is no risk of confusion. 

\medskip
\noindent
Let $\widehat{\widetilde{\nu}}$ be the real vector field on $\mathbb C^n$ associated to $\widetilde{\nu}$, given by 
\begin{equation*}
\widehat{\widetilde{\nu}} = \widetilde{\nu} +\overline{\widetilde{\nu}}, \quad \overline{\widetilde{\nu}} = \sum_{j=1}^n \overline{\widetilde{a}}_j(z)
  \partial_{\overline{z}_j}.
\end{equation*}
We have $i\widetilde{\nu}= \sum_{j=1}^n i\widetilde{a}_j(z) \partial_{z_j}$ so that 
\begin{equation*}
i \widetilde{\nu} = \widetilde{i\nu} = \sum_{j=1}^n i\widetilde{a}_j(z) \partial_{z_j}, \quad \widehat{i\widetilde{\nu}} = i\widetilde{\nu} - i 
\overline{\widetilde{\nu}}.
\end{equation*}
In the real coordinates $x,y$, where $z=x+iy$, we have 
\begin{equation*}
\widehat{\widetilde{\nu}} = \sum_{j=1}^n \left((\Re \widetilde{a}_j)\partial_{x_j} + (\Im  \widetilde{a}_j )\partial_{y_j}\right), \quad \widehat{i\widetilde{\nu}} = \sum_{j=1}^n \left(-(\Im \widetilde{a}_j)\partial_{x_j} + (\Re  \widetilde{a}_j )\partial_{y_j}\right), 
\end{equation*}
so we can identify $\widehat{\widetilde{\nu}}$, $\widehat{i\widetilde{\nu}}$ with the real vectors fields $(\Re \widetilde{a}, \Im \widetilde{a})$, $(-\Im \widetilde{a}, \Re \widetilde{a})$ on $\RR^{2n}_{x,y}\simeq \mathbb C^n_z$, respectively. We compute the commutator 
\begin{equation}
\label{eq_appA0}
\begin{split}
  [\widehat{\widetilde{\nu}},\widehat{i\widetilde{\nu}}] 
  &=
  [\widetilde{\nu} +\overline{\widetilde{\nu}},
                i\widetilde{\nu} - i\overline{\widetilde{\nu}}]
  =
  -i [\widetilde{\nu},\overline{\widetilde{\nu}}]
  +i[\overline{\widetilde{\nu}},\widetilde{\nu}]\\
  &=2i[\overline{\widetilde{\nu}},\widetilde{\nu}] 
    =2i \sum_{j,k=1}^n [\overline{\widetilde{a}}_j\partial_{\overline{z}_j},
      \widetilde{a}_k\partial_{z_k}] \\
  & = 2i \sum_{j,k=1}^n \left(
    \overline{\widetilde{a}}_j\partial_{\overline{z}_j}(\widetilde{a}_k)\partial_{z_k} 
      - \widetilde{a}_k(\partial_{z_k}\overline{\widetilde{a}}_j)\partial_{\overline{z}_j}   \right)\\
  & = \sum_{k=1}^n \mathcal{O}(|\Im z|^\infty)\partial_{z_k} 
  + \sum_{j=1}^n \mathcal{O}(|\Im z|^\infty)\partial_{\overline{z}_j},
\end{split}
\end{equation}
which is a real vector field. 

\medskip
\noindent
Consider the flow $\Phi_s=\exp (s\widehat{\widetilde{\nu}})$ of the real vector field $\widehat{\widetilde{\nu}}$, defined locally in $\mathbb C^n$ for small times $s\in\RR$. We have 
\begin{equation}
\label{A2}
\begin{cases}
\partial_s \Phi_s(z) = \widehat{\widetilde{\nu}}(\Phi_s(z) ), \quad s\in\neigh(0,\RR),\\
\Phi_0(z) =z.
\end{cases}
\end{equation}
More explicitly, by the identification $\widehat{\widetilde{\nu}} \simeq \widetilde{a}$, the first equation in (\ref{A2}) reads 
\begin{equation*}
\partial_s\Phi_s(z) = \widetilde{a}(\Phi_s(z) ). 
\end{equation*}
Applying $\partial_z$ and $\partial_{\overline{z}}$, we get 
\begin{equation}\label{eq:ahvf1}
  \partial_s\partial_z\Phi_s(z) = \partial_w\widetilde{a}(\Phi_s(z) )\partial_z\Phi_s(z)
  + \partial_{\overline{w}}\widetilde{a}(\Phi_s(z) )\partial_z\overline{\Phi}_s(z),
\end{equation}
\begin{equation*}
  \partial_s\partial_{\overline{z}}\Phi_s(z) 
  = \partial_w\widetilde{a}(\Phi_s(z) )\partial_{\overline{z}}\Phi_s(z) 
  + \partial_{\overline{w}}\widetilde{a}(\Phi_s(z) )\partial_{\overline{z}}\overline{\Phi}_s(z). 
\end{equation*}
Taking the complex conjugation of the last equation gives 
\begin{equation}\label{eq:ahvf2}
\partial_s\partial_{z}\overline{\Phi}_s(z) = 
\overline{\partial_{\overline{w}}\widetilde{a}}(\Phi_s(z) )\partial_{z}\Phi_s(z)
+ \overline{\partial_{w}\widetilde{a}}(\Phi_s(z) )\partial_{z}\overline{\Phi}_s(z).
\end{equation}
Here the linear evolution system (\ref{eq:ahvf1}), (\ref{eq:ahvf2}) for $\partial_z\Phi_s$, $\partial_{z}\overline{\Phi}_s$ 
comes with the initial condition 
\begin{equation}\label{eq:ahvf3}
\partial_z\Phi_0=1, \quad \partial_z\overline{\Phi}_0 = 0. 
\end{equation}
Recalling that $\widetilde{\nu}$ and $\widetilde{a}$ are almost holomorphic, we infer from \eqref{eq:ahvf1}, \eqref{eq:ahvf2} 
\begin{equation}\label{eq:ahvf4}
\begin{cases}
\partial_s\partial_z\Phi_s(z) =\partial_w\widetilde{a}(\Phi_s(z) )\partial_z\Phi_s(z) 
      + \mathcal{O}(|\Im\Phi_s(z)|^\infty )\partial_z\overline{\Phi}_s(z), \\ 
      \partial_s\partial_{z}\overline{\Phi}_s(z) = 
    \mathcal{O}(|\Im\Phi_s(z)|^\infty )\partial_{z}\Phi_s(z) 
    + \overline{\partial_{w}\widetilde{a}}(\Phi_s(z) )\partial_{z}\overline{\Phi}_s(z).
  \end{cases}
\end{equation}
We shall only consider trajectories $[0,s_0]\ni s \mapsto \Phi_s(z)$, with $s_0 = \mathcal{O}(1)$, that remain in some fixed compact set, and then we 
know that $\partial_z\Phi_s$, $\partial_z\overline{\Phi}_s$ are $\mathcal{O}(1)$. From \eqref{eq:ahvf3}, \eqref{eq:ahvf4}, we then get 
\begin{equation}\label{eq:ahvf5}
 \overline{ \partial_{\overline{z}}\Phi_s(z)} = 
 \partial_{z}\overline{\Phi}_s(z) = \mathcal{O}(1)\int_0^s |\Im\Phi_\sigma(z)|^\infty d\sigma. 
\end{equation}

\bigskip
\noindent
We shall next look at the flow $\exp(\widehat{t\widetilde{\nu}})(z)$ for $t\in \mathrm{neigh}(0,\CC)$\footnote{This notation means that $t$ is in a 
sufficiently small open complex neighborhood of $0$.}. Put $\Phi_{s,t}=\exp (s\widehat{t\widetilde{\nu}})$ for $0\leq s\leq 1$, 
so that $\Phi_{1,t}=\exp (\widehat{t\widetilde{\nu}})$. Replacing $\widehat{\widetilde{\nu}}$ by $\widehat{t\widetilde{\nu}}$ and 
$\widetilde{a}$ by $t\widetilde{a}$ in (\ref{A2}), we get 
$\partial_s\Phi_{s,t}(z) = \widehat{t\widetilde{\nu}}(\Phi_{s,t}(z))$ or 
equivalently, 
\begin{equation}
\label{eq:ahvf6}
  \partial_s\Phi_{s,t}(z) = t \widetilde{a}(\Phi_{s,t}(z)), 
  \quad \Phi_{0,t}(z)=z.
\end{equation}
Applying $\partial_t$ and $\partial_{\overline{t}}$ to (\ref{eq:ahvf6}), we get 
\begin{equation}\label{eq:ahvf7}
\partial_s\partial_t\Phi_{s,t}(z) = \widetilde{a}(\Phi_{s,t}(z)) + t \partial_w\widetilde{a}(\Phi_{s,t}(z))\partial_t\Phi_{s,t}(z) 
+ t \partial_{\overline{w}}\widetilde{a}(\Phi_{s,t}(z))\partial_t\overline{\Phi}_{s,t}(z),
\end{equation}
\begin{equation*}
  \partial_s\partial_{\overline{t}}\Phi_{s,t}(z) = 
   t \partial_w\widetilde{a}(\Phi_{s,t}(z))\partial_{\overline{t}}\Phi_{s,t}(z) 
  + t \partial_{\overline{w}}\widetilde{a}(\Phi_{s,t}(z))\partial_{\overline{t}}
  \overline{\Phi}_{s,t}(z).
\end{equation*}
Taking the complex conjugation of the last equation gives 
\begin{equation}\label{eq:ahvf8}
\partial_s\partial_{t}\overline{\Phi}_{s,t}(z) = \overline{t}\, \overline{\partial_w\widetilde{a}}(\Phi_{s,t}(z))\partial_{t}\overline{\Phi}_{s,t}(z) 
  + \overline{t}\, \partial_{w}\overline{\widetilde{a}}(\Phi_{s,t}(z))\partial_{t}\Phi_{s,t}(z).
\end{equation}
The initial condition for the system \eqref{eq:ahvf7}, \eqref{eq:ahvf8} is 
\begin{equation}\label{eq:ahvf9}
\partial_{t}\Phi_{0,t}=0, \quad \partial_{\overline{t}}\Phi_{0,t}=0.
\end{equation}
Using the almost holomorphy of $\widetilde{a}$ in \eqref{eq:ahvf7}, \eqref{eq:ahvf8}, we get 
\begin{equation}\label{eq:ahvf10}
  \begin{cases}
    \partial_s\partial_t\Phi_{s,t}(z) - t\, \partial_w\widetilde{a}(\Phi_{s,t}(z))\partial_t\Phi_{s,t}(z) 
    = \widetilde{a}(\Phi_{s,t}(z))
  + t\, \mathcal{O}(|\Im\Phi_{s,t}(z)|^\infty)\partial_t \overline{\Phi}_{s,t}(z),\\
  \partial_s\partial_{t}\overline{\Phi}_{s,t}(z) - 
  \overline{t} \,\overline{\partial_w\widetilde{a}}(\Phi_{s,t}(z))\partial_{t}\overline{\Phi}_{s,t}(z) 
  = \overline{t} \,\mathcal{O}(|\Im\Phi_{s,t}(z)|^\infty)\partial_{t}\Phi_{s,t}(z).
  \end{cases}
\end{equation}

\medskip
\noindent
Assume that 
\begin{equation}\label{eq:ahvf10.1}
  |\Im \Phi_{s,t}(z)| \leq \varepsilon, \quad 0\leq s \leq 1,
\end{equation}
where $\varepsilon >0$ is a small parameter. Then by integration of the second equation in (\ref{eq:ahvf10}) from $s=0$ to $s=1$, we get 
\begin{equation}\label{eq:ahvf11}
  \partial_{\overline{t}}\Phi_{1,t}(z) = \mathcal{O}(t\varepsilon^\infty). 
\end{equation}
Here we recall that 
\begin{equation}\label{eq:ahvf12}
  \Phi_{1,t}(z) = \exp(\widehat{t\widetilde{\nu}})(z), 
\end{equation}
and we have also used that $\partial_{\overline{t}}\Phi_{s,t} = \overline{\partial_t\overline{\Phi}_{s,t}}$. 

\bigskip
\noindent
We shall next discuss approximation of the map $\exp(\widehat{t\widetilde{\nu}})$ by the composition of the flows 
\[ 
\exp(t_1 \widehat{\widetilde{\nu}})\exp(t_2 \widehat{i\widetilde{\nu}}), 
\] 
when $t=t_1+it_2\in\CC$, $t_1$, $t_2 \in \mathbb R$. When doing so, let us start by making some elementary computations and remarks. Let $\Omega \subseteq \mathbb R^N$ be open and let $X$ be a $C^{\infty}$ real vector field on $\Omega$. Associated to $X$ is its flow $\Phi_s = \exp(sX)$, $s\in \mathbb R$, defined by
\[ 
\partial_s \Phi_s(y) = X(\Phi_s(y)), \quad \Phi_0(y) = y.
\] 
For each $\omega \Subset \Omega$ there exists $\varepsilon = \varepsilon_{\omega} > 0$ such that the map 
\[ 
\omega \times (-\varepsilon, \varepsilon) \ni (y,s) \mapsto \Phi_s(y) \in \Omega
\] 
is defined and is of class $C^{\infty}$. It follows that the pullback map 
\[ 
\Phi_s^* u = u \circ \Phi_s, \quad u \in C^{\infty}(\Omega), 
\] 
satisfies 
\[ 
\Phi_s^*: C^{\infty}(\Omega) \rightarrow C^{\infty}(\omega), \quad |s| < \varepsilon, 
\] 
and depends smoothly on $s$. We have a Taylor expansion at $s=0$, 
\begeq
\label{eq_appA1}
\Phi_s^* u \sim \sum_{k=0}^{\infty} \frac{s^k X^k u}{k!}. 
\endeq
Here $X$ is viewed as a homogeneous real first order differential operator on $\Omega$ with $C^{\infty}$ coefficients. Let $Y$ be a second $C^{\infty}$ real vector field on $\Omega$ and consider the composition $\Phi_s^* \circ Y \circ \Phi_{-s}^*$ acting on $C^{\infty}(\Omega)$. Here $\Phi_s^* \circ Y \circ \Phi_{-s}^*$ is a first order differential operator with coefficients depending smoothly on $s$. It follows from (\ref{eq_appA1}) that 
\begeq
\label{eq_appA2}
\partial_s|_{s=0} \left(\Phi_s^* \circ Y \circ \Phi_{-s}^*\right) = [X,Y], 
\endeq
and combining (\ref{eq_appA2}) with the group property 
\begeq
\label{eq_appA2.1}
\Phi_t^* \circ \Phi_s^* = \Phi_{s+t}^*, \quad t,s \in {\rm neigh}(0,\mathbb R), 
\endeq
we get more generally 
\begeq
\label{eq_appA3}
\partial_s \left(\Phi_s^* \circ Y \circ \Phi_{-s}^*\right) = \Phi_s^* \circ [X,Y] \circ \Phi_{-s}^*.
\endeq
Integrating (\ref{eq_appA3}) we obtain 
\begeq
\label{eq_appA4} 
\Phi_s^* \circ Y \circ \Phi_{-s}^* - Y = \int_0^s \Phi_{\sigma}^* \circ [X,Y] \circ \Phi_{-\sigma}^*\, d\sigma.
\endeq
Composing (\ref{eq_appA4}) with $\Phi_s^*$ to the right and using (\ref{eq_appA2.1}) we get therefore 
\begeq
\label{eq_appA5} 
[\Phi_s^*, Y] = \int_0^s \Phi_{\sigma}^* \circ [X,Y] \circ \Phi_{s-\sigma}^*\, d\sigma.
\endeq 

\medskip
\noindent
Next, let $\Psi_s = \exp(sY)$, $\Lambda_s = \exp(s(X+Y))$, $s\in {\rm neigh}(0,\mathbb R)$, and consider the derivative 
$\partial_s \left(\Psi_s^* \circ \ \Phi_s^* \circ \Lambda_{-s}^*\right)$. Here we first observe that  
\begeq
\label{eq_appA7} 
\partial_s|_{s=0} \left(\Psi_s^* \circ \ \Phi_s^* \circ \Lambda_{-s}^*\right) = 0,
\endeq
in view of (\ref{eq_appA1}) and its analogues for $\Psi_s$ and $\Lambda_s$. More generally, we may write using the group property, 
\begin{equation} 
\label{eq_appA8} 
\partial_s \left(\Psi_s^* \circ \ \Phi_s^* \circ \Lambda_{-s}^*\right) = \Psi_s^* \circ \left(\partial_t|_{t=0} \left( \Psi_t^* \circ \Phi^*_{s+t} \circ \Lambda^*_{-t}\right) \right) \circ \Lambda^*_{-s}. 
\end{equation}
Here we have in view of (\ref{eq_appA7}) and (\ref{eq_appA1}), 
\begin{multline} 
\label{eq_appA9}
\partial_t|_{t=0} \left( \Psi_t^* \circ \Phi^*_{s+t} \circ \Lambda^*_{-t}\right) = \partial_t|_{t=0} \left( \Psi_t^* \circ \Phi^*_s \circ \Phi^*_t \circ \Lambda^*_{-t}\right) \\
= \Phi_s^* \circ \partial_t|_{t=0} \left( \Psi_t^* \circ \Phi^*_t \circ \Lambda^*_{-t}\right) + \partial_t|_{t=0} \left([\Psi_t^*, \Phi_s^*] \circ \Phi_t^* \circ \Lambda^*_{-t}\right) = [Y, \Phi_s^*]. 
\end{multline} 
We infer, combining (\ref{eq_appA8}), (\ref{eq_appA9}), and (\ref{eq_appA5}),  
\begeq
\label{eq_appA10} 
\partial_s \left(\Psi_s^* \circ \ \Phi_s^* \circ \Lambda_{-s}^*\right) = \Psi_s^* \circ [Y, \Phi_s^*] \circ \Lambda^*_{-s} = \Psi_s^* \circ \left(
\int_0^s \Phi_{\sigma}^* \circ [Y,X] \circ \Phi_{s-\sigma}^*\, d\sigma \right) \circ \Lambda^*_{-s}.
\endeq 

\medskip
\noindent
We shall apply (\ref{eq_appA10}) when $X = t_1\widehat{\widetilde{\nu}}$, $Y = t_2\widehat{i\widetilde{\nu}}$, $\Omega$ is a suitable bounded open set in $\mathbb C^n$ with $\Omega \cap \mathbb R^n \neq \emptyset$, and $t = t_1 + it_2 \in {\rm neigh}(0,\mathbb C)$. Here we observe that 
\[
X + Y = t_1 \widehat{\widetilde{\nu}} + t_2 \widehat{i\widetilde{\nu}} = t_1 (\widetilde{\nu}+\overline{\widetilde{\nu}}) +  t_2 (i\widetilde{\nu}-i\overline{\widetilde{\nu}}) = t\widetilde{\nu} + \overline{t \widetilde{\nu}} = \widehat{t\widetilde{\nu}},
\] 
and hence $\Lambda_s = \exp(s \widehat{t \widetilde{\nu}})$. We get 
\begin{equation}
\label{eq:ahvf15}
\partial_s \left(\Psi_s^* \circ \ \Phi_s^* \circ \Lambda_{-s}^*\right) 
= \Psi_s^* \circ \left(\int_0^s \Phi_{\sigma}^* \circ  t_2t_1 [\widehat{i\widetilde{\nu}},\widehat{\widetilde{\nu}}] \circ \Phi_{s-\sigma}^*\, d\sigma \right) \circ \Lambda^*_{-s}.
\end{equation} 
Assume that 
\begeq
\label{eq_appA10.1}
|{\rm Im}\, \Phi_{\sigma}(\Psi_s(z))| \leq \mathcal O(\varepsilon), \quad 0 \leq \sigma \leq s \leq 1, 
\endeq
for all $z\in \Omega \Subset \mathbb C^n$. It follows then from (\ref{eq:ahvf15}), (\ref{eq_appA10.1}), and (\ref{eq_appA0}) that 
\begeq
\label{eq_appA10.2}
\partial_s \left(\Psi_s^* \circ \ \Phi_s^* \circ \Lambda_{-s}^*\right) = \mathcal O(\varepsilon^{\infty}), \quad 0 \leq s \leq 1, 
\endeq
in the sense of linear continuous operators: $C^{\infty}(\Omega) \rightarrow C^{\infty}(\omega)$, $\omega \Subset \Omega$. 

\medskip
\noindent
{\it Example}. Assume that the vector field $\nu$ in (\ref{A1}) is real, i.e. the coefficients $a_j$ in (\ref{A1}) are real valued. Then the vector field $X = t_1\widehat{\widetilde{\nu}}$ is tangent to $\mathbb R^n$, for $t_1 \in {\rm neigh}(0,\mathbb R)$, and we get locally 
\begeq
\label{eq_appA11} 
|{\rm Im}\, \Phi_{\sigma}(z)| = |{\rm Im}\, \exp(\sigma X)(z)| \leq \mathcal O(1) |{\rm Im}\, z|,\quad 0 \leq \sigma \leq 1. 
\endeq
Using also that 
\[ 
{\rm Im}\, \Psi_s(z) = {\rm Im}\, z + \mathcal O(|t_2|), \quad 0 \leq s \leq 1, 
\] 
we get 
\[ 
|{\rm Im}\, \Phi_{\sigma}(\Psi_s(z))| \leq \mathcal O(1) |{\rm Im}\, \Psi_s(z)| \leq \mathcal O(1) \left(|{\rm Im}\, z| + |t_2|\right) \leq \mathcal O(\varepsilon), \quad 0 \leq \sigma, s \leq 1,
\] 
provided that ${\rm Im}\, z = \mathcal O(\varepsilon)$ and $t_ 2 = \mathcal O(\varepsilon)$. Letting $\Omega$ be an $\varepsilon$--neighborhood of the real domain, we conclude that (\ref{eq_appA10.1}), (\ref{eq_appA10.2}) hold for all $z\in \Omega$, $t_1 \in {\rm neigh}(0,\mathbb R)$, and $t_2 = \mathcal O(\varepsilon)$. 

\medskip
\noindent
Assuming that (\ref{eq_appA10.1}) holds and writing 
\begin{equation}
\label{eq_appA12}
\Psi_1^* \circ \Phi_1^* \circ \Lambda_{-1}^* - 1 = \int_0^1 \partial_s \left(\Psi_s^* \circ \ \Phi_s^* \circ \Lambda_{-s}^*\right)\, ds, 
\end{equation} 
we get in view of (\ref{eq_appA10.2}), 
\[ 
\Psi_1^* \circ \Phi_1^* \circ \Lambda_{-1}^* - 1 = \mathcal O(\varepsilon^{\infty}). 
\] 
Therefore,
\[ 
\left(\exp(X) \circ \exp(Y)\right)^* = \Psi_1^* \circ \Phi_1^* = (1 + \mathcal O(\varepsilon^{\infty}))\circ \Lambda_{1}^* = 
(1 + \mathcal O(\varepsilon^{\infty}))\circ (\exp(X+Y))^*, 
\]
and we get the corresponding relation for the flows, 
\begeq
\label{eq_appA13}
\exp(t_1\widehat{\widetilde{\nu}})\exp(t_2\widehat{i\widetilde{\nu}})(z) = \exp(\widehat{t\widetilde{\nu}})(z) + \mathcal O(\varepsilon^{\infty}), \quad z\in \Omega. 
\endeq 

\paragraph{\textbf{Almost holomorphic coordinate functions}}
Let $M$ be a complex manifold of complex dimension $n$ and let $\Gamma\subseteq M$ be a compact connected real submanifold of real dimension $n$. We assume that $\Gamma$ is totally real, 
\begin{equation}
\label{eq:ahvf18}
T_z M = T_z\Gamma \oplus i T_z\Gamma, \quad \forall z\in \Gamma. 
\end{equation}
Here when defining the action of the imaginary unit $i$ on the real linear subspace $T_z \Gamma \subseteq T_z M$, we use the isomorphism $TM \simeq T^{(1,0)}M$ given by $\displaystyle \nu \mapsto \mu = \frac12 (\nu - i J \nu)$. Here $J: TM \rightarrow TM$ is the canonical almost complex structure induced by the complex structure on $M$. We have 
\[ 
\nu = 2{\rm Re}\, \mu, \quad J \nu = 2{\rm Re}\, (i\mu), 
\] 
and in local holomorphic coordinates $z = (z_1, \ldots \, z_n)$, $z_ j = x_j + iy_j$, on $M$ the isomorphism is given by 
\[
\nu = \sum_{j=1}^n \left(a_j \partial_{x_j} + b_j \partial_{y_j}\right) \mapsto \mu = \sum_{j=1}^n (a_j + i b_j)\partial_{z_j}. 
\] 

\medskip
\noindent
For every $z_0\in\Gamma$, there are neighborhoods $V\subseteq M$, $W\subseteq \CC^n$ of $z_0$ and $0$, respectively, and a $C^{\infty}$ diffeomorphism $\kappa:V \to W$, mapping $\Gamma$ to $\RR^n$, which is almost holomorphic in the sense that 
\begin{equation*}
|\overline{\partial}\kappa(z)| = \mathcal{O}(\dist(z,\Gamma)^\infty),\quad z\in V, 
\end{equation*}
\begin{equation*}
|\overline{\partial}\kappa^{-1}(w)| = \mathcal{O}(|\Im w|^\infty), \quad w\in W.
\end{equation*}
Indeed, locally in a neighborhood of $z_0$, we have $\Gamma = \{f(y); y\in \mathbb R^n\}$, where $f\in C^{\infty}(\mathbb R^n; \mathbb C^n)$ has injective differential, $f(0) = z_0$. It is then well known that $\Gamma$ is totally real precisely when the complex $n\times n$ matrix $(\partial_{y_k} f_j(0))$ is invertible, see~\cite[Section 3]{SjZw91}. Letting $\widetilde{f}$ be an almost holomorphic extension of $f$, we can take $\kappa = \widetilde{f}^{-1}$. 

\medskip
\noindent 
Let $\nu$ be a smooth complex vector field on $\Gamma$. Then locally, the pushforward $\kappa_* \nu$ is a well defined smooth vector field on $\RR^n$. Taking an almost holomorphic extension of $\kappa_* \nu$ of type $(1,0)$ to a complex domain and using the inverse $\kappa^{-1}$ to push it forward to $M$, we obtain a locally defined vector field $\widetilde{\nu}$ on $M$. The $(0,1)$ part of $\widetilde{\nu}$ is then $\mathcal{O}(\dist(z,\Gamma)^\infty)$ and can be suppressed. If we use two different diffeomorphisms as above, then in the overlap region the two extensions $\widetilde{\nu}$ agree modulo $\mathcal{O}(\dist(z,\Gamma)^\infty)$. It is then clear how to get an extension 
$\widetilde{\nu}$ to a neighborhood of $\Gamma$ in $M$ by means of a partition of unity. It is unique modulo $\mathcal{O}(\dist(z,\Gamma)^\infty)$. 

\medskip
\noindent 
Finally, let $M$ be of complex dimension $1$ and let $\Gamma \subseteq M$ be of real dimension $1$, i.e. a simple closed curve. 
Let $\nu$ be a non-vanishing real vector field on $\Gamma$ and let $T_0>0$ be the minimal period of the $\nu$--flow, so that $\exp T_0\nu = \mathrm{Id}$. An almost holomorphic complex coordinate function $t\in \RR/ T_0\ZZ + i(-\varepsilon,\varepsilon)$ in an $\varepsilon$--neighborhood of $\Gamma$ in $M$ is given by 
\begin{equation*}
z =  \exp((\Re t)\widehat{\widetilde{\nu}})\exp((\Im t)\widehat{i\widetilde{\nu}})(z_0), 
\end{equation*}
where $z_0\in \Gamma$ is fixed.

\subsection{Almost holomorphic extensions using an embedding}

Let $M$ be a compact Riemann surface equipped with a conformal Riemannian metric (\ref{eq_metric0}), (\ref{eq_metric1}). Let $\Gamma \subseteq M$ be a compact connected real submanifold of real dimension one, i.e. a simple closed curve, and let $\gamma:\mathbb{T}_\lambda \to M$ be a unit speed parametrization of $\Gamma$. Here $\mathbb{T}_\lambda =\mathbb{R}/\lambda \mathbb{ Z}$, for a suitable $\lambda >0$. The purpose of this subsection is to present an alternative approach to the construction of an almost holomorphic extension of $\gamma$, using the fact that each open Riemann surface embeds in $\mathbb C^3$. 

\medskip
\noindent
Let $\Omega \subseteq M$ be an open neighborhood of $\Gamma$ in $M$. The open Riemann surface $\Omega$ is a Stein manifold, see~\cite[Section 2.2]{Fo11}, and thus there exists a proper holomorphic embedding $f:\Omega \to \mathbb C^3$, see~\cite[Theorem 5.3.9]{Horm_CASV}. 
By \cite{DoGr60}, see also \cite[Theorem 8, p. 257]{GuRo65}, there exists an open neighborhood $\widetilde{\Omega}\subseteq \mathbb C^3$ of the closed complex submanifold $f(\Omega) \subseteq \mathbb C^3$ and a holomorphic retraction 
\begin{equation}\label{3.tb0.5}
\rho : \widetilde{\Omega} \to f(\Omega), 
\end{equation}
i.e. a holomorphic map which is the identity on $f(\Omega)$. Setting $\Sigma_\alpha = \RR+i(-\alpha,\alpha)$, $\alpha > 0$, let 
\begin{equation}
\label{3.tb0.6.0}
\widetilde{g}: \Sigma_\alpha \to \mathbb C^3
\end{equation}
be an almost holomorphic extension of $g = f\circ \gamma: \mathbb T_{\lambda} \rightarrow \mathbb C^3$, see~\cite{Ho69,MeSj74}. We have therefore  
\begin{equation}\label{3.tb0.6}
\widetilde{g}|_{\Im w =0} = g, \quad \partial_{\overline{w}}\widetilde{g} = \mO(|\Im w|^\infty),\quad w=x+iy\in \Sigma_\alpha. 
\end{equation}
Let us recall from~\cite{Ho69} that $\widetilde{g}$ may be obtained by adapting the Borel construction, setting 
\begin{equation*}
\widetilde{g}(x+iy) =  \sum_{n = 0}^{\infty} g^{(n)}(x) \frac{(iy)^n}{n!} \chi(t_n y).
\end{equation*}
Here $\chi\in C^\infty_0(\RR;[0,1])$ is such that $\chi =1$ near $0$, and $t_n\to\infty$ sufficiently rapidly. This construction allows us to assume that $\widetilde{g}$ is $\lambda$-periodic with respect to the $x$ variable. Let the group $\lambda \ZZ$ act on $\Sigma_\alpha$ by translation by integer multiples of $\lambda$ of the real part of an element in $\Sigma_\alpha$. Then $\widetilde{\mathbb{ T}}:=\Sigma_\alpha/\lambda \ZZ = \mathbb{T}_{\lambda} + i(-\alpha,\alpha)$ is a complex manifold of complex dimension one with the natural quotient complex structure, and $\mathbb{ T}_\lambda$ is a real smooth submanifold. Using the natural projection $\Sigma_\alpha \to \Sigma_\alpha/\lambda\ZZ$, 
we may therefore view \eqref{3.tb0.6.0} as a smooth function 
\begin{equation*}
\widetilde{g}: \widetilde{\mathbb{ T}} \to \mathbb C^3, 
\end{equation*}
such that \eqref{3.tb0.6} holds with $w=x+iy\in \widetilde{\mathbb{ T}}$. Since $g(\mathbb{T}_{\lambda}) = (f\circ \gamma) (\mathbb{T}_{\lambda})$ is a compact subset of $f(\Omega)\subseteq \widetilde{\Omega}$, it follows that for $\alpha>0$ small enough, $\widetilde{g}(\widetilde{\mathbb{ T}}) \subseteq \widetilde{\Omega}$. Setting 
\begin{equation}\label{3.tb0.6.1}
\widetilde{\gamma} = f^{-1}\circ \rho \circ \widetilde{g}: \widetilde{\mathbb{ T}} \to M,
\end{equation}
and using that the embedding $f$ and the retraction $\rho$ in \eqref{3.tb0.5} are holomorphic, we obtain that $\widetilde{\gamma}$ is an almost holomorphic extension of $\gamma$, 
\begin{equation}\label{3.tb0.7}
\widetilde{\gamma}|_{\mathbb{ T}_\lambda} = \gamma, \quad \text{ and }\quad \overline{\partial}\widetilde{\gamma}(w) = \mO(|\Im w|^\infty),\quad w\in \widetilde{\mathbb T}. 
\end{equation}
Here the second equation in (\ref{3.tb0.7}) is understood introducing local holomorphic coordinates near $w \in \widetilde{\mathbb T}$ and $\widetilde{\gamma}(w) \in M$, and viewing $\widetilde{\gamma}$ locally as a map: $\mathbb C \rightarrow \mathbb C$. We may also remark that when viewed more invariantly, $\overline{\partial} \widetilde{\gamma}(w)$ is a complex linear map: $T^{0,1}_w \widetilde{\mathbb T} \rightarrow T^{1,0}_{\widetilde{\gamma}(w)}M$ and there is a corresponding interpretation of (\ref{3.tb0.7}).

\section{Off-diagonal decay of Bergman kernels: Proof of Proposition \ref{prop:off-diag}} 
\label{sec:Prop3.9} 

\medskip
\noindent
The purpose of this appendix is to provide a proof of Proposition \ref{prop:off-diag}. Our starting point is the following essentially well known estimate of Agmon type, see \cite{De98}, \cite{Li01}, \cite{Br09}. 

\begin{prop}
\label{min_norm}
Let $(M,g)$ be a compact Riemann surface equipped with a conformal Riemannian metric and let $\varphi \in C^{\infty}(M; \mathbb R)$ be non-constant. Let $W \subseteq M$ be open such that $W \Subset \{x\in M; \Delta \varphi(x) >0\}$ and let 
$\beta \in  L^2_{\varphi}(W, T^*_{0,1}W) = e^{\varphi/h}L^2(W, T^*_{0,1}W)$. For each $\chi \in C^{\infty}_0(W;[0,1])$ there exist $0 < \delta < 1$ and $h_0 >0$ such that for all $h\in (0,h_0]$ the minimal norm solution $u\in L^2_{\varphi}(W,\omega)$ to the equation $h\overline{\partial}u = \beta$ satisfies 
\begeq
\label{eq_appB0}
\|u\|_{L^2_{\psi}(W,\omega)} \leq \frac{\mathcal O (1)}{h^{1/2}} \|\beta\|_{L^2_{\psi}(W,T^*_{0,1}W)}, \quad \psi = \varphi + \delta h^{1/2} \chi. 
\endeq
Here $\omega$ is the area form associated to the metric $g$, given in {\rm (\ref{eq_metric})}. 
\end{prop}
\begin{proof}
Proceeding as in the proof of~\cite[Theorem 2.1]{Br09}, let us write 
\begin{multline}
\label{eq_appB1}
\|u\|_{L^2_{\psi}(W,\omega)}^2 = \int |u(x)|^2 e^{-2\psi(x)/h}\, \omega(x, dx d\overline{x}) = \int |u(x)|^2 e^{-2(\varphi(x) + \delta h^{1/2} \chi(x))/h}\, \omega(x, dx d\overline{x}) \\ 
=  \int |v(x)|^2 e^{-2\widetilde{\psi}(x)/h}\, \omega(x, dx d\overline{x}) = \|v\|_{L^2_{\widetilde{\psi}}(W,\omega)}^2, \quad \widetilde{\psi} = \varphi - \delta h^{1/2} \chi,
\end{multline}
where 
\begeq
\label{eq_appB2}
v = u e^{-2\delta h^{1/2}\chi/h}.
\endeq
Expressing the orthogonality condition 
\[
(u,f)_{L^2_{\varphi}(W,\omega)} = 0, \quad f \in H_{\varphi}(W) = {\rm Hol}(W) \cap L^2_{\varphi}(W,\omega), 
\] 
in terms of $v$, we get 
\begeq
\label{eq_appB3} 
(v,f)_{L^2_{\widetilde{\psi}}(W,\omega)} = 0, \quad f \in H_{\widetilde{\psi}}(W) = {\rm Hol}(W) \cap L^2_{\widetilde{\psi}}(W,\omega) = H_{\varphi}(W),
\endeq 
and therefore $v \in L^2_{\widetilde{\psi}}(W,\omega)$ is the minimal norm solution of the $\overline{\partial}$--problem $h\overline{\partial} w = \alpha$, where $\alpha = h\overline{\partial} v$. Here
\[
\Delta \widetilde{\psi} = \Delta \varphi - \delta h^{1/2} \Delta \chi \geq \alpha - \mathcal O_{\chi}(h^{1/2}) \geq \frac{1}{2}\alpha, \quad \alpha > 0, 
\] 
for all $h >0$ small enough and all $\delta \in (0,1)$. An application of Proposition \ref{exist_dbar} gives therefore for all $h>0$ small enough and all $\delta \in (0,1)$, 
\begeq
\label{eq_appB4} 
\|v\|_{L^2_{\widetilde{\psi}}(W,\omega)} \leq \frac{\mathcal O (1)}{h^{1/2}} \|h\overline{\partial} v\|_{L^2_{\widetilde{\psi}}(W,T^*_{0,1}W)}.
\endeq
The left hand side in (\ref{eq_appB4}) equals $\|u\|_{L^2_{\psi}(W,\omega)}$, in view of (\ref{eq_appB1}), and when estimating the right hand side we write using (\ref{eq_appB2}), 
\begeq
\label{eq_appB5} 
h\overline{\partial} v = e^{-2\delta h^{1/2}\chi/h} \left(\beta - 2\delta h^{1/2} u \overline{\partial}\chi\right). 
\endeq
We get, combining (\ref{eq_appB4}) and (\ref{eq_appB5}), 
\begin{multline} 
\label{eq_appB6}
\|u\|_{L^2_{\psi}(W,\omega)} \\ \leq \frac{\mathcal O (1)}{h^{1/2}} \|\beta - 2\delta h^{1/2} u \overline{\partial}\chi\|_{L^2_{\psi}(W,T^*_{0,1}W)} 
\leq \frac{\mathcal O (1)}{h^{1/2}} \|\beta\|_{L^2_{\psi}(W,T^*_{0,1}W)} + \mathcal O_{\chi}(\delta) \|u\|_{L^2_{\psi}(W,\omega)}. 
\end{multline} 
Taking $\delta >0$ sufficiently small depending on $\chi$, we may absorb the second term in the right hand side of (\ref{eq_appB6}) into the left hand side. The estimate (\ref{eq_appB0}) follows and the proof is complete. 
\end{proof} 

\bigskip
\noindent
We shall now prove Proposition \ref{prop:off-diag}, proceeding along the lines of~\cite[Section 4]{Li01}. Our starting point is the following observation: for each $V \Subset W$ there exist $C>0$ and $h_0 > 0$ such that for all $h\in (0,h_0]$ and all $u\in H_{\varphi}(W)$, we have 
\begeq
\label{eq_appB7}
|u(y)|^2 e^{-2\varphi(y)/h} \leq \frac{C}{h^2} \|u\|^2_{L^2_{\varphi}(U(y,h),\omega)}, \quad y\in V. 
\endeq 
Here $U(y,h) \subseteq W$ is a neighborhood of $y$ of diameter of size $h$. Indeed, to obtain (\ref{eq_appB7}) we first fix a point $y\in V$ and choose a local holomorphic coordinate $z$ near $y$ such that $z(y) = 0$. The estimate (\ref{eq_appB7}) then follows by applying the mean value property for holomorphic functions in a disc $D(0,h) \subseteq \mathbb C$ and the Cauchy-Schwarz inequality. The uniformity of (\ref{eq_appB7}) follows from the compactness of $\overline{V}$. 

\medskip
\noindent
Let $C_j \subseteq W$, $j=1,2$, be disjoint compact sets and let $C_j \subseteq V_j \Subset W$ be open such that $\overline{V_1} \cap \overline{V_2} = \emptyset$. We get, applying (\ref{eq_appB7}) to the function $W \ni y \mapsto \overline{K(x,y)} \in H_{\varphi}(W)$ in (\ref{eq3.1.0.5}), for each $x\in W$, 
\begeq
\label{eq_appB8} 
|K(x,y)|^2 e^{-2(\varphi(y) + \varphi(x))/h} \leq \frac{C}{h^2} \int_{U(y,h)} |K(x,z)|^2 e^{-2(\varphi(z) + \varphi(x))/h}\, \omega(z,dz\,d\overline{z}),\quad y\in C_1.
\endeq
Letting $\chi_0 \in C^{\infty}_0(W;[0,1])$ be such that $\chi_0 = 1$ in a neighborhood of $\overline{V_1}$ and ${\rm supp}\, \chi_0 \cap\, \overline{V_2} = \emptyset$, we get using (\ref{eq_appB8}) for all $h>0$ small enough and all $x\in W$, $y\in C_1$, 
\begin{multline} 
\label{eq_appB9} 
|K(x,y)|^2 e^{-2(\varphi(y) + \varphi(x))/h} \leq \frac{C}{h^2} \int_{W} \chi_0(z) |K(x,z)|^2 e^{-2(\varphi(z) + \varphi(x))/h}\, \omega(z,dz\,d\overline{z}) \\
= \frac{C}{h^2} \int_{W} K(x,z) \chi_0(z) \overline{K(x,z)} e^{-2(\varphi(z) + \varphi(x))/h}\, \omega(z,dz\,d\overline{z}) \\ 
= \frac{C}{h^2} \left(\int_{W} K(x,z) \chi_0(z) K(z,x) e^{-\varphi(x)/h} e^{-2\varphi(z)/h}\, \omega(z,dz\,d\overline{z})\right) e^{-\varphi(x)/h} \\
= \frac{C}{h^2} \left|\Pi\left(\chi_0 K(\cdot, x) e^{-\varphi(x)/h}\right)(x)\right| e^{-\varphi(x)/h},\quad x\in W,\,\, y \in C_1. 
\end{multline}
Here $\Pi$ is the orthogonal projection introduced in (\ref{eq3.1.0.5}) and in the penultimate equality we have used that $\overline{K(x,z)} = K(z,x)$. 

\medskip
\noindent
Let 
\begeq
\label{eq_appB10} 
u(\zeta) = \chi_0(\zeta) K(\zeta,x) e^{-\varphi(x)/h} - \Pi\left(\chi_0 K(\cdot, x) e^{-\varphi(x)/h}\right)(\zeta),\quad \zeta \in W,
\endeq
be the solution of the $\overline{\partial}$--problem 
\begeq
\label{eq_appB10.1}
h \overline{\partial} u = h \overline{\partial} \left(\chi_0 K(\cdot,x) e^{-\varphi(x)/h}\right) = K(\cdot,x) e^{-\varphi(x)/h} h \overline{\partial} \chi_0,  
\endeq
having minimal norm in $L^2_{\varphi}(W,\omega)$. We have 
\begeq
\label{eq_appB11}
|u(x)| = |\Pi\left(\chi_0 K(\cdot, x) e^{-\varphi(x)/h}\right)(x)|, \quad x\in C_2, 
\endeq
since $\chi_0$ vanishes on $C_2$, and we get in view of (\ref{eq_appB9}), (\ref{eq_appB11}), 
\begeq
\label{eq_appB12} 
|K(x,y)|^2 e^{-2(\varphi(y) + \varphi(x))/h} \leq \frac{C}{h^2} |u(x)| e^{-\varphi(x)/h},\quad x\in C_2,\,\, y \in C_1.
\endeq
Now it follows from (\ref{eq_appB10.1}) that $u\in {\rm Hol}(V_2)$, and we get as in (\ref{eq_appB7}), for all $h>0$ small enough, 
\begeq
\label{eq_appB13} 
|u(x)| e^{-\varphi(x)/h} \leq \frac{C}{h} \|u\|_{L^2_{\varphi}(U(x,h),\omega)} \leq \frac{C}{h} \|u\|_{L^2_{\varphi}(V_2,\omega)},  \quad x\in C_2.
\endeq
We infer, in view of (\ref{eq_appB12}) and (\ref{eq_appB13}), that 
\begeq
\label{eq_appB14}
|K(x,y)|^2 e^{-2(\varphi(y) + \varphi(x))/h} \leq \frac{C}{h^3} \|u\|_{L^2_{\varphi}(V_2,\omega)}, \quad x\in C_2,\,\, y \in C_1. 
\endeq 
Let $\chi \in C^{\infty}_0(W;[0,1])$ be such that $\chi = 1$ in a neighborhood of ${\rm supp}\, \chi_0$, ${\rm supp}\, \chi \cap \overline{V_2} = \emptyset$. Applying (\ref{eq_appB0}) with $\psi = \varphi + \delta h^{1/2} \chi$ and using (\ref{eq_appB10.1}) we obtain that  
\begin{multline}
\label{eq_appB15} 
\|u\|^2_{L^2_{\varphi}(V_2,\omega)} = \|u\|^2_{L^2_{\psi}(V_2,\omega)} \leq \frac{\mathcal O(1)}{h} \|K(\cdot,x)e^{-\varphi(x)/h} h\overline{\partial} \chi_0\|^2_{L^2_{\psi}(W,T^*_{0,1}W)}, \\ 
\leq \mathcal O(1) \|K(\cdot,x)e^{-\varphi(x)/h}\|^2_{L^2_{\psi}({\rm supp}\, \overline{\partial} \chi_0,\omega)} \\ 
\leq \mathcal O(1) e^{-\frac{2\delta}{h^{1/2}}}\, 
\int_{{\rm supp}\, \overline{\partial} \chi_0} |K(z,x)|^2 e^{-2(\varphi(z) + \varphi(x))/h}\, \omega(z,dzd\overline{z}).
\end{multline}
Here the general arguments for Bergman kernels, see~\cite[Chapter 4]{Br10},~\cite{Li01}, show that 
\begeq
\label{eq_appB16}
|K(z,x)|^2 e^{-2(\varphi(z) + \varphi(x))/h} \leq \left(K(z,z) e^{-2\varphi(z)/h}\right) \left(K(x,x) e^{-2\varphi(x)/h}\right) \leq \frac{\mathcal O(1)}{h^2}, 
\endeq
uniformly on compact subsets of $W \times W$. Combining (\ref{eq_appB14}), (\ref{eq_appB15}), and (\ref{eq_appB16}), we complete the proof of Proposition \ref{prop:off-diag}.

\section{Singular values for large exponential decay rates}
\label{sec:app_large_decay}

\medskip
\noindent
Let $(M,g)$ be a compact Riemann surface equipped with a conformal Riemannian metric, see (\ref{eq_metric0}), and let $\varphi \in C^{\infty}(M;\mathbb R)$ be non-constant. The purpose of this appendix is to give a proof of Theorem \ref{decay_rate_large}. See also \cite[Section 2.2]{HeKaSu24}, \cite[Appendix]{HeKa24} for related arguments and results. 

\medskip
\noindent
When proving Theorem \ref{decay_rate_large} we let 
\begeq
\label{appA1}
\tau > \underset{M}{\rm max}\, \varphi - \underset{M}{\rm min}\, \varphi, 
\endeq 
and without loss of generality we may assume  that
\begin{equation}
\label{eq:appA2}
\min_M \varphi = 0.
\end{equation}
Let $u\in L^2_\varphi(M,\omega)$ be a normalized singular state of the operator $h\overline{\partial}:L^2_\varphi(M,\omega)\to L^2_\varphi(M,T^*_{0,1}M)$, associated to a singular value  $s \in [0, e^{-\tau/h}]$. Thus, we have
\begin{equation}
\label{eq:appA3}
\|u\|_{L^2_\varphi(M,\omega)} =1,
\quad
\|h\overline{\partial}u\|_{L^2_\varphi(M,T^*_{0,1}M)} = s \leq e^{-\tau/h}.
\end{equation}
The following proposition shows that, as a normalized element of $L^2_\varphi(M,\omega)$, the singular state $u$ is exponentially small away from any fixed neighborhood of the set $\varphi^{-1}(0)$. While not used directly in the proof of Theorem \ref{decay_rate_large}, including it here seems natural as it may be of some independent interest. 

\medskip
\noindent
\begin{prop}
Let $u\in L^2_{\varphi}(M,\omega)$ be such that {\rm (\ref{eq:appA3})} holds, for some $\tau > 0$ satisfying {\rm (\ref{appA1})}, {\rm (\ref{eq:appA2})}. Then for each $\eta > 0$ there exists $C_{\eta} > 0$ such that
\begeq
\label{eq:appA3.1}
\|u\|_{L^2(M,\omega)} \leq C_{\eta}\, e^{\eta/h}.
\endeq
\end{prop} 
\begin{proof}
Let $0< \varepsilon < \underset{M}{\rm max}\, \varphi$ be a regular value of $\varphi$ and let us define
\begeq
\label{eq:appA4}
U_{\varepsilon} = \{x\in M; \varphi(x) < \varepsilon\}.
\endeq
Then $\partial U_{\varepsilon} = \varphi^{-1}(\varepsilon)$ is smooth. Let $\Omega_{\varepsilon} = M\setminus \overline{U}_{\varepsilon}$ and let $\psi_{\varepsilon} \in C^{\infty}(\overline{\Omega}_{\varepsilon})$ be the solution of the boundary value problem
\begin{equation}
\label{eq:appA5}
\begin{cases}
-\Delta \psi_{\varepsilon} = 1 \hbox{ in }\Omega_{\varepsilon},\\
\psi_{\varepsilon}|_{\partial \Omega_{\varepsilon}} = 0.
\end{cases}
\end{equation}
In particular, $\psi_{\varepsilon}$ is strictly superharmonic in $\Omega_{\varepsilon}$ and $\psi_{\varepsilon} > 0$ in $\Omega_{\varepsilon}$, by the maximum principle. Define $\psi_{\delta, \varepsilon} = \delta \psi_{\varepsilon}$, for $0 < \delta < 1$, so that
\[
-\Delta \psi_{\delta,\varepsilon} = \delta > 0. 
\]
We get, arguing as in the proof of Proposition \ref{propCH}, cf. (\ref{eq1.19}),
\begeq
\label{eq:appA6}
\left(\frac{h\delta}{2}\right)^{1/2} \|e^{-\psi_{\delta,\varepsilon}/h} u\|_{L^2(\Omega_{\varepsilon},\omega)} \leq
\|e^{-\psi_{\delta,\varepsilon}/h} h\overline{\partial} u\|_{L^2(\Omega_{\varepsilon}, T^*_{0,1} \Omega_{\varepsilon})}, \quad u\in C^{\infty}_0(\Omega_{\varepsilon}).
\endeq
If $\Theta \subseteq \overline{\Omega}_{\varepsilon}$ is an open neighborhood of $\partial \Omega_{\varepsilon}$ in $\overline{\Omega}_{\varepsilon}$, we also get, continuing to follow the proof of Proposition \ref{propCH},
\begin{multline}
\label{eq:appA7}
\|e^{-\psi_{\delta,\varepsilon}/h} u\|_{L^2(\Omega_{\varepsilon},\omega)}
\leq \left(\frac{2}{h \delta}\right)^{1/2}\|e^{-\psi_{\delta,\varepsilon}/h} h\overline{\partial} u\|_{L^2(\Omega_{\varepsilon}, T^*_{0,1} \Omega_{\varepsilon})} \\
+ \left(\frac{C(\Theta,\varepsilon)}{\delta^{1/2}} + 1\right)\|e^{-\psi_{\delta,\varepsilon}/h} u\|_{L^2(\Theta, \omega)}, \quad u \in C^{\infty}(M).
\end{multline}
Here $C(\Theta,\varepsilon) > 0$ is a constant depending on $\Theta$ and $\varepsilon$. We shall apply the Carleman estimate (\ref{eq:appA7}) to the singular state $u$ in (\ref{eq:appA3}). When doing so, we observe first that
\begin{multline}
\label{eq:appA8}
\|e^{-\psi_{\delta,\varepsilon}/h} h\overline{\partial} u\|_{L^2(\Omega_{\varepsilon}, T^*_{0,1}\Omega_{\varepsilon})}\leq
\|h\overline{\partial} u\|_{L^2(\Omega_{\varepsilon},T^*_{0,1}\Omega_{\varepsilon})} \leq \|h\overline{\partial} u\|_{L^2(M,T^*_{0,1} M)}\\
\leq e^{\max_M \varphi/h} \|e^{-\varphi/h} h\overline{\partial} u\|_{L^2(M, T^*_{0,1} M)}
\leq  e^{(\max_M \varphi -\tau)/h} \leq e^{-1/Ch},\quad C > 0.
\end{multline} 
Here we have used (\ref{appA1}), (\ref{eq:appA2}), and (\ref{eq:appA3}). Next, we have using (\ref{eq:appA3}),
\begeq
\label{eq:appA9}
\|e^{-\psi_{\delta,\varepsilon}/h} u\|_{L^2(\Theta, \omega)}\leq \|u\|_{L^2(\Theta, \omega)} \leq
e^{\max_{\Theta} \varphi/h}\|e^{-\varphi/h}u\|_{L^2(\Theta, \omega)} \leq e^{\max_{\Theta} \varphi/h}.
\endeq
Recalling that $\varphi = \varepsilon$ along $\partial \Omega_{\varepsilon}$, we see that we can choose $\Theta = \Theta_{\varepsilon}$ so that
$\max_{\Theta} \varphi \leq 2\varepsilon$. 

\medskip
\noindent
We obtain using (\ref{eq:appA7}), (\ref{eq:appA8}), and (\ref{eq:appA9}),
\begeq
\label{eq:appA10}
\|e^{-\psi_{\delta,\varepsilon}/h} u\|_{L^2(\Omega_{\varepsilon},\omega)} \leq \frac{C(\varepsilon)}{\delta^{1/2}} e^{2\varepsilon/h},
\endeq
for some (new) constant $C(\varepsilon) > 1$ depending on $\varepsilon > 0$. Combining (\ref{eq:appA10}) with the estimate
\begeq
\label{eq:appA11}
\|u\|_{L^2(U_{\varepsilon},\omega)} \leq e^{\varepsilon/h} \|u\|_{L^2_{\varphi}(M,\omega)} = e^{\varepsilon/h},
\endeq
we get
\begeq
\label{eq:appA12}
\|u\|_{L^2(M,\omega)} \leq \frac{C(\varepsilon)}{\delta} e^{(\delta \|\psi_{\varepsilon}\|_{L^{\infty}(\Omega_{\varepsilon})} + 2\varepsilon)/h}.
\endeq
Taking $\delta > 0$ small enough depending on $\varepsilon > 0$, we conclude therefore that for each regular value $0 < \varepsilon < \underset{M}{\rm max}\, \varphi$ of $\varphi$ there exists $C(\varepsilon) > 0$ such that
\begeq
\label{eq:appA13}
\|u\|_{L^2(M,\omega)} \leq C(\varepsilon)\, e^{3\varepsilon /h}.
\endeq
Taking a sequence of regular values $0 < \varepsilon_j \rightarrow 0$, by Sard's theorem, we obtain (\ref{eq:appA3.1}).
\end{proof}

\medskip
\noindent
Let $u\in L^2_{\varphi}(M,\omega)$ be a normalized singular state of $h\overline{\partial}: L^2_{\varphi}(M,\omega) \rightarrow L^2_{\varphi}(M, T^*_{0,1}M)$, associated to a singular value $s\in [0, e^{-\tau/h}]$, for some $\tau > 0$ satisfying (\ref{appA1}), (\ref{eq:appA2}). It follows from (\ref{eq:appA8}) that
\begeq
\label{eq:appA14}
\|\overline{\partial} u\|_{L^2(M, T^*_{0,1} M)} \leq \mathcal O(1)\, e^{-1/Ch}, \quad C> 0,
\endeq
and using the representation (\ref{eq1.4.0.5}) for the Laplacian $\Delta$ on $M$ we get
\begeq
\label{eq:appA15}
(-\Delta u, u)_{L^2(M,\omega)} \leq \mathcal O(1)\, e^{-2/Ch}.
\endeq

\medskip
\noindent
Let
\begin{equation*}
0 = \lambda_0 < \lambda_1 \leq \lambda_2 \leq \dots \to +\infty
\end{equation*}
stand for the increasing sequence of eigenvalues of $-\Delta$, repeated according to their multiplicity, and let $\psi_0, \psi_1, \psi_2, \ldots$ be an orthonormal basis of $L^2(M,\omega)$ composed of associated real valued eigenfunctions of $-\Delta$. We shall take
\begeq
\label{eq:appA16}
\psi_0 = \frac{1}{{\rm vol}(M)^{1/2}},
\endeq
see also (\ref{eq:fbp0}).

\medskip
\noindent
Writing
\begeq
\label{eq:appA17}
u = \sum_{j=0}^{\infty} c_j \psi_j, \quad c_j = (u,\psi_j)_{L^2(M,\omega)},
\endeq
where the series converges in $C^{\infty}(M)$, we obtain in view of (\ref{eq:appA15}),
\begeq
\label{eq:appA18}
\sum_{j=1}^{\infty} \lambda_j \abs{c_j}^2 \leq \mathcal O(1)\, e^{-2/Ch}.
\endeq
If we let
\begeq
\label{eq:appA19}
u^\perp = \sum_{j=1}^{\infty} c_j \psi_j,
\endeq
then (\ref{eq:appA17}) gives an orthogonal decomposition in $L^2(M,\omega)$,
\begin{equation}
\label{eq:appA20}
u = c_0 \psi_0 + u^\perp,
\end{equation}
and using (\ref{eq:appA18}) we obtain that
\begin{equation}
\label{eq:appA21}
\| u^\perp\|_{H^1(M)} \leq \mathcal{O}(1)\, {e}^{-1/Ch}.
\end{equation}
The normalization condition (\ref{eq:appA3}), together with (\ref{eq:appA20}), give that
\begeq
\label{eq:appA22}
1 = \|u\|_{L^2_{\varphi}(M,\omega)}^2 = \abs{c_0}^2 \|\psi_0\|^2_{L^2_{\varphi}(M,\omega)} + 2{\rm Re}\, (c_0 \psi_0, u^{\perp})_{L^2_{\varphi}(M,\omega)} + \|u^{\perp}\|^2_{L^2_{\varphi}(M,\omega)}.
\endeq
Here
\begeq
\label{eq:appA23}
\|u^{\perp}\|_{L^2_{\varphi}(M,\omega)} \leq \|u^{\perp}\|_{L^2(M,\omega)} \leq \mathcal{O}(1)\, {e}^{-1/Ch},
\endeq
in view of (\ref{eq:appA21}), and we get using (\ref{eq:appA22}), (\ref{eq:appA23}),
\begeq
\label{eq:appA24}
1 = \abs{c_0}^2 \|\psi_0\|^2_{L^2_{\varphi}(M,\omega)} + \mathcal O(1)\, e^{-1/Ch} \abs{c_0} \|\psi_0\|_{L^2_{\varphi}(M,\omega)}  + \mathcal{O}(1)\, {e}^{-2/Ch}.
\endeq
It follows that
\begeq
\label{eq:appA25}
\abs{c_0} \|\psi_0\|_{L^2_{\varphi}(M,\omega)} = 1 + \mathcal O(1)\, e^{-1/Ch}.
\endeq
Here we have, recalling (\ref{eq:appA16}),
\begeq
\label{eq:appA26}
\|\psi_0\|_{L^2_{\varphi}(M,\omega)} = \frac{1}{{\rm vol}(M)^{1/2}} \|1\|_{L^2_{\varphi}(M,\omega)} \in [\alpha h^{1/2}, 1], \quad \alpha > 0,
\endeq
using that
\[
\|1\|^2_{L^2_{\varphi}(M,\omega)} = \int_M e^{-2\varphi(z)/h}\, \omega(z,dz\,d\overline{z}) \leq {\rm vol}(M),
\]
and
\[
\|1\|^2_{L^2_{\varphi}(M,\omega)} = \int_M e^{-2\varphi(z)/h}\, \omega(z, dz\, d\overline{z} ) \geq 
\int_{B(z_0,h^{1/2})} e^{-2\varphi(z)/h}\, \omega(z,dz\,d\overline{z}) \geq \frac{h}{C}.
\]
Here $z_0 \in \varphi^{-1}(0)$, $B(z_0, h^{1/2})$ is a small neighborhood of $z_0$ of diameter of size $h^{1/2}$, and we have also used that $\varphi$ vanishes to the second order at $z_0$. 

\bigskip
\noindent
Next, the smallest singular value of $h\overline{\partial}:L^2_\varphi(M,\omega)\to L^2_\varphi(M, T^*_{0,1} M)$ is $0$, and the corresponding singular state $u_0$, normalized in $L^2_\varphi(M,\omega)$, is a constant such that
\begeq
\label{eq:appA27}
\abs{u_0} = \frac{1}{\|1\|_{L^2_{\varphi}(M,\omega)}} \in \left[{\rm vol}(M)^{-1/2}, \mathcal O(h^{-1/2})\right],
\endeq
in view of (\ref{eq:appA26}). Assuming that $u_0$ is a real constant $>0$, we get using (\ref{eq:appA16}), (\ref{eq:appA26}), and (\ref{eq:appA27}),
\begeq
\label{eq:appA27.1}
u_0 = \frac{\psi_0}{\|\psi_0\|_{L^2_{\varphi}(M,\omega)}}.
\endeq
If the singular value $s$ in (\ref{eq:appA3}) is non-vanishing, then we have
\begin{equation}
\label{eq:appA28}
(u, u_0)_{L^2_\varphi(M,\omega)}=0,
\end{equation}
and recalling (\ref{eq:appA20}), we get
\begin{multline}
\label{eq:appA29}
0 = c_0(\psi_0,u_0)_{L^2_{\varphi}(M,\omega)} + (u^{\perp},u_0)_{L^2_{\varphi}(M,\omega)} = c_0(\psi_0,u_0)_{L^2_{\varphi}(M,\omega)} + \mathcal O(1)\, e^{-1/Ch} \\
= c_0 \|\psi_0\|_{L^2_{\varphi}(M,\omega)} + \mathcal O(1)\, e^{-1/Ch}.
\end{multline}
Here we have also used (\ref{eq:appA23}), (\ref{eq:appA27.1}). The right hand side of (\ref{eq:appA29}) is non-vanishing for all $h>0$ small enough, in view of (\ref{eq:appA25}), giving a contradiction. The proof of Theorem \ref{decay_rate_large} is complete.

\section{Sobolev and H\"older spaces on manifolds}
\label{sec:appHoSoSpace}

\noindent
We recall some notions and results for Sobolev and H\"older spaces on compact manifolds. In this section we let $s\in \RR$ and $1< p < + \infty$, unless stated otherwise.
\subsection{Sobolev and Hölder spaces on Euclidean space}
First, we work on $\mathbb R^d$. Let $m\in\RR$ and let $\Psi^m(\RR^d)$ denote the class of pseudodifferential operators on $\RR^d$ whose symbols are in the class $S^m(\RR^{2d})$, i.e. the class of smooth functions $a$ on $\RR^{2d}$ such that for all $\alpha,\beta\in\NN^d$
\begin{equation*}
\sup_{x,\xi\in\RR^d}|\langle \xi\rangle^{|\alpha|-m} \partial_\xi^\alpha
\partial_x^\beta a(x,\xi)| < +\infty.
\end{equation*}
Let $P\in \Psi^0(\mathbb R^d)$, then, see for instance \cite[Section 5, Proposition 4]{St93},
\begin{equation}\label{eq:PseudoMapping1}
    P:L^{p}(\RR^d)\to L^{p}(\RR^d)
\end{equation}
is a linear continuous operator.

\medskip
\noindent
Next, we recall that the Sobolev space $W^{s,p}(\RR^d):=(1-\Delta)^{-s/2}L^p(\RR^d)$ equipped with the norm
\begin{equation}\label{eq:SobNorm1}
  \|u\|_{W^{s,p}(\RR^d)} := \|(1-\Delta)^{s/2}u\|_{L^p(\RR^d)},
\end{equation}
is a Banach space. Notice that when $s$ is a positive integer, the space $W^{s,p}(\RR^d)$ coincides with the space of functions $f\in L^p(\RR^d)$
such that $\partial^\alpha f$, $|\alpha| \leq s$, taken in the sense of distributions, belong to $L^p(\mathbb R^d)$. Moreover, \eqref{eq:PseudoMapping1} shows that the norm \eqref{eq:SobNorm1}
is equivalent to the norm
\begin{equation}\label{eq:SobNorm2}
  \|u\|_{W^{s,p}(\RR^d)} = \sum_{|\alpha|\leq s} \|\partial^\alpha u\|_{L^p(\RR^d)}.
\end{equation}

\medskip
\noindent
{\it Remark}. When $p=\infty$ and $s\in \NN$, then we will keep the second definition of the Sobolev spaces $W^{s,\infty}(\RR^d)$ equipped with the norm \eqref{eq:SobNorm2}.

\medskip
\noindent
Going back to the case of $s\in \RR$ and $1< p < + \infty$, we get, using the pairing induced by the standard $L^2$ scalar product, that
these spaces satisfy the duality property
\begin{equation}\label{eq:daulity}
  (W^{s,p}(\RR^d))^* = W^{-s,p'}(\RR^d), \quad p^{-1}+(p')^{-1} =1.
\end{equation}
Let $P\in \Psi^m(\RR^d)$, for some $m\in \mathbb R$. Standard pseudodifferential calculus
and \eqref{eq:PseudoMapping1} then show that
\begin{equation}\label{eq:PseudoMapping}
    P:W^{s,p}(\RR^d)\to W^{s-m,p}(\RR^d)
\end{equation}
is a linear continuous operator.

\medskip
\noindent
For $k\in\NN$, we recall that the space $C^{k}(\overline{\RR^d})$ is the space of $C^k$ functions $f$ whose partial derivatives $\partial^\alpha f$, $|\alpha|\leq k$, are bounded and uniformly continuous on $\RR^d$. For $0< \lambda \leq 1$,
the space $C^{k,\lambda}(\overline{\RR^d})$ is the subspace of $C^{k}(\overline{\RR^d})$ consisting of those functions $f$ for which
$\partial^\alpha f$, for all $|\alpha|=k$, are uniformly Hölder continuous on $\RR^d$ with exponent $\lambda$, i.e.
\begin{equation*}
  |\partial^\alpha f(x)-\partial^\alpha f(x)| \leq \mO(|x-y|^\lambda),
  \quad \text{for all } x,y\in \RR^d.
\end{equation*}
When equipped with the norm
\begin{equation*}
  \|u\|_{C^{k,\lambda}(\overline{\RR^d})} =
    \sup_{|\alpha|\leq k}\sup_{x\in\RR^d}|\partial^\alpha u|
    +
    \sup_{|\alpha|= k}\sup_{\substack{x,y\in\RR^d \\ x\neq y}}
    \frac{|\partial^\alpha u(x) - \partial^\alpha u(y)| }{|x-y|^\lambda},
\end{equation*}
the space $C^{k,\lambda}(\overline{\RR^d})$ becomes a Banach space.

\medskip
\noindent
By the Sobolev embedding theorem, see for instance \cite[Theorem 5.4]{Ad75}, we know that for $s\in \NN$, $d<p<+\infty$ and $0< \lambda \leq 1 - d/p$ the embedding
\begin{equation}\label{eq:MorreyInequality0}
  W^{s+1,p}(\RR^d)\subset C^{s,\lambda}(\overline{\RR^d})
\end{equation}
is continuous. There exists therefore a constant $C=C_{s,p,d,\lambda}>0$ such that for
all $u\in W^{s,p}(\RR^d)$
\begin{equation}\label{eq:MorreyInequality}
    \|u\|_{C^{s,\lambda}(\overline{\RR^d})}
    \leq C \|u\|_{W^{s+1,p}(\RR^d)}.
\end{equation}
\subsection{Sobolev and Hölder spaces on compact manifolds}
Let $M$ be a compact smooth $d$-dimensional manifold without boundary and let $s\in \RR$ and $1< p < + \infty$.
We define the Sobolev spaces $W^{s,p}(M)\subset \mathcal{D}'(M)$ as the set of all distributions $u\in\mathcal{D}'(M)$ such that
\begin{equation*}
(\kappa^{-1})^*\chi u \in W^{s,p}(\RR^d),
\end{equation*}
for all local smooth coordinate charts $\kappa:M\supset U \to V\subset \RR^d$ and cut-off functions $\chi\in C^\infty_0(U)$. Since $M$ is compact, we may consider a finite open cover of $M$ by smooth coordinate charts $\{\kappa_k:M\supset U_k\to V_k\subset \RR^d\}_{k\in K}$
and a subordinate finite partition of unity $\{\chi_k\}_{k\in K}$, $\chi_k \in C^\infty_0(U_k)$. We can equip the Sobolev spaces $W^{s,p}(M)$
with the norm
%
\begin{equation}\label{eq:EqSobNorm}
	\| u \|_{W^{s,p}(M)} = \sum_{k\in K}
	\|(\kappa_k^{-1})^*\chi_k u \|_{W^{s,p}(\RR^d)},
\end{equation}
turning it into a Banach space. Notice that taking different coordinate patches and cut-off functions in \eqref{eq:EqSobNorm} gives an equivalent norm.

\medskip
\noindent
{\it Remark}. When $p=\infty$ and $s\in \NN$, we define Sobolev spaces $W^{s,\infty}(M)$ analogously, however, in \eqref{eq:EqSobNorm} we use the norm \eqref{eq:SobNorm2}.

\medskip
\noindent
We recall from~\cite[Chapter 3]{GrSj} that a linear continuous map $P: C^{\infty}(M)\to \mathcal{D}'(M)$ is a pseudodifferential operator of class $\Psi^{m}(M)$, $m\in\RR$, if and only if the following two conditions hold:
\begin{itemize}
\item $\phi P \psi:\mathcal{D}'(M)\to C^\infty(M)$ for all $\phi,\psi \in C^\infty_0(M)$ with
$\supp \phi \cap \supp \psi = \emptyset$;
\item for all local coordinate charts $\kappa: M\supset U \to V\subset \RR^d$ and cut-off functions $\chi\in C^\infty_0(U)$ there exists a
$P_{\kappa,\chi}\in \Psi^m(\RR^{d})$ such that
\begin{equation*}
\chi P\chi = \chi\kappa^*  P_{\kappa\chi}(\kappa^{-1})^*\chi.
\end{equation*}
\end{itemize}
It then follows from \eqref{eq:PseudoMapping}, \eqref{eq:EqSobNorm} that
\begin{equation}\label{eq:PseudoMapping2}
  P=\mO(1):W^{s,p}(M)\to W^{s-m,p}(M).
\end{equation}
Now fix a smooth Riemannian metric $g$ on $M$. Then $-\Delta \in \Psi^2(M)$ is a positive selfadjoint operator: $L^2(M)\to L^2(M)$
with domain $W^{2,2}(M)=H^2(M)$. By the functional calculus of selfadjoint operators we can define $(1-\Delta)^s$, $s\in\RR$, and a result due to Seeley \cite{Se67} shows that $(1-\Delta)^s\in \Psi^{2s}(M)$.
Thus \eqref{eq:PseudoMapping2} shows that
\begin{equation}\label{eq:sob3}
  (1-\Delta)^{-s/2}L^p(M) = W^{s,p}(M)
\end{equation}
and the norm $\|(1-\Delta)^{s/2}u\|_{L^p(M)}$ is equivalent to \eqref{eq:EqSobNorm}. Using the pairing induced by the $L^2$ scalar
product on $M$ with respect to the Riemannian volume form, we get an analogue of \eqref{eq:daulity},
\begin{equation}\label{eq:daulity2}
  (W^{s,p}(M))^* = W^{-s,p'}(M), \quad p^{-1}+(p')^{-1} =1.
\end{equation}

\medskip
\noindent
For $k\in \NN$ and $0 <\lambda \leq 1$ we define the Hölder space $C^{k,\lambda}(M)$ as the set of all $C^k(M)$ functions $u$ such that
\begin{equation*}
(\kappa^{-1})^*\chi u \in C^{k,\lambda}(\overline{\RR^d}),
\end{equation*}
for all smooth local coordinate charts $\kappa:M\supset U \to V\subset \RR^d$ and cut-off functions $\chi\in C^\infty_0(U)$. Since $M$ is compact, we may consider a finite open cover of $M$ by coordinate charts $\{\kappa_k:M\supset U_k\to V_k\subset \RR^d\}_{k\in K}$
and a subordinate finite partition of unity $\{\chi_k\}_{k\in K}$, $\chi_k \in C^\infty_0(U_k)$. We can equip the Hölder space $C^{k,\lambda}(M)$
with the norm
\begin{equation}\label{eq:EqHolNorm}
	\| u \|_{C^{k,\lambda}(M)} = \sum_{k\in K}
	\|(\kappa_k^{-1})^*\chi_k u \|_{C^{k,\lambda}(\RR^d)},
\end{equation}
turning it into a Banach space. Similarly, we can equip the space $C^k(M)$ with the Banach space norm
\begin{equation}\label{eq:EqHolNorm1}
	\| u \|_{C^{k}(M)} = \sum_{k\in K}
	\|(\kappa_k^{-1})^*\chi_k u \|_{C^{k}(\RR^d)}.
\end{equation}

\medskip
\noindent
Let $\mathrm{dist}$ denotes the geodesic distance on $M$ with respect to the metric $g$.
Notice that since $M$ is compact, the globally defined norm
\begin{equation}\label{eq:EqHolNorm2}
	\| u \|_{C^{0,\lambda}(M)} = \sup_{x\in M} |u(x)| +
  \sup_{\substack{x,y\in M \\ x\neq y}}
  \frac{| u(x) -  u(y)| }{\mathrm{dist}(x,y)^\lambda}
\end{equation}
is equivalent to \eqref{eq:EqHolNorm} when $k=0$.

\medskip
\noindent
In view of the definition of $W^{s,p}(M)$ and $C^{k,\lambda}(M)$, the Sobolev embedding \eqref{eq:MorreyInequality0} immediately yields that
for $s\in \NN$, $d<p<+\infty$ and $0< \lambda \leq 1 - d/p$, we have
\begin{equation}\label{eq:MorreyInequality2a}
  W^{s+1,p}(M)\subset C^{s,\lambda}(M).
\end{equation}
Moreover, in view of \eqref{eq:EqSobNorm} and \eqref{eq:EqHolNorm}, the inequality \eqref{eq:MorreyInequality} gives that here exists a
constant $C=C_{s,p,d,\lambda}>0$ such that for all $u\in W^{s,p}(M)$,
\begin{equation}
\label{eq:MorreyInequality2b}
\|u\|_{C^{s,\lambda}(M)} \leq C \|u\|_{W^{s+1,p}(M)},
\end{equation}
which shows that the embedding \eqref{eq:MorreyInequality2a} is continuous.

\medskip
\noindent
Finally, let us recall the Rellich-Kondrachov embedding theorem for the space of continuous functions $C(M)$ equipped
with the uniform norm $\|u\|_{C(M)} = \sup_M|u|$, see~\cite[Theorems 2.10, 2.20, 2.34]{Au82}. If $s>d/p$ with $s\in\NN$,
$d<p<\infty$, then the embedding
\begin{equation}
\label{eq:CompImbed}
W^{s,p}(M)\subset C(M).
\end{equation}
is compact.

\end{document}